\documentclass[12pt,reqno,b5paper]{amsart}
\usepackage{comment}
\usepackage{multicol}
\usepackage{a4wide}
\usepackage{amsmath, amssymb, amsthm}
\usepackage[utf8]{inputenc}
\usepackage[T1]{fontenc}
\usepackage[all]{xypic}
\usepackage{graphicx}
\usepackage{fancyhdr}
\usepackage{tabularx}
\usepackage{wrapfig}
\usepackage{xcolor}
\usepackage{array}
\usepackage{textcomp}
\usepackage{siunitx}
\usepackage{epsfig}
\usepackage{hyperref}
\usepackage{url}
\usepackage{tabularx}
\usepackage{tikz-cd} 
\usepackage{quiver}
\usepackage{todonotes}

\usepackage[metapost, mplabels,truebbox,clip]{mfpic}
\definecolor{darkgreen}{RGB}{204,0,78}

\newcommand{\darkgreen}[1]{\textcolor{darkgreen}{#1}}

\makeatletter
\def\slashedarrowfill@#1#2#3#4{%
  $\m@th\thickmuskip0mu\medmuskip\thickmuskip\thinmuskip\thickmuskip
   \relax#4#1\mkern-7mu\cleaders\hbox{$#1\mkern-2mu$}\hfill\mkern-7mu
   \mathclap{#3}\mkern-7mu\cleaders\hbox{$#1\mkern-2mu$}\hfill\mkern-7mu#1$}

\def\rightslashedarrowfill@{%
  \slashedarrowfill@\relbar\relbar\mapstochar\rightarrow}

\newcommand\xslashedrightarrow[2][]{%
  \ext@arrow 0055{\rightslashedarrowfill@}{#1}{#2}}
\makeatother

\newtheorem{lemma}{Lemma}[section]
\newtheorem{corr}[lemma]{Corollary}
\newtheorem{proposition}[lemma]{Proposition}
\newtheorem{theorem}[lemma]{Theorem}

\theoremstyle{definition}

\newtheorem{definition}[lemma]{Definition}
\newtheorem{remark}[lemma]{Remark}
\newtheorem{example}[lemma]{Example}

\newtheorem{construction}[lemma]{Construction}
\newtheorem{notation}[lemma]{Notation}

\author{Tom\'{a}\v{s} Perutka}
\address{University of Münster}
\email{tom.perutka@uni-muenster.de, tom.perutka@gmail.com}
\date{\today}

\begin{document}
\def\Ccat{\operatorname{2\textbf{-Cat}}}
\def\ICat{\operatorname{\textbf{ICat}}}
\def\cat{\operatorname{\textbf{cat}}}
\def\Prof{\operatorname{\textbf{Prof}}}
\def\SMC{\operatorname{\textbf{SMC}}}
\def\Prop{\operatorname{\textbf{Prop}}}
\def\Law{\operatorname{\textbf{Law}}}
\def\ALG{\operatorname{-\textbf{Alg}}}
\def\Cat{\operatorname{\textbf{Cat}}}
\def\End{\operatorname{\mathsf{End}}}
\def\Mod{\operatorname{\mathsf{Mod}}}
\def\Mon{\operatorname{\mathsf{Mon}}}
\def\Bimod{\operatorname{\mathsf{Bimod}}}
\def\Sketch{\operatorname{\mathbf{Sk}}}
\def\Prod{\operatorname{\textbf{Prod}}}
\def\IntAlg{\operatorname{{\mathsf{IntAlg}}}}
\def\IntCoalg{\operatorname{{\mathsf{IntCoalg}}}}
\def\IntBialg{\operatorname{{\mathsf{IntBialg}}}}
\def\ev{\operatorname{ev}}
\def\ladj{\operatorname{ladj}}
\def\radj{\operatorname{radj}}
\def\inert{\operatorname{in}}
\def\act{\operatorname{act}}
\def\el{\operatorname{el}}
\def\str{\operatorname{str}}
\def\coev{\operatorname{coev}}
\def\CMon{\operatorname{\mathsf{CMon}}}
\def\LaxFun{\operatorname{\mathsf{LaxFun}}}
\def\PsFun{\operatorname{\mathsf{PsFun}}}
\def\Fun{\operatorname{\mathsf{Fun}}}
\def\Mul{\operatorname{\mathsf{Mul}}}
\def\Set{\operatorname{\textbf{Set}}}
\def\Alg{\operatorname{\mathsf{Alg}}}
\def\Coalg{\operatorname{\mathsf{Coalg}}}
\def\Bialg{\operatorname{-Bialg}}
\def\Cc{\operatorname{\mathbb{C}}}
\def\Dd{\operatorname{\mathbb{D}}}
\def\Bb{\operatorname{\mathbb{B}}}
\def\A{\operatorname{\mathcal{A}}}
\def\B{\operatorname{\mathcal{B}}}
\def\C{\operatorname{\mathcal{C}}}
\def\Ss{\operatorname{\mathbb{S}}}
\def\D{\operatorname{\mathcal{D}}}
\def\E{\operatorname{\mathcal{E}}}
\def\V{\operatorname{\mathcal{V}}}
\def\W{\operatorname{\mathcal{W}}}
\def\P{\operatorname{\mathcal{P}}}
\def\O{\operatorname{\mathcal{O}}}
\def\mF{\operatorname{\mathcal{F}}}
\def\mM{\operatorname{\mathcal{M}}}
\def\mN{\operatorname{\mathcal{N}}}
\def\mZ{\operatorname{\mathcal{Z}}}
\def\U{\operatorname{\mathcal{U}}}
\def\X{\operatorname{\mathcal{X}}}
\def\Y{\operatorname{\mathcal{Y}}}
\def\K{\operatorname{\mathit{\widetilde{K}}}}
\def\G{\operatorname{\mathit{\widetilde{G}}}}
\def\Q{\operatorname{\mathbb{Q}}}
\def\Z{\operatorname{\mathbb{Z}}}
\def\N{\operatorname{\mathbb{N}}}
\def\M{\operatorname{\mathbb{V}}}
\def\F{\operatorname{\mathbb{F}}}
\def\T{\operatorname{\mathbb{T}}}
\def\FP{\operatorname{\mathsf{FP}}}
\def\Pr{\operatorname{\mathbf{Pr}}}
\def\Hom{\operatorname{\mathsf{Hom}}}
\def\Map{\operatorname{\mathsf{Map}}}
\def\Id{\operatorname{Id}}
\def\ho{\operatorname{ho}}
\def\Nat{\operatorname{\mathsf{Nat}}}
\def\lan{\operatorname{lan}}
\def\t{\operatorname{\times}}
\def\tens{\operatorname{\otimes}}
\def\id{\operatorname{id}}
\def\im{\operatorname{im}}
\def\cst{\operatorname{cst}}
\def\coim{\operatorname{coim}}
\def\coker{\operatorname{coker}}
\def\fup{\operatorname{\overline{\mathit{f}}}}
\def\flow{\operatorname{\underline{\mathit{f}}}}
\def\lax{\operatorname{lax}}
\def\Lur{\operatorname{Lur}}
\def\Tw{\operatorname{Tw}}
\def\comp{\operatorname{comp}}

\title{2-dimensional Lawvere theories, commutativity, and higher Day convolution}

\maketitle

\begin{abstract}
The aim of this paper is to study categorified algebraic structures and their pseudo- and lax homomorphisms using the framework of Lawvere $2$-theories, and more generally, marked $2$-dimensional sketches. The key notion we focus on is that of $2$-dimensional commutativity. Equipping a $2$-theory $\T$ with such a structure, we study the implications on the $2$-categories $\mathsf{Mod}_{\text{pseudo}}
(\T,\mathbf{Cat})$, $\mathsf{Mod}_{\lax}(\T,\mathbf{Cat})$ of $\mathbb{T}$-models, pseudo- or lax homomorphisms, and modifications, both in the ordinary and $\infty$-categorical setting. We show that these $2$-categories admit certain natural structures of a closed $2$-operad and deduce a generalization of Fox's theorem. The main technical tool is of independent interest: we construct an $(\infty,n-1)$-operad structure on the hom $(\infty,n-1)$-category $\mathsf{Hom}_{\mathbb{V}}(\mathcal{M},\mathcal{N})$, where $\mathbb{V}$ is a monoidal $(\infty,n)$-category and $\mathcal{M},\mathcal{N}$ are monoids therein, generalizing the Day convolution to the higher setting.
\end{abstract}

\setcounter{tocdepth}{1}
\tableofcontents

\newpage
\section*{Introduction}

Categorical algebra studies algebraic structures that can be put on categories: (symmetric) monoidal categories, double categories, rig categories, categories with an action, involutive categories etc. One of the reasons making categorical algebra more complex than just studying algebraic structures on sets is that the added categorical dimension allows us to study different kinds of homomorphisms between objects with given algebraic structure: \darkgreen{{strict, pseudo,}} and \darkgreen{{(co)lax}}. The weakened notions of lax and colax homomorphisms then allow us to define further algebraic structures \textit{inside a model}: for example, if $\V$ is a monoidal category, the category $\Fun_{\lax}^\otimes(*,\V)$ of lax monoidal functors from a point to $\V$ is exactly the category $\Mon(\V)$ of monoids in $\V$.

Our primary source of algebraic structures are \darkgreen{Lawvere $(\infty,2)$-theories}. For any Lawvere $(\infty,2)$-theory $\T$ and any $(\infty,2)$-category $\C$ with finite products, we construct an $(\infty,2)$-category \darkgreen{$\Mod_{w}(\T,\C)$} for any choice $w\in\{$strict, lax, colax$\}$. Similarly, if $\C$ is an ordinary $2$-category with finite products and $\T$ an ordinary Lawvere $2$-theory, we construct a $2$-category $\Mod_w(\T,\C)$ for any choice $w\in\{$strict, pseudo, lax, colax$\}$.

\begin{remark}
    Note that in the ordinary $2$-categorical situation there is a distinction between commuting up to isomorphism and commuting strictly which does not make sense in the higher categorical world, so developing both theories simultaneously poses some notational challenges. For the introduction, let us put these aside and speak simply about $2$-categories, distinguishing between the ordinary and higher case only when discussing the results. We also use $\Cat$ as a placeholder for both the ordinary and $\infty$-categories, unless specified otherwise.
\end{remark}

Although Lawvere 2-theories are our main interest, we also look at other algebraic structures and perform similar constructions: \darkgreen{marked $2$-sketches} of Arkor-Bourke-Ko \cite{ABK24} and \darkgreen{algebraic patterns} of Chu-Haugseng \cite{CH21}. The only example that is not a Lawvere 2-theory, which we study in depth, is the algebraic pattern / sketch $\Delta^{op}$ for Segal objects.

Studying the $2$-categories $\Mod_{\text{(co)lax}}(\T,\Cat)$ is more difficult than the study of their stricter versions $\Mod_{\text{strict}}(\T,\Cat)$ or $\Mod_{\text{pseudo}}(\T,\Cat)$ as the former two are in general not cocomplete in any suitable 2-categorical sense. Nevertheless, the $2$-categories $\Mod_{\text{(co)lax}}(\T,\Cat)$ are of primary interest to us: they allow us to define, for a model $\X\colon \T\to \Cat$, a category of \darkgreen{{internal algebras} }$$\darkgreen{\IntAlg(\X)}:=\Hom_{\lax}(*,\X),$$ as well as the category of \darkgreen{internal coalgebras} $\Hom_{\text{colax}}(*,\X)$. For example, if $\T$ is the $2$-theory for monoidal categories, we can identify $\IntAlg(\V)$, $\IntCoalg(\V)$ with the categories $\Mon(\V)$, $\mathsf{Comon}(\V)$ of monoids and comonoids in a monoidal category $\V$. Considering $\Delta^{op}$ as an algebraic pattern (or a marked $2$-sketch) for Segal objects, for $\Cc$ a model in $\Cat$, i.e. a double category, we get that $\IntAlg(\mathbb{C})$ is the category of horizontal monads in $\mathbb{C}$.  

 After defining the basic notions, we study more deeply the phenomena associated with \darkgreen{$2$-dimensional commutativity}. Recall that a Lawvere $1$-theory $\T$ is commutative if all the ``operations'' commute with each other, i.e. for any morphisms $\alpha\colon n\to 1$, $\beta\colon k\to 1$ in $\T$, the square below commutes.
%https://q.uiver.app/#q=WzAsNCxbMCwwLCJtXFxjZG90IGsiXSxbMSwwLCJuXFxjZG90IGsiXSxbMCwxLCJtXFxjZG90IGwiXSxbMSwxLCJuXFxjZG90IGwiXSxbMCwxLCJrXFxjZG90IFxcYWxwaGEiXSxbMiwzLCJsXFxjZG90IFxcYWxwaGEiXSxbMCwyLCJtXFxjZG90IFxcYmV0YSIsMl0sWzEsMywiblxcY2RvdCBcXGJldGEiXV0=
\begin{equation} \label{intro: comm diagram}
    \begin{tikzcd} 
	{k\cdot n} & {k} \\
	{ n} & {1}
	\arrow["{k\cdot \alpha}", from=1-1, to=1-2]
	\arrow["{\alpha}", from=2-1, to=2-2]
	\arrow["{\beta\cdot n}"', from=1-1, to=2-1]
	\arrow["{\beta}", from=1-2, to=2-2]
\end{tikzcd}
\end{equation}
This is equivalent to saying that for any model $\X\colon \T\to \C$, the $n$-ary operation $\X(\alpha)\colon$ $ \X(1)^n\to \X(1)$ gives rise to a homomorphism of models $\X^n\to \X$. For example, the theory for monoids is not commutative, whereas the theory for commutative monoids is. 

If we investigate the differences one categorical dimension higher, i.e. monoidal vs. symmetric monoidal categories, we can see that their behaviour differs more drastically: for example, the category $\CMon(\V)$ of commutative monoids in a symmetric monoidal category $\V$ inherits the symmetric monoidal structure from that on $\V$; we denote that by $\underline{\CMon}(\V)$. If $\W$ is an arbitrary monoidal category, the analogue for $\Mon(\W)$ is not true. As another example, one can form a tensor product of symmetric monoidal categories $\V\boxtimes \W$ such that the internal hom is the category of symmetric monoidal functors, i.e. $$\Fun^\otimes(\V_1\boxtimes \V_2,\W)\simeq \Fun^\otimes(\V_1,\Fun^{\otimes}(\V_2,\W)).$$ Something similar does not even make sense for arbitrary monoidal categories as $\Fun^{\otimes}(\V_2,\W)$ simply does not inherit the monoidal structure. This should hint to us that if a $2$-theory $\T$ is ``2-dimensionally commutative,'' the $2$-categories of models should inherit some extra structure.

Another observation in this direction is due to Fox \cite{Fox76}. Denote by $\SMC$ the category of symmetric monoidal categories and strong symmetric monoidal functors. 

{\textbf{Fox's theorem.}} \textit{The endofunctor} $\underline{\CMon}(-)\colon \SMC\to \SMC$ \textit{is an idempotent comonad. Moreover, coalgebras for this comonad are precisely cocartesian monoidal categories.}

\begin{remark}
    A well-known special case of Fox's theorem is the useful fact that the coproduct of two commutative rings is their tensor product.
\end{remark}

The main motivation for this paper is to explain all this surprisingly nice behaviour of symmetric monoidal categories by the fact that the Lawvere $(2,1)$-theory $\T_{E_{\infty}}$ for symmetric monoidal categories is commutative in the suitable 2-dimensional sense, as well as to prove the analogue of the pleasant behaviour of $\CMon(-)$ for the more general functor $\IntAlg(-)$. 

\section*{Main results} 
\subsection{2-dimensional commutativity}
The first goal of this paper is to define the $2$-dimensional analogue of commutativity. In the $1$-dimensional situation, commutativity is a property of a Lawvere theory: we simply ask whether the diagrams (\ref{intro: comm diagram}) commute or not. In the $2$-dimensional setting, this becomes a structure: we want to fill the squares (\ref{intro: comm diagram}) with $2$-cells, possibly even non-invertible, and these cells ought to satisfy some coherence equations. An elegant solution to package all this structure is to consider $\T$ as a \darkgreen{marked $2$-category} (by marking the product projections $n\to 1$ which we call the \darkgreen{inert $1$-cells}) and then make use of a Gray tensor product on marked $2$-categories.

\textbf{Definition.} (Def. \ref{def: w-commutativity}) For a Lawvere $2$-theory\footnote{We also define $w$-commutativity for marked 2-sketches (Def. \ref{def: w-commutativity II}) and for algebraic patterns (Def. \ref{def: w-commutativity III}).} $\T$ and $w\in \{$strict, pseudo, lax, colax$\}$, a \darkgreen{$w$-commutativity} on $\T$ is an $E_1$-monoidal structure on $\T$ such that the multiplication map $\mu\colon \T\otimes_{s,w} \T\to \T$ preserves products in each variable and the monoidal unit is $1\in \T$. Here, $\otimes_{s,w}$ is a marked Gray tensor product of marked $2$-categories.

It turns out that the $w$-commutativity structure is quite rigid. Denoting again the multiplication map by $\mu$, we prove that necessarily $\mu(m,n)=m\cdot n$. Thus, filling the squares (\ref{intro: comm diagram}) with $2$-cells amounts to specifying $\mu$ on the Gray cells (\ref{dia: gray square}). By definition and the universal property of $\otimes_{s,w}$ we obtain that a $w$-commutativity on $\T$ implies the existence of a ``lifting functor'' \begin{align} \label{intro: lifting}
    \darkgreen{\widetilde{(-)}}\colon \Mod_w(\T,\C)\to \Mod_w(\T,\Mod_w(\T,\C)) 
\end{align} which is a section to both forgetful functors evaluating at 1 in the first or the second variable. Intuitively, this functor promotes each operation $\X(\alpha)\colon \X(1)^n\to \X(1)$ to a homomorphism of models $\widetilde{\X}(\alpha)\colon \X^n\to \X$.

Quite surprisingly, if $\T$ is a Lawvere $(2,2)$-theory, we can show that filling in the diagrams (\ref{intro: comm diagram}) in a coherent way already specifies completely the map $\mu\colon \T\otimes_{s,w}\T\to \T$ and it even implies that it is associative up to a coherent isomorphism! In other words, the coherences for the Gray tensor product imply the coherences for $\mu$.

\textbf{Theorem A}. (\darkgreen{miracle associativity}) (Lemma \ref{lemma: syntactic} + Thm. \ref{prop: miracle associativity}) \textit{Let $\T$ be a Lawvere $(2,2)$-theory. The following sets are in bijection:
\begin{enumerate}
    \item $w$-commutativity structures on $\T$,
    \item (syntactic criterion) coherent systems $\sigma_{\alpha\beta}$ of $w$-cells filling the diagrams below.
    %\item (semantic criterion) sections $\Mod_w(\T,\C)\to \Mod_w(\T,\Mod_w(\T,\C))$ of the forgetful functor\footnote{In fact, there are two forgetful functors (evaluation at $1$ in the first or the second variable), and this is a section to both of them.}, functorial in $\C$. 
\end{enumerate}}
% https://q.uiver.app/#q=WzAsNCxbMCwwLCJtXFxjZG90IGsiXSxbMCwxLCIgayJdLFsxLDAsIm0iXSxbMSwxLCIxIl0sWzAsMSwiXFxhbHBoYVxcY2RvdCBrIiwyXSxbMCwyLCJtXFxjZG90IFxcYmV0YSJdLFsyLDMsIlxcYWxwaGEiXSxbMSwzLCJcXGJldGEiLDJdLFs1LDcsIlxcc2lnbWFfe1xcYWxwaGFcXGJldGF9IiwwLHsib2Zmc2V0IjoxLCJzaG9ydGVuIjp7InNvdXJjZSI6NDAsInRhcmdldCI6NDB9fV1d
\[\begin{tikzcd}
	{m\cdot k} & m \\
	{ k} & 1
	\arrow[""{name=0, anchor=center, inner sep=0}, "{m\cdot \beta}", from=1-1, to=1-2]
	\arrow["{\alpha\cdot k}"', from=1-1, to=2-1]
	\arrow["\alpha", from=1-2, to=2-2]
	\arrow[""{name=1, anchor=center, inner sep=0}, "\beta"', from=2-1, to=2-2]
	\arrow["{\sigma_{\alpha\beta}}", shift right, between={0.4}{0.6}, Rightarrow, from=0, to=1]
\end{tikzcd}\]

We are going to formulate the result above more precisely in Section \ref{sec: commutativity}. There is an interesting special case: we can consider a braiding on a monoidal category $\V$ as a pseudocommutativity structure on the delooping $B\V$. The miracle associativity is then witnessed by the Yang-Baxter equation (see Example \ref{example: Yang-Baxter}). 

The previous result concerns syntax of $w$-commutativity, i.e. some implications for the Lawvere $2$-theories themselves. Let us now turn to the semantics of $w$-commutativity: what does such a structure on $\T$ tell us about the $2$-category $\Mod_w(\T,\Cat)$? First, we can observe that the lifting functor (\ref{intro: lifting}) equips the category $\IntAlg(\X)$ with a structure of a model $\darkgreen{\underline{\IntAlg}(\X)}\colon \T\to \Cat$. With a little more work, we can show that it promotes to an endofunctor $$\underline{\IntAlg}:=\underline{\Hom}_{\lax}(*,-)\colon \Mod_{w}(\T,\Cat)\to \Mod_{w}(\T,\Cat).$$ 

In fact, we can show that much more is true.

\textbf{Theorem B.} (\darkgreen{closed 2-operad structure}) (Thm. \ref{thm: closed structure}, Prop. \ref{prop: it exists, after all}, Cor. \ref{corr: closed operad}) \textit{A $w$-commutativity on $\T$ equips} $\Mod_w(\T,\Cat)$ \textit{with a closed $2$-operad / $2$-multicategory\footnote{This is one of many examples of clashing conventions between $\infty$-categorical and classical terminology. We usually distinguish between $(\infty,2)$-operads and $(2,2)$-multicategories throughout the paper.} structure. Moreover, if we denote the internal hom by $\underline{\Hom}_w(-,-)$, we get $\underline{\Hom}_{\lax}(*,\X)=\underline{\IntAlg}(\X)$. } 

\begin{remark}
    Let us note that the proof of the theorem above is much more intricate for $w\in \{$lax, colax$\}$ than for $w\in\{$strict, pseudo$\}$. In particular, only in the latter case, we have a proof working equally well for $(\infty,2)$-categories and $(2,2)$-categories. In case $w\in \{$lax, colax$\}$, we are currently only able to construct the 2-multicategory structure for $(2,2)$-categories, although we expect it to generalize as well. Even proving the closedness is more difficult in the lax setting and requires some theory of lax ends we develop in Appendix A.
\end{remark}

\begin{remark}
    Theorem B was proved in some special cases before. In the setting of commutative monads, the analogue was proved by Keigher \cite{Kei78}, and in the $2$-dimensional setting of pseudocommutative monads, this was developed by Hyland and Power \cite{HP02}. For Lawvere $(\infty,1)$-theories, the analogue was proved by Berman \cite{Ber20} and the special case of symmetric monoidal $\infty$-categories goes back to Gepner, Groth, and Nikolaus \cite{GGN16}, although these results only describe the $1$-dimensional structure. For $w\in\{$lax, colax$\}$ and the special case of symmetric monoidal categories, this was worked out in detail by Schmitt \cite{Sch07}, but there is no more general treatment of the case $w=$ lax that we are aware of. Let us also note that in the references above, it is usually proved that one has a monoidal structure, not only an operad. We expect this to generalize for $w\in\{$strict, pseudo$\}$ but not for the lax and colax case. Nevertheless, we almost never need this in our applications.
\end{remark}

\subsection{Higher Day convolution}
To prove Theorem B above, we make use of the fact that if $\T$ has a pseudocommutativity structure $\mu$, then $\Mod_p(\T,\Cat)$ is a sub-2-category of a hom 2-category $\Fun_{s,p}(\T,\Cat)$ in a monoidal 3-category $\Cat^{\mathfrak{m}}_2$ of marked categories, and as such, it inherits certain Day convolution structure. We prove much more general result. If $\O$ is an $\infty$-operad and $\V$ an $\O$-monoid in the cartesian $\infty$-category $\Cat_{(\infty,n)}$ of $(\infty,n)$-categories, the associated unstraightening $p\colon \V^\otimes\to \O$ is an example of what we call an \darkgreen{$(\infty,n)$-operad over $\O$} (Def. \ref{def: oo,n operad over O}). Our generalized Day convolution is a structure put on the hom $(\infty,n-1)$-category between two $\O$-algebras in $\V$.

\textbf{Theorem C.} (\darkgreen{higher Day convolution}) (Thm. \ref{lemma: higher day conv}) \textit{ Let $\O$ be an $\infty$-operad, $\V$ and $\O$-monoid in }$\Cat_{(\infty,n)}$,\textit{ and $M_0,M_1\colon \O \to \V$ two $\O$-algebras in $\V$. Then there exists an $(\infty,n-1)$-operad $p\colon\Hom_{\V}^{\lax}(M_0,M_1)^\otimes\to \O$. Moreover, if $\O=\Delta^{op}$ or $\mathsf{Fin}_*$ is an operad for $E_1$- or $E_\infty$- algebras, the fiber $p^{-1}(1)$ is the hom $(\infty,n-1)$-category $\Hom_{\V}^{\lax}(M_0(1),M_1(1))$, and if }$\V=\Cat$, \textit{we recover the usual Day convolution.}

For example, if $(\V,\otimes)$ is a monoidal $(\infty,n)$-category, $\mM_0$, $\mM_1$ two monoids in $\V$, and $X_1,\dots,X_n,Y\colon \mM_0\to\mM_1$ a collection of one-cells, the a $n$-ary morphism $f\colon X_1,\dots,X_n\to Y$ in $\Hom_{\V}^{\lax}(\mM_0,\mM_1)$ looks as follows (where $m^n_0$, $m^n_1$ stand for the iterated multiplication):
% https://q.uiver.app/#q=WzAsNCxbMCwwLCJcXG1NXzBee1xcb3RpbWVzIG59Il0sWzEsMCwiXFxtTV8xXntcXG90aW1lcyBufSJdLFswLDEsIlxcbU1fMCJdLFsxLDEsIlxcbU1fMSJdLFsyLDMsIlkiLDJdLFswLDEsIlxcb3RpbWVzX2kgWF9pIl0sWzAsMiwibV8wXm4iLDJdLFsxLDMsIm1fMV5uIl0sWzUsNCwiZiIsMCx7InNob3J0ZW4iOnsic291cmNlIjozMCwidGFyZ2V0IjozMH19XV0=
\[\begin{tikzcd}
	{\mM_0^{\otimes n}} & {\mM_1^{\otimes n}} \\
	{\mM_0} & {\mM_1}
	\arrow[""{name=0, anchor=center, inner sep=0}, "{\otimes_i X_i}", from=1-1, to=1-2]
	\arrow["{m_0^n}"', from=1-1, to=2-1]
	\arrow["{m_1^n}", from=1-2, to=2-2]
	\arrow[""{name=1, anchor=center, inner sep=0}, "Y"', from=2-1, to=2-2]
	\arrow["f", between={0.3}{0.7}, Rightarrow, from=0, to=1]
\end{tikzcd}\]

\subsection{Other results}
The foundational results mentioned so far have some further applications in categorical algebra. First and foremost, we are now able to formulate and prove a generalization of Fox's theorem.

\textbf{Theorem D} (\darkgreen{Fox's theorem}) (Thm. \ref{thm: fox}, Prop. \ref{prop: restricting fox})  \textit{Let $\T$ be a Lawvere $2$-theory with a $w$-commutativity. Then the endofunctor} $$\underline{\IntAlg}\colon\Mod_w(\T,\Cat)\to \Mod_w(\T,\Cat)$$ \textit{has a structure of a $2$-comonad. Moreover, if $\T$ is \darkgreen{$w$-idempotent}, (see Def. \ref{def: Eckmann-Hilton}), this $2$-comonad is idempotent.}

\begin{remark}
    In the form written above, we are able to prove the theorem only for $(2,2)$-categories. For $(\infty,2)$-categories, we need some extra assumptions (cf. \ref{prop: restricting fox}) which should probably always hold.
\end{remark}

An example of a pseudoidempotent $2$-theory is the theory for symmetric monoidal categories. There is a more general family of examples, coming from unital commutative Lawvere $1$-theories. Then follows from a certain version of the Eckmann-Hilton argument that is not particularly difficult to prove but we have not found the full result in the literature.

\textbf{Proposition E.} (\darkgreen{Eckmann-Hilton for theories}) (Prop. \ref{subsec: 1-dim Kronecker Eckmann-Hilton}) \textit{Let $\T$ be a commutative Lawvere $1$-theory generated (in a suitable sense) by units and unital operations. Then the Kronecker product $\T\otimes_{\F^{op}}\T$ is isomorphic to $\T$.}

For any Lawvere $2$-theory $\T$, without any commutativity structure, we also construct certain generalized convolution structure. It is well-known that a set of maps from a comonoid $C$ to a monoid $M$ forms again a monoid, with a multiplication called convolution product: for any $f,g\colon C\to M$, we set $f\star g:= \mu\circ (f\otimes g)\circ \nabla$ where $\mu,$ $\nabla$ denote the multiplication and comultiplication. We generalize that to internal algebras and coalgebras.

\textbf{Proposition D.} (\darkgreen{convolution algebras}) (Prop. \ref{prop: conv product}) \textit{Let $\T$ be a Lawvere $(2,2)$-theory equipped with a suitable map $\T\to \T_{E_\infty}$ to the $2$-theory for symmetric monoidal categories. Consider a model $\X\colon \T\to \Cat$. Then for any $A\in \IntAlg(\X)$, $C\in \IntCoalg(\X)$ with underlying objects $a,c\in\X(1)$, respectively, the hom-set $\X(1)(c,a)$ naturally inherits a structure of an internal algebra in} $\Set$.

In the formulation above, $\Cat$ really means the category of ordinary categories, not $\infty$-categories. The map $\T\to \T_{E_\infty}$ above ensures that $\Set$ with its cartesian monoidal structure is a $\T$-model. For more details, see \ref{subsec: convolution}.

The last phenomenon we study is the interchange of internal algebras and coalgebras. In \ref{subsec: closed structure I}, we explain how to use the closed structure $\underline{\Hom}_w(-,-)$ to reprove some classical results about interchanging algebra and coalgebra structures in a braided monoidal $ifnty$-category. Related to that is the notion of a bilax monoidal functor: this is a pair of a lax and a colax monoidal structure on a functor, compatible with each other. We explain how this arises from pseudocommutativity (or more generally, a pseudocommutation). We study this for models in ordinary $2$-categories, utilizing a result of Arkor, Bourke, and Ko \cite[Theorem 7.4]{ABK24}, which says that we have an isomorphism of marked $2$-sketches $\Mod_{\lax}(\T,\Mod_{\mathrm{colax}}(\T,\C))\cong \Mod_{\mathrm{colax}}(\T,\Mod_{\lax}(\T,\C))$.

\textbf{Definition / Proposition E.} (\darkgreen{bilax structures}) (Lemma \ref{lemma: bilax} + Def. \ref{def:bilax})  \textit{ Let $\T$ be a pseudocommutative Lawvere $2$-theory, $\X,\Y\colon \T\to \C$ models, and $f\colon \X(1)\to \Y(1)$ a $1$-cell. A bilax structure on $f$ consists of a pair $(\fup,\flow)$ of a lax and a colax homomorphisms $\X\to \Y$ with $\fup(1)=f=\flow(1)$ such that they both come from a $1$-cell $\underline{\fup}$ in $\Mod_{\lax}(\T,\Mod_{\mathrm{colax}}(\T,\C))$.}

\section*{Outlook}\noindent
In the follow-up work, we would like to construct an $(\infty,2)$-operad structure on the $(\infty,2)$-category $\Mod_{\lax}(\T,\Cat_\infty)$ which turned out to be beyond the scope of the current paper: instead of just monoidal $(\infty,n)$ categories, one needs to use Gray categories in the process. The procedure is hinted in Remark \ref{remark: Gray structure}. We also plan to construct the $(\infty,2)$-category of marked $(\infty,2)$-sketches as an umbrella for all the choices of syntactic categories we have mentioned throughout the paper. 

\section*{Conventions and notation}
The main notational challenge of this paper is to try to keep the treatment of ordinary and higher $2$-categories uniform, at least to some extent. This is difficult mostly for two reasons:
\begin{enumerate}
    \item In ordinary $2$-categories, there are certain extra levels of strictnes (i.e., strict vs. pseudo) which higher categories do not see.
    \item Some of the conventions and terminology became very different for ordinary and higher categories; the relevant example for us is the dichotomy between the terms ``$\infty$-operad'' and ``multicategory''.
\end{enumerate}

As a solution, we use the term ``2-category'' whenever we can treat the ordinary and higher $2$-categories uniformly, and otherwise we distinguish between $(2,2)$- and $(\infty,2)$-categories. This mostly works well, except one situation, which is the clashing terminology between ``strict'' $\infty$-structures and ``pseudo'' ordinary structures. For example, let $\T_{E_\infty}$ be the Lawvere $(2,1)$-theory for symmetric monoidal categories and $\Cat$ the $(2,2)$-category of ordinary categories, then $\Mod_s(\T,\Cat)$ has different meanings if view $\T$, $\Cat$ as $(2,2)$-categories or $(\infty,2)$-categories.

Further conventions:
\begin{itemize}
    \item By $\infty$-category, we mean $(\infty,1)$-category, and we use the notation $\Cat_{\infty}$ for the (very large) $(\infty,2)$-category of $\infty$-categories, functors, and natural transformations. We also use just $\Cat$ when it is apparent from the context that we talk about higher categories, or in situations where we talk simultaneously about both ordinary and higher categories.
    \item By functors between $2$-categories, we always mean $2$-functors.
    \item To adapt to the standard terminology, we use the terms $(2,2)$-multicategory and $(\infty,2)$-operad to refer to the same concept.
    \item We write $\F$ for a skeleton of the category of finite sets and functions. We identify objects of this category with natural numbers.
    \item We write $\mathcal{S}$ for the $\infty$-category of $\infty$-groupoids (also spaces or animae). If $\C$ is an $\infty$-category, we use the term ``mapping space'' instead of ''mapping $\infty$-groupoid''.
    \item If $\C$ is a $(2,2)$-category with finite products, by \textit{monoids} or \textit{symmetric monoids} in $\C$ we always mean $E_1$- or $E_{\infty}$-monoids in the associated $(\infty,2)$-category (with all higher cells identities). In particular, by monoids in $\Cat$ we mean monoidal categories, not strict monoidal categories. The usual term in ordinary categorical literature would be ``pseudomonoid'', but we refrain from this term to achieve unity. 
\end{itemize}

\section*{Acknowledgments}
I thank John Bourke and Nathanael Arkor for many valuable discussions about the contents of this paper, as well as for reading carefully the initial draft. I am also deeply grateful to John for suggesting Fox's theorem to me as a topic of study during my Master's in Brno and for our many discussion back then.  Further, I thank Maxime Ramzi, Thorger Geiß and Phil Pützstück for discussion regarding the $\infty$-categorical setting, and to Miloslav Štěpán and Vít Jelínek for useful discussions about $2$-dimensional algebra. During my stay in Brno, I was supported by John Bourke's MASH Junior Project -- \textit{The language of higher dimensional categories}. During my stay in Münster, I was supported by the Deutsche Forschungsgemeinschaft (DFG, German Research Foundation) under Germany's Excellence Strategy EXC 2044-390685587, Mathematics Münster: Dynamics–Geometry–Structure, by the CRC 1442 \textit{Geometry: Deformations and Rigidity of the DFG}, and by the Leibniz prize of Eva Viehmann.

\section{Prelude: 1-dimensional case}

In this section, we are going to recall the main results on commutative Lawvere $1$-theories, first introduced by Linton \cite{Lin66}. We are going to use this simple setting to develop an intuition for the 2-dimensional case.% give a detailed treatment and include all the proofs for two reasons: first, it is hard to find some of these results in the literature; second, we will approach the $2$-dimensional setting in a similar way so it is worth thinking first about the classical case properly.

\subsection{Lawvere 1-theories} Denote by $\F$ a skeleton of the category of finite sets. Then $\F^{op}$ is a free category with finite powers (or products) generated by one object. More precisely, if $\C$ is a category with finite powers and we denote by $\FP(\F^{op},\C)\subset \Fun(\F^{op},\C)$ the full subcategory spanned by functors preserving finite powers, the evaluation-at-1 functor $$\ev_1\colon \FP(\F^{op},\C)\to \C$$ is an equivalence of categories. Notice that the multiplication of natural numbers becomes a coproduct in $\F^{op}$.

\begin{definition} \cite{Law63}
    A \textit{Lawvere (1-)theory} is a pair $(\T,\theta)$ where $\T$ is a category and $\theta\colon\F^{op}\to\T$ is an identity-on-objects functor that preserves finite powers. We define the category $\Law$ as a full subcategory of $\Cat$ where objects are Lawvere theories and morphisms are functors under $\F^{op}$. For a Lawvere theory $\T$, we define the category of \textit{$\T$-models in $\C$}, denoted by $\Mod(\T,\C)$, to be the category $\FP(\T,\C)$.
\end{definition}

Note that $\Mod(\T,\C)$ is equipped with the forgetful functor  $$U\colon\FP(\T,\C)\xrightarrow{\theta^*}\FP(\F^{op},\C)\xrightarrow{\ev_1}\C.$$

\begin{remark}
    (Variants of $\Mod(\T,\C)$.) There is a couple of alternative ways of defining $\Mod(\T,\C)$, all yielding the same category up to equivalence. First, one can take $\C$ to be a category with finite \textit{products} instead of powers and define $\Mod(\T,\C):=\Fun^{\times}(\T,\C)$ to be the category of functors preserving finite products. As finite products and finite powers in $\T$ coincide, the resulting category is the same. Another useful variant introduces an extra strictness condition: we can equip a category $\C$ (admitting finite powers) with an explicit choice of an inverse equivalence $J\colon \C\to \FP(\F^{op},\C)$ of $\ev_1$. This in turn yields a \textit{choice of finite powers} $c^n:= J(c)(n).$ Then, we can define $\Mod^{s}(\T,\C)$ as a pullback % https://q.uiver.app/#q=WzAsNCxbMCwxLCJcXEMiXSxbMSwxLCJcXEZQKFxcRl57b3B9LFxcQykiXSxbMSwwLCJcXEZQKFxcVCxcXEMpIl0sWzAsMCwiXFxNb2QoXFxULFxcQykiXSxbMCwxLCJcXHNpbWVxIl0sWzIsMV0sWzMsMl0sWzMsMCwiVSJdLFszLDEsIiIsMSx7InN0eWxlIjp7Im5hbWUiOiJjb3JuZXIifX1dXQ==
\[\begin{tikzcd}
	{\Mod(\T,\C)} & {\FP(\T,\C)} \\
	\C & {\FP(\F^{op},\C)}
	\arrow["J", from=2-1, to=2-2]
	\arrow["\theta^*", from=1-2, to=2-2]
	\arrow[from=1-1, to=1-2]
	\arrow["U", from=1-1, to=2-1]
	\arrow["\lrcorner"{anchor=center, pos=0.125}, draw=none, from=1-1, to=2-2]
\end{tikzcd}\]
As we are pulling back along an equivalence, $\Mod^s(\T,\C)\simeq\Mod(\T,\C)$. The difference between the two is the following subtlety: suppose that $\C$ has finite products and $\X,\Y$ are models of $\T$ in $\C$. Denote $X:=\X(1)$, $Y:=\Y(1)$. Then the product $\X\times \Y$ is a model as well. In $\FP(\T,\C)$, the product of two functors is pointwise and so for a morphism $\alpha\colon n\to 1$ in $\T$, we get an operation $X^n\times Y^n\to X\times Y$, whereas in $\Mod^s(\T,\C)$, we get an operation $(X\times Y)^n\to X\times Y$, which is slightly more favorable. None of these distinctions will really matter; we take $\C$ to be locally presentable, so we indeed get $\Fun^{\times}(\T,\C)\simeq \FP(\T,\C)\simeq \Mod^s(\T,\C)$. 
\end{remark}

The category $\Mod(\T,\C)$ has finite powers again and the forgetful functor $U\colon \Mod(\T,\C)$ preserves them. If we equip $\C$ with a choice of finite powers, we obtain one on $\Mod(\T,\C)$as well. For $\X\colon \T \to \C$, $X:=\X(1)$, we easily get the following description:

\begin{itemize}
    \item $\X^k(m)=(X^k)^m,$
    \item for any $\alpha\colon m\to n$, we define $\X^k(\alpha)$ is a composition of $\X(\alpha)^k$ with the canonical isomorphisms % https://q.uiver.app/#q=WzAsNCxbMCwwLCIoWF5rKV5tIl0sWzAsMSwiKFhebSleayJdLFsxLDEsIihYXm4pXmsiXSxbMSwwLCIoWF5rKV5uIl0sWzAsMywiXFxYXmsoXFxhbHBoYSkiXSxbMSwyLCJcXFgoXFxhbHBoYSleayIsMl0sWzAsMSwiXFxjb25nIiwyXSxbMiwzLCJcXGNvbmciLDJdXQ==
\[\begin{tikzcd}
	{(X^k)^m} & {(X^k)^n} \\
	{(X^m)^k} & {(X^n)^k}
	\arrow["{\X^k(\alpha)}", from=1-1, to=1-2]
	\arrow["\cong"', from=1-1, to=2-1]
	\arrow["{\X(\alpha)^k}"', from=2-1, to=2-2]
	\arrow["\cong"', from=2-2, to=1-2]
\end{tikzcd}\]
\end{itemize}

Let us also describe morphisms in $\Mod(\T,\C)$ explicitly. For two models $\X,\Y\colon \T\to \C$ $X=\X(1)$, $Y=\Y(1)$, a homomorphism of models $f\colon \X\to \Y$ is by definition a system of maps $f_n\colon X^n\to Y^n$ such that for each $\beta\colon k\to l$ in $\T$, the square below commutes.

% https://q.uiver.app/#q=WzAsNCxbMCwwLCJYXmsiXSxbMCwxLCJYXmwiXSxbMSwxLCJZXmwiXSxbMSwwLCJZXmsiXSxbMCwxLCJcXFgoXFxiZXRhKSIsMl0sWzMsMiwiXFxZKFxcYmV0YSkiXSxbMSwyLCJmXmwiLDJdLFswLDMsImZeayJdXQ==
\[\begin{tikzcd}
	{X^k} & {Y^k} \\
	{X^l} & {Y^l}
	\arrow["{f_k}", from=1-1, to=1-2]
	\arrow["{\X(\beta)}"', from=1-1, to=2-1]
	\arrow["{\Y(\beta)}", from=1-2, to=2-2]
	\arrow["{f_l}"', from=2-1, to=2-2]
\end{tikzcd}\]

In particular, taking $\beta\colon n\to 1$ to be any map in the image of $\theta\colon \F^{op}\to \C$, we get that $\X(\beta)$, $\Y(\beta)$ are projection maps and we obtain that $f_n=f_1^n$. This leads to the observation below.

\begin{remark} \label{rmk: U faithful}
    The forgetful functor $U\colon\Mod(\T,\C)\to \C$ is faithful. As we have seen above, a homomorphism of models $f\colon \X\to \Y$ is a map $f_1\colon \X(1)\to \Y(1)$ together with equations $f_1^n \X(\alpha)=\Y(\alpha)f_1^m$ for any map $\alpha\colon m\to n$ in $\T$. In other words, having two homomorphisms $f,g\colon \X\to \Y$, $\ev_1(f)=\ev_1(g)$ implies $f=g.$
 \end{remark}

\subsection{Commutativity for Lawvere theories}\cite{Lin66} We say that a Lawvere theory $\T$ is \textit{commutative} if for any two maps $\alpha\colon m\to n$, $\beta\colon k\to l$ in $\T$, the following commutes:

% https://q.uiver.app/#q=WzAsOCxbMSwwLCJrXFxjZG90IG0iXSxbMCwxLCJtXFxjZG90IGsiXSxbMiwwLCJrXFxjZG90IG4iXSxbMywxLCJuXFxjZG90IGsiXSxbMywyLCJuXFxjZG90IGwiXSxbMiwzLCJsXFxjZG90IG4iXSxbMSwzLCJsXFxjZG90IG0iXSxbMCwyLCJtXFxjZG90IGwiXSxbMCwxLCJcXGNvbmciLDJdLFswLDIsImtcXGNkb3QgXFxhbHBoYSJdLFs2LDUsImxcXGNkb3QgXFxhbHBoYSJdLFsyLDMsIlxcY29uZyJdLFsxLDcsIm1cXGNkb3RcXGJldGEiLDJdLFs3LDYsIlxcY29uZyIsMl0sWzMsNCwiblxcY2RvdCBcXGJldGEiXSxbNCw1LCJcXGNvbmciXV0=
\[\begin{tikzcd}
	& {k\cdot m} & {k\cdot n} \\
	{m\cdot k} &&& {n\cdot k} \\
	{m\cdot l} &&& {n\cdot l} \\
	& {l\cdot m} & {l\cdot n}
	\arrow["{k\cdot \alpha}", from=1-2, to=1-3]
	\arrow["\cong"', from=1-2, to=2-1]
	\arrow["\cong", from=1-3, to=2-4]
	\arrow["{m\cdot\beta}"', from=2-1, to=3-1]
	\arrow["{n\cdot \beta}", from=2-4, to=3-4]
	\arrow["\cong"', from=3-1, to=4-2]
	\arrow["\cong", from=3-4, to=4-3]
	\arrow["{l\cdot \alpha}", from=4-2, to=4-3]
\end{tikzcd}\]

Here, the unnamed isomorphisms are obtained from coproduct symmetries in $\F^{op}$ by applying $\theta\colon \F^{op}\to \T$. The notation $k\cdot \alpha$ simply refers to a $k$-th power of the map $\alpha$. It will be convenient to define $\beta\cdot m$ to be the following composition:

% https://q.uiver.app/#q=WzAsNCxbMCwwLCIga1xcY2RvdCBtIl0sWzAsMSwibVxcY2RvdCBrIl0sWzEsMSwibVxcY2RvdCBsIl0sWzEsMCwibFxcY2RvdCBtIl0sWzAsMSwiXFxjb25nIiwyXSxbMSwyLCJtXFxjZG90IFxcYmV0YSJdLFsyLDMsIlxcY29uZyIsMl0sWzAsMywiXFxiZXRhXFxjZG90IG0iXV0=
\[\begin{tikzcd}
	{ k\cdot m} & {l\cdot m} \\
	{m\cdot k} & {m\cdot l}
	\arrow["{\beta\cdot m}", from=1-1, to=1-2]
	\arrow["\cong"', from=1-1, to=2-1]
	\arrow["{m\cdot \beta}", from=2-1, to=2-2]
	\arrow["\cong"', from=2-2, to=1-2]
\end{tikzcd}\]

Then, we can rewrite the hexagon above as a square which we require to commute.

\begin{center}
\begin{equation} \label{diagram: comm}
%https://q.uiver.app/#q=WzAsNCxbMCwwLCJtXFxjZG90IGsiXSxbMSwwLCJuXFxjZG90IGsiXSxbMCwxLCJtXFxjZG90IGwiXSxbMSwxLCJuXFxjZG90IGwiXSxbMCwxLCJrXFxjZG90IFxcYWxwaGEiXSxbMiwzLCJsXFxjZG90IFxcYWxwaGEiXSxbMCwyLCJtXFxjZG90IFxcYmV0YSIsMl0sWzEsMywiblxcY2RvdCBcXGJldGEiXV0=
\begin{tikzcd} 
	{k\cdot m} & {k\cdot n} \\
	{l\cdot m} & {l\cdot n}
	\arrow["{k\cdot \alpha}", from=1-1, to=1-2]
	\arrow["{l\cdot \alpha}", from=2-1, to=2-2]
	\arrow["{\beta\cdot m}"', from=1-1, to=2-1]
	\arrow["{\beta\cdot n}", from=1-2, to=2-2]
\end{tikzcd}
\end{equation}
\end{center}

%We will sometimes refer to the composite map $k\cdot m\to l\cdot n$ as $\alpha\otimes \beta$. 
Note that since $\T(m,n)\cong \T(m,1)^n$, it is enough to consider maps $\alpha,\beta$ with $n=l=1$. 

One can interpret diagram (\ref{diagram: comm}) using a composition rule usually used in the context of operads:

\begin{definition}
    Let $\T$ be a Lawvere theory and $\{\beta_i\colon m_i\to 1\}_{1\le i\le n}$ some collection of maps in $\T$. For another map $\alpha\colon n\to 1$, define the \textit{operadic composition} $$\alpha(\beta_1,\dots,\beta_n):=\alpha \circ (\beta_1\times\cdots\times \beta_n).$$
\end{definition}

It makes sense to consider this kind of composition in the context of Lawvere theories as we usually want to consider only maps with target $1$. Then, the diagram (\ref{diagram: comm}) says that for any maps $\alpha\colon m\to 1$, $\beta\colon k\to 1$, we have $$\alpha(\beta,\dots,\beta)=\beta(\alpha,\dots,\alpha).$$ 

To check whether all the squares of the form (\ref{diagram: comm}) commute seems like a very tedious task. In fact, the number of squares we actually need to examine is not that large comparing to the number of maps we need to specify when defining our theory. To this end, we define a basis of a Lawvere theory.

\begin{definition} \label{remark: active inert} \begin{enumerate} 
    \item A map $\beta$ in $\T$ is called \textit{inert}\footnote{This is inspired by Lurie's terminology for operads, which is in fact closely related.} if it is in the image of $\theta\colon \F^{op}\to \T$.
    %\item A map $\alpha\colon n\to 1$ in $\T$ is called active if it is not inert.
    \item For a collection of maps $M$ in $\T$, we define $\langle M\rangle$ to be the smallest class of maps in $\T$ with target $1$ such that it contains isomorphisms, $M$, and is closed under operadic compositions.
    \item A collection of non-inert maps $B=\{\alpha_i\colon n_i\to 1\}_{i\in I}$ is called a \textit{basis} if any map $\gamma\colon m\to 1$ can be factored as $\gamma=\alpha\beta$ where $\beta$ is inert and $\alpha\in \langle B\rangle$.
    %\item A collection of maps $B=\{\alpha_i\colon n_i\to 1\}_{i\in I}$ is called a \textit{basis} if it is closed under products and compositions and every map $\gamma$ in $\T$ can be written as a composition $\alpha\beta$ where $\alpha\in B$ and $\beta$ is inert.
    %\item We define a \textit{collection of active maps in} $\T$ to be a collection of non-inert maps $\alpha\colon n\to 1$ in $\T$ such that any map $\gamma\colon m\to 1$ can be factored as $\gamma=\alpha\beta$ where $\beta$ is inert and $\alpha$ is active.
\end{enumerate}    
Given a basis $B$, we call maps in $B$ by \textit{active maps}. 
\end{definition}

Note any basis together with all the inert maps (in fact, it is enough to take those with codomain $1$) generate $\T$ under products and compositions. In other words, denoting $I$ the class of inert maps, then if $B$ is a basis then any map with target $1$ lies in $\langle B\cup I\rangle$. 

\begin{example}
    Let $\T_{mon}$ be the theory for monoids. Then the collection of iterated multiplication maps $m_n\colon n\to 1$, $n\ge 2$, together with the unit map $u\colon 0\to 1$, form a basis which we will denote $B_{can}$ (can for canonical). There are also much smaller bases, for example we can take just $B=\{m_2,u\}$ and we get $\langle B\rangle = B_{can}\cup\{\id_1\}$. As another example we can take a basis $B'=\{m_3,u\}$ which also satisfies $\langle B\rangle = B_{can}\cup\{\id_1\}$ as we can write $m_2$ as an operadic composition $m_2=m_3(u,\id_1,\id_1)$. Note that we could have taken the theory $\T_{cmon}$ for commutative monoids instead, obtaining the same bases.
\end{example}

It is thus enough to check the commutativity of squares (\ref{diagram: comm}) for $\alpha,$ $\beta$ elements of some basis $B$. That is because $(\F^{op},\id)$ as a Lawvere theory is itself commutative and thus if either $\alpha$ or $\beta$ in $\T$ is inert, the square (\ref{diagram: comm}) commutes automatically, and a condition for a square to commute is closed under products and compositions, and thus also operadic compositions.

%\begin{remark}      More concretely, let us call a map in $\T$ \textit{inert}\footnote{This is inspired by Lurie's terminology for operads, which is in fact closely related.} if it is in the image of $\theta\colon \F^{op}\to \T$. Then if either $\alpha$ or $\beta$ in the square (\ref{diagram: comm}) is inert, the square has to commute. We say that $\alpha$ in $\T$ is \textit{active} if whenever we write $\alpha$ as a composite $\alpha_1\circ\cdots\circ \alpha_n$ of a finite number of maps and some of them, say $\alpha_i$, is inert, then $\alpha_i=1$. Quite easily, we get that every map in $\T$ can be factored as into an inert map followed by an active one. Therefore, we only need to check the commutativity of (\ref{diagram: comm}) for $\alpha,\beta$ both active. Also, since $\T(m,n)\cong \T(m,1)^n$, we can assume that $\alpha$ and $\beta$ have $1$ as a codomain.\end{remark}

What does this imply for models of $\T$? Suppose we have a $\Set$-valued model $\X\colon \T\to \Set$, $X:=\X(1)$. Consider an element of $(X^m)^k$; we will write such element as an $m\times k$ matrix of elements $x_{ij}\in X$: $$(x_{ij})=\begin{pmatrix}
    x_{11} & x_{12} & \cdots & x_{1k} \\
    x_{21} & x_{22} & \cdots & x_{2k} \\
    \vdots & \vdots & \ddots & \vdots \\
    x_{m1} & x_{m2} & \cdots & x_{mk}
\end{pmatrix}$$

For an map $\alpha\colon m\to 1$ in $\T$, the map $k\cdot \alpha\colon mk\to k$ induces an operation $\X(\alpha)^k\colon(X^m)^k\to X^k$. We can depict this as an action of $\alpha$ on matrices (written as a left action) $$\alpha\cdot \begin{pmatrix}
    x_{11} & x_{12} & \cdots & x_{1k} \\
    x_{21} & x_{22} & \cdots & x_{2k} \\
    \vdots & \vdots & \ddots & \vdots \\
    x_{m1} & x_{m2} & \cdots & x_{mk}
\end{pmatrix}:=\begin{pmatrix}
    \alpha\begin{pmatrix} x_{11} \\ x_{21} \\ \vdots \\ x_{m1} \end{pmatrix} &
    \alpha\begin{pmatrix} x_{12} \\ x_{22} \\ \vdots \\ x_{m2} \end{pmatrix} &
    \cdots &
    \alpha\begin{pmatrix} x_{1k} \\ x_{2k} \\ \vdots \\ x_{mk} \end{pmatrix}
\end{pmatrix}
$$

Similarly, a map $\beta\colon k\to 1$ gives an operation $\X(\beta)^m\colon (X^k)^m\to X^m$. Composing with the natural isomorphism $(X^m)^k\cong (X^k)^m$, which we can visualize as a matrix transpose, we can depict the effect of $\X(\beta\cdot m)$ as an action (written as a right action) on $m\times k$ matrices $$\begin{pmatrix}
    x_{11} & x_{12} & \cdots & x_{1k} \\
    x_{21} & x_{22} & \cdots & x_{2k} \\
    \vdots & \vdots & \ddots & \vdots \\
    x_{m1} & x_{m2} & \cdots & x_{mk}
\end{pmatrix}\cdot \beta:=\begin{pmatrix}
    \beta(x_{11},\dots,x_{1k})\\ \beta(x_{21},\dots,x_{2k}) \\ \vdots \\ \beta(x_{m1},\dots,x_{mk})
\end{pmatrix} $$
Now, the commutativity of $\T$ implies that for every such $\alpha,\beta$ as above and any matrix $(x_{ij})\in (X^m)^k$, we have $(\alpha\cdot (x_{ij}))\cdot \beta=\alpha\cdot((x_{ij}\cdot\beta)$. Exploring the meaning of this equation, we arrive at the ``semantic'' criterion for commutativity of $\T$.

\begin{proposition} \label{prop: comm criteria}
    The following are equivalent:
    \begin{enumerate}
        \item (Syntactic criterion) $\T$ is commutative, i.e. the squares as in (\ref{diagram: comm}) commute.
        \item (Semantic criterion) For each model $\X\colon \T\to \C$, $X=\X(1)$ and each $\alpha\colon m\to n$ in $\T$, the map $\X(\alpha)\colon \X(1)^m\to \X(1)^n$ lifts to a homomorphism of models $\X^m\to \X^n$.
        %\item $\T$ is a commutative monoid in $\Law$.
    \end{enumerate}
\end{proposition}

\begin{proof}
    $(1)\Leftarrow(2)$. Consider the model $\Id\colon \T\to\T$. As each $\alpha\colon m \to n$ in $\T$ lifts to a homomorphism of models $\Id^m\to \Id^n$, we obtain the desired equalities.
    
    $(1)\Rightarrow (2).$ We need to show that for each $\X\colon \T \to \C$ and each $\alpha\colon m\to n$, $\beta\colon k\to l$ in $\T$, the following commutes:
    % https://q.uiver.app/#q=WzAsNCxbMCwxLCJYKDEpXnttbH0iXSxbMSwxLCJYKDEpXntubH0iXSxbMCwwLCJYKDEpXntta30iXSxbMSwwLCJYKDEpXntua30iXSxbMCwxLCJYKFxcYWxwaGEpXmwiXSxbMiwwLCJYXm0oXFxiZXRhKSIsMl0sWzIsMywiWChcXGFscGhhKV5rIl0sWzMsMSwiWF5uKFxcYmV0YSkiXV0=
    \[\begin{tikzcd}
	{X^{mk}} & {X^{nk}} \\
	{X^{ml}} & {X^{nl}}
	\arrow["{\X(\alpha)^l}", from=2-1, to=2-2]
	\arrow["{\X^m(\beta)}"', from=1-1, to=2-1]
	\arrow["{\X(\alpha)^k}", from=1-1, to=1-2]
	\arrow["{\X^n(\beta)}", from=1-2, to=2-2]
    \end{tikzcd}\]
    But that is just $\X$ applied to the square (\ref{diagram: comm}). \end{proof}

    If $\T$ is commutative, we have a functor $\widetilde{(-)}\colon \Mod(\T,\C)\to \Mod(\T,\Mod(\T,\C))$, $\X\mapsto \widetilde{\X}$ where $\widetilde{\X}(n)=\X^n$ and $\widetilde{\X}(\alpha)$ is the lift described above.

\begin{corr} \label{rmk: 1-dim lifting functor}
    The functor $\widetilde{(-)}$ is fully faithful. 
\end{corr}

\begin{proof}
    For any $f\colon \widetilde{\X}\to \widetilde{\Y}$ with $f_1:=f(1)$, we have $f(1)=\widetilde{f_1}(1)$ and thus $f=\widetilde{f_1}$ by Remark \ref{rmk: U faithful}.
\end{proof}

Let us discuss some examples of commutative Lawvere theories.

\begin{example}
    Let $M$ be a monoid; then the monoid multiplication $m\colon M^2\to M$ is a homomorphism of monoids if and only if $M$ is commutative. Thus, the Lawvere theory $\T_{mon}$ for monoids is not commutative, but the theory $\T_{cmon}$ for commutative monoids is. Indeed, matching the syntactic criterion for $\T_{cmon}$ boils down to the commutativity of the following square (where $m$ stands for multiplication): % https://q.uiver.app/#q=WzAsNCxbMCwwLCI0Il0sWzEsMCwiMiJdLFsxLDEsIjEiXSxbMCwxLCIyIl0sWzAsMSwiMlxcY2RvdCBtIl0sWzEsMiwibSJdLFswLDMsIm1cXGNkb3QyIiwyXSxbMywyLCJtIl1d
\[\begin{tikzcd}
	4 & 2 \\
	2 & 1
	\arrow["{2\cdot m}", from=1-1, to=1-2]
	\arrow["{m\cdot2}"', from=1-1, to=2-1]
	\arrow["m", from=1-2, to=2-2]
	\arrow["m", from=2-1, to=2-2]
\end{tikzcd}\]
But $m\cdot 2$ can be written as $4\xrightarrow{1s1}4\xrightarrow{2\cdot m}2$, where $s\colon 2\to 2$ is the symmetry (coming from $\F^{op}$. But as $ms=m$ in $\T_{cmon}$ (this is the commutativity condition), we can rewrite the square above in a way such that the commutativity is clear: 
\begin{equation}
    \label{dia: comm mon}
    \begin{tikzcd} 
	4 & 3 & 2 \\
	4 & 2 & 1
	\arrow["1m1", from=1-1, to=1-2]
	\arrow["1s1"', from=1-1, to=2-1]
	\arrow["m1", from=1-2, to=1-3]
	\arrow["m", from=1-3, to=2-3]
	\arrow["1m1"{description}, from=2-1, to=1-2]
	\arrow["{2\cdot m}"', from=2-1, to=2-2]
	\arrow["m"', from=2-2, to=2-3]
\end{tikzcd}
\end{equation}
% https://q.uiver.app/#q=WzAsNixbMCwwLCI0Il0sWzEsMCwiMyJdLFsxLDEsIjIiXSxbMCwxLCI0Il0sWzIsMSwiMSJdLFsyLDAsIjIiXSxbMCwxLCIxbTEiXSxbMCwzLCIxczEiLDJdLFszLDIsIjJcXGNkb3QgbSIsMl0sWzMsMSwiMW0xIiwxXSxbMiw0LCJtIiwyXSxbNSw0LCJtIl0sWzEsNSwibTEiXV0=
Indeed, the triangle on the left commutes as $ms=m$ and the pentagon on the right commutes because $m$ is associative.

\end{example}

\begin{example} \label{example: 1dim monoids}
    Denote by $\T_M$ the Lawvere theory whose models in $\Set$ are sets equipped with an action of a monoid $M$. Hence, $\T_M$ is generated by unary maps $\rho_m\colon 1\to 1$ for each $m\in M$ such that $\rho_{mn}=\rho_m\circ \rho_n$ and $\rho_1=1$. Commutativity of $\T_M$ amounts to commutativity of the squares % https://q.uiver.app/#q=WzAsNCxbMCwwLCIxIl0sWzAsMSwiMSJdLFsxLDEsIjEiXSxbMSwwLCIxIl0sWzAsMSwiXFxyaG9fbSIsMl0sWzEsMiwiXFxyaG9fbiJdLFszLDIsIlxccmhvX20iXSxbMCwzLCJcXHJob19uIl1d
\[\begin{tikzcd}
	1 & 1 \\
	1 & 1
	\arrow["{\rho_n}", from=1-1, to=1-2]
	\arrow["{\rho_m}"', from=1-1, to=2-1]
	\arrow["{\rho_m}", from=1-2, to=2-2]
	\arrow["{\rho_n}", from=2-1, to=2-2]
\end{tikzcd}\]
In other words, $\T_M$ is commutative if and only if $M$ is. We get the analogous result for a family of theories $\T_R$ where $R$ is a ring and $\Mod(\T_R,\Set)$ is the category of $R$-modules.
\end{example}

\subsection{Categorical criterion for commutativity} \label{subsec: categorical criterion}
We have exhibited two equivalent criteria for a Lawvere theory to be commutative: one syntactic and the other semantic. There is a third criterion — which we shall call the \emph{categorical} criterion — and it will serve as the starting point for our approach in the $2$-dimensional setting. %Although the statement below is well-known to experts, it is not easy to find it in a literature and several subtleties are hard to reconstruct. For that reason we provide a complete proof here, since these considerations will be needed later.

\begin{proposition} \label{prop: ord formal}
    $\T$ is commutative if and only if $\T$ has a structure of a symmetric monoidal category such that the tensor product $\odot\colon \T\times \T \to \T$ preserves powers in each variable.
\end{proposition}

The key part is the following lemma for $\T=\F^{op}$ which turns out to be quite subtle. We use a term \textit{magmal category} (\cite{Dav07}) for categories with a monoidal structure which is not necessarily associative: i.e., we have a tensor product $\odot\colon \C^2\to\C$, a unit $I$, and unitors $\lambda_X\colon X\to I\odot X$, $\rho_X\colon X\to X\odot I$ such that $\rho_I=\lambda_I$ and the usual unitality axioms hold, but no associators. 

\begin{lemma} \label{lemma: unique on F}
    There exists a unique magmal structure $\odot\colon \F^{op} \times\F^{op}\to \F^{op},$ $u\colon1\to \F^{op}$ which preserves products in each variable, and it is the cocartesian monoidal structure. In particular, it is symmetric, and the induced monoid structure on $\N$ is the classical multiplication, $m\odot n=m\cdot n$, $u(*)=1$.
\end{lemma}

\begin{proof}
    By definition, $\odot$ has to preserve powers in each variable, i.e. $(m\cdot 1)\odot n=m\cdot (1\odot n)$, $m\odot (n\cdot 1)=n\cdot (m\odot1)$. %Therefore, if $a\otimes n=b\otimes n$ then $a=b$. 
    Denote $v:=u(*)$; this means that $v\odot n=n=n\odot v$. Then $$1=v\odot1=v\cdot(1\odot1),$$ hence $v=1$. This yields $m\odot n=(m\cdot 1)\odot n=m\cdot (1\odot n)=m\cdot n.$ Therefore, the functor $\odot$ is just a coproduct in $\F^{op}$. Then, by \cite[Thm. 1.1]{Ark25}, the magmal structure is necessarily the cocartesian monoidal structure. %On morphisms, we then easily get $\alpha\odot \id_n=n\cdot \alpha=\id_n\odot \alpha$, again from the product preservation.
\end{proof}

\textit{Proof of Proposition \ref{prop: ord formal}.} If $\T$ is commutative, define $\odot\colon \T\times\T\to \T$ by sending $(m,n)$ to $m\cdot n$, $(\id_n,\alpha)$ to $n\cdot \alpha$, $(\beta, \id_k)$ to $\beta\cdot k$. This is well-defined thanks to the syntactic criterion for commutativity. The associators etc. are inherited from the cocartesian monoidal structure on $\F^{op}$. On the other hand, since $\odot$ preserves products in each variable, precomposition with $\odot$ gives us a functor $\Mod(\T,\C)\to \Mod(\T,\Mod(\T,\C))$ which is easily seen to be a section of the forgetful functor, and thus $\T$ is commutative via the semantic criterion. \qed

One useful implication of this viewpoint on commutativity is the following observation. Let $\T$ be a Lawvere theory $\T$ and denote by $F$ and $U$ the free and forgetful functor between $\Mod(\T,\Set)$ and $\Set$. Then the set $UF1$ has a natural monoid structure induced by the composition in $\T$ under the isomorphism $UF1\cong \Set(1,UF1)\cong \T(1,1)$.

\begin{corr}
     If $\T$ is commutative then $UF1$ is a commutative monoid.
\end{corr}

\begin{proof}
    If $\T$ is commutative, then by Proposition \ref{prop: ord formal}, it is a monoidal category with $1$ being a unit. But for any unit in any monoidal category, its endomorphism monoid is commutative. Therefore, $\T(1,1)$ is commutative and under the isomorphism above, so is $UF1$.
\end{proof}

\subsection{Units and unital operations} \label{subsec: 1-cat EH} We say that a \textit{unit} in a Lawvere theory $\T$ is a map $u\colon 0\to 1$, i.e. a ``nullary operation symbol''. 

\begin{lemma} \label{lemma: unitality}
    Let $\T$ be a commutative Lawvere theory. Then there is at most one unit in $\T$.
\end{lemma}

\begin{proof}
    Suppose there are two of them, say $u,v\colon 0\to 1$. Then the commutativity square (\ref{diagram: comm}) for $u,v$ is of the following form: % https://q.uiver.app/#q=WzAsNCxbMCwxLCIwIl0sWzEsMSwiMSJdLFsxLDAsIjAiXSxbMCwwLCIwIl0sWzAsMSwidSJdLFsyLDEsInYiXSxbMywyLCIwXFxjZG90IHUiXSxbMywwLCJ2XFxjZG90MCIsMl1d
\[\begin{tikzcd}
	0 & 0 \\
	0 & 1
	\arrow["{0\cdot u}", from=1-1, to=1-2]
	\arrow["{v\cdot0}"', from=1-1, to=2-1]
	\arrow["v", from=1-2, to=2-2]
	\arrow["u", from=2-1, to=2-2]
\end{tikzcd}\]
 Since $0\cdot u=\id_0=v\cdot 0$, we get that $u=v$.
\end{proof}

For a unit $u$ and positive integers $k\le n$, $n>1$, denote by $u_{k,n}\colon 1\to n$ the map of the form $u+\cdots +\id_1+\cdots +u$ where $\id_1$ is at $k$-th position.

\begin{definition} \label{def: unitality}
    Let $\T$ be a Lawvere theory, $\alpha\colon n\to 1$ a map in $\T$, $n> 1$. We say that $\alpha$ is \textit{unital }with a unit $u\colon 0\to 1$ if for any $k\le n$, we get $\alpha\circ u_{k,n}=1$.
\end{definition}

%In other words, $\alpha$ is unital if and only if it exists a unit $u$ such that for any model $\X\colon \T\to \C$ any any
For example, for a model $\X\colon \T\to \Set$, $X:=\X(1)$, $\X(u)$ just picks an object in $u_X\in X$. Unitality of $\alpha$ then implies that for any $x\in X$, we get that $\X(\alpha)(u_X,\cdots,x,\cdots,u_X)=x.$

\begin{example}
    Consider the Lawvere theory for commutative rings. Then the binary maps $a,m\colon 2\to 1$ modeling addition and multiplication are unital, each with a different unit. Therefore, by Lemma \ref{lemma: unitality}, this theory cannot be commutative.
\end{example}

\begin{lemma} \label{lemma: unital maps}
    Let $\T$ be a commutative Lawvere theory and suppose that $\alpha,\beta\colon n\to 1$, $n>1$, are unital maps with the same unit $u$. Then $\alpha=\beta$.
\end{lemma}

\begin{proof}
    Consider the map $\sum_{i=1}^n u_{i,n}\colon n\to  n\cdot n$. Then, by definition of unitality, the triangles in the diagram below commutes.

    % https://q.uiver.app/#q=WzAsNSxbMSwxLCJuXFxjZG90IG4iXSxbMiwxLCJuIl0sWzEsMiwibiJdLFsyLDIsIjEiXSxbMCwwLCJuIl0sWzAsMSwiblxcY2RvdCBcXGFscGhhIl0sWzIsMywiXFxhbHBoYSJdLFswLDIsIlxcYmV0YSBcXGNkb3QgbiIsMl0sWzEsMywiIFxcYmV0YSJdLFs0LDAsIlxcc3VtX2kgdV97aSxufSJdLFs0LDEsIlxcaWRfbiIsMCx7ImN1cnZlIjotM31dLFs0LDIsIlxcaWRfbiIsMix7ImN1cnZlIjozfV1d
\[\begin{tikzcd}
	n \\
	& {n\cdot n} & n \\
	& n & 1
	\arrow["{\sum_i u_{i,n}}", from=1-1, to=2-2]
	\arrow["{\id_n}", curve={height=-18pt}, from=1-1, to=2-3]
	\arrow["{\id_n}"', curve={height=18pt}, from=1-1, to=3-2]
	\arrow["{n\cdot \alpha}", from=2-2, to=2-3]
	\arrow["{\beta \cdot n}"', from=2-2, to=3-2]
	\arrow["{ \beta}", from=2-3, to=3-3]
	\arrow["\alpha", from=3-2, to=3-3]
\end{tikzcd}\]

The square also commutes as our theory is commutative. Hence, $\alpha=\beta.$
\end{proof}

Using matrices, we can explain an intuition behind this. Let $\X$ be a $\Set$-valued model, $X=\X(1)$. Seeing elements of $X^{n\cdot n}$ as $n\times n$ matrices, we get that $\sum_i u_{i,n}\colon X^n\to X^{n\cdot n}$ send an $n$-tuple $(x_1,\dots,x_n)$ to a matrix $(x_{ij})$ with $x_{ij}=x_i$ if $i=j$ and $x_{ij}=u_X$ otherwise. Then $\alpha \cdot (x_{ij})$ gives us an $n$-tuple of elements of the form $\alpha(u_X,\dots,x_i,\dots,u_X)=x_i$, and similary $(x_{ij})\cdot \beta$. This exhibits the commutativity of the triangles above.

\begin{remark}
    Note that the core of the proof above is a slight generalization of a proof of the classical Eckmann-Hilton argument, which says that is we have two group structures on a set $X$ which commute and share the same unit, then they necessarily have to be the same (and commutative). Based on that, we call the statement below (and its 2-dimensional generalization later) an Eckmann-Hilton argument.
\end{remark}

\subsection{Kronecker product and the Eckmann-Hilton argument} \label{subsec: 1-dim Kronecker Eckmann-Hilton} We have seen in Lemma \ref{lemma: unique on F} that $\F^{op}$ possesses a unique monoidal structure compatible with products. Therefore, we can view any Lawvere theory $\T$ as a module over $\F^{op}$ (viewing $\F^{op}$ as a symmetric pseudomonoid in $\Cat$) via the action $a\colon\F^{op}\times \T\to \T$ given by sending $(m,n)$ to $m\cdot n$ and $(\id_n,\alpha)$ to $n\cdot \alpha$. Then one can consider a tensor product of Lawvere theories $\T$, $\Ss$ as $\F^{op}$-modules: $$\T\otimes_{\F^{op}}\Ss:=\text{coeq}(\T\times\F^{op}\times \Ss\rightrightarrows\T\times \Ss).$$ It is well-known (e.g. \cite[3.11]{Bor94}) that $\T\otimes_{\F^{op}}\Ss$ is again a Lawvere theory, and if $B$ is a basis for $\T$ and $B'$ one for $\Ss$,  $\T\otimes_{\F^{op}}\Ss$ admits $B\coprod B'$ as a basis. This tensor product of Lawvere theories is often called the \textit{Kronecker product.}

From the discussion in \ref{subsec: categorical criterion}, it follows that if $\T$ is commutative, then the structure map $\theta\colon \F^{op}\to \T$ is symmetric monoidal. Therefore, we can view $\T$ as a symmetric pseudomonoid in $\F^{op}$-modules. In other words, writing the tensor product on $\T$ as $\mu\colon\T\times \T\to \T$, then $\mu$ factors through $\T\otimes_{\F^{op}}\T$ as a quotient map followed by some $\mu'$, and if we write $u'\colon \T\cong\F^{op}\otimes_{\F^{op}}\T\xrightarrow{\theta\otimes 1}\T\otimes_{\F^{op}} \T$, then $\mu', u'$ gives a pseudomonoid structure on $\T$ with respect to the Kronecker product. Note that $\mu'$ is identity on objects.
 
\begin{proposition} \label{prop: 1-dim EH}
    (Eckmann-Hilton argument) Let $\T$ be a commutative Lawvere theory. Suppose that there is a basis $B$ of $\T$ such that every map $\alpha\colon n\to 1$, $\alpha\in B$, $n>1$, is unital and there are no unary maps $1\to 1$ in $B$. Then the map $\mu'\colon \T\otimes_{\F^{op}}\T\to \T$ is an isomorphism.
\end{proposition}

\begin{proof}
    We know that $\mu'u'= 1$ so it is enough to show that $u'\mu'\cong 1$. As a tensor product of two symmetric pseudomonoids is again one,\footnote{The extension of the $1$-dimensional result is \cite[Appendix A]{Sch14}.} $\T\otimes_{\F^{op}}\T$ is again commutative. The basis $B\coprod B$ satisfies the same conditions as $B$ does. Let us take any $\alpha\colon n\to1$ in $B\coprod B$. Then either $\alpha$ is a unit or a unital operation. As any morphism of Lawvere theories preserves units and unital operations, we get that $u'\mu'(\alpha)\colon n\to 1$ has to be $\alpha$ itself by the combined effort of Lemmas \ref{lemma: unitality} and \ref{lemma: unital maps}. 
\end{proof}

\begin{remark}
    The conditions on $\T$ in the proposition are sufficient but not necessary; consider the Lawvere $1$-theory $\T_{ab}$ is for abelian groups. Then although there is a unary map $1\to1$, the group inverse, the Eckmann-Hilton argument still works. This is connected to the fact that being an abelian group is a property of a commutative monoid, not a structure.
\end{remark}

\begin{example}
    To see that we cannot drop the condition about unary maps in Definition \ref{def: Eckmann-Hilton}, we use the theory $\T_{M}$ from Example \ref{example: 1dim monoids}. Then one can check that $\T_{M}\otimes_{\F^{op}} \T_M= \T_{M\times M}$.%Consider a model of $\T_{inv}$ in $\Cat$, i.e. a category $\C$ equipped an involution functor $\iota\colon \C\to \C$. Suppose that $\iota$ is not an identity. Then we can lift $(\C,\iota)$ to a model $\T_{inv}\to \Mod_p(\T_{inv},\Cat)$ in two different ways: either we equip $(\C,\iota)$ with $\iota$, which is easily shown to be an endomorphism $\Mod_p(\T_{inv},\Cat)$, or with the identity. Therefore, the Eckmann-Hilton argument does not hold for $\T_{inv}$. 
\end{example}

\section{2-dimensional algebraic structures}

We now proceed to define 2-dimensional analogue of Lawvere theories and their models. The semantic theory now becomes richer as we have various flavours of morphisms of models: strict, pseudo, lax, colax. We use the framework of \cite{ABK24} to consider Lawvere 2-theories as \textit{marked $2$-sketches}\footnote{Also enhanced $2$-sketches or $\mF$-sketches}. The goal of this section is to define the category of algebras internal to a $\T$-model $\X$, denoted by $\IntAlg(\X)$, both in $(2,2)$ and $(\infty,2)$-categorical contexts. To achieve that, we need to recall the theory of marked (or enhanced) $2$-categories and the Gray tensor products thereof.

We want to develop the theory simultaneously in the setting of $(2,2)$-categories and $(\infty,2)$, but due to different conventions in these two research fields, we need to do it separately to some extent. We follow the notational conventions explained at the beginning: in particular, whenever we can treat both cases uniformly, we talk about $2$-categories, otherwise we talk about $(2,2)$- or $(\infty,2)$-categories. 

\subsection{Gray tensor products} \label{sub: recollections}

Recall that if $F,G\colon \C\to \D$ are functors between $(2,2)$-categories, we have several flavours of natural transformations we can consider: strict (or $2$-natural), pseudo, lax, and colax.\footnote{Also called oplax; we try to fix the notation in such a way that we do not need to address this too much.} For example, a lax natural transformation $\phi\colon F\to G$ consists of a collection of $1$-cells $\phi_X\colon F(X)\to G(X)$ indexed by objects $X$ of $\C$, together with a collection of $2$-cells $\phi_f$ as below indexed by $1$-cells $f$ in $\C$, satisfying some coherence equations.

\[\begin{tikzcd}
	{F(X)} & {G(X)} \\
	{F(X')} & {G(X')}
	\arrow[""{name=0, anchor=center, inner sep=0}, "{\phi_X}", from=1-1, to=1-2]
	\arrow["Ff"', from=1-1, to=2-1]
	\arrow["Gf", from=1-2, to=2-2]
	\arrow[""{name=1, anchor=center, inner sep=0}, "{\phi_{X'}}"', from=2-1, to=2-2]
	\arrow["{\phi_f}", between={0.4}{0.6}, Rightarrow, from=0, to=1]
\end{tikzcd}\]

For a colax natural transformation, $2$-cells $\phi_f$ point in the opposite direction. In case of pseudonaturality, all of the $\phi_f$ are invertible, and finally, for a strict natural transformation, we want the square above to commute, i.e. all the $\phi_f$ to be identities. Sometimes, \emph{funny} or unnatural natural transformations can be considered: these are simply collections $\{\phi_X\}_{X\in ob \C}$ with no other data or compatibilites whatsoever. If $F, G\colon \C\to \D$ are functors between $(\infty,2)$-categories, the same above still applies with the exception that it does not make sense to distinguish between strict and pseudo; see the notation below

\begin{notation}(Weaknesses.) \label{notation: weaknesses}
    In order to handle all of the flavours above uniformly, we introduce a set of \textit{weaknesses}\footnote{This is inspired by \cite[Notation 4.6]{ABK24}.} $W_{ord}=\{s,p,l,c\}$, standing for strict, pseudo, lax and colax. Hence, we can talk about $w$-natural transformations for any $w\in W_{ord}$. We also fix a set of weaknesses for $(\infty,2)$-categories $W_{\infty}=\{s,l,c\}$ We also occasionally speak about the set $\widetilde W_{ord}:=W_{ord}\cup\{f\}$, where $f$ stands for funny. For any $W\in\{W_{ord}, W_{\infty}, \widetilde{W}_{ord}\}$ we define an involution $(-)^*\colon W\to W$ such that $l^*:=c$, $c:=l^*$ and all the other elements stay fixed. We are going to denote by $\Fun_w(\C,\D)$ the $2$-category (either $(2,2)$- or $(\infty,2)$-) of functors, $w$-natural transformations, and modifications. To avoid the ambiguity between the terms colax and oplax, we will usually use the notation $\Fun_{l^*}(\C,\D)$, using the involution above, instead of $\Fun_c(\X,\D)$. We also introduce a partial order on $\widetilde{W}_{ord}$ such that $s\le p$, $p\le l, c$, and $f$ is maximal, restricting to a partial order on $W_{ord}$ and $W_\infty$.
\end{notation}

For each weakness $w\in \widetilde W_{ord}$, there is a notion of a \textit{$w$-Gray tensor product} $\otimes_w$, introduced in \cite{Gra06}, making $(\Cat_2, \otimes_w, 1, \Fun_w(-,-), \Fun_{w^*}(-,-))$ into a biclosed monoidal category, symmetric for $w\le p$. In other words, we have natural isomorphisms \begin{align*}
    \Fun_w(\A\otimes_w \B, \C)&\cong \Fun_w(\A,\Fun_w(\B,\C)),\\
    \Fun_{w^*}(\A\otimes_w \B, \C)&\cong \Fun_{w^*}(\B,\Fun_{w^*}(\A,\C)),\\    \Fun_{w}(\A,\Fun_{w^*}(\B,\C))&\cong \Fun_{w^*}(\B,\Fun_{w}(\A,\C)).
\end{align*} In particular, $\otimes_s$ is just a cartesian product $\times$. In general, on objects we have $ob(\A\otimes_w\B)=ob\A\times ob\B$. The 1-cells are in all cases generated by those of the form $(f,1),$ $(1,g)$ composing in the obvious way. We have the 2-cells $(\alpha,1)$, $(1,\beta)$ in either case, composing the obvious way. The difference is in the following square:
% https://q.uiver.app/#q=WzAsNCxbMCwwLCIoYSxiKSJdLFswLDEsIihhJyxiKSJdLFsxLDAsIihhLGInKSJdLFsxLDEsIihhJyxiJykiXSxbMCwxLCIoZiwxKSIsMl0sWzAsMiwiKDEsZykiXSxbMSwzLCIoMSxnKSJdLFsyLDMsIihmLDEpIl1d
\[\begin{tikzcd}
	{(a,b)} & {(a,b')} \\
	{(a',b)} & {(a',b')}
	\arrow["{(f,1)}"', from=1-1, to=2-1]
	\arrow["{(1,g)}", from=1-1, to=1-2]
	\arrow["{(1,g)}", from=2-1, to=2-2]
	\arrow["{(f,1)}", from=1-2, to=2-2]
\end{tikzcd}\]

For $\otimes_s$, i.e. cartesian product, this square commutes. For lax Gray tensor product, it is filled by the $2$-cell $\Sigma_{fg}$:
% https://q.uiver.app/#q=WzAsNCxbMCwwLCIoYSxiKSJdLFswLDEsIihhJyxiKSJdLFsxLDAsIihhLGInKSJdLFsxLDEsIihhJyxiJykiXSxbMCwxLCIoZiwxKSIsMl0sWzAsMiwiKDEsZykiXSxbMSwzLCIoMSxnKSJdLFsyLDMsIihmLDEpIl0sWzUsNiwiXFxzaWdtYV97Zmd9IiwwLHsic2hvcnRlbiI6eyJzb3VyY2UiOjQwLCJ0YXJnZXQiOjQwfX1dXQ==
\begin{equation}\label{dia: gray square} \begin{tikzcd}
	{(a,b)} & {(a,b')} \\
	{(a',b)} & {(a',b')}
	\arrow[""{name=0, anchor=center, inner sep=0}, "{(1,g)}", from=1-1, to=1-2]
	\arrow["{(f,1)}"', from=1-1, to=2-1]
	\arrow["{(f,1)}", from=1-2, to=2-2]
	\arrow[""{name=1, anchor=center, inner sep=0}, "{(1,g)}", from=2-1, to=2-2]
	\arrow["{\Sigma_{fg}}", between={0.4}{0.6}, Rightarrow, from=0, to=1]
\end{tikzcd}\end{equation}
For colax Gray tensor product, there is a $2$-cell in the opposite way, and for $\otimes_p$, the cell is always invertible. For funny tensor product, there is no cell at all filling the square above. The $2$-cells $\Sigma_{fg}$ ought to satisfy a number of coherence equations, described e.g. in \cite{Gur13} (we see them explicitly in Lemma \ref{lemma: syntactic}). We are going to sometimes call the $2$-cells by \textit{Gray cells} or Gray $w$-cells. Note that we have canonical maps
% https://q.uiver.app/#q=WzAsNSxbMCwxLCJcXEFcXG90aW1lc19mXFxCIl0sWzEsMCwiXFxBXFxvdGltZXNfbFxcQiJdLFsxLDIsIlxcQVxcb3RpbWVzX2NcXEIiXSxbMiwxLCJcXEFcXG90aW1lc19wXFxCIl0sWzMsMSwiXFxBXFxvdGltZXNfc1xcQiJdLFswLDFdLFswLDJdLFsxLDNdLFsyLDNdLFszLDRdXQ==
\[\begin{tikzcd}
	& {\A\otimes_l\B} \\
	{\A\otimes_f\B} && {\A\otimes_p\B} & {\A\otimes_s\B} \\
	& {\A\otimes_c\B}
	\arrow[from=2-1, to=1-2]
	\arrow[from=2-1, to=3-2]
	\arrow[from=1-2, to=2-3]
	\arrow[from=3-2, to=2-3]
	\arrow[from=2-3, to=2-4]
\end{tikzcd}\]
In other words, for $w',w\in \widetilde W_{ord}$ with $w'\le w$, the identity functor $(\Cat_2, \otimes_w,1)\to (\Cat_2, \otimes_{w'},1)$ is lax monoidal.

For $w\in W_{\infty}$, we have a Gray tensor product making $(\Cat_{(\infty,2)},\otimes_w,1,\Fun_w(-,-),\Fun_{w^*}(-,-))$ into a biclosed monoidal $\infty$-category, a homotopy coherent enhancement of $\otimes_p,$ $\otimes_l,$ $\otimes_{l^*}$ described above. Again, we get $\otimes_s=\times.$ The formal construction of $\otimes_l$ is fairly subtle; see \cite{CM23}, \cite{LR25} and references therein. %Both $\otimes_{l}$, $\otimes_{l^{*}}$ is are compatible with the one on $(2,2)$-categories introduced above.

\subsection{Marked 2-categories.} To define the $2$-category $\Mod_{\lax}(\T,\C)$, we need to be able to talk about lax natural transformations which become strict on a specified collection of $1$-cells in $\T$, called a marking. To do so, we use the theory of \textit{marked $(\infty,2)$-categories} in the form that is developed in \cite{AGH24} and \cite{AHM26}, together with its $(2,2)$-categorical variant called \textit{enhanced (2,2)-categories} or $\mF$-categories in the form developed in \cite{ABK24}. We choose to use the term marked (instead of enhanced) throughout the paper.

Recall that a marked $2$-category is a pair $(\C,I)$ where $\C$ is a $2$-category and $I$ a collection of $1$-cells called \textit{marking} containing equivalences and closed under compositions. We have a $1$-category $\Cat^{\mathfrak{m}}_2$ with objects marked $2$-categories and morphisms $(\C,I)\to (\D,J) $ being functors $F\colon \C\to \D$ preserving marking, i.e. $F(I)\subset J$. We have two natural functors $\Cat_2^{\mathfrak{m}}\to \Cat_2$ -- a forgetful functor which forgets the marking and a functor sending $(\C,I)$ to the locally full 2-subcategory $\C_I\subset \C$ spanned by the marked $1$-morphisms.

Having two marked $(2,2)$-categories $(\C,I)$, $(\D,J)$, we can consider their \textit{marked Gray tensor product} $\C\otimes_{w',w}\D$ for any $w',w\in {W}_{ord}$, $w'\le w$. This is obtained from the unmarked Gray tensor product $\C\otimes_w\D$ by requiring that the Gray $w$-cells $\Sigma_{fg}$ are $w'$-cells whenever $f$ or $g$ are marked. We have a canonical marking $I\otimes J$ on $\C\otimes_{w',w}\D$ generated by $1$-cells $(f,1)$, $(1,g)$ where $f\in I$ and $g\in J$. We can also consider the functor $(2,2)$-categories $\Fun_{w',w}(\C,\D)$ defined as follows: \begin{itemize}
       \item objects are marking-preserving functors $\C\to \D$,
       \item $1$-cells are $w$-natural transformations of functors $\phi\colon F\Rightarrow G$ such that for a $1$-cell $f\colon X\to X'$ in $\C$, the component $\phi_f$ is a $w'$-cell whenever $f$ is marked,
       % https://q.uiver.app/#q=WzAsNCxbMCwwLCJGKFgpIl0sWzAsMSwiRihYJykiXSxbMSwwLCJHKFgpIl0sWzEsMSwiRyhYJykiXSxbMCwyLCJcXHBoaV9YIl0sWzAsMSwiRmYiLDJdLFsyLDMsIkdmIl0sWzEsMywiXFxwaGlfe1gnfSIsMl0sWzQsNywiXFxwaGlfZiIsMCx7InNob3J0ZW4iOnsic291cmNlIjo0MCwidGFyZ2V0Ijo0MH19XV0=
\[\begin{tikzcd}
	{F(X)} & {G(X)} \\
	{F(X')} & {G(X')}
	\arrow[""{name=0, anchor=center, inner sep=0}, "{\phi_X}", from=1-1, to=1-2]
	\arrow["Ff"', from=1-1, to=2-1]
	\arrow["Gf", from=1-2, to=2-2]
	\arrow[""{name=1, anchor=center, inner sep=0}, "{\phi_{X'}}"', from=2-1, to=2-2]
	\arrow["{\phi_f}", between={0.4}{0.6}, Rightarrow, from=0, to=1]
\end{tikzcd}\]
\item $2$-cells are modifications.
   \end{itemize}
This $(2,2)$-category also admits a natural marking $[I,J]$ where we mark $w'$-natural transformations. Similarly as in the unmarked case, we have the following result.

\begin{proposition} \label{prop: Gray universal} (\cite{ABK24}, Thm. 6.12) 
    For any marked $(2,2)$-categories $(\A,I),$ $(\B,J),$ $(\C,K)$, we have the family of isomorphisms of marked $(2,2)$-categories \begin{align}
        \Fun_{w',w}(\A\otimes_{w',w}\B,\C)&\cong \Fun_{w',w}(\A,\Fun_{w',w}(\B,\C)),\\
        \Fun_{w'^*,w^*}(\A\otimes_{w',w}\B,\C)&\cong \Fun_{w'^*,w^*}(\B,\Fun_{w'^*, w^*}(\A,\C)),\\
        \Fun_{w',w}(\A,\Fun_{w'^*, w^*}(\B,\C))&\cong \Fun_{w'^*,w^*}(\B,\Fun_{w',w}(\A,\C)),
    \end{align} natural in each variable.
\end{proposition}

Having two marked $(\infty,2)$-categories $(\C,I)$, $(\D,J)$, the constructions and results above are still available, see \cite{AHM26} and references therein. In particular, we can consider the marked Gray tensor product $\C\otimes_{w',w}\D$ for any $w',w\in W_{\infty}$, $w'\le w$. As $\otimes_{s,s}=\times$, we focus on the tensor product $\C\otimes_{s,l}\D$. One can obtain it from the non-marked Gray tensor product $\C\otimes_l\D$ by inverting the Gray cells $\Sigma_{fg}$ whenever $f$ or $g$ is marked. We have also the functor $(\infty,2)$-categories $\Fun_{w',w}(\C,\D)$ and canonical markings $I\otimes J$, $[I,J]$, together with the following result.

\begin{proposition} \label{prop: marked tensor-hom}
    \cite[Cor. 2.3.6]{AHM26} For any marked $(\infty,2)$-categories $(\A,I),$ $(\B,J),$ $(\C,K)$, and $w\in \{l,l^*\},$ we have the equivalences of marked $(\infty,2)$-categories \begin{align}
        \Fun_{s,l}(\A\otimes_{s,l}\B,\C)&\simeq \Fun_{s,l}(\A,\Fun_{s,l}(\B,\C)),\\
        \Fun_{s,l^*}(\A\otimes_{s,l}\B,\C)&\simeq \Fun_{s,l^*}(\B,\Fun_{s,l^*}(\A,\C)),\\
        \Fun_{s,l}(\A,\Fun_{s, l^*}(\B,\C))&\simeq \Fun_{s,l^*}(\B,\Fun_{s,l}(\A,\C)),
    \end{align} natural in each variable.
\end{proposition}

\begin{remark} Although we will not need it extensively, it is worth it to discuss funny tensor product $\otimes_f$ and its variants $\otimes_{w,f}$ for other weaknesses $w\in W_{ord}$; this is omitted in \cite{ABK24}. Intuitively, for example $\C\otimes_{l,f} \D$ should have the lax Gray cells $\Sigma_{fg}$ only when $f$ or $g$ is marked and no relationship between $(1,f)\circ (g,1)$ and $(g,1)\circ(1,f)$ in general, so we can see it behaves differently than its aforementioned variants.  A way of defining $\otimes_{w',w}$ for any weaknesses $w',w$, including $f$, is the following definition.
\end{remark}
\begin{definition}
    Let $(\A,I)$, $(\B,J)$ be marked $(2,2)$-categories, $w', w\in \widetilde{W}_{ord}$ any weaknesses. Define $\A\otimes_{w',w}\B$ to be a colimit of the following diagram: 
    % https://q.uiver.app/#q=WzAsNSxbMiwxLCJcXEFeXFxsYW1iZGFcXG90aW1lc193XFxCXlxcbGFtYmRhIl0sWzEsMCwiXFxBXlxcbGFtYmRhXFxvdGltZXNfd1xcQl5cXHRhdSJdLFswLDEsIlxcQV5cXGxhbWJkYVxcb3RpbWVzX3t3J31cXEJeXFx0YXUiXSxbMywwLCJcXEFeXFx0YXVcXG90aW1lc193XFxCXlxcbGFtYmRhIl0sWzQsMSwiXFxBXlxcdGF1XFxvdGltZXNfe3cnfVxcQl5cXGxhbWJkYSJdLFsxLDBdLFszLDBdLFszLDRdLFsxLDJdXQ==
\[\begin{tikzcd}
	& {\A\otimes_w\B_J} && {\A_I\otimes_w\B} \\
	{\A\otimes_{w'}\B_J} && {\A\otimes_w\B} && {\A_I\otimes_{w'}\B}
	\arrow[from=1-2, to=2-1]
	\arrow[from=1-2, to=2-3]
	\arrow[from=1-4, to=2-3]
	\arrow[from=1-4, to=2-5]
\end{tikzcd}\]
\end{definition}
It is an easy exercise to see that this agrees with the usual definition for $w,w'\ne f$ and also that this colimit has a correct universal property in the sense of Proposition \ref{prop: Gray universal}. In the context of $(\infty,2)$-categories for $w'=s$, $w=l$, the approach above was utilized in \cite[Def. 2.3.1]{AHM26}.

\subsection{Lawvere 2-theories} \label{subsec: Lawvere theories}

 To define Lawvere $(\infty,2)$-theories, we follow the $(\infty,1)$-dimensional definition \cite[Def. 3.2]{Cra10}.

\begin{definition} 
    A \textit{Lawvere $(\infty,2)$-theory} is a pair $(\T,\theta)$ where $\T$ is an $(\infty,2)$-category with finite products and a distinguished object $1_{\T}$ and $\theta\colon \F^{op}\to \T$ is an essentially surjective, product preserving functor sending $1$ to $1_{\T}$.
\end{definition}

\begin{notation}
    We are often going to abuse notation and write $1$ for $1_{\T}$ and $n$ for the $n$-th power $(1_{\T})^n$.
\end{notation}

We view any Lawvere $(\infty,2)$-theory $\T$ as a marked $(\infty,2)$-category with marking generated by equivalences and the image of $\theta$; we call these maps \textit{inert}. Let us fix the following notation: for $(\infty,2)$-categories $\C$, $\D$ together with a marking $I$ on $\C$, we denote $$\Fun_{w',w}(\C,\D):=\Fun_{w',w}((\C,I),(\D,\text{all}))$$ to be the hom-category in $\Cat^\mathfrak{m}_{(\infty,2)}$ where $\D$ is viewed as marked $(\infty,2)$ with the maximal marking (consisting of all $1$-cells).% More generally, we view any $(\infty,2)$-category $\C$ with finite products as a marked $(\infty,2)$-category with marking generated by equivalences and all the projection maps in any product diagram in $\C$.

\begin{definition}
    For any $\T$ a Lawvere $(\infty,2)$-theory viewed as a marked $(\infty,2)$-category as above, any $(\infty,2)$-category $\C$ with finite products, and for any $w\in \{s,l,l^*\}$, define the $(\infty,2)$-category $\Mod_w(\T,\C)$ to be the full $(\infty,2)$-subcategory of $\Fun_{s,w}(\T,\C)$ spanned by functors preserving finite products. We denote the hom-categories in $\Mod_w(\T,\C)$ by $\Hom_w(-,-)$.
\end{definition}

Take $w=l,$ $\C=\Cat_{\infty}$ and let us look more concretely on how $\Mod_w(\T,\Cat_\infty)$ looks. Objects are functors $\X\colon\T\to\Cat_\infty$ satisfying the condition that the inert maps (i.e., the maps in the image of $\theta\colon \F^{op}\to \T)$ $p_i\colon n\to 1$ induce an equivalence \begin{equation}\label{eq: Segal condition}
    \X(n)\xrightarrow{\simeq}\X(1)^n, 
\end{equation} i.e. (a variant of) the Segal condition. Consider now a $1$-cell $f\colon \X\to \Y$; this is more generally a morphism in $\Fun_{s,l}(\T,\Cat)$, i.e. a marked lax natural transformation. The marked condition forces the following square to commute (up to homotopy) for any inert $p_i$:
% https://q.uiver.app/#q=WzAsNCxbMCwwLCJcXFgobikiXSxbMCwxLCJcXFgoMSkiXSxbMSwxLCJcXFkoMSkiXSxbMSwwLCJcXFkobikiXSxbMSwyLCJmKDEpIiwyXSxbMCwzLCJmKG4pIl0sWzAsMSwiXFxYKHBfaSkiLDJdLFszLDIsIlxcWShwX2kpIl0sWzMsMSwiXFxzaW1lcSIsMSx7InN0eWxlIjp7ImJvZHkiOnsibmFtZSI6Im5vbmUifSwiaGVhZCI6eyJuYW1lIjoibm9uZSJ9fX1dXQ==
\[\begin{tikzcd}
	{\X(n)} & {\Y(n)} \\
	{\X(1)} & {\Y(1)}
	\arrow["{f(n)}", from=1-1, to=1-2]
	\arrow["{\X(p_i)}"', from=1-1, to=2-1]
	\arrow["\simeq"{description}, draw=none, from=1-2, to=2-1]
	\arrow["{\Y(p_i)}", from=1-2, to=2-2]
	\arrow["{f(1)}"', from=2-1, to=2-2]
\end{tikzcd}\]
Hence, we get that $f(n)\simeq f(1)^n$. Therefore, denoting $X=\X(1)$, $Y=\Y(1)$, we can view $f$ as a data of a functor of $(\infty,1)$-categories $f_1\colon X\to Y$ together with $2$-cells $f_\alpha\colon \Y(\alpha)f_1^m\to f_1\X(\alpha)$ for $\alpha$ active\footnote{We are using the same notation as in the $1$-dimensional case \ref{remark: active inert}.} $1$-cells with respect to some basis $B$ of $1$-cells. Finally, for a modification $t\colon f\Rrightarrow g$, the equation below yields $t(n)\simeq t(1)^n$, so we can view $t$ as a natural transformation of functors $t_1\colon f_1\Rightarrow g_1$ together with some higher coherence data with respect to the cells $f_\alpha,$ $g_\alpha$.

% https://q.uiver.app/#q=WzAsOCxbMCwwLCJcXFgobikiXSxbMSwwLCJcXFkobikiXSxbMCwxLCJcXFgoMSkiXSxbMSwxLCJcXFkoMSkiXSxbMywwLCJcXFgobikiXSxbMywxLCJcXFgoMSkiXSxbNCwwLCJcXFkobikiXSxbNCwxLCJcXFkoMSkiXSxbMCwxLCJmKG4pIiwwLHsiY3VydmUiOi0yfV0sWzAsMSwiZyhuKSIsMix7ImN1cnZlIjoyfV0sWzAsMiwiXFxYKHBfaSkiLDJdLFsyLDMsImcoMSkiLDIseyJjdXJ2ZSI6Mn1dLFsxLDMsIlxcWShwX2kpIl0sWzQsNiwiZihuKSIsMCx7ImN1cnZlIjotMn1dLFs1LDcsImcoMSkiLDIseyJjdXJ2ZSI6Mn1dLFs1LDcsImYoMSkiLDAseyJjdXJ2ZSI6LTJ9XSxbNCw1LCJcXFgocF9pKSIsMl0sWzYsNywiXFxZKHBfaSkiXSxbOCw5LCJ0KG4pIiwwLHsic2hvcnRlbiI6eyJzb3VyY2UiOjIwLCJ0YXJnZXQiOjIwfX1dLFsxMiwxNiwiXFxzaW1lcSIsMyx7InNob3J0ZW4iOnsic291cmNlIjoyMCwidGFyZ2V0IjoyMH0sInN0eWxlIjp7ImJvZHkiOnsibmFtZSI6Im5vbmUifSwiaGVhZCI6eyJuYW1lIjoibm9uZSJ9fX1dLFsxNSwxNCwidCgxKSIsMCx7InNob3J0ZW4iOnsic291cmNlIjoyMCwidGFyZ2V0IjoyMH19XV0=
\[\begin{tikzcd}
	{\X(n)} & {\Y(n)} && {\X(n)} & {\Y(n)} \\
	{\X(1)} & {\Y(1)} && {\X(1)} & {\Y(1)}
	\arrow[""{name=0, anchor=center, inner sep=0}, "{f(n)}", curve={height=-12pt}, from=1-1, to=1-2]
	\arrow[""{name=1, anchor=center, inner sep=0}, "{g(n)}"', curve={height=12pt}, from=1-1, to=1-2]
	\arrow["{\X(p_i)}"', from=1-1, to=2-1]
	\arrow[""{name=2, anchor=center, inner sep=0}, "{\Y(p_i)}", from=1-2, to=2-2]
	\arrow["{f(n)}", curve={height=-12pt}, from=1-4, to=1-5]
	\arrow[""{name=3, anchor=center, inner sep=0}, "{\X(p_i)}"', from=1-4, to=2-4]
	\arrow["{\Y(p_i)}", from=1-5, to=2-5]
	\arrow["{g(1)}"', curve={height=12pt}, from=2-1, to=2-2]
	\arrow[""{name=4, anchor=center, inner sep=0}, "{g(1)}"', curve={height=12pt}, from=2-4, to=2-5]
	\arrow[""{name=5, anchor=center, inner sep=0}, "{f(1)}", curve={height=-12pt}, from=2-4, to=2-5]
	\arrow["{t(n)}", between={0.2}{0.8}, Rightarrow, from=0, to=1]
	\arrow["\simeq"{marking, allow upside down}, draw=none, from=2, to=3]
	\arrow["{t(1)}", between={0.2}{0.8}, Rightarrow, from=5, to=4]
\end{tikzcd}\]

\begin{remark} \label{remark: 2-dim faithful I}
    If both $\C$ and $\T$ are $(2,2)$-categories, the forgetful functor $\Mod_w(\T,\C)\to \C$ is locally faithful, i.e. the functors on hom-categories $\Hom_w(\X,\Y)\to \C(\X(1),\Y(1))$ are faithful. Indeed, having two modifications $t,s\colon f\Rrightarrow g$ in $\Hom_w(\X,\Y)$ with $s(1)=t(1)$, from the discussion above we necessarily obtain $s=t$. 
\end{remark}

\begin{remark}
    For $\C=\Cat_\infty$, we can use the marked $(\infty,2)$-categorical unstraightening equivalence \cite[Prop. 3.5.7]{AGH24} to view $\Mod_{l}(\T,\Cat)$, $\Mod_{l^*}(\T,\Cat)$ as full subcategories of certain $(\infty,2)$-categories $\mathsf{Fib}^l_{(0,1)}(\T)$, $\mathsf{Fib}^l_{1,1}(\T^{coop})$ of fibrations satisfying the Segal condition (\ref{eq: Segal condition}). More concretely, objects of $\mathsf{Fib}^l_{(0,1)}(\T)$ are functors of $(\infty,2)$-categories $\pi\colon \mathbb{X}\to\T$ such that the functor of underlying $\infty$-categories is a cocartesian fibration and the functors on hom-categories $\Hom_{\mathbb{X}}(x,y)\to \Hom_{\T}(\pi(x),\pi(y))$ are cartesian fibrations satisfying some coherences. A map $\mathbb{X}\to \mathbb{Y}$ is a functor of $(\infty,2)$-categories over $\T$ which preserves the cocartesian lifts of marked $1$-cells. Although this approach can be often useful when constructing $\T$-models, it will turn out more useful to view models as functors from $\T$, rather than fibrations over $\T$, to construct certain Day convolution structure in Section \ref{sec: day conv}.
\end{remark}

\begin{example}
    Consider $\T=\T_{mon}$. For the respective choices of $w\in \{s,l,l^*\}$,  the $(\infty,2)$-category $\Mod_w(\T,\Cat_\infty)$ is the $(\infty,2)$-category of monoidal $\infty$-categories, strict / lax / colax monoidal functors, and monoidal natural transformations. Although this presentation differs from the usual one using the formalism of $\infty$-operads where one would use $\Delta^{op}$ instead of $\T_{mon}$, one can check that it is equivalent, for example by the results of \cite{Ber18}. One way of phrasing it precisely is the following: precomposition with the inclusion $\Delta^{op}\to \T_{mon}$ induces an equivalence between finite product preserving functors out of $\T_{mon}$ and functors out of $\Delta^{op}$ satisfying the Segal condition.
\end{example}

\begin{example}
    As a more general family of examples we can consider Lawvere $(\infty,1)$-theories $\T_{E_n}$ such that $\Mod_w(\T,\Cat{\infty})$ is the $(\infty,2)$-category of $E_n$-monoidal categories, $E_n$-monoidal functors, and $E_n$-monoidal natural transformations. These theories exists by \cite[Appendix B]{GGN16}, and for $n=2$ or $n=\infty$, they are even $(2,1)$-theories.
\end{example}

\begin{remark} \label{rmk: pointwise products}
    All the categories $\Mod_w(\T,\Cat)$ have pointwise products. One can check it by  unstraightening and seeing that the pointwise product $\X\times \Y$ of two models correspond to the fiber product $\mathbb{X}\times_{\T} \mathbb{Y}$ of the respoective fibrations.
\end{remark}

Comparing to the $(\infty,1)$-categorical situation, the extra categorical dimension in the $(\infty,2)$-setting allows us to observe more phenomena. For any Lawvere $(\infty,2)$-theory, denote by $*$ the constant model $\T\to \Cat_{\infty}$ sending everything to $\Delta^0$ and identities therein. Then, for $\T=\T_{mon}$ and $\V\colon \T_{mon} \to \Cat_{\infty}$ a model, i.e. a monoidal $\infty$-category, we get that the $\infty$-category $\Hom_l(*,\V)$ of lax monoidal functors from $*$ to $\V$ and monoidal natural transformations is exactly the the $\infty$-category $\Alg(\V)$ of $\mathbb{E}_1$-monoids in $\V$ and homomorphisms between them. Motivated by that, consider the following definition.

\begin{definition}
    Let $\T$ be Lawvere $(\infty,2)$-theory, $\C$ an $(\infty,2)$-category with finite products, $\X\colon \T\to\C$ a model. We define the category $\IntAlg(\X)$ of \textit{internal algebras} in $\X$ to be the $(\infty,1)$-category of lax homomorphisms of $\T$-models $\Hom_l(*,\X)$. Similarly, define \textit{internal coalgebras} $\IntCoalg(\X)$ to be the $(\infty,1)$-category $\Hom_{l^*}(*,\X)$ of colax (or oplax) homomorphisms of $\T$-models.
\end{definition}

We are going to see more examples later in \ref{subsec: examples}, \ref{subsec: example segal}.

\begin{remark}
    No matter whether we want to study $(\infty,2)$-categories or $(2,2)$-categories on the semantic side, it is often enough to study $(2,2)$-categories on syntactic side: in fact, all our examples (e.g. $\T_{mon},$ $\T_{cmon}$) of Lawvere $(\infty,2)$-theories turn out to be $(2,2)$-categories. This principle has been often called  \textit{animation} in the recent literature \cite{CS24}: for a Lawvere $(1,1)$-theory $\T$, the animation of the category $\Mod(\T,\Set)$ is the $\infty$-category $\Mod(\T,\mathcal{S})$. Therefore, although our main syntactic result \ref{prop: miracle associativity} only applies to Lawvere $(2,2)$-theories, it is still interesting from the $\infty$-categorical standpoint.
\end{remark}

In the $(2,2)$-categorical setting, there are usually extra levels of strictness one wants to keep track of. We spend the rest of this subsection describing variants thereof.

\begin{definition}
    Define a \textit{Lawvere $(2,2)$-theory} to be a pair $(\T,\theta)$ where $\T$ is a $(2,2)$-category and $\theta\colon\F^{op}\to\T$ an identity-on-objects functor that preserves finite products. 
\end{definition}

Again, we view $\T$ as a marked $(2,2)$-category with marking generated by the image of $\theta$ (in the strict setting, we do not need to add isomorphisms).

\begin{definition}
    For any $w\in \{s,p,l,c,f\}$ and any $(2,2)$-category $\C$ with finite products, define $\Mod_w(\T,\C)$ to be the full $(2,2)$-subcategory of $\Fun_{s,w}(\T,\C)$ spanned by functors preserving finite products. We denote the hom-categories therein by $\Hom_w(-,-)$. For a model $\X\colon \T\to \C$, we define the category of internal algebras and coalgebras by $\IntAlg(\X):=\Hom_l(*,\X)$, $\IntCoalg(\X):=\Hom_{l^*}(*,\X)$.
\end{definition}

\begin{remark}
    There is an alternative notion which also deserves the name ``Lawvere $2$-theory'': an identity-on-objects $2$-functor $\Cat^{\omega}\to \T$ where $\Cat^\omega$ is the $2$-category of finitely presentable categories. This falls under the framework of enriched Lawvere theories developed by Power in \cite{Pow99}. To distinguish these two notions, one may call the one from definition above a \textit{discrete} Lawvere $2$-theory. However, as Power's notion will not be used or referred to throughout this paper, we will drop the adjective 'discrete'. Whenever we speak about finite powers, we always mean powers by finite discrete categories $n$.
\end{remark}

There are variants of the previous constructions -- similarly as in Section 1, we can speak about finite powers instead of finite products and we can also equip $\C$ with a choice of powers, resulting in a slightly different (but equivalent) $(2,2)$-category of models. 

The most important subtlety in the $(2,2)$-categorical case is whether to treat models as pseudofunctors or honest $2$-functors. For example, the functors of $(2,2)$-categories $\T_{mon}\to \Cat$ correspond to strict monoidal categories; a monoidal category would correspond to a pseudofunctor, strict on inert cells. One way of getting around this is to introduce a Lawvere $(2,1)$-theory $\T_{mon}^\flat$ for pseudomonoids together with a map $\T_{mon}^\flat\to \T_{mon}$ inducing an equivalence $\PsFun_{s,p}^\times(\T,-)\cong\Fun_{s,p}^\times(\T^\flat,-)$.  We will make use of this in \ref{subsec: examples}.

\subsection{2-dimensional commutativity}
In this subsection, we are going to introduce the notion of a $w$-commutativity for Lawvere $2$-theories. %This works equally well for both $(\infty,2)$- and $(2,2)$-theories. 

\begin{definition} \label{def: Lawvere tensor}
    For Lawvere $(\infty,2)$-theories $\T, \Ss$ (with their natural marking by inert $1$-cells) and any $(\infty,2)$-category $\C$ with finite products, define $\Mod_w(\T\otimes_{s,w}\Ss,\C)$ to be the full subcategory of $\Fun_{s,w}(\T\otimes_{s,w}\Ss,\C)$ spanned by functors preserving finite products in each variable. Namely, we require that for any objects $T\in \T$, $S\in \Ss$, both precompositions of $F\colon \T\otimes_{s,w}\Ss\to \C$ with $\T\simeq \T\otimes_{s,w}1\xrightarrow{1\otimes S} \T\otimes_{s,w}\Ss$, $\Ss\xrightarrow{T\otimes 1} \T\otimes_{s,w}\Ss$ preserve finite products.
\end{definition}

\begin{lemma} \label{lemma: sketchy equivalence}
    We have the following equivalences of categories:
    \begin{align}
        \Mod_{l}(\T\otimes_{s,l}\Ss,\C)&\simeq \Mod_{l}(\T,\Mod_{l}(\Ss,\C)), \\
        \Mod_{l^*}(\T\otimes_{s,l}\Ss,\C)&\simeq \Mod_{l^*}(\Ss,\Mod_{l^*}(\T,\C)),\\
        \Mod_{l}(\T,\Mod_{ l^*}(\Ss,\C))&\simeq \Mod_{l^*}(\Ss,\Mod_{l}(\T,\C)).
    \end{align}
\end{lemma}
\begin{proof}
    We prove only the first equivalence as the other are proved similarly. By Proposition \ref{prop: marked tensor-hom}, we have an equivalence $$\Fun_{s,l}(\T\otimes_{s,l}\Ss,\C)\simeq \Fun_{s,l}(\T,\Fun_{s,l}(\Ss,\C)).$$ By definition, $\Mod_{l}(\T\otimes_{s,l}\Ss,\C)$ embeds fully faithfully into the left hand side. By \cite[Cor. 2.8.13]{AGH24}, $\Mod_{l}(\T,\Mod_{l}(\Ss,\C))$ embeds fully faithfully into the right hand side as well. Therefore, to prove the equivalence, it is enough to compare the conditions defining the categories of models, which is straightforward.
\end{proof}

We need the previous lemma to make sense of the following definition. It applies to both $(\infty,2)$-categorical and $(2,2)$-categorical setting.
\begin{definition} \label{def: w-commutativity}
    Let $\T$ be any Lawvere $2$-theory, $w$ any weakness. A \textit{$w$-commutativity} on $\T$ is an $E_1$-monoid structure on $\T$ in $(\Cat_2^{\mathfrak{m}},\otimes_{s,w},1)$ such that the multiplication $\mu\colon \T\otimes_{s,w}\T\to \T$ preserves products in each variable and the unit is the distinguished object $1_{\T}$. Moreover, in case that $\otimes_{s,w}$ is symmetric, a symmetric $w$-commutativity on $\T$ is a symmetric (or $E_\infty$-) monoid structure.
\end{definition}

Denote by $U_1, U_2$ the forgetful functors $$\Mod_w(\T,\Mod_w(\T,\C))\simeq\Mod_{w}(\T\otimes_{s,w}\T,\C)\to \Mod_w(\T,\C)$$ obtained by composing the equivalence from Lemma \ref{lemma: sketchy equivalence} with evaluating at $1_{\T}$ in first or second variable, respectively.

\begin{proposition} \label{prop: lifting the models functor}
    For any $w$-commutativity $\mu\colon \T\otimes_{s,w}\T\to \T$ on $\T$, the functor $$\Mod_w(\T,\C)\to \Mod_w(\T,\Mod_w(\T,\C))$$ induced by precomposition with $\mu$ and the equivalence from Lemma \ref{lemma: sketchy equivalence} is a section to both $U_1$ and $U_2$.
\end{proposition}
\begin{proof}
    Denote by $u$ the map $1\to \T$ picking the distinguished object $1_{\T}$. By definition of the $w$-commutativity, we have $\mu\circ(1\otimes_{s,w} u)\simeq1\simeq\mu\circ(u\otimes_{s,w} 1)$. Precomposition with $1\otimes_{s,w} u$, $u\otimes_{s,w} 1$ induces functors $\Mod_w(\T\otimes_{s,w}\T,\C)\to \Mod_w(\T,\C)$ evaluating at $1_{\T}$ and the first or second variable and after composing with the equivalence from \ref{lemma: sketchy equivalence}, we get exactly $U_1$ and $U_2$, finishing the proof. 
\end{proof}
We denote the functor above by $\widetilde{(-)}$ and for a model $\X$, we think about $\widetilde{\X}$ as a lift of $\X$ along the forgetful functor. In particular, we are going to see in the next section that for any $w$-commutativity $\mu$, we necessarily have $\mu(m,n)=m\cdot n$. Hence, if $\alpha\colon m\to n$, $\beta\colon k\to l$ are $1$-cells in $\T$, we get that the image of the Gray $w$-cell $\X(\mu(\Sigma_{\alpha\beta}))$ yields a $2$-cell as follows (where $X=\X(1)$): 
% https://q.uiver.app/#q=WzAsNixbMCwwLCIoWF5rKV5tIl0sWzIsMCwiKFhebClebSJdLFswLDEsIihYXm0pXmsiXSxbMiwxLCIoWF5tKV5sIl0sWzAsMiwiWF5rIl0sWzIsMiwiWF5sIl0sWzAsMSwiXFxYKFxcYmV0YSlebSJdLFswLDIsIlxcY29uZyJdLFsxLDMsIlxcY29uZyIsMl0sWzIsNCwiXFxYKFxcYWxwaGEpXmsiXSxbMyw1LCJcXFgoXFxhbHBoYSlebCIsMl0sWzQsNSwiXFxYKFxcYmV0YSkiLDJdLFswLDQsIlxcWF5rKFxcYWxwaGEpIiwyLHsiY3VydmUiOjR9XSxbMSw1LCJcXFhebChcXGFscGhhKSIsMCx7ImN1cnZlIjotNH1dLFs2LDExLCJcXFgoXFxzaWdtYV9cXGFscGhhXFxiZXRhKSIsMCx7InNob3J0ZW4iOnsic291cmNlIjo0MCwidGFyZ2V0Ijo0MH19XV0=
\[\begin{tikzcd}
	{(X^k)^m} && {(X^l)^m} \\
	{(X^m)^k} && {(X^m)^l} \\
	{(X^k)^n} && {(X^l)^n}
	\arrow[""{name=0, anchor=center, inner sep=0}, "{\X(\beta)^m}", from=1-1, to=1-3]
	\arrow["\cong", from=1-1, to=2-1]
	\arrow["{\X^k(\alpha)}"', curve={height=24pt}, from=1-1, to=3-1]
	\arrow["\cong"', from=1-3, to=2-3]
	\arrow["{\X^l(\alpha)}", curve={height=-24pt}, from=1-3, to=3-3]
	\arrow["{\X(\alpha)^k}", from=2-1, to=3-1]
	\arrow["{\X(\alpha)^l}"', from=2-3, to=3-3]
	\arrow[""{name=1, anchor=center, inner sep=0}, "{\X(\beta)^n}"', from=3-1, to=3-3]
	\arrow["{\X(\sigma_{\alpha\beta})}", between={0.4}{0.6}, Rightarrow, from=0, to=1]
\end{tikzcd}\]
Letting $\alpha$ vary, this makes $\X(\beta)$ into a $w$-homomorphism of models $\widetilde{\X}(\beta)\colon\X^k\to \X^l$.

\begin{lemma} \label{rmk: 2-dim faithful}
    If $\T$ and $\C$ are $(2,2)$-categories and $\mu\colon\T\otimes_{s,w}\T\to\T$ a $w$-commutativity structure, the associated functor $\widetilde{(-)}\colon \Mod_w(\T,\C)\to \Mod_w(\T,\Mod_w(\T,\C))$ is locally fully faithful, i.e. for any models $\X,\Y$ in $\C$, the induced functor $\Hom_w(\X,\Y)\to \Hom_w(\widetilde{\X},\widetilde{\Y})$ is fully faithful. 
\end{lemma}

\begin{proof}
    By Remark \ref{remark: 2-dim faithful I}, any modification $\hat s\colon \tilde f\to \tilde g$ is necessarily of the form $\widetilde{\hat s(1)}$ as both of these have the same component at $1$. Faithfulness follows from being a section to the forgetful functors.
\end{proof}

\subsection{Marked (2,2)-sketches as $\mF$-sketches} In the $(2,2)$-categorical context, we can generalize our constructions from Lawvere theories to marked $(2,2)$-sketches. By insight of \cite{ABK24}, we can study those by methods of enriched category theory.
\begin{definition} \cite[Def. 5.1]{ABK24}
    Let $\V$ be a symmetric monoidal category. A (limit) \emph{$\V$-sketch} $\Ss$ comprises a $\V$-enriched category $\Ss_0$ equipped with a class $\Gamma_{\Ss}$ of weighted cones in $\Ss_0$. 
    \begin{equation}
		%\label{cone}
		\Big\{ \big(W_i \colon  J_i \to \V , \quad D_i \colon J_i \to \Ss_0 , \quad X_i \in  \Ss_0 , \quad \gamma_i \colon W_i \Rightarrow  \Ss_0(X_i, D_i{-}) \big) \Big\}_{i \in \Gamma_{\Ss}}
	\end{equation}    
    A morphism $F \colon  \Ss \to \T$ of $\V$-sketches is a $\V$-functor $F_0 \colon \Ss_0 \to \T_0$ sending weighted cones in $\Gamma_{\Ss}$ to those in $\Gamma_{\T}$. Limit sketches, their morphisms, and natural transformations form a 2-category $\V\text{-}\Sketch$. We denote by $\Mod(\Ss,\T)$ the full subcategory of $\Fun^{\V}(\Ss_0,\T_0)$ spanned by morphisms of $\V$-sketches. If $\Phi=\{W_k\colon J_k\to \V\}$ is a class of weights in $\V$, denote by $\V\text{-}\Sketch^{\Phi}$ the full subcategory of those $\V$-sketches $\Ss$ with all the weights for cones in $\Gamma_{\Ss}$ belonging to $\Phi.$
\end{definition}

Consider the full subcategory $\mF$ of the arrow category $\Cat^{\to}$ spanned by functors $\C'\to \C$ that are fully faithful and injective on objects, together with its cartesian monoidal structure. One can then interpret marked $(2,2)$-categories as $\mF$-enriched categories. Using the definition above, one can define marked $(2,2)$-sketches to be $\mF$-sketches. 

The tensor product $\otimes_{w',w}$ of marked $(2,2)$-categories lifts to $\mF$-sketches (\cite{ABK24}, 6.15); for $\mF$-sketches $\Ss_1, \Ss_2$, we can equip $\Ss_1\otimes_{w',w}\Ss_2$ with a sketch structure in such a way that for any sketch $\T$, an $\mF$-functor $\Ss_1\otimes_{w',w}\Ss_2\to \T$ is a model if and only if the precomposition with $\Ss_1\otimes_{w',w}1\xrightarrow{1\otimes S_2}\Ss_1\otimes_{w',w}\Ss_2$, $1\otimes_{w',w}\Ss_2\xrightarrow{S_1\otimes1 }\Ss_1\otimes_{w',w}\Ss_2$ gives us models $\Ss_1\to \T$, $\Ss_2\to \T$ for any objects $S_i$ of of $\Ss_i$. Note that this generalizes \ref{def: Lawvere tensor}. Defining $\Mod_{w',w}(\Ss,\T)$ as a full subcategory of $\Fun_{w',w}(\Ss,\T)$ spanned by models, one can make $(\mF\text{-}\Sketch, \otimes_{w',w},1)$ into a biclosed monoidal category again, generalizing \ref{lemma: sketchy equivalence}:% So, for example if $\T$ is a Lawvere $2$-theory viewed as an $\Psi$-$\mF$-sketch and $\C$ is a $2$-category with finite powers, then a model $\T\otimes_{s,w}\T\to \C$ is a functor of $2$-variables which is a model when restricted to either of the variables. There is a useful version of the previous proposition using the tensor product of sketches that we are going to exploit later: 

\begin{proposition} \cite[Prop. 7.1]{ABK24} \label{prop: swapping lax and colax}
    For any sketches $\Ss_1,$ $\Ss_2,$ $\T$ and any weaknesses $w', w\in W$, we have the following isomorphisms:  \begin{align}
        \Mod_{w',w}(\Ss_1\otimes_{w',w}\Ss_2,\T)&\simeq \Mod_{w',w}(\Ss_1,\Mod_{w',w}(\Ss_2,\T)), \\
        \Mod_{w'^*,w^*}(\Ss_1\otimes_{w',w}\Ss_2,\T)&\simeq \Mod_{w'^*,w^*}(\Ss_2,\Mod_{w'^*,w^*}(\Ss_1,\T)),\\
        \Mod_{w',w}(\Ss_1,\Mod_{w'^*,w^*}(\Ss_2,\T))&\simeq \Mod_{w'^*,w^*}(\Ss_2,\Mod_{w',w}(\Ss_1,\T)).
    \end{align}
\end{proposition}

%To upgrade this to a statement about $\Psi$-$\mF$-sketches, we need to argue that $\Mod_{w,s}(-,-)$ is closed under the property of being an $\Psi$-$\mF$-sketch. This follows in bigger generality from \cite[A.2]{ABK24}. 

\begin{notation}
    For any pair of $\mF$-sketches $\Ss$, $\T$, put $\Mod_w(\Ss,\T):=\Mod_{s,w}(\Ss,\T)$. We denote the hom-categories in $\Mod_w(\Ss,\T)$ by $\Hom_w(-,-)$.
\end{notation}

The upshot of working with sketches is that we can make the semantic side into a sketch as well. Namely, for any class $\Phi$ of weights, we can make any $\Phi$-complete $(2,2)$-category $\C$ into an $\mF$-sketch by choosing the collection of weighted cones $\Gamma_{\C}$ to be all weighted $\Phi$-limit cones and marking all the maps appearing in these cones. Then, the objects in the internal hom $\Mod_w(\Ss, \C)$ of sketches are exactly the functors $\Ss\to \C$ which sends the weighted cones in $\Gamma_{\Ss}$ weighted limit cones in $\C$. In particular, if $\Ss$ is a Lawvere $(2,2)$-theory $(\T,\theta)$, viewed as an $\mF$-sketch with marked cells being those in the image of $\theta$, and cones being all cones for finite products (alternatively, weighted cones for finite powers), we get that $\Mod_w(\T,\C)$ as an internal hom in $\mF$-sketches agrees with $\Mod_w(\T,\C)$ defined in \ref{subsec: Lawvere theories}. Indeed, although we now consider more restrictive marking on $\C$ than in the subsection, the product preservation force any model $\T\to \C$ to respect this marking.

\begin{definition} \label{def: w-commutativity II}
    Let $\Ss$ be an $\Phi$-$\mF$-sketch (i.e., marked $\Phi$-$(2,2)$-sketch), $\X\colon \Ss\to \C$ a model. 
    \begin{enumerate}
        \item Define $\IntAlg(\X):=\Hom_l(*,\X)$, $\IntCoalg(\X):=\Hom_{l^*}(*,\X)$.
        \item Define a $w$-commutativity on $\Ss$ to be a pseudomonoid structure $\mu\colon \Ss\otimes_{s,w}\Ss\to \Ss$, $u\colon 1\to \Ss$ in $(\mF\text{-}\Sketch,\otimes_{s,w},1)$.
    \end{enumerate}
\end{definition}

We can easily see that this generalizes our definitions for Lawvere $(2,2)$-theories. 

\begin{remark}
    A pleasant consequence of working with $\mF$-sketches in that we get a converse to Proposition \ref{prop: lifting the models functor}. Observe that $\Mod_{w}(\T,-)$ is an $(\Cat^{\mathfrak{m}}_{(2,2)},\otimes_{s,w},1)$-enriched presheaf on $\mF\text{-}\Sketch$. Denote by $\U_1,$ $\U_2$ the forgetful natural transformations of presheaves $\Mod_w(\T,\Mod_w(\T,-))\to \Mod_w(\T,-)$ induced by evaluation at $1_{\T}$ in first or second variable, respectively. Then by the enriched Yoneda lemma, a natural transformation which is a section to both $\U_1,\U_2$ is equivalent to a morphism of $\mF$-sketches $\mu\colon \T\otimes_{s,w}\T\to \T$ strictly unital with the unit $1_{\T}$. We will see in \ref{prop: miracle associativity} that this data uniquely extends to a $w$-commutativity structure on $\T$. 
\end{remark}

\subsection{Algebraic patterns}
In the $\infty$-categorical literature, there is a choice of a syntactic category used more often that Lawvere theories: algebraic patterns. Comparing to the subsection above, algebraic patterns are essentially marked sketches with extra structure. 

\begin{definition} \cite[Def. 2.1, 2.7]{CH21} An algebraic pattern is an $\infty$-category $\O$ equipped with the data of a factorization system $(\O^{\inert},\O^{\act})$ and a full subcategory $\O^{\el}\subseteq \O^{\inert}$. For any $\infty$-category $\C$ having all limits of the shape $\O^{\el}_{X/}:=\O^{\el}\times_{\O^{\inert}}\O^{\inert}_{X/}$, $X\in \O$, a Segal $\O$-object in $\C$ is a functor $F\colon \O\to \C$ such that for every $X\in \O$ the canonical map $F(X)\to \lim_{E\in \O^{\el}_{X/}}F(E)$ is an equivalence. 
\end{definition}

Note that the definition of a Segal object makes sense also for $(\infty,2)$-categories having all limits of said shapes; we call them $\O$-complete. We can view any algebraic pattern as a marked $(\infty,2)$-category (with all $2$-cells invertible) by marking the equivalences and inert $1$-cells (i.e., morphisms in $\O^{\inert}$). 

\begin{definition}
    For an algebraic pattern $\O$ and an $\O$-complete $(\infty,2)$-category, define $\Mod_{w}(\O,\C)$ to be the full $(\infty,2)$-subcategory of $\Fun_{s,w}(\O,\C)$ spanned by Segal objects.
\end{definition}

In Section $5$ of \textit{loc. cit.}, the category of algebraic patterns is defined and shown to have products. In particular, for patterns $\O,$ $\P$, we have $\Mod_s(\O\times \P, \C)\simeq \Mod_s(\O,\Mod_s(\P,\C))$. 
\begin{definition} \label{def: w-commutativity III}
    A commutativity structure on an algebraic pattern $\O$ is a monoidal structure on $(\O,\text{inert})$ in $(\Cat^{\mathfrak{m}}_{(\infty,2)},\times)$ such that $\mu\colon \O\times \O\to \O$ is a Segal morphism (Def. 4.2. in \textit{loc. cit.}) and the unit is an elementary object.
\end{definition}

The requirement that $\mu$ be a Segal morphism yields a lifting functor $\Mod_s(\O,\C)\xrightarrow{\mu^*}\Mod_s(\O\times \O,\C)\simeq \Mod_s(\O,\Mod_s(\O,\C))$ similarly as in the previous sections.

%\begin{example}    One can check that the category $\F_*$ has a commutativity structure with $1$ being a unit and $\mu(m,n)=m\cdot n$, \end{example}

\section{Syntax of $w$-commutativity} \label{sec: commutativity}

Recall that $w$-commutativity on a Lawvere $2$-theory $\T$ is an $E_1$-monoid structure $\mu\colon \T\otimes_{s,w}\to \T$, $u\colon 1\to \T$ such that $\mu$ preserves products in each variable and $u(1)=1_{\T}$. In particular, a part of the structure is to give $w$-cells $\mu(\Sigma_{\alpha\beta})$ as below for any $\alpha\colon m\to n$, $\beta\colon k\to l$.
% https://q.uiver.app/#q=WzAsNCxbMCwwLCJcXG11KG0sIGspIl0sWzAsMSwiXFxtdShuLCBrKSJdLFsxLDAsIlxcbXUobSwgbCkiXSxbMSwxLCJcXG11KG4sbCkiXSxbMCwxLCJcXG11KFxcYWxwaGEsIGspIiwyXSxbMCwyLCJcXG11KG0sIFxcYmV0YSkiXSxbMiwzLCJcXG11KFxcYWxwaGEsIGwpIl0sWzEsMywiXFxtdShuLCBcXGJldGEpIiwyXSxbNSw3LCJcXG11KFxcU2lnbWFfe1xcYWxwaGFcXGJldGF9KSIsMCx7Im9mZnNldCI6MSwic2hvcnRlbiI6eyJzb3VyY2UiOjQwLCJ0YXJnZXQiOjQwfX1dXQ==
\[\begin{tikzcd}
	{\mu(m, k)} & {\mu(m, l)} \\
	{\mu(n, k)} & {\mu(n,l)}
	\arrow[""{name=0, anchor=center, inner sep=0}, "{\mu(m, \beta)}", from=1-1, to=1-2]
	\arrow["{\mu(\alpha, k)}"', from=1-1, to=2-1]
	\arrow["{\mu(\alpha, l)}", from=1-2, to=2-2]
	\arrow[""{name=1, anchor=center, inner sep=0}, "{\mu(n, \beta)}"', from=2-1, to=2-2]
	\arrow["{\mu(\Sigma_{\alpha\beta})}", shift right, between={0.4}{0.6}, Rightarrow, from=0, to=1]
\end{tikzcd}\]

We will see that $\mu(m,n)=m\cdot n$ and thus we are filling the squares of the form \ref{diagram: comm} with come coherent system of $w$-cells. In fact, Lemma \ref{lemma: syntactic} will tell us that except filling these cells, $\mu$ is uniquely determined. Theorem \ref{prop: miracle associativity} will tell us even more: surprisingly, $\mu$ turns out to be associative ``for free'', without even specifying the associativity in the definition. To this end, and also for the further discussion, we define some weaker variants of $w$-commutativity. In the following, we view a Lawvere $(2,2)$-theory as a marked $(2,2)$-sketch. Write $u\colon 1\to \T$ for a map picking the object $1_{\T}$.

\begin{definition} Let $\rho\colon \Ss\to\T$ be a product-preserving functor mapping $1_{\Ss}$ to $1_{\T}$. A \textit{$w$-commutation} of $\Ss$ over $\T$ is a map $\mu\colon \Ss\otimes_{s,w}\Ss\to \T$ together with isomorphisms $\lambda^L\colon \mu (1\otimes u)\cong \rho$, $\lambda^R\colon \mu(u\otimes1)\cong \rho$ such that $\lambda^L_{1_{\T}}=\lambda^R_{1_{\T}}$. A \textit{magmal $w$-commutativity} on $\T$ is a $w$-commutation with respect to $\rho=\id\colon \T\to\T$.
\end{definition}

We see that a magmal $w$-commutativity on $\T$ is a $w$-commutativity without any associativity requirements. We will soon see that any magmal commutativity is strictly unital, i.e. we can choose $\lambda^L$, $\lambda^R$ to be identities.

\subsection{Syntactic exploration: Lawvere $(2,2)$-theories} We are going to see that if $\T$ is a Lawvere $(2,2)$-theory, the notion of a $w$-commutativity is quite rigid. In particular, if $\mu\colon \T\otimes_{s,w}\T\to \T$ is a $w$-commutativity, we are going to see that on objects, necessarily $\mu(m,n)=m\cdot n$. In fact, the only variability we have is in how we fill in the squares (\ref{diagram: comm}).

\begin{lemma}
    Let $\mu$ be any magmal $w$-commutativity on $\T$. Then $\theta\colon \F^{op}\to \T$ is a homomorphism of magmas with respect to $\mu$ and the cocartesian monoidal structure on $\F^{op}$.
\end{lemma}
\begin{proof}
Recall we have the marking $I$ on $\T$ by inert cells. We are going to factor $\theta\colon \F^{op}\to \T$ as a composition of four maps $\F^{op}\to \T_I^{\le1}\to \T_I\to \T$ and show that each of them is a homomorphism. Here, $\T_I$ is the locally full $(2,2)$-subcategory of $\T$ spanned my marked $1$-cells, $\T_{I}^{\le1}$ its underlying $1$-category.

First, note that if we restrict $\mu$ to $\T_I\times \T_I\cong \T_I\otimes_{s,w}\T_I\subset \T\otimes_{s,w}\T$, the image will land in $\T_I$ as $\mu$ is a morphism of marked $(2,2)$-categories. Therefore, the inclusion $\T_I\to\T$ is a homomorphism of magmas. The inclusion $\T_I^{\le1}\to \T_I$ is a again a homomorphism of magmas for obvious reasons. Clearly, the map $\theta$ factors through $\T_I^{\le1}$ and the latter still have finite products (as all the product projections in $\T$ comes from $\F^{op}$). We will conclude by showing that either $\T_I^{\le1}$ is contractible (and hence $\theta$ is a homomorphism for trivial reasons) or the map $\F^{op}\to \T_I^{\le1}$ is an equivalence. 

The restriction $\theta$ to a functor $\F^{op}\to \T_I^{\le1}$ is bijective on objects and full (as $\T_{I}^{\le1}$ is the essential image of $\F^{op}$ in $\T$). Suppose it is not faithful, i.e. we have an inert map in $\T$ which comes from two different maps in $\F^{op}$. WLOG we can assume that it is a map $n\to 1$, so we get that two projections from the $n$-th power of $1$ are the same. But that implies (in any category with products, really) that $1$ is the terminal object, and so all the objects of $\T$, being precisely the powers of $1$, are isomorphic, and $\T_I^{\le1}$ becomes contractible. 
\end{proof}

Leveraging the fact that $\theta$ is a homomorphism of magmas $(\F^{op},\text{cocartesian})\to (\T,\mu)$, we obtain the following:
\begin{itemize}
    \item $u(*)=1$, $\mu(m,n)=m\cdot n$, (previous lemma),
    \item for any $1$-cell $\alpha$ in $\T$, $\mu(\id_n,\alpha)=n\cdot \alpha,$ $\mu(\alpha,\id_n)=\alpha\cdot n$, (preservation of powers), 
    \item for any $2$-cell $s$ in $\T$,  $\mu(\id_n,s)=n\cdot s,$ $\mu(s,\id_n)=s\cdot n$ (preservation of powers).
\end{itemize} 

This observation can be improved using $f$-commutativity.

\begin{lemma} \label{lemma: f-commutativity}
    There exists a unique magmal $f$-commutativity structure $\mu\colon \T\otimes_{s,f}\T\to \T$, completely specified by the requirements above. Moreover, it is associative.
\end{lemma}

\begin{proof}
    The criteria above define a map $\T\otimes_{f}\T\to \T$, which in fact factors through $\T\otimes_{f}\T\to \T\otimes_{s,f}\T$ as all the squares (\ref{diagram: comm}) with $\alpha$ or $\beta$ inert commute. So, there exists a magmal $f$-commutativity, and there is at most one satisfying the criteria above, so it is unique.
    
    To see that $\mu$ is associative, we equip it with associators $(m\cdot n)\cdot k\cong m\cdot (n\cdot k)$ coming from $\F^{op}$. Again, these are clearly the only compatible choice. We just need to show that they are natural with respect to all $1$- and $2$-cells in $\T$. For a $1$-cell $\alpha\colon n\to n'$, we need to show that the following square (and its variants with $\alpha$ at the first or the third position) commutes:

    % https://q.uiver.app/#q=WzAsNCxbMCwwLCIobVxcY2RvdCBuKVxcY2RvdCBrIl0sWzIsMCwiKG1cXGNkb3QgbicpXFxjZG90IGsiXSxbMiwxLCJtXFxjZG90IChuJ1xcY2RvdCBrKSJdLFswLDEsIm1cXGNkb3QgKG5cXGNkb3QgaykiXSxbMCwzLCJcXGNvbmciLDJdLFsxLDIsIlxcY29uZyJdLFswLDEsIihtXFxjZG90IFxcYWxwaGEpXFxjZG90IGsiXSxbMywyLCJtXFxjZG90IChcXGFscGhhXFxjZG90IGspIiwyXV0=
\begin{equation} \label{dia: assoc  comp in T}\begin{tikzcd}
	{(m\cdot n)\cdot k} && {(m\cdot n')\cdot k} \\
	{m\cdot (n\cdot k)} && {m\cdot (n'\cdot k)}
	\arrow["{(m\cdot \alpha)\cdot k}", from=1-1, to=1-3]
	\arrow["\cong"', from=1-1, to=2-1]
	\arrow["\cong", from=1-3, to=2-3]
	\arrow["{m\cdot (\alpha\cdot k)}"', from=2-1, to=2-3]
\end{tikzcd}\end{equation}

To see that, first suppose $\alpha$ is in fact at the third position (as in the middle square below); then the commutativity follows from coherences for finite powers in $\T$. Then, we write the diagram above as a composition of three squares:

% https://q.uiver.app/#q=WzAsOCxbMCwwLCIobVxcY2RvdCBuKVxcY2RvdCBrIl0sWzAsMSwibVxcY2RvdCAoblxcY2RvdCBrKSJdLFsyLDEsIm1cXGNkb3QgKGtcXGNkb3QgbicpIl0sWzIsMCwia1xcY2RvdCAobVxcY2RvdCBuJykiXSxbMywwLCIobVxcY2RvdCBuJylcXGNkb3QgayJdLFszLDEsIm1cXGNkb3QgKG4nXFxjZG90IGspIl0sWzEsMCwia1xcY2RvdCAobVxcY2RvdCBuKSJdLFsxLDEsIm1cXGNkb3QgKGtcXGNkb3QgbikiXSxbMCwxLCJcXGNvbmciLDJdLFszLDIsIlxcY29uZyJdLFszLDQsIlxcY29uZyJdLFs0LDUsIlxcY29uZyJdLFsyLDUsIlxcY29uZyIsMl0sWzYsMywia1xcY2RvdCAobVxcY2RvdCBcXGFscGhhKSJdLFs3LDIsIm1cXGNkb3QgKGtcXGNkb3QgXFxhbHBoYSkiLDJdLFsxLDcsIlxcY29uZyIsMl0sWzAsNiwiXFxjb25nIl0sWzYsNywiXFxjb25nIiwyXV0=
\[\begin{tikzcd}
	{(m\cdot n)\cdot k} & {k\cdot (m\cdot n)} & {k\cdot (m\cdot n')} & {(m\cdot n')\cdot k} \\
	{m\cdot (n\cdot k)} & {m\cdot (k\cdot n)} & {m\cdot (k\cdot n')} & {m\cdot (n'\cdot k)}
	\arrow["\cong", from=1-1, to=1-2]
	\arrow["\cong"', from=1-1, to=2-1]
	\arrow["{k\cdot (m\cdot \alpha)}", from=1-2, to=1-3]
	\arrow["\cong"', from=1-2, to=2-2]
	\arrow["\cong", from=1-3, to=1-4]
	\arrow["\cong", from=1-3, to=2-3]
	\arrow["\cong", from=1-4, to=2-4]
	\arrow["\cong"', from=2-1, to=2-2]
	\arrow["{m\cdot (k\cdot \alpha)}"', from=2-2, to=2-3]
	\arrow["\cong"', from=2-3, to=2-4]
\end{tikzcd}\]

We already commented on the commutativity of the middle square. The right and left squares are in the image of $\theta$ and one can easily show that they commute in $\F^{op}$ (coherences for coproducts). The case where $\alpha$ would be at the first position in (\ref{dia: assoc  comp in T}) is similar.

Finally, let $t\colon \alpha\Rightarrow \beta$ be any $2$-cell in $\T$. We need to show the following equation holds, together with its variants with $\alpha,\beta, t$ at the second or third position.

% https://q.uiver.app/#q=WzAsOCxbMCwwLCIoblxcY2RvdCBrKVxcY2RvdCBsIl0sWzAsMSwiblxcY2RvdCAoa1xcY2RvdCBsKSJdLFsxLDAsIihuJ1xcY2RvdCBrKVxcY2RvdCBsIl0sWzEsMSwibidcXGNkb3QgKGtcXGNkb3QgbCkiXSxbMiwwLCIoblxcY2RvdCBrKVxcY2RvdCBsIl0sWzIsMSwiblxcY2RvdCAoa1xcY2RvdCBsKSJdLFszLDAsIihuJ1xcY2RvdCBrKVxcY2RvdCBsIl0sWzMsMSwibidcXGNkb3QgKGtcXGNkb3QgbCkiXSxbMCwxLCJcXGNvbmciLDJdLFswLDIsIihcXGFscGhhXFxjZG90IGspXFxjZG90IGwiLDAseyJjdXJ2ZSI6LTN9XSxbMCwyLCIoXFxiZXRhXFxjZG90IGspXFxjZG90IGwiLDIseyJjdXJ2ZSI6M31dLFsxLDMsIlxcYmV0YVxcY2RvdCAoa1xcY2RvdCBsKSIsMix7ImN1cnZlIjozfV0sWzIsMywiXFxjb25nIl0sWzUsNywiXFxiZXRhXFxjZG90IChrXFxjZG90IGwpIiwyLHsiY3VydmUiOjN9XSxbNCw2LCIoXFxhbHBoYVxcY2RvdCBrKVxcY2RvdCBsIiwwLHsiY3VydmUiOi0zfV0sWzQsNSwiXFxjb25nIiwyXSxbNiw3LCJcXGNvbmciXSxbNSw3LCJcXGFscGhhXFxjZG90IChrXFxjZG90IGwpIiwyLHsiY3VydmUiOi0zfV0sWzksMTAsIih0XFxjZG90IGspXFxjZG90IGwiLDFdLFsxNywxMywidFxcY2RvdCAoa1xcY2RvdCBsw7oiLDFdLFsxMiwxNSwiPSIsMyx7InNob3J0ZW4iOnsic291cmNlIjoyMCwidGFyZ2V0IjoyMH0sInN0eWxlIjp7ImJvZHkiOnsibmFtZSI6Im5vbmUifSwiaGVhZCI6eyJuYW1lIjoibm9uZSJ9fX1dXQ==
\[\begin{tikzcd}
	{(n\cdot k)\cdot l} & {(n'\cdot k)\cdot l} & {(n\cdot k)\cdot l} & {(n'\cdot k)\cdot l} \\
	{n\cdot (k\cdot l)} & {n'\cdot (k\cdot l)} & {n\cdot (k\cdot l)} & {n'\cdot (k\cdot l)}
	\arrow[""{name=0, anchor=center, inner sep=0}, "{(\alpha\cdot k)\cdot l}", curve={height=-18pt}, from=1-1, to=1-2]
	\arrow[""{name=1, anchor=center, inner sep=0}, "{(\beta\cdot k)\cdot l}"', curve={height=18pt}, from=1-1, to=1-2]
	\arrow["\cong"', from=1-1, to=2-1]
	\arrow[""{name=2, anchor=center, inner sep=0}, "\cong", from=1-2, to=2-2]
	\arrow["{(\alpha\cdot k)\cdot l}", curve={height=-18pt}, from=1-3, to=1-4]
	\arrow[""{name=3, anchor=center, inner sep=0}, "\cong"', from=1-3, to=2-3]
	\arrow["\cong", from=1-4, to=2-4]
	\arrow["{\beta\cdot (k\cdot l)}"', curve={height=18pt}, from=2-1, to=2-2]
	\arrow[""{name=4, anchor=center, inner sep=0}, "{\beta\cdot (k\cdot l)}"', curve={height=18pt}, from=2-3, to=2-4]
	\arrow[""{name=5, anchor=center, inner sep=0}, "{\alpha\cdot (k\cdot l)}", curve={height=-18pt}, from=2-3, to=2-4]
	\arrow["{(t\cdot k)\cdot l}"{description}, Rightarrow, from=0, to=1]
	\arrow["{=}"{marking, allow upside down}, draw=none, from=2, to=3]
	\arrow["{t\cdot (k\cdot l)}"{description}, Rightarrow, from=5, to=4]
\end{tikzcd}\]
The proof goes through analogously to the case of $1$-cells: if $\alpha,\beta, t$ are at the third position, then the equation holds via coherences for powers in $\T$, and for the other two situations we can reduce to this case using some squares in $\F^{op}$ that commute.
\end{proof}

\begin{remark}
    Notice that from the previous lemma it follows that any $w$-commutativity structure on a Lawvere $2$-theory $\T$ is in fact strictly unital, as we have $n\cdot 1=1$. It is also almost never strictly associative (only for Lawvere theories with $0\cong 1$) as the associator isomorphisms $m\cdot(n\cdot k)\cong (m\cdot n)\cdot k$ in $\F^{op}$ are in general nontrivial permutations of $mnk$.
\end{remark}

We can formulate the observations made above as follows: in order to define magmal $w$-commutativity on $\T$ for $w\in\{s,p,l,c\}$, it in fact suffices to fill in the squares featuring the Gray $w$-cells $\Sigma_{\alpha\beta}$ as below in a coherent way.

% https://q.uiver.app/#q=WzAsNCxbMCwwLCJtXFxjZG90IGsiXSxbMCwxLCJuXFxjZG90IGsiXSxbMSwwLCJtXFxjZG90IGwiXSxbMSwxLCJuXFxjZG90IGwiXSxbMCwxLCJcXGFscGhhXFxjZG90IGsiLDJdLFswLDIsIm1cXGNkb3QgXFxiZXRhIl0sWzIsMywiXFxhbHBoYVxcY2RvdCBsIl0sWzEsMywiblxcY2RvdCBcXGJldGEiLDJdLFs1LDcsIlxcbXUoXFxTaWdtYV97XFxhbHBoYVxcYmV0YX0pIiwwLHsib2Zmc2V0IjoxLCJzaG9ydGVuIjp7InNvdXJjZSI6NDAsInRhcmdldCI6NDB9fV1d
\[\begin{tikzcd}
	{m\cdot k} & {m\cdot l} \\
	{n\cdot k} & {n\cdot l}
	\arrow[""{name=0, anchor=center, inner sep=0}, "{m\cdot \beta}", from=1-1, to=1-2]
	\arrow["{\alpha\cdot k}"', from=1-1, to=2-1]
	\arrow["{\alpha\cdot l}", from=1-2, to=2-2]
	\arrow[""{name=1, anchor=center, inner sep=0}, "{n\cdot \beta}"', from=2-1, to=2-2]
	\arrow["{\mu(\Sigma_{\alpha\beta})}", shift right, between={0.4}{0.6}, Rightarrow, from=0, to=1]
\end{tikzcd}\]

We are going to make that precise in the lemma below.

\begin{notation}
   For $\alpha\colon m\to n$, $\beta\colon k\to l$ in $\T$, we denote by $\beta\otimes \alpha$ the composition $m\cdot k\xrightarrow{\alpha\cdot k}n\cdot k\xrightarrow{n\cdot \beta}n\cdot l$. 
\end{notation} 

\begin{lemma} \label{lemma: syntactic}
    (Syntactic criterion for $w$-commutativity.) To give a magmal $w$-commutativity structure $\mu\colon \T\otimes_{s,w}\T\to \T$ on $\T$ amounts to specifying a family of $w$-cells $\sigma_{\alpha\beta}\colon \alpha\otimes \beta\to \beta\otimes \alpha$ for all $1$-cells $\alpha,\beta$ in $\T$, satisfying the following coherence equations for any $1$-cells $\alpha\colon \beta,\gamma,\gamma'$ and any $2$-cell $s\colon \gamma'\to \gamma$ in $\T$:
    \begin{enumerate}
        \item $\sigma_{\alpha\beta}=1$ whenever $\alpha$ or $\beta$ is inert,
        \item for any $\alpha\colon m\to n$, denoting by $p_i$ the $i$-th projection map $n\to 1$, we have $\sigma_{\alpha\beta}=(\sigma_{(p_i\alpha)\beta})_i$ and similarly in the second variable,
        \item for any composable $\alpha\colon m\to n$, $\beta\colon n \to k$ and any $\gamma\colon p\to q$, the following equation (and its analog in the second variable):
        % https://q.uiver.app/#q=WzAsMTAsWzAsMCwicFxcY2RvdCBtIl0sWzEsMCwicFxcY2RvdCBuIl0sWzIsMCwicFxcY2RvdCBrIl0sWzAsMSwicVxcY2RvdCBtIl0sWzEsMSwicVxcY2RvdCBuIl0sWzIsMSwicVxcY2RvdCBrIl0sWzQsMCwicFxcY2RvdCBtIl0sWzUsMCwicFxcY2RvdCBrIl0sWzQsMSwicVxcY2RvdCBtIl0sWzUsMSwicVxcY2RvdCBrIl0sWzAsMywiXFxnYW1tYVxcY2RvdCBtIiwyXSxbMSw0LCJcXGdhbW1hXFxjZG90IG4iLDFdLFsyLDUsIlxcZ2FtbWFcXGNkb3QgayJdLFswLDEsInBcXGNkb3QgXFxhbHBoYSJdLFsxLDIsInBcXGNkb3QgXFxiZXRhIl0sWzQsNSwicVxcY2RvdCBcXGJldGEiLDJdLFszLDQsInFcXGNkb3QgXFxhbHBoYSIsMl0sWzYsOCwiXFxnYW1tYVxcY2RvdCBtIiwyXSxbNyw5LCJcXGdhbW1hXFxjZG90IGsiXSxbNiw3LCJwXFxjZG90KFxcYmV0YVxcYWxwaGEpIl0sWzgsOSwicVxcY2RvdCAoXFxiZXRhXFxhbHBoYSkiLDJdLFsxMywxNiwiXFxzaWdtYV97XFxhbHBoYVxcZ2FtbWF9IiwyLHsic2hvcnRlbiI6eyJzb3VyY2UiOjQwLCJ0YXJnZXQiOjQwfX1dLFsxNCwxNSwiXFxzaWdtYV97XFxiZXRhXFxnYW1tYX0iLDAseyJzaG9ydGVuIjp7InNvdXJjZSI6NDAsInRhcmdldCI6NDB9fV0sWzE5LDIwLCJcXHNpZ21hX3soXFxiZXRhXFxhbHBoYSlcXGdhbW1hfSIsMCx7InNob3J0ZW4iOnsic291cmNlIjo0MCwidGFyZ2V0Ijo0MH19XSxbMTIsMTcsIj0iLDMseyJzaG9ydGVuIjp7InNvdXJjZSI6MjAsInRhcmdldCI6MjB9LCJzdHlsZSI6eyJib2R5Ijp7Im5hbWUiOiJub25lIn0sImhlYWQiOnsibmFtZSI6Im5vbmUifX19XV0=
\begin{equation}
    \label{dia: gray1}
\begin{tikzcd}
	{p\cdot m} & {p\cdot n} & {p\cdot k} && {p\cdot m} & {p\cdot k} \\
	{q\cdot m} & {q\cdot n} & {q\cdot k} && {q\cdot m} & {q\cdot k}
	\arrow[""{name=0, anchor=center, inner sep=0}, "{p\cdot \alpha}", from=1-1, to=1-2]
	\arrow["{\gamma\cdot m}"', from=1-1, to=2-1]
	\arrow[""{name=1, anchor=center, inner sep=0}, "{p\cdot \beta}", from=1-2, to=1-3]
	\arrow["{\gamma\cdot n}"{description}, from=1-2, to=2-2]
	\arrow[""{name=2, anchor=center, inner sep=0}, "{\gamma\cdot k}", from=1-3, to=2-3]
	\arrow[""{name=3, anchor=center, inner sep=0}, "{p\cdot(\beta\alpha)}", from=1-5, to=1-6]
	\arrow[""{name=4, anchor=center, inner sep=0}, "{\gamma\cdot m}"', from=1-5, to=2-5]
	\arrow["{\gamma\cdot k}", from=1-6, to=2-6]
	\arrow[""{name=5, anchor=center, inner sep=0}, "{q\cdot \alpha}"', from=2-1, to=2-2]
	\arrow[""{name=6, anchor=center, inner sep=0}, "{q\cdot \beta}"', from=2-2, to=2-3]
	\arrow[""{name=7, anchor=center, inner sep=0}, "{q\cdot (\beta\alpha)}"', from=2-5, to=2-6]
	\arrow["{\sigma_{\gamma\alpha}}"', between={0.4}{0.6}, Rightarrow, from=0, to=5]
	\arrow["{\sigma_{\gamma\beta}}", between={0.4}{0.6}, Rightarrow, from=1, to=6]
	\arrow["{=}"{marking, allow upside down}, draw=none, from=2, to=4]
	\arrow["{\sigma_{\gamma(\beta\alpha)}}", between={0.4}{0.6}, Rightarrow, from=3, to=7]
\end{tikzcd}\end{equation}

    \item for any $\alpha\colon m\to n$, $\gamma,\gamma'\colon p\to q$ and $s\colon \gamma\to \gamma'$, the following equations:
    % https://q.uiver.app/#q=WzAsOCxbMCwwLCJwXFxjZG90IG0iXSxbMCwxLCJxXFxjZG90IG0iXSxbNCwwLCJwXFxjZG90IG0iXSxbNCwxLCJxXFxjZG90IG0iXSxbMiwwLCJwXFxjZG90IG4iXSxbMiwxLCJxXFxjZG90IG4iXSxbNiwwLCJwXFxjZG90IG4iXSxbNiwxLCJxXFxjZG90IG4iXSxbMCwxLCJcXGdhbW1hJ1xcY2RvdCBtIiwwLHsiY3VydmUiOi0yfV0sWzIsMywiXFxnYW1tYVxcY2RvdCBtIiwwLHsiY3VydmUiOjJ9XSxbMSw1LCJxXFxjZG90IFxcYWxwaGEiLDJdLFswLDQsInBcXGNkb3QgXFxhbHBoYSJdLFs0LDUsIlxcZ2FtbWEnXFxjZG90IG4iLDAseyJjdXJ2ZSI6LTJ9XSxbMCwxLCJcXGdhbW1hXFxjZG90IG0iLDIseyJjdXJ2ZSI6Mn1dLFsyLDYsInBcXGNkb3QoXFxiZXRhXFxhbHBoYSkiXSxbMyw3LCJxXFxjZG90IChcXGJldGFcXGFscGhhKSIsMl0sWzYsNywiXFxnYW1tYSdcXGNkb3QgbiIsMCx7ImN1cnZlIjotMn1dLFs2LDcsIlxcZ2FtbWFcXGNkb3QgbiIsMix7ImN1cnZlIjoyfV0sWzExLDEwLCJcXHNpZ21hX3tcXGFscGhhXFxnYW1tYSd9IiwwLHsic2hvcnRlbiI6eyJzb3VyY2UiOjQwLCJ0YXJnZXQiOjQwfX1dLFs4LDEzLCJzXFxjZG90IG0iLDIseyJzaG9ydGVuIjp7InNvdXJjZSI6MjAsInRhcmdldCI6MjB9fV0sWzE0LDE1LCJcXHNpZ21hX3tcXGFscGhhXFxnYW1tYX0iLDIseyJzaG9ydGVuIjp7InNvdXJjZSI6NDAsInRhcmdldCI6NDB9fV0sWzE2LDE3LCJzXFxjZG90IG4iLDIseyJzaG9ydGVuIjp7InNvdXJjZSI6MjAsInRhcmdldCI6MjB9fV0sWzEyLDksIj0iLDMseyJsYWJlbF9wb3NpdGlvbiI6NzAsInNob3J0ZW4iOnsic291cmNlIjoyMCwidGFyZ2V0IjoyMH0sInN0eWxlIjp7ImJvZHkiOnsibmFtZSI6Im5vbmUifSwiaGVhZCI6eyJuYW1lIjoibm9uZSJ9fX1dXQ==
\begin{equation} \label{dia: gray2}
   \begin{tikzcd}
	{p\cdot m} && {p\cdot n} && {p\cdot m} && {p\cdot n} \\
	{q\cdot m} && {q\cdot n} && {q\cdot m} && {q\cdot n}
	\arrow[""{name=0, anchor=center, inner sep=0}, "{p\cdot \alpha}", from=1-1, to=1-3]
	\arrow[""{name=1, anchor=center, inner sep=0}, "{\gamma'\cdot m}", curve={height=-12pt}, from=1-1, to=2-1]
	\arrow[""{name=2, anchor=center, inner sep=0}, "{\gamma\cdot m}"', curve={height=12pt}, from=1-1, to=2-1]
	\arrow[""{name=3, anchor=center, inner sep=0}, "{\gamma'\cdot n}", curve={height=-12pt}, from=1-3, to=2-3]
	\arrow[""{name=4, anchor=center, inner sep=0}, "{p\cdot(\beta\alpha)}", from=1-5, to=1-7]
	\arrow[""{name=5, anchor=center, inner sep=0}, "{\gamma\cdot m}"', curve={height=12pt}, from=1-5, to=2-5]
	\arrow[""{name=6, anchor=center, inner sep=0}, "{\gamma'\cdot n}", curve={height=-12pt}, from=1-7, to=2-7]
	\arrow[""{name=7, anchor=center, inner sep=0}, "{\gamma\cdot n}"', curve={height=12pt}, from=1-7, to=2-7]
	\arrow[""{name=8, anchor=center, inner sep=0}, "{q\cdot \alpha}"', from=2-1, to=2-3]
	\arrow[""{name=9, anchor=center, inner sep=0}, "{q\cdot (\beta\alpha)}"', from=2-5, to=2-7]
	\arrow["{s\cdot m}", between={0.2}{0.8}, Rightarrow, from=2, to=1]
	\arrow["{\sigma_{\gamma'\alpha}}", between={0.4}{0.6}, Rightarrow, from=0, to=8]
	\arrow["{=}"{marking, allow upside down, pos=0.5}, draw=none, from=3, to=5]
	\arrow["{\sigma_{\gamma\alpha}}"', between={0.4}{0.6}, Rightarrow, from=4, to=9]
	\arrow["{s\cdot n}", between={0.2}{0.8}, Rightarrow, from=7, to=6]
\end{tikzcd} 
\end{equation}

% https://q.uiver.app/#q=WzAsOCxbMCwwLCJtXFxjZG90IHAiXSxbMSwwLCJtXFxjZG90IHEiXSxbMCwxLCJuXFxjZG90IHAiXSxbMSwxLCJuXFxjZG90IHEiXSxbMywwLCJtXFxjZG90IHEiXSxbMywxLCJuXFxjZG90IHAiXSxbNCwwLCJtXFxjZG90IHEiXSxbNCwxLCJuXFxjZG90IHEiXSxbMCwyLCJcXGFscGhhXFxjZG90IHAiLDJdLFsxLDMsIlxcYWxwaGFcXGNkb3QgcSJdLFswLDEsIiBtXFxjZG90XFxnYW1tYSIsMCx7ImN1cnZlIjotMn1dLFswLDEsIiBtXFxjZG90XFxnYW1tYSciLDIseyJjdXJ2ZSI6Mn1dLFsyLDMsIiBuXFxjZG90XFxnYW1tYSciLDIseyJjdXJ2ZSI6Mn1dLFs0LDUsIlxcYWxwaGFcXGNkb3QgcCIsMl0sWzYsNywiXFxhbHBoYVxcY2RvdCBxIl0sWzUsNywiIG5cXGNkb3RcXGdhbW1hJyIsMix7ImN1cnZlIjoyfV0sWzUsNywiIG5cXGNkb3RcXGdhbW1hIiwwLHsiY3VydmUiOi0yfV0sWzQsNiwiIG1cXGNkb3RcXGdhbW1hIiwwLHsiY3VydmUiOi0yfV0sWzEwLDExLCJtXFxjZG90IHMiLDAseyJzaG9ydGVuIjp7InNvdXJjZSI6MjAsInRhcmdldCI6MjB9fV0sWzE3LDE2LCJcXHNpZ21hX3tcXGFscGhhXFxnYW1tYX0iLDAseyJzaG9ydGVuIjp7InNvdXJjZSI6NDAsInRhcmdldCI6NDB9fV0sWzE2LDE1LCJuXFxjZG90IHMiLDAseyJzaG9ydGVuIjp7InNvdXJjZSI6MjAsInRhcmdldCI6MjB9fV0sWzksMTMsIj0iLDMseyJzaG9ydGVuIjp7InNvdXJjZSI6MjAsInRhcmdldCI6MjB9LCJzdHlsZSI6eyJib2R5Ijp7Im5hbWUiOiJub25lIn0sImhlYWQiOnsibmFtZSI6Im5vbmUifX19XSxbMTEsMTIsIlxcc2lnbWFfe1xcYWxwaGFcXGdhbW1hJ30iLDAseyJzaG9ydGVuIjp7InNvdXJjZSI6NDAsInRhcmdldCI6NDB9fV1d
\begin{equation} \label{dia: gray3}\begin{tikzcd}
	{m\cdot p} & {m\cdot q} && {m\cdot q} & {m\cdot q} \\
	{n\cdot p} & {n\cdot q} && {n\cdot p} & {n\cdot q}
	\arrow[""{name=0, anchor=center, inner sep=0}, "{ m\cdot\gamma}", curve={height=-12pt}, from=1-1, to=1-2]
	\arrow[""{name=1, anchor=center, inner sep=0}, "{ m\cdot\gamma'}"', curve={height=12pt}, from=1-1, to=1-2]
	\arrow["{\alpha\cdot p}"', from=1-1, to=2-1]
	\arrow[""{name=2, anchor=center, inner sep=0}, "{\alpha\cdot q}", from=1-2, to=2-2]
	\arrow[""{name=3, anchor=center, inner sep=0}, "{ m\cdot\gamma}", curve={height=-12pt}, from=1-4, to=1-5]
	\arrow[""{name=4, anchor=center, inner sep=0}, "{\alpha\cdot p}"', from=1-4, to=2-4]
	\arrow["{\alpha\cdot q}", from=1-5, to=2-5]
	\arrow[""{name=5, anchor=center, inner sep=0}, "{ n\cdot\gamma'}"', curve={height=12pt}, from=2-1, to=2-2]
	\arrow[""{name=6, anchor=center, inner sep=0}, "{ n\cdot\gamma'}"', curve={height=12pt}, from=2-4, to=2-5]
	\arrow[""{name=7, anchor=center, inner sep=0}, "{ n\cdot\gamma}", curve={height=-12pt}, from=2-4, to=2-5]
	\arrow["{m\cdot s}", between={0.2}{0.8}, Rightarrow, from=0, to=1]
	\arrow["{\sigma_{\alpha\gamma'}}", between={0.4}{0.6}, Rightarrow, from=1, to=5]
	\arrow["{=}"{marking, allow upside down}, draw=none, from=2, to=4]
	\arrow["{\sigma_{\alpha\gamma}}", between={0.4}{0.6}, Rightarrow, from=3, to=7]
	\arrow["{n\cdot s}", between={0.2}{0.8}, Rightarrow, from=7, to=6]
\end{tikzcd}\end{equation}
    
    \end{enumerate}
For symmetric pseudocommutativity, we moreover require $\sigma_{\beta\alpha}\sigma_{\alpha\beta}=1$.
\end{lemma}

\begin{proof}
    (1) follows from the description of $\otimes_{s,w}$ and (2) follows from the fact that the magmal $w$-commutativity $\mu$ is required to preserve products in each variable. The equations in (3) and (4) are just coherence equations for the Gray cells $\Sigma_{\alpha\beta}$.
    Finally, for any marked $(2,2)$-categories $\A,\B$, the symmetry $s\colon \A\otimes_{s,p} \B\xrightarrow{\sim} \B\otimes_{s,p} \A$ swaps sends $(a,b)$ to $(b,a)$, $(f,1)$ to $(1,f)$ etc. and the cell $\Sigma_{fg}$ gets mapped to $\Sigma_{gf}^{-1}$.
    % https://q.uiver.app/#q=WzAsOCxbMCwwLCIoYSxiKSJdLFsxLDAsIihhLGInKSJdLFswLDEsIihhJyxiKSJdLFsxLDEsIihhJyxiJykiXSxbMywwLCIoYixhKSJdLFszLDEsIihiJyxhKSJdLFs0LDAsIihiLGEnKSJdLFs0LDEsIihiJyxhJykiXSxbMCwxLCIoMSxnKSJdLFswLDIsIihmLDEpIiwyXSxbMiwzLCIoMSxnKSIsMl0sWzEsMywiKGYsMSkiXSxbNCw1LCIoZywxKSIsMl0sWzQsNiwiKDEsZikiXSxbNSw3LCIoMSxmKSIsMl0sWzYsNywiKGcsMSkiXSxbOCwxMCwiXFxTaWdtYV97Zmd9IiwwLHsic2hvcnRlbiI6eyJzb3VyY2UiOjMwLCJ0YXJnZXQiOjMwfX1dLFsxNCwxMywiXFxTaWdtYV97Z2Z9XnstMX0iLDIseyJzaG9ydGVuIjp7InNvdXJjZSI6MzAsInRhcmdldCI6MzB9fV0sWzExLDEyLCJcXG1hcHN0byIsMSx7InNob3J0ZW4iOnsic291cmNlIjoyMCwidGFyZ2V0IjoyMH0sInN0eWxlIjp7ImJvZHkiOnsibmFtZSI6Im5vbmUifSwiaGVhZCI6eyJuYW1lIjoibm9uZSJ9fX1dXQ==
\[\begin{tikzcd}
	{(a,b)} & {(a,b')} && {(b,a)} & {(b,a')} \\
	{(a',b)} & {(a',b')} && {(b',a)} & {(b',a')}
	\arrow[""{name=0, anchor=center, inner sep=0}, "{(1,g)}", from=1-1, to=1-2]
	\arrow["{(f,1)}"', from=1-1, to=2-1]
	\arrow[""{name=1, anchor=center, inner sep=0}, "{(f,1)}", from=1-2, to=2-2]
	\arrow[""{name=2, anchor=center, inner sep=0}, "{(1,f)}", from=1-4, to=1-5]
	\arrow[""{name=3, anchor=center, inner sep=0}, "{(g,1)}"', from=1-4, to=2-4]
	\arrow["{(g,1)}", from=1-5, to=2-5]
	\arrow[""{name=4, anchor=center, inner sep=0}, "{(1,g)}"', from=2-1, to=2-2]
	\arrow[""{name=5, anchor=center, inner sep=0}, "{(1,f)}"', from=2-4, to=2-5]
	\arrow["{\Sigma_{fg}}", between={0.3}{0.7}, Rightarrow, from=0, to=4]
	\arrow["\mapsto"{description}, draw=none, from=1, to=3]
	\arrow["{\Sigma_{gf}^{-1}}"', between={0.3}{0.7}, Rightarrow, from=5, to=2]
\end{tikzcd}\]
If $\mu s=\mu$, we get for any $\alpha,\beta$ in $\T$ that $\sigma_{\beta\alpha}=\sigma_{\alpha\beta}^{-1}$, finishing the proof.
\end{proof}

\begin{remark}
    Analogously to the $1$-dimensional setting, it is enough to define $\sigma_{\alpha\beta}$ for $\alpha$ and $\beta$ active with respect to some basis of $1$-cells in $\T$. That is because from the lemma above, formation of the cells $\sigma_{\alpha\beta}$ respects (operadic) compositions and products.
\end{remark}

\begin{remark}
    We can view $w$-commutativity as specifying a relationship between $\alpha\otimes \beta$ and $\beta\otimes \alpha,$ as well as between $(\alpha\circ \alpha')\otimes (\beta\circ\beta')$ and $(\alpha\otimes \beta)\circ(\alpha'\otimes\beta')$. In the case $w=p$, they are coherently isomorphic.
\end{remark}

Finally, we prove certain ``miracle associativity'' phenomenon: quite surprisingly, any magmal $w$-commutativity is automatically associative in a unique way. It follows from a simple observation: the $2$-cells $\sigma_{\alpha\beta}$ in $\T$ from Lemma \ref{lemma: syntactic} ought to satisfy some equations with respect to any other $2$-cell $s$ (see (\ref{dia: gray2}), (\ref{dia: gray3})); in particular, we can take $s$ to be some other cell of the form $\sigma_{\alpha'\beta'}$.

\begin{theorem} \label{prop: miracle associativity}
    Any magmal $w$-commutativity structure $\mu\colon\T\otimes_{s,w}\T\to \T$ on $\T$ is automatically associative with the associators coming from $\F^{op}$.
\end{theorem}

\begin{proof}
    We already know from Lemma \ref{lemma: f-commutativity} that the underlying unique $f$-commutativity is associative with the associators coming from $\F^{op}$, so to extend this to $\mu$, the only task is to verify that $2$-cells $\sigma_{\alpha\beta}$ behave with respect to the associators same way the Gray cells do. That means, following \cite{Gra76}, we have to check the following equation for any $\alpha\colon m\to 1$, $\beta\colon k\to 1$, $\gamma\colon p\to 1$: 

     % % https://q.uiver.app/#q=WzAsMTQsWzAsMSwibWsiXSxbMSwyLCJrIl0sWzIsMiwiMSJdLFsxLDEsIm0iXSxbMSwwLCJwbSJdLFswLDAsInBtayJdLFsyLDEsInAiXSxbMywxLCJtayJdLFs0LDIsImsiXSxbMywwLCJwbWsiXSxbNCwxLCJwayJdLFs0LDAsInBtIl0sWzUsMSwicCJdLFs1LDIsIjEiXSxbMCwxLCJcXGFscGhhIGsiLDJdLFsxLDIsIlxcYmV0YSAiLDJdLFswLDMsIm1cXGJldGEgIl0sWzMsMiwiXFxhbHBoYSJdLFs0LDMsIlxcZ2FtbWEgbSJdLFs1LDQsInBtXFxiZXRhICJdLFs1LDAsIlxcZ2FtbWEgbWsiLDJdLFs0LDYsInBcXGFscGhhICJdLFs2LDIsIlxcZ2FtbWEiXSxbNyw4LCJcXGFscGhhIGsiLDJdLFs5LDcsIlxcZ2FtbWEgbWsiLDJdLFs5LDEwLCJwXFxhbHBoYSBrIiwyXSxbMTAsOCwiXFxnYW1tYSBrIiwyXSxbOSwxMSwicG1cXGJldGEgIl0sWzExLDEyLCJwXFxhbHBoYSAiXSxbMTAsMTIsInAgXFxiZXRhICIsMl0sWzgsMTMsIlxcYmV0YSAiLDJdLFsxMiwxMywiXFxnYW1tYSJdLFs2LDcsIj0iLDEseyJzdHlsZSI6eyJib2R5Ijp7Im5hbWUiOiJub25lIn0sImhlYWQiOnsibmFtZSI6Im5vbmUifX19XSxbMTYsMTUsIlxcc2lnbWFfe1xcYmV0YVxcYWxwaGEgfSIsMCx7Im9mZnNldCI6Miwic2hvcnRlbiI6eyJzb3VyY2UiOjQwLCJ0YXJnZXQiOjQwfX1dLFsxOSwxNiwiXFxzaWdtYV97bVxcYmV0YSxcXGdhbW1hfSIsMix7InNob3J0ZW4iOnsic291cmNlIjo0MCwidGFyZ2V0Ijo0MH19XSxbMjEsMTcsIlxcc2lnbWFfe1xcYWxwaGEgXFxnYW1tYX0iLDAseyJvZmZzZXQiOjIsInNob3J0ZW4iOnsic291cmNlIjo0MCwidGFyZ2V0Ijo0MH19XSxbMjksMzAsIlxcc2lnbWFfe1xcYmV0YVxcZ2FtbWF9IiwwLHsib2Zmc2V0IjoyLCJzaG9ydGVuIjp7InNvdXJjZSI6NDAsInRhcmdldCI6NDB9fV0sWzI1LDIzLCJcXHNpZ21hX3tcXGFscGhhXFxnYW1tYX1rIiwwLHsib2Zmc2V0IjoyLCJzaG9ydGVuIjp7InNvdXJjZSI6NDAsInRhcmdldCI6NDB9fV0sWzI3LDI5LCJwXFxzaWdtYV97XFxiZXRhIFxcYWxwaGF9IiwwLHsib2Zmc2V0IjoyLCJzaG9ydGVuIjp7InNvdXJjZSI6NDAsInRhcmdldCI6NDB9fV1d
\begin{equation} \label{dia: associativity}\begin{tikzcd}
	pmk & pm && pmk & pm \\
	mk & m & p & mk & pk & p \\
	& k & 1 && k & 1
	\arrow[""{name=0, anchor=center, inner sep=0}, "{pm\beta }", from=1-1, to=1-2]
	\arrow["{\gamma mk}"', from=1-1, to=2-1]
	\arrow["{\gamma m}", from=1-2, to=2-2]
	\arrow[""{name=1, anchor=center, inner sep=0}, "{p\alpha }", from=1-2, to=2-3]
	\arrow[""{name=2, anchor=center, inner sep=0}, "{pm\beta }", from=1-4, to=1-5]
	\arrow["{\gamma mk}"', from=1-4, to=2-4]
	\arrow[""{name=3, anchor=center, inner sep=0}, "{p\alpha k}"', from=1-4, to=2-5]
	\arrow["{p\alpha }", from=1-5, to=2-6]
	\arrow[""{name=4, anchor=center, inner sep=0}, "{m\beta }", from=2-1, to=2-2]
	\arrow["{\alpha k}"', from=2-1, to=3-2]
	\arrow[""{name=5, anchor=center, inner sep=0}, "\alpha", from=2-2, to=3-3]
	\arrow["{=}"{description}, draw=none, from=2-3, to=2-4]
	\arrow["\gamma", from=2-3, to=3-3]
	\arrow[""{name=6, anchor=center, inner sep=0}, "{\alpha k}"', from=2-4, to=3-5]
	\arrow[""{name=7, anchor=center, inner sep=0}, "{p \beta }"', from=2-5, to=2-6]
	\arrow["{\gamma k}"', from=2-5, to=3-5]
	\arrow["\gamma", from=2-6, to=3-6]
	\arrow[""{name=8, anchor=center, inner sep=0}, "{\beta }"', from=3-2, to=3-3]
	\arrow[""{name=9, anchor=center, inner sep=0}, "{\beta }"', from=3-5, to=3-6]
	\arrow["{\sigma_{\gamma,m\beta}}"', between={0.4}{0.6}, Rightarrow, from=0, to=4]
	\arrow["{\sigma_{ \gamma\alpha}}", shift right=2, between={0.4}{0.6}, Rightarrow, from=1, to=5]
	\arrow["{p\sigma_{\alpha\beta }}", shift right=2, between={0.4}{0.6}, Rightarrow, from=2, to=7]
	\arrow["{\sigma_{\gamma\alpha}k}", shift right=2, between={0.4}{0.6}, Rightarrow, from=3, to=6]
	\arrow["{\sigma_{\alpha\beta }}", shift right=2, between={0.4}{0.6}, Rightarrow, from=4, to=8]
	\arrow["{\sigma_{\gamma\beta}}", shift right=2, between={0.4}{0.6}, Rightarrow, from=7, to=9]
\end{tikzcd}\end{equation}

We can regard $\sigma_{\alpha\beta}$ as a $2$-cell $\alpha\otimes \beta \to \beta\otimes\alpha$ in $\T$. Thus, by the equation (\ref{dia: gray3}), the following has to hold:
% https://q.uiver.app/#q=WzAsOCxbMCwxLCJtayJdLFsyLDEsIjEiXSxbMCwwLCJwbWsiXSxbMiwwLCJwIl0sWzMsMCwicG1rIl0sWzMsMSwibWsiXSxbNSwwLCJwIl0sWzUsMSwiMSJdLFswLDEsIlxcYWxwaGFcXG90aW1lc1xcYmV0YSIsMSx7ImN1cnZlIjotMn1dLFswLDEsIlxcYmV0YVxcb3RpbWVzXFxhbHBoYSIsMix7ImN1cnZlIjoyfV0sWzIsMCwiXFxnYW1tYVxcY2RvdCBtayIsMl0sWzIsMywicFxcY2RvdCAoXFxhbHBoYVxcb3RpbWVzIFxcYmV0YSkiLDAseyJjdXJ2ZSI6LTJ9XSxbMywxLCJcXGdhbW1hIl0sWzQsNSwiXFxnYW1tYVxcY2RvdCBtayIsMV0sWzQsNiwicFxcY2RvdCAoXFxhbHBoYVxcb3RpbWVzIFxcYmV0YSkiLDAseyJjdXJ2ZSI6LTJ9XSxbNCw2LCJwXFxjZG90IChcXGJldGFcXG90aW1lcyBcXGFscGhhKSIsMSx7ImN1cnZlIjoyfV0sWzYsNywiXFxnYW1tYSJdLFs1LDcsIlxcYmV0YVxcb3RpbWVzIFxcYWxwaGEiLDIseyJjdXJ2ZSI6Mn1dLFs4LDksIlxcc2lnbWFfe1xcYWxwaGFcXGJldGF9IiwwLHsic2hvcnRlbiI6eyJzb3VyY2UiOjMwLCJ0YXJnZXQiOjMwfX1dLFsxMSw4LCJcXHNpZ21hX3tcXGFscGhhIFxcb3RpbWVzIFxcYmV0YSxcXGdhbW1hfSIsMCx7InNob3J0ZW4iOnsic291cmNlIjozMCwidGFyZ2V0IjozMH19XSxbMTQsMTUsInBcXGNkb3QgXFxzaWdtYV97XFxhbHBoYVxcYmV0YX0iLDAseyJzaG9ydGVuIjp7InNvdXJjZSI6MzAsInRhcmdldCI6MzB9fV0sWzE1LDE3LCJcXHNpZ21hX3tcXGJldGFcXG90aW1lc1xcYWxwaGEsXFxnYW1tYX0iLDAseyJzaG9ydGVuIjp7InNvdXJjZSI6MzAsInRhcmdldCI6MzB9fV0sWzEyLDEzLCI9IiwzLHsic2hvcnRlbiI6eyJzb3VyY2UiOjIwLCJ0YXJnZXQiOjIwfSwic3R5bGUiOnsiYm9keSI6eyJuYW1lIjoibm9uZSJ9LCJoZWFkIjp7Im5hbWUiOiJub25lIn19fV1d
\begin{equation}
    \label{dia: associativity2}
\begin{tikzcd}
	pmk && p & pmk && p \\
	mk && 1 & mk && 1
	\arrow[""{name=0, anchor=center, inner sep=0}, "{p\cdot (\alpha\otimes \beta)}", curve={height=-12pt}, from=1-1, to=1-3]
	\arrow["{\gamma\cdot mk}"', from=1-1, to=2-1]
	\arrow[""{name=1, anchor=center, inner sep=0}, "\gamma", from=1-3, to=2-3]
	\arrow[""{name=2, anchor=center, inner sep=0}, "{p\cdot (\alpha\otimes \beta)}", curve={height=-12pt}, from=1-4, to=1-6]
	\arrow[""{name=3, anchor=center, inner sep=0}, "{p\cdot (\beta\otimes \alpha)}"{description}, curve={height=12pt}, from=1-4, to=1-6]
	\arrow[""{name=4, anchor=center, inner sep=0}, "{\gamma\cdot mk}"{description}, from=1-4, to=2-4]
	\arrow["\gamma", from=1-6, to=2-6]
	\arrow[""{name=5, anchor=center, inner sep=0}, "{\alpha\otimes\beta}"{description}, curve={height=-12pt}, from=2-1, to=2-3]
	\arrow[""{name=6, anchor=center, inner sep=0}, "{\beta\otimes\alpha}"', curve={height=12pt}, from=2-1, to=2-3]
	\arrow[""{name=7, anchor=center, inner sep=0}, "{\beta\otimes \alpha}"', curve={height=12pt}, from=2-4, to=2-6]
	\arrow["{\sigma_{\gamma,\alpha \otimes \beta}}", between={0.3}{0.7}, Rightarrow, from=0, to=5]
	\arrow["{=}"{marking, allow upside down}, draw=none, from=1, to=4]
	\arrow["{p\cdot \sigma_{\alpha\beta}}", between={0.3}{0.7}, Rightarrow, from=2, to=3]
	\arrow["{\sigma_{\gamma,\beta\otimes\alpha}}", between={0.3}{0.7}, Rightarrow, from=3, to=7]
	\arrow["{\sigma_{\alpha\beta}}", between={0.3}{0.7}, Rightarrow, from=5, to=6]
\end{tikzcd}\end{equation}

Substituting $\beta\otimes \alpha=\beta \circ (\alpha\cdot k)$, we can rewrite $\sigma_{\gamma,\beta\otimes\alpha}$ as a pasting of $\sigma_{\gamma,\alpha \cdot k}=\sigma_{\gamma\alpha}\cdot k$ and $\sigma_{\gamma\beta}$, using (\ref{dia: gray1}).

% https://q.uiver.app/#q=WzAsNixbMCwwLCJwbWsiXSxbMCwxLCJtayJdLFsxLDEsImsiXSxbMiwxLCIxIl0sWzIsMCwicCJdLFsxLDAsInBrIl0sWzAsMSwiXFxnYW1tYVxcY2RvdCBtayIsMl0sWzUsMiwiXFxnYW1tYVxcY2RvdCBrIiwxXSxbNCwzLCJcXGdhbW1hIl0sWzEsMiwiXFxhbHBoYVxcY2RvdCBrIiwyXSxbMiwzLCJcXGJldGEiLDJdLFswLDUsInBcXGNkb3QgXFxhbHBoYVxcY2RvdCBrIl0sWzUsNCwicFxcY2RvdCBcXGJldGEiXSxbMTEsOSwiXFxzaWdtYV97XFxhbHBoYVxcY2RvdCBrLCBcXGdhbW1hfSIsMix7InNob3J0ZW4iOnsic291cmNlIjozMCwidGFyZ2V0IjozMH19XSxbMTIsMTAsIlxcc2lnbWFfe1xcYmV0YVxcZ2FtbWF9IiwwLHsic2hvcnRlbiI6eyJzb3VyY2UiOjMwLCJ0YXJnZXQiOjMwfX1dXQ==
\[\begin{tikzcd}
	pmk & pk & p \\
	mk & k & 1
	\arrow[""{name=0, anchor=center, inner sep=0}, "{p\cdot \alpha\cdot k}", from=1-1, to=1-2]
	\arrow["{\gamma\cdot mk}"', from=1-1, to=2-1]
	\arrow[""{name=1, anchor=center, inner sep=0}, "{p\cdot \beta}", from=1-2, to=1-3]
	\arrow["{\gamma\cdot k}"{description}, from=1-2, to=2-2]
	\arrow["\gamma", from=1-3, to=2-3]
	\arrow[""{name=2, anchor=center, inner sep=0}, "{\alpha\cdot k}"', from=2-1, to=2-2]
	\arrow[""{name=3, anchor=center, inner sep=0}, "\beta"', from=2-2, to=2-3]
	\arrow["{\sigma_{\gamma,\alpha} k}"', between={0.3}{0.7}, Rightarrow, from=0, to=2]
	\arrow["{\sigma_{\gamma\beta}}", between={0.3}{0.7}, Rightarrow, from=1, to=3]
\end{tikzcd}\]
Doing the same for $\sigma_{\alpha\otimes\beta,\gamma}$, we get that equations (\ref{dia: associativity}) and (\ref{dia: associativity2}) coincide.
\end{proof}

\begin{remark}
    \label{rmk: syntax of oo,2 theories}

  If $\T$ is more generally a Lawvere $(\infty,2)$-theory and $\mu$ a $w$-commutativity structure, the discussion above still applies to some extent. Denote by $\ho\colon \Cat^{\mathfrak{m}}_{(\infty,2)}\to \Cat^{\mathfrak{m}}_{(2,2)}$ the homotopy marked $(2,2)$-category functor (left adjoint to the inclusion). One can see that this is monoidal with respect to either of the Gray tensor products $\otimes_{s,w}$ and if we view $\T$ as a marked $(\infty,2)$-category, $\ho(\T)$ is equivalent to a (marked $(2,2)$-category associated to a) Lawvere $(2,2)$-theory. In particular, we get that $\mu(m,n)\simeq m\cdot n$. We can also see that a part of specifying $\mu$ is to find the fillers $\sigma_{\alpha\beta}$ to the squares (\ref{diagram: comm}). %Nevertheless, we probably cannot expect an analogue of \ref{prop: miracle associativity} to hold as there are of course much more coherences to handle.
\end{remark}

\subsection{Units and unital operations}
As a next example of how rigid the $w$-commutativity structures are, we are going to generalize some observations from the first section. To this end, let $\T$ be any Lawvere $2$-theory (applies to both $(\infty,2)$- and $(2,2)$-) and $\mu\colon \T\otimes_{s,w}\T\to \T$ a (magmal) lax commutativity structure.

We start the discussion with units. Same as in \ref{subsec: 1-cat EH}, we say that a unit in $\T$ is a map $u\colon 0\to 1$.

\begin{lemma} \label{lemma: comm units} 
    For any two units $u, v$ in $\T$, there exists a $2$-cell $ u\Rightarrow v$. 
\end{lemma}

\begin{proof}
    The lax commutativity implies an existence of a cell $\sigma_{uv}$ as below.
    % https://q.uiver.app/#q=WzAsNCxbMCwwLCIwXFxjZG90MCJdLFswLDEsIjFcXGNkb3QwIl0sWzEsMCwiMFxcY2RvdDEiXSxbMSwxLCIxXFxjZG90IDEiXSxbMSwzLCJ2IiwyXSxbMiwzLCJ1Il0sWzAsMiwiMFxcY2RvdCB2Il0sWzAsMSwidVxcY2RvdCAwIiwyXSxbNiw0LCJzX3t1dn0iLDAseyJzaG9ydGVuIjp7InNvdXJjZSI6NDAsInRhcmdldCI6NDB9fV1d
\[\begin{tikzcd}
	{0\cdot0} & {0\cdot1} \\
	{1\cdot0} & {1\cdot 1}
	\arrow[""{name=0, anchor=center, inner sep=0}, "{0\cdot v}", from=1-1, to=1-2]
	\arrow["{u\cdot 0}"', from=1-1, to=2-1]
	\arrow["u", from=1-2, to=2-2]
	\arrow[""{name=1, anchor=center, inner sep=0}, "v"', from=2-1, to=2-2]
	\arrow["{\sigma_{uv}}", between={0.4}{0.6}, Rightarrow, from=0, to=1]
\end{tikzcd}\]
    As $0\cdot v=\id_0=u\cdot 0$, $\sigma_{uv}$ is in fact a cell $u\to v$. 
\end{proof}

Recall from \ref{subsec: 1-cat EH} that for a unit $u$ and positive integers $k\le n$, $n>1$, we denote by $u_{k,n}\colon 1\to n$ the map of the form $u+\cdots +\id_1+\cdots +u$ where $\id_1$ is at $k$-th position. We can define unitality of a map $\alpha\colon n\to 1$ similarly as with $1$-theories: we say that $\alpha$ is unital with a unit $u$ if for any positive integer $k\le n$, we have an isomorphism $\id_1\Rightarrow \alpha\circ u_{k,n}$.

% https://q.uiver.app/#q=WzAsMyxbMCwwLCIxIl0sWzIsMCwiMSJdLFsxLDEsIm4iXSxbMCwxLCIiLDAseyJsZXZlbCI6Miwic3R5bGUiOnsiaGVhZCI6eyJuYW1lIjoibm9uZSJ9fX1dLFswLDIsInVfe2ssbn0iLDJdLFsyLDEsIlxcYWxwaGEiLDJdLFszLDIsIlxcY29uZyIsMSx7InNob3J0ZW4iOnsic291cmNlIjozMCwidGFyZ2V0IjozMH0sInN0eWxlIjp7ImJvZHkiOnsibmFtZSI6Im5vbmUifSwiaGVhZCI6eyJuYW1lIjoibm9uZSJ9fX1dXQ==
\[\begin{tikzcd}
	1 && 1 \\
	& n
	\arrow[""{name=0, anchor=center, inner sep=0}, equals, from=1-1, to=1-3]
	\arrow["{u_{k,n}}"', from=1-1, to=2-2]
	\arrow["\alpha"', from=2-2, to=1-3]
	\arrow["\cong"{description}, draw=none, from=0, to=2-2]
\end{tikzcd}\]

\begin{lemma} \label{lemma: comm unital}
    Suppose that $\alpha,\beta\colon n\to 1$ are unital $1$-cells in $\T$. Then there exists a map $\beta \Rightarrow\alpha$.
\end{lemma}

\begin{proof}
    Denote by $u,v$ the units for $\alpha,\beta,$ respectively; we know that there exists a $2$-cell $u\Rightarrow v$. Denote $U:= \sum_i u_{in}$, $V:= \sum_i v_{in}$. We know that we have a $2$-cell $U\Rightarrow V$ and also an isomorphism $\id_n\Rightarrow (n\cdot \alpha)\circ U.$ It is straightforward to check that we have an isomorphism $\id_n\Rightarrow (\beta\cdot n)\circ V$ as well. Then we obtain a cell $\alpha\Rightarrow\beta$ as a pasting of the $2$-cells in the diagram below.
    % https://q.uiver.app/#q=WzAsNSxbMSwxLCJuXjIiXSxbMSwyLCJuIl0sWzIsMSwibiJdLFsyLDIsIjEiXSxbMCwwLCJuIl0sWzAsMSwiXFxiZXRhXFxjZG90IG4iLDJdLFswLDIsIm5cXGNkb3QgXFxhbHBoYSJdLFsyLDMsIlxcYmV0YSJdLFsxLDMsIlxcYWxwaGEiLDJdLFs0LDEsIiIsMCx7ImN1cnZlIjoyLCJsZXZlbCI6Miwic3R5bGUiOnsiaGVhZCI6eyJuYW1lIjoibm9uZSJ9fX1dLFs0LDIsIiIsMCx7ImN1cnZlIjotMiwibGV2ZWwiOjIsInN0eWxlIjp7ImhlYWQiOnsibmFtZSI6Im5vbmUifX19XSxbNCwwLCJWIiwxLHsiY3VydmUiOjJ9XSxbNCwwLCJVIiwxLHsiY3VydmUiOi0yfV0sWzYsOCwiXFxzaWdtYV97XFxiZXRhXFxhbHBoYX0iLDAseyJvZmZzZXQiOjEsInNob3J0ZW4iOnsic291cmNlIjo0MCwidGFyZ2V0Ijo0MH19XSxbMTIsMTEsIiIsMSx7InNob3J0ZW4iOnsic291cmNlIjozMCwidGFyZ2V0IjozMH19XSxbMTAsMCwiXFxzaW1lcSIsMCx7InNob3J0ZW4iOnsic291cmNlIjoyMH0sInN0eWxlIjp7ImJvZHkiOnsibmFtZSI6Im5vbmUifSwiaGVhZCI6eyJuYW1lIjoibm9uZSJ9fX1dLFswLDksIlxcc2ltZXEiLDAseyJzaG9ydGVuIjp7InRhcmdldCI6MjB9LCJzdHlsZSI6eyJib2R5Ijp7Im5hbWUiOiJub25lIn0sImhlYWQiOnsibmFtZSI6Im5vbmUifX19XV0=
\[\begin{tikzcd}
	n && \\
	& {n^2} & n \\
	& n & 1
	\arrow[""{name=0, anchor=center, inner sep=0}, "V"{description}, curve={height=12pt}, from=1-1, to=2-2]
	\arrow[""{name=1, anchor=center, inner sep=0}, "U"{description}, curve={height=-12pt}, from=1-1, to=2-2]
	\arrow[""{name=2, anchor=center, inner sep=0}, curve={height=-12pt}, equals, from=1-1, to=2-3]
	\arrow[""{name=3, anchor=center, inner sep=0}, curve={height=12pt}, equals, from=1-1, to=3-2]
	\arrow[""{name=4, anchor=center, inner sep=0}, "{n\cdot \alpha}", from=2-2, to=2-3]
	\arrow["{\beta\cdot n}"', from=2-2, to=3-2]
	\arrow["\beta", from=2-3, to=3-3]
	\arrow[""{name=5, anchor=center, inner sep=0}, "\alpha"', from=3-2, to=3-3]
	\arrow[between={0.3}{0.7}, Rightarrow, from=1, to=0]
	\arrow["\simeq", draw=none, from=2, to=2-2]
	\arrow["\simeq", draw=none, from=2-2, to=3]
	\arrow["{\sigma_{\beta\alpha}}", shift right, between={0.4}{0.6}, Rightarrow, from=4, to=5]
\end{tikzcd}\]
\end{proof}

Let us call the composite $2$-cell $\beta\Rightarrow\alpha$ from the diagram above by $\iota_{\beta\alpha}$.

\begin{proposition}
    Let $\T$ be a Lawvere $(2,2)$-theory with a pesudocommutativity $\mu$. Then the following holds:
    \begin{enumerate}
        \item There exists at most one unit $u$ in $\T$ which moreover admits no nontrivial automorphisms.
        \item Any two unital maps $\alpha,\beta\colon n \to 1$ are isomorphic. 
        %\item If $\mu$ is symmetric, the $2$-cell $\iota_{\beta\alpha}$ induces an isomorphism $\End(\beta)\xrightarrow{\sim}\End(\alpha)$. 
    \end{enumerate}
\end{proposition}

\begin{proof}
    We know from Lemma \ref{lemma: comm units} that for any two units $u,v$ we have a $2$-cell $\sigma_{uv}\colon u\Rightarrow v$. As $\mu$ is a pseudocommutativity, it is invertible. Suppose that $t\colon u\Rightarrow u$ is an automorphism of the unit. Then the equation \ref{dia: gray2} tells us that $\sigma_{uv}\circ t=\sigma_{uv}$. This proves (1). Proof of (2) follows immediately from the lemma above as the $2$-cell $\iota_{\beta\alpha}$ is now invertible. 
\end{proof}

\subsection{Examples} \label{subsec: examples}

First class of examples comes from ordinary commutative Lawvere 1-theories. Commutative $1$-theory is obviously $s$-commutative regarded as a $2$-theory. We also want to show that if a $1$-theory describing some sort of structure is commutative -- e.g., theory for commutative monoids --, then the associated $2$-theory describing weakening of that structure in the $2$-dimensional setting -- e.g., $2$-theory for symmetric pseudomonoids -- is pseudocommutative.

\begin{lemma} \label{lemma: sharpening commutativity}
    Let $\T$ be a Lawvere $(1,1)$-theory and let $\T^\flat$ be a $(2,1)$-theory together with a map $\T^{\flat}\to \T$ inducing an isomorphism $\PsFun_{s,p}^\times(\T,-)\cong\Fun_{s,p}^\times(\T^\flat,-)$. Then if $\T$ is commutative, we have a symmetric pseudocommutativity on $\T^{\flat}$.
\end{lemma}

\begin{proof}
    We can show that for $(\T\otimes_{s,p}\T)^\flat\cong \T^\flat\otimes_{s,p}\T^\flat$ by means of the Yoneda lemma:\begin{align*}
        \Fun_{s,p}(\T^\flat \otimes_{s,p} \T^\flat,-)&\cong \Fun_{s,p}(\T^\flat,\Fun_{s,p}(\T^{\flat},-)\\
        &\cong \PsFun_{s,p}(\T,\PsFun_{s,p}(\T,-))\\
        &\cong \PsFun_{s,p}(\T\otimes_{s,p}\T,-)\\
        &\cong \Fun_{s,p}((\T\otimes_{s,p} \T)^\flat,-) 
    \end{align*}
Since $\T$ is commutative, it is a commutative monoid with respect to $\times,$ hence also $\otimes_{s,p}$. We have just shown that $(-)^{\flat}$ is strong monoidal, and hence $\T^{\flat}$ is a commutative monoid with respect to $\otimes_{s,p}$. \end{proof}

%\begin{remark}    The lemma above is not surprising: as mentioned before, using $2$-functors instead of pseudofunctors throughout this paper is mostly an aesthetical choice. Also, in the $\infty$-categorical setting, the distinction between strict and pseudofunctors becomes meaningles, and therefore so does the distinction between $\T$ and $\T^\flat$.\end{remark}

\begin{example} \label{example: Eoo monoids}
    Consider the $(2,1)$-theory $\T$ for $E_\infty$-monoids; that is, a functor $\theta\colon \F^{op}\to \mathsf{Span}(\F)$ which is an identity on objects and sends a map $n\to m$ in $\F^{op}$ to a span $n\leftarrow m\xrightarrow{\id} m$. By \cite{Cra10}, $\Mod_{s}(\T,{\Cat_\infty})$ is the category of symmetric monoidal categories, symmetric monoidal functors, and symmetric monoidal natural transformations. In the context of the previous lemma, we can also write $\T=\T_{cmon}^{\flat}$. This theory is pseudocommutative.
\end{example}

\begin{example} \label{example: braided}
    An important example of a Lawvere $(2,1)$-theory which is not of the form $\T^{\flat}$ for any $1$-theory is $\T_{E_2}$, the $2$-theory for braided pseudomonoids. Explicitly, it is generated by $\T_{mon}^\flat$, the theory for pseudomonoids, together with an invertible $2$-cell $b\colon m\to ms$, satisfying the usual coherence equations for braiding (see e.g. \cite[XI.1]{Mac98}). It may appear it admits a pseudocommutativity structure, because we can actually find a system of cells $\sigma_{\alpha\beta}$; for example, we use $b$ to define $\sigma_{mm}$ the same way as for $\T_{cmon}^\flat$, in virtue of diagram (\ref{dia: comm mon}). 
    % https://q.uiver.app/#q=WzAsNixbMCwwLCI0Il0sWzEsMCwiMyJdLFsxLDEsIjIiXSxbMCwxLCI0Il0sWzIsMSwiMSJdLFsyLDAsIjIiXSxbMCwxLCIxbTEiXSxbMCwzLCIxczEiLDJdLFszLDIsIjJcXGNkb3QgbSIsMl0sWzMsMSwiMW0xIiwyXSxbMiw0LCJtIiwyXSxbNSw0LCJtIl0sWzEsNSwibTEiXSxbNiwzLCIxYjEiLDEseyJzaG9ydGVuIjp7InNvdXJjZSI6MjAsInRhcmdldCI6MTB9fV0sWzEsMTAsIlxcY29uZyIsMSx7InNob3J0ZW4iOnsic291cmNlIjozMCwidGFyZ2V0IjozMH0sInN0eWxlIjp7ImJvZHkiOnsibmFtZSI6Im5vbmUifSwiaGVhZCI6eyJuYW1lIjoibm9uZSJ9fX1dXQ==
\[\begin{tikzcd}
	4 & 3 & 2 \\
	4 & 2 & 1
	\arrow[""{name=0, anchor=center, inner sep=0}, "1m1", from=1-1, to=1-2]
	\arrow["1s1"', from=1-1, to=2-1]
	\arrow["m1", from=1-2, to=1-3]
	\arrow["m", from=1-3, to=2-3]
	\arrow["1m1"', from=2-1, to=1-2]
	\arrow["{2\cdot m}"', from=2-1, to=2-2]
	\arrow[""{name=1, anchor=center, inner sep=0}, "m"', from=2-2, to=2-3]
	\arrow["1b1"{description}, between={0.2}{0.9}, Rightarrow, from=0, to=2-1]
	\arrow["\cong"{description}, draw=none, from=1-2, to=1]
\end{tikzcd}\]
    
    However, the compatibilities of Lemma \ref{lemma: syntactic} are not satisfied, as we can see in the diagram below (or its vertical analogue):

    % https://q.uiver.app/#q=WzAsMTEsWzAsMCwiNCJdLFswLDEsIjQiXSxbMCwyLCIyIl0sWzIsMiwiMSJdLFsyLDAsIjIiXSxbMywxLCJcXG5lIl0sWzQsMCwiNCJdLFs0LDEsIjQiXSxbNCwyLCIyIl0sWzYsMCwiMiJdLFs2LDIsIjEiXSxbMCwxLCIxczEiLDJdLFsxLDIsIm1tIiwyXSxbNCwzLCJtIl0sWzAsNCwiMlxcY2RvdCBtIiwwLHsiY3VydmUiOi0yfV0sWzIsMywibSIsMCx7ImN1cnZlIjotMn1dLFsyLDMsIm1zIiwyLHsiY3VydmUiOjJ9XSxbNiw5LCIyXFxjZG90IG0iLDAseyJjdXJ2ZSI6LTJ9XSxbNiw5LCIyXFxjZG90IG1zIiwyLHsiY3VydmUiOjJ9XSxbNiw3LCIxczEiLDJdLFs3LDgsIjJcXGNkb3QgbSIsMl0sWzksMTAsIm0iXSxbOCwxMCwibXMiLDIseyJjdXJ2ZSI6Mn1dLFsxNCwxNSwiIiwwLHsic2hvcnRlbiI6eyJzb3VyY2UiOjQwLCJ0YXJnZXQiOjQwfX1dLFsxNSwxNiwiYiIsMCx7InNob3J0ZW4iOnsic291cmNlIjo0MCwidGFyZ2V0IjozMH19XSxbMTgsMjIsIiIsMix7InNob3J0ZW4iOnsic291cmNlIjo0MCwidGFyZ2V0Ijo0MH19XSxbMTcsMTgsImIiLDAseyJzaG9ydGVuIjp7InNvdXJjZSI6MzAsInRhcmdldCI6MzB9fV1d
\[\begin{tikzcd}
	4 && 2 && 4 && 2 \\
	4 &&& \ne & 4 \\
	2 && 1 && 2 && 1
	\arrow[""{name=0, anchor=center, inner sep=0}, "{2\cdot m}", curve={height=-12pt}, from=1-1, to=1-3]
	\arrow["1s1"', from=1-1, to=2-1]
	\arrow["m", from=1-3, to=3-3]
	\arrow[""{name=1, anchor=center, inner sep=0}, "{2\cdot m}", curve={height=-12pt}, from=1-5, to=1-7]
	\arrow[""{name=2, anchor=center, inner sep=0}, "{2\cdot ms}"', curve={height=12pt}, from=1-5, to=1-7]
	\arrow["1s1"', from=1-5, to=2-5]
	\arrow["m", from=1-7, to=3-7]
	\arrow["2\cdot m"', from=2-1, to=3-1]
	\arrow["{2\cdot m}"', from=2-5, to=3-5]
	\arrow[""{name=3, anchor=center, inner sep=0}, "m", curve={height=-12pt}, from=3-1, to=3-3]
	\arrow[""{name=4, anchor=center, inner sep=0}, "ms"', curve={height=12pt}, from=3-1, to=3-3]
	\arrow[""{name=5, anchor=center, inner sep=0}, "ms"', curve={height=12pt}, from=3-5, to=3-7]
	\arrow[between={0.4}{0.6}, Rightarrow, from=0, to=3]
	\arrow["b", between={0.3}{0.7}, Rightarrow, from=1, to=2]
	\arrow[between={0.4}{0.6}, Rightarrow, from=2, to=5]
	\arrow["b", between={0.4}{0.7}, Rightarrow, from=3, to=4]
\end{tikzcd}\]

The semantic way of seeing this failure is the observation that for monoids $M, N$ in a braided monoidal category $\V$, the braiding $b_{MN}\colon M\otimes N\to N\otimes M$ is not a monoid homomorphism. This would violate the eixistence of the lifting functor $\Mod_p(\T,\Cat)\to \Mod_p(\Mod_p(\T,\Cat))$. 

Nevertheless, the obstruction above is the only one, and hence we can only consider a map $\mu\colon\T_{E_1}\otimes_p \T_{E_1}\to \T_{E_2}$, giving us a $p$-commutation of $\T_{E_1}$ over $\T_{E_2}$ (using analogous reasoning as in Lemma \ref{lemma: syntactic}). This is of course closely related to the relationship $\mathbb{E}_1\otimes_{BV} \mathbb{E}_1\cong \mathbb{E}_2$ for operads (where $\otimes_{BV}$ is the Boardman-Vogt tensor product of operads). 
\end{example}

\begin{example}
    For a group $G$, we can consider the Lawvere $2$-theory $\T_G^\flat$, where $\T_G$ is as in \ref{example: 1dim monoids}. In fact, it is not necessary to take the whole Lawvere theory; we can just consider $BG^\flat$ where $BG$ is delooping of $G$. In other words, $BG^\flat$ is a bicategory with one object $\bullet$, $1$-cells $\rho_g$ for $g\in G$ and coherent invertible $2$-cells $\rho_{g}\rho_h\xrightarrow{\sim}\rho_{gh}$. As any bicategory\footnote{Strictly speaking, we should be talking here only about $2$-categories, so $\V$ should be a strict monoidal category. Nevertheless, in the last section, we will  generalize many results to $(\infty,2)$-categories, so the reader will hopefully believe me now that generalizing to bicategories causes no issues whatsoever.}, we can view it as an $\mF$-sketch with the empty set of specified cones and only tight map being the identity. In particular, $B\V\otimes_{s,w} B\V=B\V\otimes_w B\V$, the usual $w$-Gray tensor product. Following \ref{example: 1dim monoids} and \ref{lemma: sharpening commutativity}, we get that $BG^\flat$ is pseudocommutative if and only if $G$ is commutative. Models of $BG^\flat$ are categories with $G$-action as in \cite{BGM19}. An interesting object is the category of $G$-fixpoints, denoted by $\C^{G}$ in \textit{loc. cit.}, which in our formalism we can define as $\Hom_p(*,\C)$. For example, if $\C=\Set$ is equipped with the trivial $G$-action, then $\Hom_p(*,\Set)\cong G\text{-}\Set$, the category of sets with a $G$-action.
\end{example}

\begin{example} \label{example: Yang-Baxter} Let $\V$ be a monoidal category and $B\V$ its delooping, i.e. a bicategory with one object $\bullet$ and the hom-category $B\V(\bullet,\bullet)=\V$. By (simplified version of the proof of) Lemma \ref{lemma: syntactic}, a pseudocommutativity for $B\V$ is given by a system of invertible $2$-cells $b_{XY}\colon X\otimes Y\to Y\otimes X$; it is easy to compare the coherence equations to see that this is exactly braiding. 
    % https://q.uiver.app/#q=WzAsNCxbMCwwLCJcXGJ1bGxldCJdLFsxLDAsIlxcYnVsbGV0Il0sWzAsMSwiXFxidWxsZXQiXSxbMSwxLCJcXGJ1bGxldCJdLFswLDEsIlgiXSxbMSwzLCJZIl0sWzAsMiwiWSIsMl0sWzIsMywiWCIsMl0sWzQsNywiYl97WFl9IiwwLHsic2hvcnRlbiI6eyJzb3VyY2UiOjQwLCJ0YXJnZXQiOjQwfX1dXQ==
\[\begin{tikzcd}
	\bullet & \bullet \\
	\bullet & \bullet
	\arrow[""{name=0, anchor=center, inner sep=0}, "Y", from=1-1, to=1-2]
	\arrow["X"', from=1-1, to=2-1]
	\arrow["X", from=1-2, to=2-2]
	\arrow[""{name=1, anchor=center, inner sep=0}, "Y"', from=2-1, to=2-2]
	\arrow["{b_{XY}}", between={0.4}{0.6}, Rightarrow, from=0, to=1]
\end{tikzcd}\]

    Proposition \ref{prop: miracle associativity} tells us that this is associative. Let $X,Y,Z$ be any objects of $\V$; substituting $X=\gamma$, $Y=\alpha,$ $Z=\beta$ into the associativity equation \ref{dia: associativity}, we get that the following diagram has to commute (we omit the associators and brackets):

    % https://q.uiver.app/#q=WzAsNixbMSwwLCJYXFxvdGltZXMgWVxcb3RpbWVzIFoiXSxbMCwxLCJZXFxvdGltZXMgWFxcb3RpbWVzIFoiXSxbMCwyLCJZXFxvdGltZXMgWlxcb3RpbWVzIFgiXSxbMSwzLCJaXFxvdGltZXMgWVxcb3RpbWVzIFgiXSxbMiwxLCJYXFxvdGltZXMgWlxcb3RpbWVzIFkiXSxbMiwyLCJaXFxvdGltZXMgWFxcb3RpbWVzIFkiXSxbMCwxLCJiX3tYWX1cXG90aW1lcyAxIiwyXSxbMSwyLCIxXFxvdGltZXMgYl97WFp9IiwyXSxbMiwzLCJiX3tZWn1cXG90aW1lczEiLDJdLFswLDQsIjFcXG90aW1lcyBiX3tZWn0iXSxbNCw1LCJiX3tYWn1cXG90aW1lcyAxIl0sWzUsMywiMVxcb3RpbWVzIGJfe1hZfSJdXQ==
\[\begin{tikzcd}
	& {X\otimes Y\otimes Z} & \\
	{Y\otimes X\otimes Z} && {X\otimes Z\otimes Y} \\
	{Y\otimes Z\otimes X} && {Z\otimes X\otimes Y} \\
	& {Z\otimes Y\otimes X}
	\arrow["{b_{XY}\otimes 1}"', from=1-2, to=2-1]
	\arrow["{1\otimes b_{YZ}}", from=1-2, to=2-3]
	\arrow["{1\otimes b_{XZ}}"', from=2-1, to=3-1]
	\arrow["{b_{XZ}\otimes 1}", from=2-3, to=3-3]
	\arrow["{b_{YZ}\otimes1}"', from=3-1, to=4-2]
	\arrow["{1\otimes b_{XY}}", from=3-3, to=4-2]
\end{tikzcd}\]
    
    Amazingly, these are precisely the Yang-Baxter equations! Therefore, our framework of $w$-commutativity offers an alternative proof of the fact that for any braiding, these equations hold, going back to \cite{JS93}. By similar reasoning, we get that lax commutativity on $B\V$ is a non-invertible generalization of braiding, i.e. a \textit{lax braiding} as introduced in \cite{DPS07}. Proposition \ref{prop: miracle associativity} then tells us that the Yang-Baxter equations still hold for this weaker notion.
    
    By Lemma \ref{lemma: syntactic}, the pseudocommutativity structure is symmetric if and only if $b_{XY}b_{YX}=1$, which is the case if and only if the braiding is in fact a symmetry.

    A model of $B\V$ in $\Cat$ is a category with an action of $\V$, generalizing the previous example.
\end{example}

\begin{example} \label{example: involutive}
    One easy example of strictly commutative $2$-theory is the theory for involutive categories developed in \cite{Jac12}. Consider a theory $\T_{inv}$ generated by a map $\alpha\colon 1\to 1$, called involution, and an invertible $2$-cell $\iota\colon 1\to \alpha\circ\alpha$, satisfying the following equation: 
% https://q.uiver.app/#q=WzAsOCxbMCwwLCIxIl0sWzEsMCwiMSJdLFswLDEsIjEiXSxbMSwxLCIxIl0sWzIsMCwiMSJdLFszLDAsIjEiXSxbMiwxLCIxIl0sWzMsMSwiMSJdLFswLDEsIlxcYWxwaGEiXSxbMSwyLCJcXGFscGhhIiwyXSxbMiwzLCJcXGFscGhhIiwyXSxbNCw1LCJcXGFscGhhIl0sWzUsNiwiXFxhbHBoYSJdLFs2LDcsIlxcYWxwaGEiLDJdLFsxLDMsIiIsMCx7ImxldmVsIjoyLCJzdHlsZSI6eyJoZWFkIjp7Im5hbWUiOiJub25lIn19fV0sWzQsNiwiIiwyLHsibGV2ZWwiOjIsInN0eWxlIjp7ImhlYWQiOnsibmFtZSI6Im5vbmUifX19XSxbMTQsOSwiXFxpb3RhIiwyLHsib2Zmc2V0IjotMiwic2hvcnRlbiI6eyJzb3VyY2UiOjIwLCJ0YXJnZXQiOjIwfX1dLFsxNSwxMiwiXFxpb3RhIiwwLHsib2Zmc2V0IjotMSwic2hvcnRlbiI6eyJzb3VyY2UiOjIwLCJ0YXJnZXQiOjIwfX1dLFsxNCwxNSwiPSIsMSx7InNob3J0ZW4iOnsic291cmNlIjoyMCwidGFyZ2V0IjoyMH0sInN0eWxlIjp7ImJvZHkiOnsibmFtZSI6Im5vbmUifSwiaGVhZCI6eyJuYW1lIjoibm9uZSJ9fX1dXQ==
\[\begin{tikzcd}
	1 & 1 & 1 & 1 \\
	1 & 1 & 1 & 1
	\arrow["\alpha", from=1-1, to=1-2]
	\arrow[""{name=0, anchor=center, inner sep=0}, "\alpha"', from=1-2, to=2-1]
	\arrow[""{name=1, anchor=center, inner sep=0}, equals, from=1-2, to=2-2]
	\arrow["\alpha", from=1-3, to=1-4]
	\arrow[""{name=2, anchor=center, inner sep=0}, equals, from=1-3, to=2-3]
	\arrow[""{name=3, anchor=center, inner sep=0}, "\alpha", from=1-4, to=2-3]
	\arrow["\alpha"', from=2-1, to=2-2]
	\arrow["\alpha"', from=2-3, to=2-4]
	\arrow["\iota"', shift left=2, between={0.2}{0.8}, Rightarrow, from=1, to=0]
	\arrow["{=}"{description}, draw=none, from=1, to=2]
	\arrow["\iota", shift left, between={0.2}{0.8}, Rightarrow, from=2, to=3]
\end{tikzcd}\]

Then $\Mod_l(\T_{inv},\Cat)$ is the $2$-category of involutive categories from loc.cit. An internal algebra in an involutive category $\C$ is an object $X$ equipped with a map $\overline{X}\to X$ (where $\overline{X}$ is the involution); these are called self-conjugate objects in \textit{loc. cit.} This theory is strictly commutative as we can choose $s_{\alpha\alpha}=1$, and that choice certainly satisfies the equations from Lemma \ref{lemma: syntactic}. In fact, the results of sections 3-5 in \textit{loc. cit.} follows nicely from our general results (e.g. see \cite[Lemma 3.2]{Jac12} as an easy special case of Theorem \ref{thm: closed structure}).
\end{example}

\begin{example} \label{example: double delooping}
    Consider a Lawvere $(2,2)$-theory $\T$ where all $1$-cells are inert, i.e. only interesting data comes from endomorphisms of $\id_1\colon 1\to 1$. Models of such theory are equivalently functors out of $B^2M$, the double delooping of the monoid $M=\End(\id_1)$. Then any element in the center of $M$ induces a lax commutativity structure on $B^2M$, as one can see from Lemma \ref{lemma: syntactic} by considering the equation below.
    % https://q.uiver.app/#q=WzAsOCxbMCwwLCIxIl0sWzEsMCwiMSJdLFswLDEsIjEiXSxbMSwxLCIxIl0sWzIsMCwiMSJdLFszLDAsIjEiXSxbMiwxLCIxIl0sWzMsMSwiMSJdLFswLDEsIiIsMCx7ImxldmVsIjoyLCJzdHlsZSI6eyJoZWFkIjp7Im5hbWUiOiJub25lIn19fV0sWzEsMywiIiwwLHsibGV2ZWwiOjIsInN0eWxlIjp7ImhlYWQiOnsibmFtZSI6Im5vbmUifX19XSxbMiwzLCIiLDIseyJsZXZlbCI6Miwic3R5bGUiOnsiaGVhZCI6eyJuYW1lIjoibm9uZSJ9fX1dLFswLDIsIiIsMix7ImN1cnZlIjoyLCJsZXZlbCI6Miwic3R5bGUiOnsiaGVhZCI6eyJuYW1lIjoibm9uZSJ9fX1dLFswLDIsIiIsMSx7ImN1cnZlIjotMiwibGV2ZWwiOjIsInN0eWxlIjp7ImhlYWQiOnsibmFtZSI6Im5vbmUifX19XSxbNCw1LCIiLDIseyJsZXZlbCI6Miwic3R5bGUiOnsiaGVhZCI6eyJuYW1lIjoibm9uZSJ9fX1dLFs1LDcsIiIsMix7ImN1cnZlIjoyLCJsZXZlbCI6Miwic3R5bGUiOnsiaGVhZCI6eyJuYW1lIjoibm9uZSJ9fX1dLFs0LDYsIiIsMCx7ImxldmVsIjoyLCJzdHlsZSI6eyJoZWFkIjp7Im5hbWUiOiJub25lIn19fV0sWzYsNywiIiwwLHsibGV2ZWwiOjIsInN0eWxlIjp7ImhlYWQiOnsibmFtZSI6Im5vbmUifX19XSxbNSw3LCIiLDEseyJjdXJ2ZSI6LTIsImxldmVsIjoyLCJzdHlsZSI6eyJoZWFkIjp7Im5hbWUiOiJub25lIn19fV0sWzEyLDExLCJhIiwyLHsic2hvcnRlbiI6eyJzb3VyY2UiOjIwLCJ0YXJnZXQiOjIwfX1dLFs5LDEyLCJiIiwyLHsic2hvcnRlbiI6eyJzb3VyY2UiOjIwLCJ0YXJnZXQiOjIwfX1dLFsxNywxNCwiYSIsMix7InNob3J0ZW4iOnsic291cmNlIjoyMCwidGFyZ2V0IjoyMH19XSxbMTQsMTUsImIiLDIseyJzaG9ydGVuIjp7InNvdXJjZSI6MjAsInRhcmdldCI6MjB9fV0sWzksMTUsIj0iLDEseyJzaG9ydGVuIjp7InNvdXJjZSI6MjAsInRhcmdldCI6MjB9LCJzdHlsZSI6eyJib2R5Ijp7Im5hbWUiOiJub25lIn0sImhlYWQiOnsibmFtZSI6Im5vbmUifX19XV0=
\[\begin{tikzcd}
	1 & 1 & 1 & 1 \\
	1 & 1 & 1 & 1
	\arrow[equals, from=1-1, to=1-2]
	\arrow[""{name=0, anchor=center, inner sep=0}, curve={height=12pt}, equals, from=1-1, to=2-1]
	\arrow[""{name=1, anchor=center, inner sep=0}, curve={height=-12pt}, equals, from=1-1, to=2-1]
	\arrow[""{name=2, anchor=center, inner sep=0}, equals, from=1-2, to=2-2]
	\arrow[equals, from=1-3, to=1-4]
	\arrow[""{name=3, anchor=center, inner sep=0}, equals, from=1-3, to=2-3]
	\arrow[""{name=4, anchor=center, inner sep=0}, curve={height=12pt}, equals, from=1-4, to=2-4]
	\arrow[""{name=5, anchor=center, inner sep=0}, curve={height=-12pt}, equals, from=1-4, to=2-4]
	\arrow[equals, from=2-1, to=2-2]
	\arrow[equals, from=2-3, to=2-4]
	\arrow["a"', between={0.2}{0.8}, Rightarrow, from=1, to=0]
	\arrow["b"', between={0.2}{0.8}, Rightarrow, from=2, to=1]
	\arrow["{=}"{description}, draw=none, from=2, to=3]
	\arrow["b"', between={0.2}{0.8}, Rightarrow, from=4, to=3]
	\arrow["a"', between={0.2}{0.8}, Rightarrow, from=5, to=4]
\end{tikzcd}\]
\end{example}

\subsection{Extended example: Segal objects} \label{subsec: example segal} Although we mostly focus on Lawvere $2$-theories in this paper, our constructions makes sense for general marked sketches or algebraic patterns as well. We will demonstrate it now on the algebraic pattern for Segal objects. In fact, we will only consider it as a marked $(2,2)$-sketch.

     Recall that by $\Delta$ we mean the category of nonempty finite ordinals and order-preserving maps, using the notation $[n]:=\{0<1<\cdots<n\}$. We construct the $\Phi$-$\mF$-sketch for Segal objects / internal categories where $\Phi$ is the class of shapes for wide pullbacks. The underlying $2$-category of this sketch is in fact a $1$-category, namely $\Delta^{op}$. The marking consists of the inert maps in the sense of Lurie, generated by identities and the face maps $d_0,d_n\colon [n]\to [n-1]$. It is a simple combinatorial exercise to see that the diagrams iteratively built from tight maps as below commute:
     % https://q.uiver.app/#q=WzAsMTMsWzAsMiwiWzFdIl0sWzIsMiwiWzFdLCJdLFsxLDMsIlswXSJdLFsxLDEsIlsyXSJdLFszLDIsIlsxXSJdLFs1LDIsIlsxXSJdLFs0LDMsIlswXSJdLFs2LDMsIlswXSJdLFs3LDIsIlsxXSwiXSxbNCwxLCJbMl0iXSxbNiwxLCJbMl0iXSxbNSwwLCJbM10iXSxbOCwyLCJcXGRvdHMiXSxbMCwyLCJkXzAiLDJdLFsxLDIsImRfMSJdLFszLDAsImRfMiIsMl0sWzMsMSwiZF8wIl0sWzQsNiwiZF8wIiwyXSxbNSw2LCJkXzEiXSxbNSw3LCJkXzAiLDJdLFs4LDcsImRfMSJdLFsxMSw5LCJkXzMiLDJdLFs5LDQsImRfMiIsMl0sWzksNSwiZF8wIl0sWzEwLDUsImRfMiIsMl0sWzEwLDgsImRfMCJdLFsxMSwxMCwiZF8wIl1d
\[\begin{tikzcd}
	&&&&& {[3]} &&& \\
	& {[2]} &&& {[2]} && {[2]} \\
	{[1]} && {[1],} & {[1]} && {[1]} && {[1],} & \dots \\
	& {[0]} &&& {[0]} && {[0]}
	\arrow["{d_3}"', from=1-6, to=2-5]
	\arrow["{d_0}", from=1-6, to=2-7]
	\arrow["{d_2}"', from=2-2, to=3-1]
	\arrow["{d_0}", from=2-2, to=3-3]
	\arrow["{d_2}"', from=2-5, to=3-4]
	\arrow["{d_0}", from=2-5, to=3-6]
	\arrow["{d_2}"', from=2-7, to=3-6]
	\arrow["{d_0}", from=2-7, to=3-8]
	\arrow["{d_0}"', from=3-1, to=4-2]
	\arrow["{d_1}", from=3-3, to=4-2]
	\arrow["{d_0}"', from=3-4, to=4-5]
	\arrow["{d_1}", from=3-6, to=4-5]
	\arrow["{d_0}"', from=3-6, to=4-7]
	\arrow["{d_1}", from=3-8, to=4-7]
\end{tikzcd}\]

Therefore, we make the marked $(2,2)$-category ($\Delta^{op}$, inert maps) into a $\Phi$-$\mF$-sketch $\Ss_{Seg}$ by choosing cones of the following form for any $n\ge2$ (where the maps $[n]\to [1]$ are built as above): 

% https://q.uiver.app/#q=WzAsMTAsWzQsMiwiWzFdIl0sWzYsMiwiWzFdIl0sWzUsMywiWzBdIl0sWzcsMywiWzBdIl0sWzgsMiwiWzFdIl0sWzMsMiwiXFxjZG90cyJdLFswLDIsIlsxXSJdLFsxLDMsIlswXSJdLFsyLDIsIlsxXSJdLFs0LDAsIltuXSJdLFswLDIsImRfMCIsMl0sWzEsMiwiZF8xIl0sWzEsMywiZF8wIiwyXSxbNCwzLCJkXzEiXSxbNiw3LCJkXzAiLDJdLFs4LDcsImRfMSJdLFs5LDYsIiIsMix7ImN1cnZlIjozfV0sWzksOCwiIiwwLHsiY3VydmUiOjF9XSxbOSwwXSxbOSwxLCIiLDAseyJjdXJ2ZSI6LTF9XSxbOSw0LCIiLDAseyJjdXJ2ZSI6LTN9XSxbMTcsMTgsIlxcZG90cyIsMSx7InNob3J0ZW4iOnsic291cmNlIjoyMCwidGFyZ2V0IjoyMH0sInN0eWxlIjp7ImJvZHkiOnsibmFtZSI6Im5vbmUifSwiaGVhZCI6eyJuYW1lIjoibm9uZSJ9fX1dXQ==
\[\begin{tikzcd}
	&&&& {[n]} &&&& \\
	\\
	{[1]} && {[1]} & \cdots & {[1]} && {[1]} && {[1]} \\
	& {[0]} &&&& {[0]} && {[0]}
	\arrow[curve={height=18pt}, from=1-5, to=3-1]
	\arrow[""{name=0, anchor=center, inner sep=0}, curve={height=6pt}, from=1-5, to=3-3]
	\arrow[""{name=1, anchor=center, inner sep=0}, from=1-5, to=3-5]
	\arrow[curve={height=-6pt}, from=1-5, to=3-7]
	\arrow[curve={height=-18pt}, from=1-5, to=3-9]
	\arrow["{d_0}"', from=3-1, to=4-2]
	\arrow["{d_1}", from=3-3, to=4-2]
	\arrow["{d_0}"', from=3-5, to=4-6]
	\arrow["{d_1}", from=3-7, to=4-6]
	\arrow["{d_0}"', from=3-7, to=4-8]
	\arrow["{d_1}", from=3-9, to=4-8]
	\arrow["\dots"{description}, draw=none, from=0, to=1]
\end{tikzcd}\]

For any $2$-category $\C$ with pullbacks, models of $\Ss_{Seg}$ in $\C$ (depending only on the underlying $1$-category) are Segal objects in $\C$, which we an also interpret as categories internal to $\C$: i.e., for a model $\X$, we view $X_0:=\X([0])$ as an ``object of objects'', $X_1:=\X([1])$ as an ``object of morphisms'', the tight maps $X_1 \to X_0$ induced by $d_0$ and $d_1$ as source and target maps, the map $X_1\times_{X_0}X_1\cong X_2\to X_1$ induced by $d_1$ as a composition map etc. In particular, the objects of $\Mod_{l}(\Ss_{Seg},\Cat)$ are double categories. What are the $1$-cells? A lax homomorphism $f\colon \X\to \Y$ is determined by the functors $f_0\colon X_0\to Y_0$, $f_1\colon X_1\to Y_1$ together with some conditions: they have to commute with the source and target maps (as these are tight) and for the unit and composition maps, we get the following squares: 

% https://q.uiver.app/#q=WzAsOCxbMCwwLCJYXzFcXHRpbWVzX3tYXzB9IFhfMSJdLFsxLDAsIllfMVxcdGltZXNfe1lfMH0gWV8xIl0sWzAsMSwiWF8xIl0sWzEsMSwiWV8xLCJdLFsyLDAsIlhfMCJdLFsyLDEsIlhfMSJdLFszLDAsIllfMCJdLFszLDEsIllfMSJdLFswLDIsIlxcWChkXzEpIiwyXSxbMiwzLCJmXzEiLDJdLFsxLDMsIlxcWShkXzEpIl0sWzAsMSwiKGZfMSxmXzEpIl0sWzQsNSwiXFxYKHNfMCkiLDJdLFs2LDcsIlxcWShzXzApIl0sWzQsNiwiZl8wIl0sWzUsNywiZl8xIiwyXSxbMTEsOSwiIiwwLHsic2hvcnRlbiI6eyJzb3VyY2UiOjQwLCJ0YXJnZXQiOjQwfX1dLFsxNCwxNSwiIiwwLHsic2hvcnRlbiI6eyJzb3VyY2UiOjQwLCJ0YXJnZXQiOjQwfX1dXQ==
\[\begin{tikzcd}
	{X_1\times_{X_0} X_1} & {Y_1\times_{Y_0} Y_1} & {X_0} & {Y_0} \\
	{X_1} & {Y_1,} & {X_1} & {Y_1}
	\arrow[""{name=0, anchor=center, inner sep=0}, "{(f_1,f_1)}", from=1-1, to=1-2]
	\arrow["{\X(d_1)}"', from=1-1, to=2-1]
	\arrow["{\Y(d_1)}", from=1-2, to=2-2]
	\arrow[""{name=1, anchor=center, inner sep=0}, "{f_0}", from=1-3, to=1-4]
	\arrow["{\X(s_0)}"', from=1-3, to=2-3]
	\arrow["{\Y(s_0)}", from=1-4, to=2-4]
	\arrow[""{name=2, anchor=center, inner sep=0}, "{f_1}"', from=2-1, to=2-2]
	\arrow[""{name=3, anchor=center, inner sep=0}, "{f_1}"', from=2-3, to=2-4]
	\arrow[between={0.4}{0.6}, Rightarrow, from=0, to=2]
	\arrow[between={0.4}{0.6}, Rightarrow, from=1, to=3]
\end{tikzcd}\]

That gives us exactly a lax functor of double categories in the sense of \cite[7.2]{GP99}. From that, we get that for a double category $\X$, the category $\IntAlg(\X):=\Hom_l(*,\X)$ is precisely the category of horizontal monads and their homomorphisms.

What about $w$-commutativity structures on $\Ss_{Seg}$? Our sketch being a $1$-category, we have $\Ss_{Seg}\otimes_{s,w} \Ss_{Seg}=\Ss_{Seg}\times \Ss_{Seg}$, so we are looking for a monoidal structures on $\Delta^{op}$ satisfying some extra conditions. As we are going to see in the next subsection, if a sketch $\Ss$ admits a $w$-commutativity structure $\mu\colon \Ss\otimes_{s,w}\Ss\to \Ss$, $u\colon 1\to \Ss$, the induced functor $u^*\colon \Mod_w(\Ss,\Cat)\to \Cat$ should be thought of as a forgetful functor, and hence it our case, it makes sense to look for monoidal structures with the unit being $[0]$. In fact, this is the only possible choice: for any unit $I$ of any monoidal category, it has to hold that the endomorphism monoid of $I$ is commutative, as pointed out by Tim Campion on mathoverflow. Following a similar reasoning as in his answer\footnote{See Tim Campion's answer at \href{https://mathoverflow.net/questions/402257/is-oplus-the-only-monoidal-structure-on-the-simplex-category}{mo.402257}.}, we can then in fact show that there is no such monoidal structure.

\begin{lemma}
    There is no monoidal structure on $\Delta$.
\end{lemma}
\begin{proof}
    Suppose there is such structure $\otimes\colon \Delta\times \Delta\to \Delta$. We saw above that the unit has to be $[0]$. Maps $[0]\to [n]$ can be identified with elements of the linearly ordered set $[n]$: therefore, the maps $\partial_i\otimes \partial_j \colon [0]=[0]\otimes [0]\to [1]\otimes [1]$, $i=0,1$, correspond to four (not necessarily distinct) points $\partial_{i,j}\in [1]\otimes [1]$. As these maps preserve ordering, we necessarily get that $\partial_{0,0}\le \partial_{0,1},\partial_{1,0}\le\partial_{1,1}$.    

    Now, the map $\sigma_0\colon [1]\to[0]$ induces two order-preserving maps $\sigma_0\otimes 1,1\otimes \sigma_0\colon [1]\otimes[1]\to [1]$ and we can directly compute how they behave on $\partial_{i,j}$. In particular, we have the following commutative diagrams:% For $i=0,1$ $\sigma_0\otimes1$ sends $\partial_{i,0}$ to $\perp$ and $\partial_{i,1}$ to $\top$ and $1\otimes \sigma_0$ sends $\partial_{0,i}$ to $\perp$ and $\partial_{1,i}$ to $\top$. The enforces the relations $\partial_{0,1}\le\partial_{1,0}$, $\partial_{1,0}\le\partial_{0,1}$ in $[1]\otimes [1]$ and hence $\partial_{1,0}=\partial_{0,1}$

    % https://q.uiver.app/#q=WzAsMTIsWzAsMCwiWzBdXFxvdGltZXNbMF0iXSxbMSwwLCJbMF1cXG90aW1lc1sxXSJdLFswLDEsIlsxXVxcb3RpbWVzWzBdIl0sWzEsMSwiWzFdXFxvdGltZXNbMV0iXSxbMCwyLCJbMF1cXG90aW1lc1swXSJdLFsxLDIsIlswXVxcb3RpbWVzWzFdIl0sWzMsMCwiWzBdXFxvdGltZXNbMF0iXSxbNCwwLCJbMV1cXG90aW1lc1swXSJdLFszLDEsIlswXVxcb3RpbWVzWzFdIl0sWzQsMSwiWzFdXFxvdGltZXNbMV0iXSxbMywyLCJbMF1cXG90aW1lc1swXSJdLFs0LDIsIlsxXVxcb3RpbWVzWzBdIl0sWzAsMiwiXFxwYXJ0aWFsX2pcXG90aW1lcyAxIiwyXSxbMiw0LCJcXHNpZ21hXzBcXG90aW1lcyAxIiwyXSxbNCw1LCIxXFxvdGltZXNcXHBhcnRpYWxfaSIsMl0sWzEsMywiXFxwYXJ0aWFsX2pcXG90aW1lcyAxIl0sWzIsMywiMVxcb3RpbWVzXFxwYXJ0aWFsX2kiLDJdLFswLDEsIjFcXG90aW1lc1xccGFydGlhbF9pIl0sWzMsNSwiXFxzaWdtYV8wXFxvdGltZXMgMSJdLFs2LDgsIjFcXG90aW1lc1xccGFydGlhbF9pIiwyXSxbOCwxMCwiMVxcb3RpbWVzIFxcc2lnbWFfMCIsMl0sWzEwLDExLCJcXHBhcnRpYWxfalxcb3RpbWVzIDEiLDJdLFs5LDExLCIxXFxvdGltZXMgXFxzaWdtYV8wIl0sWzgsOSwiXFxwYXJ0aWFsX2pcXG90aW1lcyAxIiwyXSxbNiw3LCJcXHBhcnRpYWxfalxcb3RpbWVzIDEiXSxbNyw5LCIxXFxvdGltZXNcXHBhcnRpYWxfaSJdLFswLDMsIlxccGFydGlhbF97aml9Il0sWzYsOSwiXFxwYXJ0aWFsX3tqaX0iXV0=
\[\begin{tikzcd}
	{[0]\otimes[0]} & {[0]\otimes[1]} && {[0]\otimes[0]} & {[1]\otimes[0]} \\
	{[1]\otimes[0]} & {[1]\otimes[1]} && {[0]\otimes[1]} & {[1]\otimes[1]} \\
	{[0]\otimes[0]} & {[0]\otimes[1]} && {[0]\otimes[0]} & {[1]\otimes[0]}
	\arrow["{1\otimes\partial_i}", from=1-1, to=1-2]
	\arrow["{\partial_j\otimes 1}"', from=1-1, to=2-1]
	\arrow["{\partial_{ji}}", from=1-1, to=2-2]
	\arrow["{\partial_j\otimes 1}", from=1-2, to=2-2]
	\arrow["{\partial_j\otimes 1}", from=1-4, to=1-5]
	\arrow["{1\otimes\partial_i}"', from=1-4, to=2-4]
	\arrow["{\partial_{ji}}", from=1-4, to=2-5]
	\arrow["{1\otimes\partial_i}", from=1-5, to=2-5]
	\arrow["{1\otimes\partial_i}"', from=2-1, to=2-2]
	\arrow["{\sigma_0\otimes 1}"', from=2-1, to=3-1]
	\arrow["{\sigma_0\otimes 1}", from=2-2, to=3-2]
	\arrow["{\partial_j\otimes 1}"', from=2-4, to=2-5]
	\arrow["{1\otimes \sigma_0}"', from=2-4, to=3-4]
	\arrow["{1\otimes \sigma_0}", from=2-5, to=3-5]
	\arrow["{1\otimes\partial_i}"', from=3-1, to=3-2]
	\arrow["{\partial_j\otimes 1}"', from=3-4, to=3-5]
\end{tikzcd}\]

Notice that all the vertical compositions are equal to identity. Denote the elements of $[1]$ by $\perp<\top.$ From the diagram on the left, we get that $\sigma_0\otimes 1$ sends $\partial_{0,0}$, $\partial_{1,0}$ to $\perp$ and $\partial_{0,1}$, $\partial_{1,1}$ to $\top$. From the diagram on the right, we get that $1\otimes \sigma_0$ sends $\partial_{0,0}$, $\partial_{0,1}$ to $\perp$ and $\partial_{1,0}$, $\partial_{1,1}$ to $\top$. This implies that $\partial_{0,1}$ and $\partial_{1,0}$ are distinct and one of the maps $\sigma_0\otimes 1$, $1\otimes \sigma_0$ reverses their ordering, which is a contradiction.
\end{proof}

\subsection{Kronecker product and idempotent theories} We are going to generalize the observations from \ref{subsec: 1-dim Kronecker Eckmann-Hilton}. Let $\T$ be a Lawvere $(2,2)$-theory and $\mu\colon \T\otimes_{s,w}\T\to\T$ a $w$-commutativity. So far we have viewed $(\T,\mu)$ as a pseudomonoid in marked $(2,2)$-sketches, but utilizing all our results on the syntax of $w$-commutativity, we will show that $\T$ is a pseudomonoid in the $(2,2)$-category of Lawvere $(2,2)$-theories with a $2$-dimensional version of a Kronecker product. 

The marked $(2,2)$-sketch $\F^{op}$ with its cocartesian monoidal structure is a symmetric pseudomonoid in $(\Cat_{(2,2)}^{\mathfrak{m}},\otimes_{s,w})$ (even in $(\mF\text{-}\Sketch,\otimes_{s,w})$). Therefore, one can consider the $(2,2)$-category $\Mod_{\F^{op}}$ of $\F^{op}$-modules. This category does not depend on $w$ since $\F^{op}\otimes_{s,w} \mM=\F^{op}\times \mM$. As $\Cat_{(2,2)}^{\mathfrak{m}}$ is cocomplete, $\Mod_{\F^{op}}$ has a tensor product $\otimes^{s,w}_{\F^{op}}$ for each $w\in W_{ord}$ defined for $\F^{op}$-modules $\mM,$ $\mN$ as a coeqalizer $$\mM\otimes_{s,w}\F^{op}\otimes_{s,w} \mN\rightrightarrows \mM\otimes_{s,w} \mN.$$ 

Any Lawvere theory is an $\F^{op}$-module (in $\mF$-sketches, so also in $\Cat_{(2,2)}^{\mathfrak{m}}$) with the action $a\colon\F^{op}\times \T\to \T$ with $a(m,n)=m\cdot n$; this follows from Lemma \ref{lemma: f-commutativity} by precomposing the unique $f$-commutativity structure $\T\otimes_{s,f}\T\to\T$ with $\theta\times 1\colon \F^{op}\times \T\to\T\otimes_{s,f}\T$.

\begin{lemma}
    For any Lawvere $(2,2)$-theories $\T,$ $\Ss$, their tensor product $\T\otimes_{\F^{op}}^{s,w}\Ss$ is again a Lawvere $(2,2)$-theory. The quotient map $\T\otimes_{s,w}\Ss\to\T\otimes_{\F^{op}}^{s,w}\Ss$ induces an equivalence $\Mod_w(\T\otimes_{\F^{op}}^{s,w}\Ss,\C)\simeq \Mod_w(\T,\Mod_w(\Ss,\C))$.
\end{lemma}

\begin{proof}
    A conceptual proof would use the fact that the $(2,2)$-category of $\mF$-sketches is presentable, but we have not found it in current literature. What we can do instead is to prove the statement manually: one knows how to construct coequalizers of $(2,2)$-categories (by taking a free $(2,2)$-category on a coequalizer in 2-globular sets), and then we can check that objects of $\T\otimes_{\F^{op}}^{s,w}\Ss$ are natural numbers as on objects, we are quotienting out by the relation $(n,1)\sim (1,n)$. Then, we get that $\T\otimes_{\F^{op}}^{s,w}\Ss(m,n)\cong\T\otimes_{\F^{op}}^{s,w}\Ss(m,1)^n$ by induction on $n$ using that the $1$- and $2$-cells in $\T\otimes_{\F^{op}}^{s,w}\Ss$ are generated by those of $\T$, $\Ss$, and the Gray cells. The second statement is straightforward.
\end{proof}

Let $\mu$ be a $w$-commutativity on a Lawvere $(2,2)$-theory $\T$. We know that $\theta\colon \F^{op}\to \T$ is a homomorphism of pseudomonoids in $\mF\text{-}\Sketch$, which implies that $\T$ is a $\F^{op}$-pseudoalgebra, i.e. a pseudomonoid object in $(\Mod_{\F^{op}},\otimes_{s,w})$. In other words, the map $\mu$ factors as $\T\otimes_{s,w}\T\to \T\otimes_{\F^{op}}^{s,w}\T\xrightarrow{\mu'}\T$.

\begin{definition} \label{def: Eckmann-Hilton}
    Let $\T$ be a Lawvere $(2,2)$-theory with a pseudocommutativity $\mu$. We say that $(\T,\mu)$ is \textit{$w$-idempotent} if $\mu'\colon \T\otimes_{\F^{op}}^{s,w}\T\to\T$ is an equivalence of $(2,2)$-categories. 
\end{definition}

\begin{example}
    If $\T$ is a commutative Lawvere 1-theory with the basis of $1$-cells where every map is either unit or unital, then following Proposition, one can check that $\T^{\flat}$ is pseudo-idempotent.
\end{example}

\begin{remark}
    If $\T$ is a Lawvere $(2,2)$-theory, pseudoidempotence of its underlying $(2,1)$-theory does not imply the pseudoidempotence of $\T$. Consider the Lawvere theory $\T^2_M$ associated to the double delooping $B^2M$ of some monoid $M$, as in Example \ref{example: double delooping}. If $M$ is nontrivial, this is not going to be an idempotent theory.
\end{remark}

\begin{lemma} \label{lemma: idempotence}    If $(\T,\mu)$ is $w$-idempotent then the lifting functor $$\widetilde{(-)}\colon \Mod_w(\T,\C)\to \Mod_w(\T,\Mod_w(\T,\C))$$ is an equivalence of categories.
\end{lemma}
\begin{proof}
    This follows from the previous lemma.
\end{proof}

\section{Semantics of $w$-commutativity}

After exploring the syntax of $w$-commutativity, we now turn to the semantics. We have already observed that a $w$-commutativity structure $\mu$ on a Lawvere $2$-theory $\T$ equips the category $\Mod_w(\T,\C)$ with a functor $$\widetilde{(-)}\colon\Mod_w(\T,\C)\to \Mod_w(\T,\Mod_w(\T,\C))$$ which is a section to the forgetful functors evaluating at $1_{\T}$ at the first or second variable. In this section, we are going to explore some useful implications of having such a functor. Notably, this enables us to prove some algebra-coalgebra interchange type results. We also observe the existence of certain internal $\underline{\Hom}_w(-,-)$ which will be crucial for discussing Fox's theorem.

\subsection{Closed structure} \label{subsec: closed structure I} Let $\T$ be a Lawvere $2$-theory equipped with a $w$-commutativity structure $\mu$. For any two models $\X,\Y\colon \T\to \C$, we can lift the category of $w$-homomorphisms of models $\Hom_w(\X,\Y)$ to a model $\underline{\Hom}_w(\X,\Y)\colon\T\to \Cat$ with $\underline{\Hom}_w(\X,\Y)(1)=\Hom_w(\X,\Y)$.   

\begin{definition}
    Let $\T$ be a Lawvere $2$-theory with a $w$-commutativity $\mu$ and $\X,\Y\colon \T\to \C$ two models of $\T$ in a $2$-category $\C$ with finite products. Define $\underline{\Hom}_w(\X,\Y)$ to be the following composition:
    % https://q.uiver.app/#q=WzAsNCxbMCwwLCJcXFQiXSxbMSwwLCJcXE1vZF93KFxcVCxcXEMpIl0sWzIsMCwiXFxtYXRoc2Z7UFNofShcXE1vZF93KFxcVCxcXEMpKSJdLFszLDAsIlxcQ2F0X3tcXGluZnR5fSJdLFswLDEsIlxcd2lkZXRpbGRle1xcWX0iXSxbMSwyLCJ5Il0sWzIsMywiXFxldl97XFxYfSJdXQ==
\[\begin{tikzcd}
	\T & {\Mod_w(\T,\C)} & {\mathsf{PSh}(\Mod_w(\T,\C))} & {\Cat_{\infty}}
	\arrow["{\widetilde{\Y}}", from=1-1, to=1-2]
	\arrow["y", from=1-2, to=1-3]
	\arrow["{\ev_{\X}}", from=1-3, to=1-4]
\end{tikzcd}\]
\end{definition} 

Here, $\widetilde{\Y}$ is the lift of $\Y$ corresponding to $\mu$ and $\ev_{\X}$ is the evaluation functor $F\mapsto F(\X)$. Since all the functors $\widetilde{\Y}$, $y$, and $\ev_{\X}$ preserves finite products, the composition above is a model. We can easily describe the low-dimensional data: \begin{itemize}
    \item for any $m\in \N$, $\underline{\Hom}_w(\X,\Y)(m)=\Hom_w(\X,\Y)^m$,
    \item for any $1$-cell $\alpha\colon m\to n$ in $\T$,  $\underline{\Hom}_w(\X,\Y)(\alpha)$ is the composition % https://q.uiver.app/#q=WzAsNCxbMCwxLCJcXE1vZChcXFQpX3coXFxYLFxcWV5tKSJdLFsyLDEsIlxcTW9kKFxcVClfdyhcXFgsXFxZXm4pIl0sWzAsMCwiXFxNb2QoXFxUKV93KFxcWCxcXFkpXm0iXSxbMiwwLCJcXE1vZChcXFQpX3coXFxYLFxcWSlebiJdLFswLDEsIlxcd2lkZXRpbGRle1xcWX0oXFxhbHBoYSlfKiJdLFsyLDAsIlxcY29uZyJdLFsxLDMsIlxcY29uZyJdLFsyLDMsIltcXFgsXFxZXV93KFxcYWxwaGEpIl1d
\[\begin{tikzcd}
	{\Hom_w(\X,\Y)^m} && {\Hom_w(\X,\Y)^n} \\
	{\Hom_w(\X,\Y^m)} && {\Hom_w(\X,\Y^n)}
	\arrow["{\underline{\Hom}_w(\X,\Y)(\alpha)}", from=1-1, to=1-3]
	\arrow["\cong", from=1-1, to=2-1]
	\arrow["{\widetilde{\Y}(\alpha)_*}", from=2-1, to=2-3]
	\arrow["\cong", from=2-3, to=1-3]
\end{tikzcd}\]
    \item for a $2$-cell $s\colon \alpha\to \alpha'$, $\underline{\Hom}_w(\X,\Y)(s)$ is the post-composition with $\widetilde{\Y}(s)$, similarly as with $1$-cells.
\end{itemize}

The above can be made functorial by defining $\underline{\Hom}_w(\X,-):\Mod_w(\T,\C)\to \Mod_w(\T,\Cat_{\infty})$ as a composition $$\Mod_w(\T,\C)\xrightarrow{\widetilde{(-)}}\Mod_w(\T,\Mod_w(\T,\C))\xrightarrow{(\ev_{\X}\circ y)^*}\Mod_w(\T,\Cat_{\infty}).$$ Indeed, as $\ev_{\X}\circ y$ preserves products, the precomposition $\Fun_{s,w}(\T,\Mod_w(\T,\C))\to \Fun_{s,w}(\T,\Cat_{\infty})$ sends models to models. %Note that of $f\colon \X\to \Y$ is in fact a strict homomorphism, i.e. the $2$-cells $f_\alpha$ are always invertible, then the same is true for

The previous construction enables us to promote the functor $\IntAlg\colon \Mod_w(\T,\C)\to \Cat$ to a 2-functor $\underline{\IntAlg}\colon\Mod_w(\T,\C)\to \Mod_w(\T,\Cat)$, $\X\mapsto \underline{\Hom}_l(*,\X)$, whenever our theory is $w$-commutative for $w\le l$, and similarly with $\IntCoalg$ for $w\le l^*.$ There is another variation as follows: 
\begin{definition}\label{def: variations of intalg}
    Let $\T$ be a Lawvere $(\infty,2)$-theory with a strict commutativity $\mu$. Denote by $\lambda^+, \lambda^-$ the inclusions of $\Mod_s(\T,\C)$ to $\Mod_l(\T,\C)$, $\Mod_{l^*}(\T,\Cat_\infty)$, respectively. Define the functors $\underline{\IntAlg}, \underline{\IntCoalg}\colon \Mod_s(\T,\C)\to \Mod_s(\T,\Cat_\infty)$ as the composition $$\Mod_s(\T,\C)\xrightarrow{\widetilde{(-)}}\Mod_s(\T,\Mod_s(\T,\C))\xrightarrow{(\ev_{*}\circ y\circ\lambda^\bullet)^*}\Mod_s(\T,\Cat_{\infty}).$$ for respective choices $\bullet=+,-$. For a Lawvere $(2,2)$-theory and a $(2,2)$-category $\C$ with finite products, we define $\underline{\IntAlg}, \underline{\IntCoalg}\colon\Mod_p(\T,\C)\to \Mod_p(\T,\Cat)$ analogously.
\end{definition}
Indeed, both $\lambda^+,$ $\lambda^-$ preserve products (which are pointwise, see Remark \ref{rmk: pointwise products}) so the definition above makes sense.

\begin{example}
    Consider the example $\T=\T_{E_\infty}$. Recall that for any $w\in W$, the $(\infty,2)$-category $\Mod_w(\T,\Cat_\infty)$ is the $2$-category of symmetric monoidal $\infty$-categories, symmetric $w$-monoidal functors, and symmetric monoidal natural transformations. Let $X, Y$ be symmetric monoidal categories and denote the corresponding functors $\T_{E_\infty}\to \Cat$ by $\X,\Y$, respectively. As $\T_{E_\infty}$ has a commutativity structure by Example \ref{example: Eoo monoids}, we can consider a symmetric monoidal category $\underline{\Hom}_s(\X,\Y)$. 
    
    By definition, the underlying category is the category $\Fun^\otimes(X,Y)$ of symmetric monoidal functors and monoidal natural transformations, and it is easy to see that the monoidal unit is the constant functor sending everything in $X$ to the monoidal unit $u_Y$ of $Y$. How does the tensor product look like? Take some functors $F, G\in \Fun^\otimes(X,Y)$ with the structure maps $F_{x,x'}\colon F(x)\otimes F(x')\xrightarrow{\sim} F(x\otimes x')$, $F_0\colon u_Y\xrightarrow{\sim} F(u_X)$, and similarly for $G$. Then the tensor product $F\circledast G$ has the underlying functor $X\to Y$ pointwise sending $x$ to $F(x)\otimes G(x)$. The structure map  $(F\circledast G)_0$ is given by the tensor product $F_0\otimes G_0$ and $(F\circledast G)_{x,x'}$ is given by composing $F_{x,x'}\otimes G_{x,x'}$ with the symmetry $F(x)\otimes G(x)\otimes F(x') \otimes G(x') \xrightarrow{1s1}F(x)\otimes F(x') \otimes G(x) \otimes G(x')$.

    As commutativity implies (op)lax commutativity, we can also consider the symmetric monoidal category $\underline{\Hom}_l(\X,\Y)$ which works exactly the same as the above, just with symmetric \textit{lax} monoidal functors (i.e., the structure maps $F_{x,x'}$, $F_0$ being not necessarily invertible). In particular, we see that $\underline{\IntAlg}(\X)=\underline{\Hom}_l(*,\X)\simeq \mathsf{CAlg}(X)$ is the category of commutative algebras in $X$ with the same tensor product and unit as in $X$, and similarly for $\underline{\IntCoalg}(\X)$ and coalgebras.
\end{example}

More generally, consider a map of Lawvere $2$-theories $\rho\colon \Ss\to \T$ together with a $w$-commutation structure $\mu\colon \Ss\otimes_{s,w}\Ss\to \T$. Then for any $\T$-model $\Y$ and an $\Ss$-model $\X$, we get an $\Ss$-model $\underline{\Hom}_w(\X,\Y)\colon \Ss\to \Cat_{\infty}$ defined in the analogous fashion, functorial in $\Y$.

\begin{definition}    Let $\mu$ be a $w$-commutation as above. For an $\T$-model $\X$ and a $\T$-model $\Y$ in $\C$, define $\underline{\Hom}_w(\X,-)$ to be the following composition:$$\Mod_w(\T,\C)\xrightarrow{\widetilde{(-)}}\Mod_w(\Ss,\Mod_w(\Ss,\C))\xrightarrow{\ev_{\X}\circ\, y}\Mod_w(\Ss,\Cat).$$
\end{definition}

For $\Ss$-models $\X,\Y$ and a $\T$-model $\mZ$, define $\Mul_w(\X,\Y;\mZ):= \Nat_{s,w}(\Y\circ \mu, \times\circ \X\otimes_{s,w}\Y)$.  Objects of this category are $2$-cells $f$  filling the square below and morphisms are modifications between those.

% https://q.uiver.app/#q=WzAsNCxbMCwwLCJcXFNzXFxvdGltZXNfe3Msd31cXFNzIl0sWzAsMSwiXFxUIl0sWzEsMSwiXFxDYXQiXSxbMSwwLCJcXENhdFxcb3RpbWVzX3tzLHd9XFxDYXQiXSxbMCwxLCJcXG11IiwyXSxbMSwyLCJcXG1aIiwyXSxbMywyLCJcXHRpbWVzIl0sWzAsMywiXFxYXFxvdGltZXNfe3Msd31cXFkiXSxbNyw1LCJmIiwwLHsic2hvcnRlbiI6eyJzb3VyY2UiOjIwLCJ0YXJnZXQiOjIwfX1dXQ==
\[\begin{tikzcd}
	{\Ss\otimes_{s,w}\Ss} & {\Cat\otimes_{s,w}\Cat} \\
	\T & \Cat
	\arrow[""{name=0, anchor=center, inner sep=0}, "{\X\otimes_{s,w}\Y}", from=1-1, to=1-2]
	\arrow["\mu"', from=1-1, to=2-1]
	\arrow["\times", from=1-2, to=2-2]
	\arrow[""{name=1, anchor=center, inner sep=0}, "\mZ"', from=2-1, to=2-2]
	\arrow["f", between={0.2}{0.8}, Rightarrow, from=0, to=1]
\end{tikzcd}\]

\begin{proposition} \label{prop: closed structure without operad}
 Let $\mu$ be a strict commutation. Then we have the following equivalence: $\Mul_s(\X,\Y;\mZ)\cong \Hom_s(\X,\underline{\Hom}_s(\Y,\mZ))$. If it is symmetric, we get $\Hom_s(\X,\underline{\Hom}_s(\Y,\mZ))\simeq \Hom_s(\Y,\underline{\Hom}_s(\X,\mZ))$.
\end{proposition} 

\begin{proof}
We are going to proceed by the calculus of ends, developed in the $\infty$-categorical setting by \cite{Lor18}. Recall that for a functor $T\colon \C^{op}\times \C\to \Cat$, we define its end $\int_cT(c,c):=\lim_{\Tw(\C)} T\circ \Sigma$ where $\Sigma\colon \Tw(\C)\to \C^{op}\times \C$ is the unstraightening of a (two-sided) hom-functor. Hence, if $F,G\colon \C\to \Cat$ are any functors, we can compute a category of natural transformations $F\rightarrow G$ as an end $\int_c\Fun(F(c),G(c)).$ Using this, we compute as follows:

 \begin{align}\Nat(\times \circ (\X\times \Y),\mZ\mu)&\cong
    \int_{a,b} \Fun(X^a\times Y^b,Z^{ab})\nonumber\\
    &\cong \int_{a} \int_b\Fun(X^a\times Y^b,Z^{ab})\nonumber\\
    &\cong \int_{a} \int_b\Fun(X^a,\Fun(Y^b,Z^{ab}))\nonumber\\
    &\cong \int_a\Fun(X^a,\int_b\Fun(Y^b,Z^{ab})) \nonumber\\
    &\cong \int_a \Fun(X^a,\Nat(\Y,\mZ^a))\nonumber\\
    &\cong \Nat(\X,\Nat(\Y,\mZ^\bullet))\nonumber
\end{align}
To get from the first line to the second, we used the Fubini rule \cite[Thm. 2.2]{Lor18}, and the rest is standard. Since $\Mod_s(\T,\Cat)$ is a full sub-2-category of $\Fun(\T,\Cat)$, we can identify the right hand side at the end with $\Hom_s(\X,\underline{\Hom}_s(\Y,\mZ)$. The observation about the symmetry is clear.
\end{proof}

We obtain some immediate application by reproving some algebra-coalgebra interchange results below.

\begin{example}
    As usual, consider first the commutativity $\mu$ on $\T_{E_\infty}$. We use the result above to reprove the interchange of $E_\infty$-algebras and coalgebras. For a symmetric monoidal $\infty$-category $\V$ we can write $\Alg_{E_\infty}(\V)\simeq \Fun^{\otimes}(\P,\V)$ where $\P:=\mathrm{Env}(E_\infty^{\otimes})$ is the symmetric monoidal $\infty$-category arising as an envelope of the $E_\infty$-operad. We have a symmetric monoidal structure $\underline{\Alg}_{E_\infty}(\V):=\underline{\Hom}_s(\P,\V)$ on $\Alg_{E_\infty}(\V)$ induced by $\mu$, and similarly we have $\underline{\Coalg}_{E_\infty}(\V):=\underline{\Hom}_s(\P^{op},\V)$. Then, we compute: 
\begin{align*}
        \Alg_{E_{\infty}}(\underline{\Coalg_{E_\infty}}(\V))&\simeq \Hom_s(\P,\underline{\Hom}_s(\P^{op},\V))\\
        &\simeq \Mul_s(\P,\P^{op};\V)\\
        &\simeq \Mul_s(\P^{op},\P;\V)\\
        &\simeq \Hom_s(\P^{op},\underline{\Hom}_s(\P,\V))\\
        &\simeq \Coalg_{E_{\infty}}(\underline{\Alg}_{E_\infty}(\V)).
    \end{align*}
\end{example}

\begin{example}
    Let $\V$ be a braided monoidal $\infty$-category, i.e. a model of $\T_{E_2}$. There exists a commutation $\T_{E_1}\times \T_{E_1}\to \T_{E_2}$ (see Example \ref{example: braided}). For a monoidal $\infty$-category $\W$, $\underline{\Hom}_s(\W,\V)$ is precisely the category of braided monoidal functors $\Fun^\otimes_{E_2}(\Y,\mZ)$ which is monoidal, although no longer braided monoidal. We can use Proposition \ref{prop: closed structure without operad} to prove the interchange of $E_1$-algebras and coalgebras in $\V$, using the similar computation as above. This result was proved in \cite[Appendix A]{Erg22} by related methods, but in a language of operads.
\end{example}

In the next section, we are going to develop a framework to promote the equivalence above to a part of a closed $(\infty,2)$-operad (or $(2,2)$-multicategory) structure on $\Mod_w(\T,\Cat)$. As an application, we will obtain that the endofunctor $\underline{\IntAlg}$ is in fact a $2$-comonad. This generalizes the first part of Fox's theorem, which states the latter for $\T=\T_{cmon}^\flat$.

\subsection{Bilax homomorphisms}

Let $\mu\colon \T\otimes_{s,p}\T\to \T$ be a pseudocommutativity structure on a Lawvere $(2,2)$-theory $\T$ and $\C$ a $(2,2)$-category with finite products. We are now going to take an advantage of the isomorphism %So far, we have only made use of the right closed structure on $\mF\text{-}\Sketch^{\Psi}$. Combined with the right structure, we get the isomorphism (as explain in ??) 
$$i\colon\Mod_l(\T,\Mod_{l^*}(\T,\C))\xrightarrow{\sim} \Mod_{l^*}(\Mod_l(\T,\C))$$ from Proposition \ref{prop: swapping lax and colax}. Since any pseudocommutativity
is also a (co)lax commutativity, any (co)lax homomorphism $f\colon \X\to \Y$ lifts to a (co)lax homomorphism $\widetilde{f}\colon\widetilde{\X}\to \widetilde{\Y}$.

\begin{lemma} \label{lemma: bilax}
    Let $\X,\Y\colon\T\to \C$ be models, $X=\X(1)$, $Y=\Y(1)$. Suppose we have a lax homomorphism $\fup\colon \X\to \Y$ together with a colax homomorphism $\flow\colon \X\to \Y$ such that they both have the same underlying $1$-cell $f\colon X\to Y$.

    %Let $f\colon X\to Y$ be a $1$-cell in $\C$. Suppose that we have structure of a lax homomorphism on $f$ consisting of the family of maps $\fup_\alpha\colon f^n\circ \X(\alpha)\Rightarrow \Y(\alpha) \circ f^m$ and a colax homomorphism structure given by maps $\flow_\alpha\colon \Y(\alpha) \circ f^m \Rightarrow f^n\circ \X(\alpha)$ for each $\alpha\colon m\to n$ in $\T$. 
    
    Then the following are equivalent: 

    \begin{enumerate}
        \item For each $\alpha\colon m\to n$ in $\T$, the lax structure cell $\fup_\alpha\colon f^n\circ \X(\alpha)\Rightarrow \Y(\alpha) \circ f^m$ lifts to a $2$-cell ${\flow}^n\circ \widetilde{\X}(\alpha)\Rightarrow\widetilde{\Y}(\alpha)\circ {\flow}^m$ in $\Mod_{l^*}(\T,\C)$.
        \item For each $\alpha\colon m\to n$ in $\T$, the colax structure cell $\flow_\alpha\colon \Y(\alpha) \circ f^m \Rightarrow f^n\circ \X(\alpha)$ lifts to a 2-cell $\widetilde{\Y}(\alpha)\circ {\fup}^m \Rightarrow {\fup}^n\circ \widetilde{\X}(\alpha)$ in $\Mod_l(\T,\C)$.
    \end{enumerate}
\end{lemma}

\begin{proof}
    Condition $(1)$ means that the data $\{\flow^n,\fup_\alpha\}$ gives us a $1$-cell in $\Mod_l(\T,\Mod_{l^*}(\T,\C))$. Condition $(2)$ then tells us that the data $\{\fup^n,\flow_\alpha\}$ assembles into a $1$-cell in $\Mod_{l^*}(\T,\Mod_l(\T,\C)$. Clearly, these two correspond to each other via the isomorphism $i$. 
\end{proof}

In more concrete terms, checking both of the conditions above boils down to checking the same equation for each $\beta\colon k\to l$, $\alpha\colon m\to n$ in $\T$ (the unnamed isomorphisms below come from the pseudocommutativity structure):

\[\begin{tikzcd}
	& {(X^n)^k} &&&& {(X^n)^k} \\
	{(X^m)^k} && {(Y^n)^k} && {(X^m)^k} && {(Y^n)^k} \\
	& {(X^n)^l} &&&& {(Y^m)^k} \\
	{(X^m)^l} && {(Y^n)^l} && {(X^m)^l} && {(Y^n)^l} \\
	& {(Y^m)^l} &&&& {(Y^m)^l}
	\arrow[""{name=0, anchor=center, inner sep=0}, "{(f^n)^k}", curve={height=-12pt}, from=1-2, to=2-3]
	\arrow["{\X^n(\beta)}"', from=1-2, to=3-2]
	\arrow["{(f^n)^k}", curve={height=-12pt}, from=1-6, to=2-7]
	\arrow[""{name=1, anchor=center, inner sep=0}, "{\X(\alpha)^k}", curve={height=-12pt}, from=2-1, to=1-2]
	\arrow["{\X^m(\beta)}"', from=2-1, to=4-1]
	\arrow[""{name=2, anchor=center, inner sep=0}, "{\Y^n(\beta)}", from=2-3, to=4-3]
	\arrow["{\X(\alpha)^k}", curve={height=-12pt}, from=2-5, to=1-6]
	\arrow[""{name=3, anchor=center, inner sep=0}, "{(f^m)^k}", curve={height=12pt}, from=2-5, to=3-6]
	\arrow[""{name=4, anchor=center, inner sep=0}, "{\X^m(\beta)}"', from=2-5, to=4-5]
	\arrow["{\Y^n(\beta)}", from=2-7, to=4-7]
	\arrow[""{name=5, anchor=center, inner sep=0}, "{(f^n)^l}"', curve={height=-12pt}, from=3-2, to=4-3]
	\arrow["{(\fup_\alpha)^k}"', between={0.3}{0.8}, Rightarrow, from=3-6, to=1-6]
	\arrow[""{name=6, anchor=center, inner sep=0}, "{\Y(\alpha)^k}", curve={height=12pt}, from=3-6, to=2-7]
	\arrow["{\Y^m(\beta)}"{description}, from=3-6, to=5-6]
	\arrow[""{name=7, anchor=center, inner sep=0}, "{\X(\alpha)^l}"', curve={height=-12pt}, from=4-1, to=3-2]
	\arrow["{(f^m)^l}"', curve={height=12pt}, from=4-1, to=5-2]
	\arrow[""{name=8, anchor=center, inner sep=0}, "{(f^m)^l}"', curve={height=12pt}, from=4-5, to=5-6]
	\arrow["{(\fup_\alpha)^l}"', between={0.3}{0.8}, Rightarrow, from=5-2, to=3-2]
	\arrow["{\Y(\alpha)^l}"', curve={height=12pt}, from=5-2, to=4-3]
	\arrow[""{name=9, anchor=center, inner sep=0}, curve={height=12pt}, from=5-6, to=4-7]
	\arrow["\cong", draw=none, from=1, to=7]
	\arrow["{=}"{marking, allow upside down}, draw=none, from=2, to=4]
	\arrow["{\flow^n_\beta}", shift left=3, between={0.4}{0.7}, Rightarrow, from=5, to=0]
	\arrow["\cong"', draw=none, from=6, to=9]
	\arrow["{\flow_\beta^m}"', shift right=3, between={0.3}{0.7}, Rightarrow, from=8, to=3]
\end{tikzcd}\]

Whenever the conditions of Lemma \ref{lemma: bilax} are satisfied, we have a homomorphism $\overline{\underline{f}}\colon \widetilde{\X}\to \widetilde{\Y}$, which is a $1$-cell in $\Mod_l(\T,\Mod_{l^*}(\T,\C))\cong \Mod_{l^*}(\T,\Mod_l(\T,\C)$ and gets sent to $\fup$ and $\flow$ by the respective forgetful functors.

\begin{definition} \label{def:bilax}
    Let $\X,\Y\colon\T\to \C$ be models, $X=\X(1)$, $Y=\Y(1)$. We define a \textit{bilax structure} on $f\colon X\to Y$ to be a pair of compatible lax and colax homomorphism structures on $f$ as in the previous lemma -- i.e., a lax homomorphism $\fup$, a colax homomorphism $\flow$ with $\fup(1)=\flow(1)=f$ such that the lift $\overline{\flow}$ as above exists.
\end{definition}

As a corollary, we get that for a pseudocommutative Lawvere $2$-theory $\T$, we have an isomorphism $$\IntCoalg(\widetilde{\X}_l)\cong \IntAlg(\widetilde{\X}_{l^*}),$$ where $\widetilde{\X}_w$ denotes the composition $\T\xrightarrow{\widetilde{\X}}\Mod_p(\T,\C)\subset \Mod_w(\T,\C)$ for $w=l,l^*.$ By the previous lemma, both capture the bilax structures on a $1$-cell $*\to X$ in $\C$.

\begin{definition}
    Let $\T$ be a Lawvere $(2,2)$-theory with a pseudocommutativity $\mu$, $\X\colon \T\to \C$ a model. Define $\IntBialg(\X):=\IntCoalg(\widetilde{\X}_l)$ where $\widetilde{\X}_l\colon \T\to \Mod_p(\T,\C)\subset \Mod_l(\T,\C)$ is the lift of $\X$. %This category is isomorphic to $\IntAlg(\widetilde{\X})$ for the lift $\widetilde{X}\colon \T\to \Mod_p(\T,\C)\subset \Mod_c(\T,\C).$
\end{definition}

\begin{example}
    If $\T=\T_{E_\infty}$ and $\V,\V'$ models in $\Cat$, i.e.\ symmetric monoidal categories, then our definition of a bilax structure on a map $\V\to \V'$ corresponds to the usual definition of a bilax symmetric monoidal functor $\V\to\V'$. Also, $\IntBialg(\V)$ is the category of bicommutative bimonoids in $\V$. %The same holds for $\T_{braid}$, the theory for braided monoids.
\end{example}

\subsection{Convolution algebras} \label{subsec: convolution}
It is well known that if $\V$ is a monoidal 1-category, $C$ a comonoid and $M$ a monoid, then the hom-set $\V(C,M)$ inherits a monoid structure called convolution. We are going to see that this generalizes to any Lawvere $(2,2)$-theory $\T$ equipped with an identity on objects, product preserving functor $\rho \colon \T\to \T_{E_\infty}$ and a basis $B$ of $1$-cells such that $\rho$ sends $B$ into $\langle B_{can}\rangle$ (recall that the latter consists of a unit and iterated multiplication maps). The reason we need this is that under these conditions, $\Set$ with the cartesian monoidal structure becomes a $\T$-model in $\Cat$ via the composition $\T\xrightarrow{\rho}\T_{E_\infty}\to\Cat$ (where the second map specifies the cartesian monoidal structure on $\Set$).

\begin{proposition} \label{prop: conv product}
    Let $\T$ be a Lawvere $(2,2)$-theory and suppose we have a basis $B$ and a morphism $\rho\colon \T\to \T_{E_\infty}$ of Lawvere $(2,2)$-theories satisfying the condition above. Let $\X\colon \T\to \Cat$ be a model, $A\colon *\to \X$ an internal algebra and $C\colon *\to \X$ an internal coalgebra. Put $X:=\X(1)$, $a:=A(*),$ $c:=C(*)$. Then the hom-set $X(c,a)$ has a structure of an internal algebra in the $\T$-model $\Set$ defined above.
\end{proposition}

\begin{proof}
    We start with a general observation: if $F\colon \C\to \D$ is any functor between $1$-categories, the maps $F_{x,y}\colon\C(x,y)\to \D(Fx,Fy)$ assemble into a natural transformation $\phi$ fitting in the triangle below.
    % https://q.uiver.app/#q=WzAsMyxbMCwwLCJcXENee29wfVxcdGltZXMgXFxDIl0sWzAsMSwiXFxEXntvcH1cXHRpbWVzIFxcRCJdLFsxLDEsIlxcYnVsbGV0Il0sWzAsMSwiRl57b3B9XFx0aW1lcyBGIiwyXSxbMCwyLCJcXEMoLSwtKSJdLFsxLDIsIlxcRCgtLC0pIiwyXSxbNCwxLCJcXHBoaSIsMix7InNob3J0ZW4iOnsic291cmNlIjoyMH19XV0=
\[\begin{tikzcd}
	{\C^{op}\times \C} \\
	{\D^{op}\times \D} & \Set
	\arrow["{F^{op}\times F}"', from=1-1, to=2-1]
	\arrow[""{name=0, anchor=center, inner sep=0}, "{\C(-,-)}", from=1-1, to=2-2]
	\arrow["{\D(-,-)}"', from=2-1, to=2-2]
	\arrow["\phi"', between={0.2}{1}, Rightarrow, from=0, to=2-1]
\end{tikzcd}\]

If $\X\colon \T\to \Cat$ is a model, then the hom-functor $X^n(-,-)\colon (X^n)^{op}\times X^n\to \Set$ factors through the product functor $\prod_{i=1}^n\colon \Set^n\to \Set$. We are going to denote the functor $(X^n)^{op}\times X^n\to \Set^n$ by $X^n(-,-)$ as well.

Using the cell $\phi$ above, we can make the hom-functors into a lax homomorphism $\mathcal{H}om\colon\X^{op}\times \X\to \Set$: we define $\mathcal{H}om_n=X^n(-,-)$ and for any $\alpha\colon n\to 1$ in our basis $B$ (therefore, $\alpha$ which gets sent to $n$-fold product by $\rho$), we define $\mathcal{H}om_\alpha$ to be the cell $\phi$ above for choices $\C=X^n$, $\D=X$, $F=\X(\alpha)$, as depicted in the following diagram:
% https://q.uiver.app/#q=WzAsNCxbMCwyLCJYXntvcH1cXHRpbWVzIFgiXSxbMCwwLCIoWF57b3B9KV5uXFx0aW1lcyBYXm4iXSxbMSwwLCJcXFNldF5uIl0sWzEsMiwiXFxTZXQiXSxbMSwwLCJcXFgoXFxhbHBoYSlee29wfVxcdGltZXMgXFxYKFxcYWxwaGEpIiwyXSxbMSwyLCJYXm4oLSwtKSJdLFswLDMsIlgoLSwtKSIsMl0sWzIsMywiXFxwcm9kX3tpPTF9Xm4iXSxbNSw2LCJcXG1hdGhjYWx7SH1vbV9cXGFscGhhIiwwLHsib2Zmc2V0IjozLCJzaG9ydGVuIjp7InNvdXJjZSI6MzAsInRhcmdldCI6NDB9fV1d
\[\begin{tikzcd}
	{(X^{op})^n\times X^n} & {\Set^n} \\
	\\
	{X^{op}\times X} & \Set
	\arrow[""{name=0, anchor=center, inner sep=0}, "{X^n(-,-)}", from=1-1, to=1-2]
	\arrow["{\X(\alpha)^{op}\times \X(\alpha)}"', from=1-1, to=3-1]
	\arrow["{\prod_{i=1}^n}", from=1-2, to=3-2]
	\arrow[""{name=1, anchor=center, inner sep=0}, "{X(-,-)}"', from=3-1, to=3-2]
	\arrow["{\mathcal{H}om_\alpha}", shift right=3, between={0.3}{0.6}, Rightarrow, from=0, to=1]
\end{tikzcd}\]

An $i$-th projection map $p_i\colon n\to 1$ in $\T$ induces an $i$-th projection maps on $X^n$ and $\Set$ and the Hom clearly commutes with these, as in the diagram below (vertical maps are $i$-th projections).

% https://q.uiver.app/#q=WzAsNCxbMCwyLCJYXntvcH1cXHRpbWVzIFgiXSxbMCwwLCIoWF57b3B9KV5uXFx0aW1lcyBYXm4iXSxbMSwwLCJcXFNldF5uIl0sWzEsMiwiXFxTZXQiXSxbMSwwLCJcXFgocF9pKV57b3B9XFx0aW1lcyBcXFgocF9pKSIsMl0sWzEsMiwiWF5uKC0sLSkiXSxbMCwzLCJYKC0sLSkiLDJdLFsyLDMsIlxccGlfaSJdXQ==
\[\begin{tikzcd}
	{(X^{op})^n\times X^n} & {\Set^n} \\
	\\
	{X^{op}\times X} & \Set
	\arrow["{X^n(-,-)}", from=1-1, to=1-2]
	\arrow["{\X(p_i)^{op}\times \X(p_i)}"', from=1-1, to=3-1]
	\arrow["{\pi_i}", from=1-2, to=3-2]
	\arrow["{X(-,-)}"', from=3-1, to=3-2]
\end{tikzcd}\] 

Thus, we have indeed specified a lax homomorphism (functoriality is clear). Then, we define the internal algebra structure $\X(c,a)$ on $X(c,a)$ is a composition $$\X(c,a):=\mathcal{H}om\circ(C^{op}\times A)\colon *\cong *\times *\rightarrow \X^{op}\times \X\rightarrow \Set.$$ As $C\colon *\to \X$ is colax, $C^{op}$ is lax, so the above is well-defined, and clearly $n$ gets mapped to $X(c,a)^n$. \end{proof}

For reader's convenience, we add one detail to the proof above. Namely, consider any map $\alpha\colon n\to 1$ in $B$. Then the components of $C$ and $A$ produce natural transformations $o_a\colon \X(\alpha)\circ A^n\to A$, $o_c\colon C\to \X(\alpha)\circ C^n$ as below.
% https://q.uiver.app/#q=WzAsNixbMCwxLCIqIl0sWzEsMCwiWF5uIl0sWzEsMSwiWCwiXSxbMiwxLCIqIl0sWzMsMSwiWCJdLFszLDAsIlhebiJdLFswLDEsIkFebiJdLFsxLDIsIlxcWChcXGFscGhhKSJdLFswLDIsIkEiLDJdLFszLDQsIkMiLDJdLFs1LDQsIlxcWChcXGFscGhhKSJdLFszLDUsIkNebiJdLFs2LDIsIm9fYSIsMCx7InNob3J0ZW4iOnsic291cmNlIjoyMCwidGFyZ2V0IjoyMH19XSxbNCwxMSwib19jIiwyLHsic2hvcnRlbiI6eyJzb3VyY2UiOjIwLCJ0YXJnZXQiOjIwfX1dXQ==
\[\begin{tikzcd}
	& {X^n} && {X^n} \\
	{*} & {X,} & {*} & X
	\arrow["{\X(\alpha)}", from=1-2, to=2-2]
	\arrow["{\X(\alpha)}", from=1-4, to=2-4]
	\arrow[""{name=0, anchor=center, inner sep=0}, "{A^n}", from=2-1, to=1-2]
	\arrow["A"', from=2-1, to=2-2]
	\arrow[""{name=1, anchor=center, inner sep=0}, "{C^n}", from=2-3, to=1-4]
	\arrow["C"', from=2-3, to=2-4]
	\arrow["{o_a}", between={0.2}{0.8}, Rightarrow, from=0, to=2-2]
	\arrow["{o_c}"', between={0.2}{0.8}, Rightarrow, from=2-4, to=1]
\end{tikzcd}\]

From that, we can construct the following natural transformation: 

% https://q.uiver.app/#q=WzAsNSxbMCwyLCIqIl0sWzEsMiwiWF57b3B9XFx0aW1lcyBYIl0sWzEsMCwiKFhee29wfSleblxcdGltZXMgWF5uIl0sWzIsMCwiXFxTZXRebiJdLFsyLDIsIlxcU2V0Il0sWzAsMSwiKENee29wfSxBKSIsMl0sWzAsMiwiKChDXntvcH0pXm4sQV5uKSJdLFsyLDEsIlxcWChcXGFscGhhKV57b3B9XFx0aW1lcyBcXFgoXFxhbHBoYSkiLDFdLFsyLDNdLFsxLDRdLFszLDQsIlxccHJvZF97aT0xfV5uIl0sWzYsMSwib19jXntvcH1cXHRpbWVzIG9fYSAiLDIseyJzaG9ydGVuIjp7InNvdXJjZSI6MjAsInRhcmdldCI6MjB9fV0sWzgsOSwiXFxwaGkiLDAseyJzaG9ydGVuIjp7InNvdXJjZSI6MzAsInRhcmdldCI6NDB9fV1d
\[\begin{tikzcd}
	& {(X^{op})^n\times X^n} & {\Set^n} \\
	\\
	{*} & {X^{op}\times X} & \Set
	\arrow[""{name=0, anchor=center, inner sep=0}, from=1-2, to=1-3]
	\arrow["{\X(\alpha)^{op}\times \X(\alpha)}"{description}, from=1-2, to=3-2]
	\arrow["{\prod_{i=1}^n}", from=1-3, to=3-3]
	\arrow[""{name=1, anchor=center, inner sep=0}, "{((C^{op})^n,A^n)}", from=3-1, to=1-2]
	\arrow["{(C^{op},A)}"', from=3-1, to=3-2]
	\arrow[""{name=2, anchor=center, inner sep=0}, from=3-2, to=3-3]
	\arrow["\mathcal{H}om_\alpha", between={0.3}{0.6}, Rightarrow, from=0, to=2]
	\arrow["{o_c^{op}\times o_a }"', between={0.2}{0.8}, Rightarrow, from=1, to=3-2]
\end{tikzcd}\]

This gives us $\X(c,a)$ evaluated on $\alpha$.

\section{Higher Day convolution and commutativity} \label{sec: day conv}

In this section, we will examine certain $(\infty,2)$-operad or $(2,2)$-multicategory structures on $2$-categories of models $\Mod_w(\T,\Cat).$ Our main tool is a result of independent interest: we develop certain generalized Day convolution structure. Let us first informally describe how this structure should look like without defining what we precisely mean by an $(\infty,n)$-operad; roughly, this is something like an $\infty$-operad but the hom-objects $\Mul(X_1,\dots,X_n;Y)$ are in fact $(\infty,n)$-categories instead of just $\infty$-groupoids. We will make this precise in the next subsection.

Let $\V$ be any monoidal $(\infty,n)$-category; for our application, we need at least the case $n=3$ so that we can access our main example $\V=(\Cat^{\mathfrak{m}}_{(\infty,2)},\times)$ which is self-enriched and hence also $(\Cat_{(\infty,2)},\times)$-enriched. Consider two monoids $M_0,M_1\colon \Delta^{op}\to \V$. We introduce the following notation for $i=0,1$:\begin{itemize}
    \item $\mM_i:=M_i(1)$,
    \item we denote the image of the unique active morphism $n\to 1$ in $\Delta^{op}$ by $m_i^n$,
    \item we denote by $\gamma^i_{k_1,\dots,k_n}$ the associativity constraints as below.% https://q.uiver.app/#q=WzAsNCxbMCwxLCJcXG1NX2lee1xcb3RpbWVzIFxcc3VtX3tpPTF9Xm4ga19pfSJdLFsxLDEsIlxcbU1faSJdLFsxLDAsIlxcbU1faV57XFxvdGltZXMgbn0iXSxbMCwwLCJcXGJpZ290aW1lc197aT0xfV5uIFxcbU1faV57XFxvdGltZXMga19pfSJdLFswLDMsIlxcY29uZyJdLFszLDIsIlxcb3RpbWVzX3tpPTF9Xm4gbV97a19pfSJdLFsyLDEsIm1fbiJdLFswLDEsIm1fe1xcc3VtX2kga19pfSIsMl0sWzUsNywiXFxnYW1tYV5pX3tuLGtfaX0iLDIseyJzaG9ydGVuIjp7InNvdXJjZSI6NDAsInRhcmdldCI6NDB9fV0sWzcsNSwiXFxzaW1lcSIsMix7InNob3J0ZW4iOnsic291cmNlIjoyMCwidGFyZ2V0IjoyMH0sInN0eWxlIjp7ImJvZHkiOnsibmFtZSI6Im5vbmUifSwiaGVhZCI6eyJuYW1lIjoibm9uZSJ9fX1dXQ==
\[\begin{tikzcd}
	{\bigotimes_{i=1}^n \mM_i^{\otimes k_i}} & {\mM_i^{\otimes n}} \\
	{\mM_i^{\otimes \sum_{i=1}^n k_i}} & {\mM_i}
	\arrow[""{name=0, anchor=center, inner sep=0}, "{\otimes_{i=1}^n m_{k_i}}", from=1-1, to=1-2]
	\arrow["{m_n}", from=1-2, to=2-2]
	\arrow["\cong", from=2-1, to=1-1]
	\arrow[""{name=1, anchor=center, inner sep=0}, "{m_{\sum_i k_i}}"', from=2-1, to=2-2]
	\arrow["{\gamma^i_{n,k_i}}"', between={0.4}{0.6}, Rightarrow, from=0, to=1]
	\arrow["\simeq"', draw=none, from=1, to=0]
\end{tikzcd}\]
\end{itemize}

Then, we will show that the hom $(\infty,n-1)$-category $\Hom_{\V}(\mM_0,\mM_1)$ admits an $(\infty,n-1)$-operad structure. We can describe the low-dimensional data as follows:
\begin{enumerate}
    \item For any tuple of $1$-cells $X_1,\dots, X_n,Y\colon \mM_0\to \mM_1$, an $n$-ary map $f\colon X_1,\dots, X_n\to Y$ is a cell as on the left below for $n\ge1$, and for $n=0$, a nullary map to $Y$ is a $2$-cell $m_1^0\to Ym_0^0$ (as on the right below). 
    % https://q.uiver.app/#q=WzAsOCxbMCwwLCJcXG1NXntcXG90aW1lcyBufV8wIl0sWzAsMSwiXFxtTV8wIl0sWzIsMSwiXFxtTV8xLCJdLFsyLDAsIlxcbU1ee1xcb3RpbWVzIG59XzEiXSxbNCwwLCJJIl0sWzMsMCwiSSJdLFs0LDEsIlxcbU1fMiJdLFszLDEsIlxcbU1fMSJdLFswLDEsIm1fMF57XFxvdGltZXMgbn0iXSxbMywyLCJtXzFee1xcb3RpbWVzIG59Il0sWzAsMywiWF8xXFxvdGltZXNcXGNkb3RzXFxvdGltZXMgWF9uIl0sWzUsNCwiIiwxLHsibGV2ZWwiOjIsInN0eWxlIjp7ImhlYWQiOnsibmFtZSI6Im5vbmUifX19XSxbNCw2LCJtXzFeMCJdLFs1LDcsIm1fMF4wIiwyXSxbNyw2LCJZIiwyXSxbMSwyLCJZIiwyXSxbMTEsMTQsInUiLDAseyJzaG9ydGVuIjp7InNvdXJjZSI6NDAsInRhcmdldCI6NDB9fV0sWzEwLDE1LCJmIiwwLHsic2hvcnRlbiI6eyJzb3VyY2UiOjQwLCJ0YXJnZXQiOjQwfX1dXQ==
\[\begin{tikzcd}
	{\mM^{\otimes n}_0} && {\mM^{\otimes n}_1} & I & I \\
	{\mM_0} && {\mM_1,} & {\mM_0} & {\mM_1}
	\arrow[""{name=0, anchor=center, inner sep=0}, "{X_1\otimes\cdots\otimes X_n}", from=1-1, to=1-3]
	\arrow["{m_0^{\otimes n}}", from=1-1, to=2-1]
	\arrow["{m_1^{\otimes n}}", from=1-3, to=2-3]
	\arrow[""{name=1, anchor=center, inner sep=0}, equals, from=1-4, to=1-5]
	\arrow["{m_0^0}"', from=1-4, to=2-4]
	\arrow["{m_1^0}", from=1-5, to=2-5]
	\arrow[""{name=2, anchor=center, inner sep=0}, "Y"', from=2-1, to=2-3]
	\arrow[""{name=3, anchor=center, inner sep=0}, "Y"', from=2-4, to=2-5]
	\arrow["f", between={0.4}{0.6}, Rightarrow, from=0, to=2]
	\arrow["u", between={0.4}{0.6}, Rightarrow, from=1, to=3]
\end{tikzcd}\]

%\item If $f\colon X_1,\dots,X_n\to Y$ and $g\colon X_1',\dots,X_n'\to Y'$ are $n$-ary maps, a morphism $f\Rightarrow g$ consists of $2$-cells $x_i\colon X_i\Rightarrow X_i'$, $y\colon Y\Rightarrow Y'$ and a $3$-cell $\eta\colon g\circ(\otimes_ix_i)\Rrightarrow y\circ f$.

    \item Consider multimaps $f\colon Y_1,\dots,Y_n\to Z$ and $f_i\colon X_{i1},\dots,X_{ik_i}\to Y_i$, $i=1,\dots,n$. Define the composition $f(f_1,\dots,f_n)\colon X_{11},\dots,X_{nk_n}\to Z$ to be the pasting of the following $2$-cells:

% https://q.uiver.app/#q=WzAsOCxbMSwxLCJcXG1NXzBee1xcb3RpbWVzIG59Il0sWzEsMiwiXFxtTV8wIl0sWzMsMiwiXFxtTV8xIl0sWzMsMSwiXFxtTV8xXntcXG90aW1lcyBufSJdLFsxLDAsIlxcYmlnb3RpbWVzX2kgXFxtTV8wXntcXG90aW1lcyBrX2l9Il0sWzMsMCwiXFxiaWdvdGltZXNfaSBcXG1NXzFee1xcb3RpbWVzIGtfaX0iXSxbMCwwLCJcXG1NXzBee1xcb3RpbWVzXFxzdW1faSBrX2l9Il0sWzQsMCwiXFxtTV8xXntcXG90aW1lc1xcc3VtX2kga19pfSJdLFswLDEsIm1fMV57bn0iXSxbMywyLCJtXzJeeyBufSJdLFsxLDIsIloiLDJdLFswLDMsIllfMVxcb3RpbWVzXFxjZG90c1xcb3RpbWVzIFlfbiJdLFs0LDAsIlxcb3RpbWVzX2kgbV8xXnsga19pfSJdLFs1LDMsIlxcb3RpbWVzX2kgbV8yXnsga19pfSJdLFs0LDUsIlxcb3RpbWVzX3tpfVxcb3RpbWVzX2pYX3tpan0iXSxbNiw0LCJcXGNvbmciXSxbNiwxLCJtXzFee1xcc3VtX2lrX2l9IiwyLHsiY3VydmUiOjR9XSxbNSw3LCJcXGNvbmciXSxbNywyLCJtXzJee1xcb3RpbWVzXFxzdW1faWtfaX0iLDAseyJjdXJ2ZSI6LTR9XSxbMTEsMTAsImYiLDAseyJzaG9ydGVuIjp7InNvdXJjZSI6NDAsInRhcmdldCI6NDB9fV0sWzE0LDExLCJcXG90aW1lc19pIGZfaSIsMCx7InNob3J0ZW4iOnsic291cmNlIjo0MCwidGFyZ2V0Ijo0MH19XSxbMTgsNSwiKFxcZ2FtbWFeMl97a18xLFxcZG90cyxrX259KV57LTF9IiwyLHsic2hvcnRlbiI6eyJzb3VyY2UiOjQwLCJ0YXJnZXQiOjQwfX1dLFsxMiwxNiwiXFxnYW1tYV4xX3trXzEsXFxkb3RzLGtfbn0iLDIseyJzaG9ydGVuIjp7InNvdXJjZSI6NDAsInRhcmdldCI6NDB9fV1d
\[\begin{tikzcd}
	{\mM_0^{\otimes\sum_i k_i}} & {\bigotimes_i \mM_0^{\otimes k_i}} && {\bigotimes_i \mM_1^{\otimes k_i}} & {\mM_1^{\otimes\sum_i k_i}} \\
	& {\mM_0^{\otimes n}} && {\mM_1^{\otimes n}} \\
	& {\mM_0} && {\mM_1}
	\arrow["\cong", from=1-1, to=1-2]
	\arrow[""{name=0, anchor=center, inner sep=0}, "{m_1^{\sum_ik_i}}"', curve={height=24pt}, from=1-1, to=3-2]
	\arrow[""{name=1, anchor=center, inner sep=0}, "{\otimes_{i}\otimes_jX_{ij}}", from=1-2, to=1-4]
	\arrow[""{name=2, anchor=center, inner sep=0}, "{\otimes_i m_1^{ k_i}}", from=1-2, to=2-2]
	\arrow["\cong", from=1-4, to=1-5]
	\arrow["{\otimes_i m_2^{ k_i}}", from=1-4, to=2-4]
	\arrow[""{name=3, anchor=center, inner sep=0}, "{m_2^{\otimes\sum_ik_i}}", curve={height=-24pt}, from=1-5, to=3-4]
	\arrow[""{name=4, anchor=center, inner sep=0}, "{Y_1\otimes\cdots\otimes Y_n}", from=2-2, to=2-4]
	\arrow["{m_1^{n}}", from=2-2, to=3-2]
	\arrow["{m_2^{ n}}", from=2-4, to=3-4]
	\arrow[""{name=5, anchor=center, inner sep=0}, "Z"', from=3-2, to=3-4]
	\arrow["{\otimes_i f_i}", between={0.4}{0.6}, Rightarrow, from=1, to=4]
	\arrow["{\gamma^1_{k_1,\dots,k_n}}"', between={0.4}{0.6}, Rightarrow, from=2, to=0]
	\arrow["{(\gamma^2_{k_1,\dots,k_n})^{-1}}"', between={0.4}{0.6}, Rightarrow, from=3, to=1-4]
	\arrow["f", between={0.4}{0.6}, Rightarrow, from=4, to=5]
\end{tikzcd}\]
\end{enumerate}

\subsection{Notion of an $(\infty,n)$-operad}
In this subsection, we are going to clarify what we mean by an $(\infty,n)$-operad over an $(\infty,1)$-operad $\O$. By default, our $\infty$-operads are \textit{non-symmetric} or \textit{planar}, i.e. they are not defined as certain $\infty$-categories over $\F_*$ but over $\Delta^{op}$ (see \cite[Def. 4.1.3.2]{Lur17}). We keep the following notation: for an $\infty$-operad $p\colon\O\to \Delta^{op}$, an active map $\alpha\colon n\to 1$ in $\Delta^{op}$, and objects $X_1,\dots,X_n$, $Y$ in $p^{-1}(1)$, we denote by $\Mul(X_1,\dots,X_n;Y)$ the space of maps $(X_1,\dots,X_n)\to Y$ in $\mathcal O$, where $(X_1,...,X_n)\in p^{-1}(1)^n\simeq p^{-1}(n)$, which get sent to $\alpha$ by $p$. For a monoidal $\infty$-category $\V$, we write the morphism specifying the monoidal structure as $\V^\otimes\to \Delta^{op}$.

For us, $(\infty,n)$-categories are complete $n$-fold Segal spaces, see for example \cite[Def. 4.1]{Nui24} (this is our default reference for higher categories with $n>2$). If $\C$ is an $(\infty,n)$-category, we can still talk about a marking on $\C$ -- again, we are only marking 1-cells, not higher cells. We adapt \cite[Def. 5.1]{Nui24} to the marked setting. 
\begin{definition}
    Let $p\colon\D\to \C$ be a functor of $(\infty,n)$-categories and $I$ a marking on $\C$. We say that $p$ if a \textit{marked pre-cocartesian fibration} if for any marked $1$-cell $\alpha\colon c\to c'$ in $\C$ and any object $d\in p^{-1}(c)$ there exists its cocartesian lift, i.e. a $1$-cell $\tilde \alpha\colon d\to d'$ with $p(\tilde \alpha)=\alpha$ such that the following square of hom $(\infty,n-1)$-categories is cartesian for any $d''\in \D$. 
    % https://q.uiver.app/#q=WzAsNCxbMCwwLCJcXEhvbShkJyxkJycpIl0sWzEsMCwiXFxIb20oZCxkJycpIl0sWzAsMSwiXFxIb20ocGQnLHBkJycpIl0sWzEsMSwiXFxIb20ocGQscGQnJykiXSxbMCwxLCJcXHRpbGRlXFxhbHBoYV4qIl0sWzEsMywicCJdLFsyLDMsIlxcYWxwaGFeKiIsMl0sWzAsMiwicCIsMl0sWzAsMywiIiwxLHsic3R5bGUiOnsibmFtZSI6ImNvcm5lciJ9fV1d
\[\begin{tikzcd}
	{\Hom(d',d'')} & {\Hom(d,d'')} \\
	{\Hom(pd',pd'')} & {\Hom(pd,pd'')}
	\arrow["{\tilde\alpha^*}", from=1-1, to=1-2]
	\arrow["p"', from=1-1, to=2-1]
	\arrow["\lrcorner"{anchor=center, pos=0.125}, draw=none, from=1-1, to=2-2]
	\arrow["p", from=1-2, to=2-2]
	\arrow["{\alpha^*}"', from=2-1, to=2-2]
\end{tikzcd}\]
\end{definition}

\begin{remark}
    If in the situation above, $\C$ is only an $(\infty,1)$-category, then $p$ is automatically a hom-wise $(n-1)$-fibration in the sense of \cite[Def. 5.10.(1)]{Nui24} as the target of the functor $\Hom(d,d')\xrightarrow{p}\Hom(pd,pd')$ is only an $\infty$-groupoid. 
\end{remark}

Let $\O$ be an $\infty$-operad. Recall that for a cartesian monoidal $\infty$-category $\C$ we can consider a category $\Mon_{\O}(\C)$ of $\O$-monoids in $\C$. Take $\C$ to be the $(\infty,1)$-category $\Cat_{(\infty,n)}$, defined following \cite{Nui24}, this is cartesian monoidal. For example if $\O$ is $\Delta^{op}$ or $\F_*$, $\Mon_{\O}(\Cat_{(\infty,n)})$ consists of monoidal (resp. symmetric monoidal) $(\infty,n)$-categories. By $n$-categorical unstraightening, one can view an $\O$-monoid $\V$ in $\Cat_{(\infty,n)}$ as a marked as $n$-fibration $\V^\otimes\to \O$ satisfying the Segal condition. We call these \textit{$\O$-monoidal $(\infty,n)$-categories}. More generally, let us make the following definition:

\begin{definition} \label{def: oo,n operad over O}
    Let $\O$ be an $(\infty,1)$-operad with its natural marking given by inert $1$-cells and let $p\colon \V^{\otimes}\to \O$ be a functor of $(\infty,n)$-categories. We say that $p$ exhibit $\V^{\otimes}$ as an \textit{$(\infty,n)$-operad over $\O$} if $p$ is a marked pre-cocartesian fibration satisfying the Segal condition, i.e. the following holds:
    \begin{enumerate}
        \item for every object $X$ of $\O$, the map $p^{-1}(X)\to \prod_{i=1}^n p^{-1}(X_i)$ induced by the cocartesian lifts $X\to X_i$ of inert maps $n\to 1$ in $\Delta^{op}$ is an equivalence.
        \item If $X,$ $X_i$ are as above, $\bar X\in p^{-1}$ and $\bar X\to \bar X_i$ are cocartesian lifts of inert maps $X\to X_i$, the the following square of hom $(\infty,n-1)$-categories is cartesian for any $Y\in \O$ and any $\bar Y\in p^{-1}(Y)$.
        % https://q.uiver.app/#q=WzAsNCxbMCwwLCJcXEhvbV97XFxWXntcXG90aW1lc319KFxcYmFyIFksXFxiYXIgWCkiXSxbMSwwLCJcXHByb2Rfe2k9MX1eblxcSG9tX3tcXFZee1xcb3RpbWVzfX0oXFxiYXIgWSxcXGJhciBYX2kpIl0sWzEsMSwiXFxwcm9kX3tpPTF9Xm5cXEhvbV97XFxPfShcXGJhciBZLFxcYmFyIFhfaSkiXSxbMCwxLCJcXEhvbV97XFxPKFksWCkiXSxbMCwxXSxbMSwyXSxbMywyXSxbMCwzXSxbMCw2LCIiLDEseyJsZXZlbCI6MSwic3R5bGUiOnsibmFtZSI6ImNvcm5lciJ9fV1d
\[\begin{tikzcd}
	{\Hom_{\V^{\otimes}}(\bar Y,\bar X)} & {\prod_{i=1}^n\Hom_{\V^{\otimes}}(\bar Y,\bar X_i)} \\
	{\Hom_{\O}(Y,X)} & {\prod_{i=1}^n\Hom_{\O}(\bar Y,\bar X_i)}
	\arrow[from=1-1, to=1-2]
	\arrow[from=1-1, to=2-1]
	\arrow[from=1-2, to=2-2]
	\arrow[""{name=0, anchor=center, inner sep=0}, from=2-1, to=2-2]
	\arrow["\lrcorner"{anchor=center, pos=0.125}, draw=none, from=1-1, to=0]
\end{tikzcd}\]
    \end{enumerate}
    A morphism of $(\infty,n)$-operads $\V^{\otimes}\to \W^{\otimes}$ over $\O$ is a functor over $\O$ which preserves cocartesian lifts of inert $1$-cells.
\end{definition}

\begin{remark}
Generalizing an observation of \cite{CCRY25}, the functor $\Fun_w(\C,-)\colon \Cat_{(\infty,n)}\to \Cat_{(\infty,n)}$ preserves products and so for any $\O$-monoidal $(\infty,n)$-category $\V$, $\Fun_w(\C,\V)$ inherits the $\O$-monoidal structure. 
\end{remark}

\subsection{Higher Day convolution $\infty$-operad} \label{subsec: higher Day conv}

Let $\V$ be an $\O$-monoidal $(\infty,n)$-category and write $\V^\otimes\to \O$ for the associated $(\infty,n)$-operad over $\O$. Let $M_0, M_1\colon \O\to \V^\otimes$ be $\O$-algebras in $\V^\otimes$, i.e. morphisms of operads over $\O$. The source and target functors $d_0, d_1\colon \Fun_w(\Delta^1,\V)\to \V$ are $\O$-monoidal, resulting in a map of $(\infty,n)$-operads $\Fun_w(\Delta^1,\V^\otimes)\to \V^{\otimes}\times_{\O} \V^\otimes$ over $\O$. Define $\Hom_{\V}^w(M_0,M_1)^\otimes$ to be the following pullback:

% https://q.uiver.app/#q=WzAsNCxbMCwxLCJcXEZ1bl97XFxiYXIgd30oXFxEZWx0YV4xLFxcTSlee1xcb3RpbWVzfSJdLFsxLDEsIlxcTV57XFxvdGltZXN9XFx0aW1lc197XFxEZWx0YV57b3B9fVxcTV57XFxvdGltZXN9Il0sWzEsMCwiXFxEZWx0YV^{op}}Il0sWzAsMCwiXFx1bmRlcmxpbmV7XFxIb219X3tcXE19Xnt3fShNXzEsTV8yKSJdLFsyLDEsIihNXzEsTV8yKSJdLFswLDEsIihkXzAsZF8xKSIsMl0sWzMsMF0sWzMsMiwicCJdLFszLDEsIiIsMSx7InN0eWxlIjp7Im5hbWUiOiJjb3JuZXIifX1dXQ==
\begin{equation} \label{dia: operad}\begin{tikzcd}
	{\Hom_{\V}^{w}(M_0,M_1)^\otimes} & {\O} \\
	{\Fun_{w^*}(\Delta^1,\V)^{\otimes}} & {\V^{\otimes}\times_{\O}\V^{\otimes}}
	\arrow["p", from=1-1, to=1-2]
	\arrow["j"', from=1-1, to=2-1]
	\arrow["\lrcorner"{anchor=center, pos=0.125}, draw=none, from=1-1, to=2-2]
	\arrow["{(M_0,M_1)}", from=1-2, to=2-2]
	\arrow["{(d_0,d_1)}"', from=2-1, to=2-2]
\end{tikzcd}\end{equation}

\begin{remark}
    If $\O=\Delta^{op}$ or $\F_*$, write $\mM_i:=M_i(1)$, $i=0,1$. In these cases, we are going to abuse the notation and write $ {{\Hom}_{\V}^{w}(\mM_0,\mM_1)^\otimes}$ instead of ${\Hom_{\V}^{w}(M_0,M_1)^\otimes}$.
\end{remark}

Let us take a closer look at the situation $w=\lax$ and $\O=\Delta^{op}$ -- i.e., $\V$ is a monoidal $(\infty,n)$-category and $\mM_0$, $\mM_1$ are monoids in it. Then by definition, the fiber $p^{-1}(1)$ is the hom $(\infty,n-1)$-category $\Hom_{\V}(\mM_0,\mM_1)$. Note our convention regarding the lax / oplax cells: a $1$-cell in $\Fun_{l^*}(\Delta^1,\V)$ which gets sent to $\id_{\mM_i}$ by $d_i$ is of the form of the square below, producing the $2$-morphism $X\Rightarrow Y$ in the correct direction. 

% https://q.uiver.app/#q=WzAsNixbMCwwLCJcXG1NXzAiXSxbMSwwLCJcXG1NXzEiXSxbMiwwLCJcXG1NXzAiXSxbMiwxLCJcXG1NXzEiXSxbMywxLCJcXG1NXzEiXSxbMywwLCJcXG1NXzAiXSxbMCwxLCJZIiwyLHsiY3VydmUiOjJ9XSxbMCwxLCJYIiwwLHsiY3VydmUiOi0yfV0sWzIsNSwiIiwwLHsibGV2ZWwiOjIsInN0eWxlIjp7ImhlYWQiOnsibmFtZSI6Im5vbmUifX19XSxbMyw0LCIiLDAseyJsZXZlbCI6Miwic3R5bGUiOnsiaGVhZCI6eyJuYW1lIjoibm9uZSJ9fX1dLFsyLDMsIlgiLDJdLFs1LDQsIlkiXSxbMyw1LCJcXGV0YSIsMCx7InNob3J0ZW4iOnsic291cmNlIjozMCwidGFyZ2V0IjozMH0sImxldmVsIjoyfV0sWzEsMiwiIiwxLHsic2hvcnRlbiI6eyJzb3VyY2UiOjIwLCJ0YXJnZXQiOjMwfSwic3R5bGUiOnsidGFpbCI6eyJuYW1lIjoiYXJyb3doZWFkIn0sImJvZHkiOnsibmFtZSI6InNxdWlnZ2x5In19fV0sWzcsNiwiXFxldGEiLDAseyJzaG9ydGVuIjp7InNvdXJjZSI6MjAsInRhcmdldCI6MjB9fV1d
\[\begin{tikzcd}
	{\mM_0} & {\mM_1} & {\mM_0} & {\mM_0} \\
	&& {\mM_1} & {\mM_1}
	\arrow[""{name=0, anchor=center, inner sep=0}, "Y"', curve={height=12pt}, from=1-1, to=1-2]
	\arrow[""{name=1, anchor=center, inner sep=0}, "X", curve={height=-12pt}, from=1-1, to=1-2]
	\arrow[between={0.2}{0.7}, squiggly, tail reversed, from=1-2, to=1-3]
	\arrow[equals, from=1-3, to=1-4]
	\arrow["X"', from=1-3, to=2-3]
	\arrow["Y", from=1-4, to=2-4]
	\arrow["\eta", between={0.3}{0.7}, Rightarrow, from=2-3, to=1-4]
	\arrow[equals, from=2-3, to=2-4]
	\arrow["\eta", between={0.2}{0.8}, Rightarrow, from=1, to=0]
\end{tikzcd}\]

\begin{theorem} \label{lemma: higher day conv} %\begin{enumerate}
    %\item 
    The functor $p$ exhibits ${\Hom}^w_{\V}(\mM_0, \mM_1)^\otimes$ as an $(\infty,n-1)$-operad over $\O$.
\end{theorem}

\begin{proof}First, we are going to show that $p$ is marked pre-cocartesian fibration (where we mark $\O$ by inert maps). This follows from (more general) Lemma \ref{lemma: cocart lifts} below. Then, we are going to check the Segal condition. We do that slightly more generally in Lemma \ref{lemma: checking segal}.
\end{proof}

\begin{lemma} \label{lemma: cocart lifts}
Consider any diagram of $(\infty,d)$-categories as below such that $vM\simeq1$. Let $f\colon m\to n$ be a $1$-cell in $\O$. Suppose that the following holds:
\begin{itemize}
    \item $\V,\W$ admits all cocartesian lifts of $f$,
    \item $q$ preserves these lifts,
    \item $M$ sends $f$ to a cocartesian lift of $f$.
\end{itemize} Then $p$ admits all cocartesian lifts of $f$ as well.
    % https://q.uiver.app/#q=WzAsNSxbMCwxLCJcXFciXSxbMSwxLCJcXFYiXSxbMSwwLCJcXE8iXSxbMSwyLCJcXE8iXSxbMCwwLCJcXFdcXHRpbWVzX3tcXFZ9XFxPIl0sWzAsMSwicSJdLFsyLDEsIk0iXSxbMSwzLCJ2Il0sWzAsMywidyIsMl0sWzQsMCwiaiIsMl0sWzQsMiwicCJdLFs0LDEsIiIsMCx7InN0eWxlIjp7Im5hbWUiOiJjb3JuZXIifX1dLFsyLDMsIiIsMSx7ImN1cnZlIjotMywibGV2ZWwiOjIsInN0eWxlIjp7ImhlYWQiOnsibmFtZSI6Im5vbmUifX19XV0=
\begin{equation}\label{dia: higher operads} \begin{tikzcd}
	{\W\times_{\V}\O} & \O \\
	\W & \V \\
	& \O
	\arrow["p", from=1-1, to=1-2]
	\arrow["j"', from=1-1, to=2-1]
	\arrow["\lrcorner"{anchor=center, pos=0.125}, draw=none, from=1-1, to=2-2]
	\arrow["M", from=1-2, to=2-2]
	\arrow["q", from=2-1, to=2-2]
	\arrow["w"', from=2-1, to=3-2]
	\arrow["v", from=2-2, to=3-2]
\end{tikzcd}\end{equation}
\end{lemma}
\begin{proof}
    Let $X$ be an object in the fiber $p^{-1}(m)$ and consider the cocartesian lift $f_!\colon jX\to Y'$ in $\W$. From our assumptions on $q$ and $M$ we deduce that both $f_!$ and $f$ get mapped to the same morphism in $\V$ (up to homotopy), giving us a $1$-cell $\tilde f_!\colon X\to Y$ in $\W\times_{\V}\O$ such that $p\tilde f_!\simeq f,$ $j\tilde f_!\simeq f_!$. We claim that $\tilde f_!$ is the desired cocartesian lift of $f$. For that, we need to show that the square $(\heartsuit)$ in the diagram below is cartesian for any object $A$ of $\W\times_{\V}\O$. This follows by certain diagram chase: more concretely, in the commutative cube below, enough of the squares are cartesian to deduce that all of them are. (We denote $h:=qj\simeq Mp$.)

% https://q.uiver.app/#q=WzAsNyxbMCwwLCJcXEhvbShZLEEpIl0sWzEsMSwiXFxIb20ocFkscEEpIl0sWzEsMCwiXFxIb20oWCxBKSJdLFsyLDEsIlxcSG9tKHBYLHBBKSJdLFswLDEsIlxcSG9tKGpZLCBqQSkiXSxbMSwyLCJcXEhvbShoWSwgaEEpIl0sWzIsMiwiXFxIb20oaFgsIGhBKSJdLFswLDFdLFswLDIsIihcXHRpbGRlIGZfISleKiJdLFsyLDNdLFswLDRdLFs0LDVdLFsxLDVdLFs1LDYsIihoXFx0aWxkZSBmXyEpXioiLDJdLFszLDZdLFsxLDMsImZeKiIsMl0sWzEsNiwiIiwxLHsic3R5bGUiOnsibmFtZSI6ImNvcm5lciJ9fV0sWzIsMSwiKFxcaGVhcnRzdWl0KSIsMSx7InN0eWxlIjp7ImJvZHkiOnsibmFtZSI6Im5vbmUifSwiaGVhZCI6eyJuYW1lIjoibm9uZSJ9fX1dLFswLDExLCIiLDEseyJsZXZlbCI6MSwic3R5bGUiOnsibmFtZSI6ImNvcm5lciJ9fV1d
\[\begin{tikzcd}
	{\Hom(Y,A)} & {\Hom(X,A)} \\
	{\Hom(jY, jA)} & {\Hom(pY,pA)} & {\Hom(pX,pA)} \\
	& {\Hom(hY, hA)} & {\Hom(hX, hA)}
	\arrow["{(\tilde f_!)^*}", from=1-1, to=1-2]
	\arrow[from=1-1, to=2-1]
	\arrow[from=1-1, to=2-2]
	\arrow["{(\heartsuit)}"{description}, draw=none, from=1-2, to=2-2]
	\arrow[from=1-2, to=2-3]
	\arrow[""{name=0, anchor=center, inner sep=0}, from=2-1, to=3-2]
	\arrow["{f^*}"', from=2-2, to=2-3]
	\arrow[from=2-2, to=3-2]
	\arrow["\lrcorner"{anchor=center, pos=0.125}, draw=none, from=2-2, to=3-3]
	\arrow[from=2-3, to=3-3]
	\arrow["{(h\tilde f_!)^*}"', from=3-2, to=3-3]
	\arrow["\lrcorner"{anchor=center, pos=0.125}, draw=none, from=1-1, to=0]
\end{tikzcd}\]
% https://q.uiver.app/#q=WzAsOCxbMSwwLCJcXEhvbShZLEEpIl0sWzEsMSwiXFxIb20oalksIGpBKSJdLFsyLDEsIlxcSG9tKGpYLCBqQSkiXSxbMiwwLCJcXEhvbShYLEEpIl0sWzIsMiwiXFxIb20oaFksIGhBKSJdLFszLDIsIlxcSG9tKGhYLCBoQSkiXSxbMywxLCJcXEhvbShwWCxwQSkiXSxbMCwxLCJcXHNpbWVxIl0sWzAsMV0sWzEsMiwiKGZfISleKiJdLFswLDMsIihcXHRpbGRlIGZfISleKiJdLFszLDJdLFsxLDRdLFsyLDVdLFszLDZdLFs2LDVdLFs0LDUsIihoXFx0aWxkZSBmXyEpXioiLDJdLFsxLDUsIiIsMSx7InN0eWxlIjp7Im5hbWUiOiJjb3JuZXIifX1dLFszLDEzLCIiLDEseyJsZXZlbCI6MSwic3R5bGUiOnsibmFtZSI6ImNvcm5lciJ9fV1d
\[\begin{tikzcd}
	& {\Hom(Y,A)} & {\Hom(X,A)} \\
	\simeq & {\Hom(jY, jA)} & {\Hom(jX, jA)} & {\Hom(pX,pA)} \\
	&& {\Hom(hY, hA)} & {\Hom(hX, hA)}
	\arrow["{(\tilde f_!)^*}", from=1-2, to=1-3]
	\arrow[from=1-2, to=2-2]
	\arrow[from=1-3, to=2-3]
	\arrow[from=1-3, to=2-4]
	\arrow["{(f_!)^*}", from=2-2, to=2-3]
	\arrow[from=2-2, to=3-3]
	\arrow["\lrcorner"{anchor=center, pos=0.125, rotate=45}, draw=none, from=2-2, to=3-4]
	\arrow[""{name=0, anchor=center, inner sep=0}, from=2-3, to=3-4]
	\arrow[from=2-4, to=3-4]
	\arrow["{(h\tilde f_!)^*}"', from=3-3, to=3-4]
	\arrow["\lrcorner"{anchor=center, pos=0.125}, draw=none, from=1-3, to=0]
\end{tikzcd}\]
\end{proof}

\begin{lemma} \label{lemma: checking segal}
    Suppose that in the diagram (\ref{dia: higher operads}), the following holds:
    \begin{enumerate}
        \item $v,w$ exhibit $\V,\W$ as $(\infty,n)$-operads over $\O$,
        \item $q$ is a morphism of operads,
        \item $M$ is an $\O$-algebra in $\V$.
    \end{enumerate} 
    Then $p$ exhibit $\W\times_{\V}\O$ as an $(\infty,n)$-operad over $\O$. Moreover, $j$ is a morphism of operads over $\O$.
\end{lemma}
\begin{proof}
We already know $p$ is a marked pre-cocartesian fibration from the previous lemma. It remains to check the conditions (1) and (2) of Definition \ref{def: oo,n operad over O}. The first condition follows from the fact that $p^{-1}(Y)\simeq q^{-1}(M(Y))$ for any $Y$ of $\O$. To see the second one, we perform a diagram chasing along a commutative cube again: in particular, in the cube below, we can see that enough squares are cartesian to deduce that the square ($\spadesuit$) is cartesian as well. Denote $\P:=\W\times_{\V}\O$.
    % https://q.uiver.app/#q=WzAsNyxbMCwwLCJcXEhvbV97XFxQfShcXGJhciBZLFxcYmFyIFgpIl0sWzEsMSwiXFxwcm9kX3tpPTF9Xm5cXEhvbV97XFxQfShcXGJhciBZLFxcYmFyIFhfaSkiXSxbMSwwLCJcXEhvbV97XFxPfShZLFgpIl0sWzIsMSwiXFxwcm9kX3tpPTF9Xm5cXEhvbV97XFxPfShZLFhfaSkiXSxbMCwxLCJcXEhvbV97XFxXfShqXFxiYXIgWSwgalxcYmFyIFgpIl0sWzEsMiwiXFxwcm9kX3tpPTF9Xm5cXEhvbV97XFxXfShqXFxiYXIgWSwgalxcYmFyIFhfaSkiXSxbMiwyLCJcXHByb2Rfe2k9MX1eblxcSG9tX3tcXFZ9KGhcXGJhciBZLGhcXGJhciBYX2kpIl0sWzAsMV0sWzAsMl0sWzIsM10sWzAsNF0sWzQsNV0sWzEsNV0sWzUsNl0sWzMsNl0sWzEsM10sWzIsMSwiKFxcc3BhZGVzdWl0KSIsMSx7InN0eWxlIjp7ImJvZHkiOnsibmFtZSI6Im5vbmUifSwiaGVhZCI6eyJuYW1lIjoibm9uZSJ9fX1dLFsxLDEzLCIiLDEseyJsZXZlbCI6MSwic3R5bGUiOnsibmFtZSI6ImNvcm5lciJ9fV1d
\[\begin{tikzcd}
	{\Hom_{\P}(\bar Y,\bar X)} & {\Hom_{\O}(Y,X)} & \\
	{\Hom_{\W}(j\bar Y, j\bar X)} & {\prod_{i=1}^n\Hom_{\P}(\bar Y,\bar X_i)} & {\prod_{i=1}^n\Hom_{\O}(Y,X_i)} \\
	& {\prod_{i=1}^n\Hom_{\W}(j\bar Y, j\bar X_i)} & {\prod_{i=1}^n\Hom_{\V}(h\bar Y,h\bar X_i)}
	\arrow[from=1-1, to=1-2]
	\arrow[from=1-1, to=2-1]
	\arrow[from=1-1, to=2-2]
	\arrow["{(\spadesuit)}"{description}, draw=none, from=1-2, to=2-2]
	\arrow[from=1-2, to=2-3]
	\arrow[from=2-1, to=3-2]
	\arrow[from=2-2, to=2-3]
	\arrow[from=2-2, to=3-2]
	\arrow[from=2-3, to=3-3]
	\arrow[""{name=0, anchor=center, inner sep=0}, from=3-2, to=3-3]
	\arrow["\lrcorner"{anchor=center, pos=0.125}, draw=none, from=2-2, to=0]
\end{tikzcd}\]
% https://q.uiver.app/#q=WzAsOCxbMSwwLCJcXEhvbV97XFxQfShcXGJhciBZLFxcYmFyIFgpIl0sWzIsMCwiXFxIb21fe1xcT30oWSxYKSJdLFszLDEsIlxccHJvZF97aT0xfV5uXFxIb21fe1xcT30oWSxYX2kpIl0sWzEsMSwiXFxIb21fe1xcV30oalxcYmFyIFksIGpcXGJhciBYKSJdLFsyLDIsIlxccHJvZF97aT0xfV5uXFxIb21fe1xcV30oalxcYmFyIFksIGpcXGJhciBYX2kpIl0sWzMsMiwiXFxwcm9kX3tpPTF9Xm5cXEhvbV97XFxWfShoXFxiYXIgWSxoXFxiYXIgWF9pKSJdLFswLDEsIlxcc2ltZXEiXSxbMiwxLCJcXEhvbV97XFxWfShoXFxiYXIgWSxoXFxiYXIgWCkiXSxbMCwxXSxbMSwyXSxbMCwzXSxbMyw0XSxbNCw1XSxbMiw1XSxbMSw3XSxbMyw3XSxbNyw1XSxbMyw1LCIiLDAseyJzdHlsZSI6eyJuYW1lIjoiY29ybmVyIn19XSxbMCwxNSwiIiwyLHsibGV2ZWwiOjEsInN0eWxlIjp7Im5hbWUiOiJjb3JuZXIifX1dXQ==
\[\begin{tikzcd}
	& {\Hom_{\P}(\bar Y,\bar X)} & {\Hom_{\O}(Y,X)} & \\
	\simeq & {\Hom_{\W}(j\bar Y, j\bar X)} & {\Hom_{\V}(h\bar Y,h\bar X)} & {\prod_{i=1}^n\Hom_{\O}(Y,X_i)} \\
	&& {\prod_{i=1}^n\Hom_{\W}(j\bar Y, j\bar X_i)} & {\prod_{i=1}^n\Hom_{\V}(h\bar Y,h\bar X_i)}
	\arrow[from=1-2, to=1-3]
	\arrow[from=1-2, to=2-2]
	\arrow[from=1-3, to=2-3]
	\arrow[from=1-3, to=2-4]
	\arrow[""{name=0, anchor=center, inner sep=0}, from=2-2, to=2-3]
	\arrow[from=2-2, to=3-3]
	\arrow["\lrcorner"{anchor=center, pos=0.125, rotate=45}, draw=none, from=2-2, to=3-4]
	\arrow[from=2-3, to=3-4]
	\arrow[from=2-4, to=3-4]
	\arrow[from=3-3, to=3-4]
	\arrow["\lrcorner"{anchor=center, pos=0.125}, draw=none, from=1-2, to=0]
\end{tikzcd}\]
That $j$ is a morphism of operads over $\O$ follow from how we constructed the cocartesian lifts in the previous lemma.
\end{proof}

\begin{remark}
    Take $\O=\Delta^{op}$ and let $m^n\colon n\to 1$ be the unique active morphism $n\to1$ in $\Delta^{op}$. Take any $X_1,\dots,X_n,Y\in p^{-1}$ and denote $\bar X\in p^{-1}(n)\simeq p^{-1}(1)^n$ the object corresponding to $X_1,\dots,X_n$. By definition, $\Mul(X_1,\dots,X_n;Y)$ sits in the following cartesian diagram:
    % https://q.uiver.app/#q=WzAsNixbMCwxLCJcXEhvbShcXGJhciBYLFkpIl0sWzAsMiwiXFxIb20oalxcYmFyIFgsalkpIl0sWzEsMiwiXFxIb20oaFxcYmFyIFgsaFkpIl0sWzEsMSwiXFxIb20obiwxKSJdLFsxLDAsIlxcRGVsdGFeMSJdLFswLDAsIlxcTXVsKFhfMSxcXGRvdHMsWF9uO1kpIl0sWzUsMF0sWzAsMV0sWzEsMl0sWzAsM10sWzMsMl0sWzUsNF0sWzQsMywiXFxtdV5uIl0sWzAsOCwiIiwxLHsibGV2ZWwiOjEsInN0eWxlIjp7Im5hbWUiOiJjb3JuZXIifX1dLFs1LDksIiIsMix7ImxldmVsIjoxLCJzdHlsZSI6eyJuYW1lIjoiY29ybmVyIn19XV0=
\[\begin{tikzcd}
	{\Mul(X_1,\dots,X_n;Y)} & {\Delta^1} \\
	{\Hom(\bar X,Y)} & {\Hom(n,1)} \\
	{\Hom(j\bar X,jY)} & {\Hom(h\bar X,hY)}
	\arrow[from=1-1, to=1-2]
	\arrow[from=1-1, to=2-1]
	\arrow["{\mu^n}", from=1-2, to=2-2]
	\arrow[""{name=0, anchor=center, inner sep=0}, from=2-1, to=2-2]
	\arrow[from=2-1, to=3-1]
	\arrow[from=2-2, to=3-2]
	\arrow[""{name=1, anchor=center, inner sep=0}, from=3-1, to=3-2]
	\arrow["\lrcorner"{anchor=center, pos=0.125}, draw=none, from=1-1, to=0]
	\arrow["\lrcorner"{anchor=center, pos=0.125}, draw=none, from=2-1, to=1]
\end{tikzcd}\]
From that, we can deduce that the objects and $1$-cells of $\Mul(X_1,\dots,X_n;Y)$ looks as described in the introduction to this section.
\end{remark}

\begin{remark}
    Consider an $\O$-algebra in $\Hom_{\V}^l(M_0,M_1)^\otimes.$ By definition, this is a map $\O\to \Fun_{s,l^*}(\Delta^1,\V^\otimes)$ preserving cocartesian lifts of inerts such that its restriction to source and target is $M_0$ and $M_1$ (up to homotopy). Currying $\O\to \Fun_{s,l^*}(\Delta^1,\V^\otimes)$ to $\Delta^1\otimes_{s,l}\O\to \V^\otimes$, we get that this data correspond to a lax natural transformation $M_0\to M_1$ which is strict on the inert cells. If $\O=\Delta^{op}$, one can observe that this is a lax homorphism of monoids $M_0\to M_1$, i.e. a $1$-cell $f\colon \mM_0\to \mM_1$ together with a map $f,f\to f$ etc. This generalizes the classical observation that if $\V,\W$ are monoidal $\infty$-categories, a monoid in $\Fun(\V,\W)$ with respect to the Day convolution structure is precisely a lax monoidal functor $\V\to \W$. As we will not need this result, we do not elaborate on that further, but we would like to prove a more rigorous statement in the follow-up work.
\end{remark}

\begin{example}
    Let $\V=(\Cat_{(\infty,n)},\times,1)$, $\mM_0=1$, $\mM_1$ any monoidal $(\infty,n)$-category. Then ${\Hom}^l_{\V}(1,\mM_1)^\otimes$ is the operad $\V^{\otimes}$ associated to $\V$.
\end{example}

\begin{example}
    Let $\V=(\Cat_\infty,\times,1)$, $(\mM_0,\otimes)$ a small monoidal $\infty$-category and $(\mM_1,\boxtimes)$ a cocomplete monoidal $\infty$-category such that the tensor product $\boxtimes$ preserves colimits in each variable. Then ${\Hom}^l_{\V}(\mM_0,\mM_1)^\otimes$ is the usual Day convolution monoidal structure on $\Fun(\mM_0,\mM_1)$ constructed in \cite{Gla16}, and originally by \cite{Day70} for ordinary monoidal categories. By requirements on $\mM_0$, $\mM_1$, we can consider the left Kan extension $\lan_{\otimes}(\boxtimes\circ (X\times Y))$ as below serving as a tensor product $X\odot Y$ of two functors $X,Y\colon \mM_0\to\mM_1$.  % https://q.uiver.app/#q=WzAsNCxbMCwxLCJcXG1NXzAiXSxbMSwwLCJcXG1NXzFcXHRpbWVzIFxcbU1fMSJdLFswLDAsIlxcbU1fMFxcdGltZXMgXFxtTV8wIl0sWzIsMCwiXFxtTV8xIl0sWzIsMCwiXFxvdGltZXMiLDJdLFsyLDEsIlhcXHRpbWVzIFkiXSxbMSwzLCJcXGJveHRpbWVzIl0sWzAsMywiWFxcb2RvdCBZIiwyLHsic3R5bGUiOnsiYm9keSI6eyJuYW1lIjoiZGFzaGVkIn19fV0sWzIsNywiXFxhbHBoYV97WFl9IiwyLHsic2hvcnRlbiI6eyJzb3VyY2UiOjMwLCJ0YXJnZXQiOjMwfX1dXQ==
\[\begin{tikzcd}
	{\mM_0\times \mM_0} & {\mM_1\times \mM_1} & {\mM_1} \\
	{\mM_0}
	\arrow["{X\times Y}", from=1-1, to=1-2]
	\arrow["\otimes"', from=1-1, to=2-1]
	\arrow["\boxtimes", from=1-2, to=1-3]
	\arrow[""{name=0, anchor=center, inner sep=0}, "{X\odot Y}"', dashed, from=2-1, to=1-3]
	\arrow["{\alpha_{XY}}"', between={0.3}{0.7}, Rightarrow, from=1-1, to=0]
\end{tikzcd}\]
Our theorem above says that if $\mM_0,$ $\mM_1$ are any monoidal $\infty$-categories, then even if we cannot form the left Kan extension $X\odot Y$, there still exists an $\infty$-operad structure on $\Fun(\mM_0,\mM_1)$.
\end{example}

\subsection{The $(\infty,2)$-operad structure on $\Mod_{s}(\T,\Cat_{\infty})$} \label{subsec: 4.2}

Let $\T$ be a Lawvere $2$-theory with a (pseudo)commutativity structure $\mu$. This implies that $\T$ is a $E_1$-monoid in the monoidal $3$-category $\Cat^{\mathfrak{m}}_2$. But so is $\Cat$ with the maximal marking and product given by the actual product $\times$. Therefore, the hom $2$-category $\Cat^{\mathfrak{m}}_2(\T,\Cat)$, which we would write as $\Fun_{s,p}(\T,\Cat)$ in the $(2,2)$-case and $\Fun_{s,s}(\T,\Cat_{\infty})$ in the $(\infty,2)$-case, has a structure of a $2$-operad / $2$-multicategory by Theorem \ref{lemma: higher day conv}. In fact, the same holds more generally for any cartesian monoidal $2$-category $\C$ in place of $\Cat$.
\begin{definition}
    Let $\T$, $\mu$, $\C$ be as above. Define a $(\infty,2)$-operad $\Mod_{s}(\T,\C)^\otimes$ as a full sub-$(\infty,2)$-operad of $\Hom^l_{\Cat^{\mathfrak{m}}}(\T,\C)^\otimes$ spanned by $\T$-models. 
\end{definition}
Writing $q\colon \Mod_s(\T,\C)^\otimes \to \Delta^{op}$, we get that $q^{-1}(1)=\Mod_s(\T,\C)$. If $\mu$ is symmetric, we get a symmetric operad $\Mod_s(\T,\C)^\otimes\to \F_*$.

\begin{example}
    As always, we are going to analyze the structure above for $\T=\T_{E_\infty}$. Let $X,$ $Y$, $Z$ be $\Cat$-valued models, i.e. symmetric monoidal ($\infty$-)categories, and denote the corresponding functors $\T_{E_\infty}\to \Cat$ by $\X$, $\Y$, $\mZ$, respectively. Consider an element $f\in \Mul(\X,\Y;\mZ)$. We are going to check how the components of $f$ look for $m\colon 2\to 1$, $u\colon 0\to 1$: 
    % https://q.uiver.app/#q=WzAsMTYsWzAsMCwiWF4yXFx0aW1lcyBZIl0sWzAsMSwiWFxcdGltZXMgWSJdLFsxLDAsIlpeMiJdLFsxLDEsIloiXSxbMywwLCJYXFx0aW1lcyBZXjIiXSxbMywxLCJYXFx0aW1lcyBZIl0sWzQsMCwiWl4yIl0sWzQsMSwiWiJdLFswLDIsIjFcXHRpbWVzIFkiXSxbMCwzLCJYXFx0aW1lcyBZIl0sWzEsMiwiMSJdLFsxLDMsIloiXSxbMywyLCJYXFx0aW1lcyAxIl0sWzMsMywiWFxcdGltZXMgWSJdLFs0LDIsIjEiXSxbNCwzLCJaIl0sWzAsMSwiXFxYKG0pXFx0aW1lczEiLDJdLFsyLDMsIlxcbVoobSkiXSxbMSwzLCJmX3sxLDF9IiwyXSxbMCwyLCJmX3syLDF9Il0sWzQsNSwiMVxcdGltZXMgXFxZKG0pIiwyXSxbNiw3LCJcXG1aKG0pIl0sWzQsNiwiZl97MSwyfSJdLFs1LDcsImZfezEsMX0iLDJdLFs4LDksIlxcWCh1KVxcdGltZXMgMSIsMl0sWzksMTEsImZfezEsMX0iLDJdLFs4LDEwXSxbMTAsMTEsIlxcbVoodSkiXSxbMTIsMTMsIjFcXHRpbWVzIFxcWShtKSIsMl0sWzEyLDE0XSxbMTMsMTUsImZfezEsMX0iLDJdLFsxNCwxNSwiXFxtWih1KSJdLFsxNiwxNywiXFxzaW1lcSIsMSx7InNob3J0ZW4iOnsic291cmNlIjoyMCwidGFyZ2V0IjoyMH0sInN0eWxlIjp7ImJvZHkiOnsibmFtZSI6Im5vbmUifSwiaGVhZCI6eyJuYW1lIjoibm9uZSJ9fX1dLFsyMCwyMSwiXFxzaW1lcSIsMSx7InNob3J0ZW4iOnsic291cmNlIjoyMCwidGFyZ2V0IjoyMH0sInN0eWxlIjp7ImJvZHkiOnsibmFtZSI6Im5vbmUifSwiaGVhZCI6eyJuYW1lIjoibm9uZSJ9fX1dLFsyOCwzMSwiXFxzaW1lcSIsMSx7InNob3J0ZW4iOnsic291cmNlIjoyMCwidGFyZ2V0IjoyMH0sInN0eWxlIjp7ImJvZHkiOnsibmFtZSI6Im5vbmUifSwiaGVhZCI6eyJuYW1lIjoibm9uZSJ9fX1dLFsyNCwyNywiXFxzaW1lcSIsMSx7InNob3J0ZW4iOnsic291cmNlIjoyMCwidGFyZ2V0IjoyMH0sInN0eWxlIjp7ImJvZHkiOnsibmFtZSI6Im5vbmUifSwiaGVhZCI6eyJuYW1lIjoibm9uZSJ9fX1dXQ==
\[\begin{tikzcd}
	{X^2\times Y} & {Z^2} && {X\times Y^2} & {Z^2} \\
	{X\times Y} & Z && {X\times Y} & Z \\
	{1\times Y} & 1 && {X\times 1} & 1 \\
	{X\times Y} & Z && {X\times Y} & Z
	\arrow["{f_{2,1}}", from=1-1, to=1-2]
	\arrow[""{name=0, anchor=center, inner sep=0}, "{\X(m)\times1}"', from=1-1, to=2-1]
	\arrow[""{name=1, anchor=center, inner sep=0}, "{\mZ(m)}", from=1-2, to=2-2]
	\arrow["{f_{1,2}}", from=1-4, to=1-5]
	\arrow[""{name=2, anchor=center, inner sep=0}, "{1\times \Y(m)}"', from=1-4, to=2-4]
	\arrow[""{name=3, anchor=center, inner sep=0}, "{\mZ(m)}", from=1-5, to=2-5]
	\arrow["{f_{1,1}}"', from=2-1, to=2-2]
	\arrow["{f_{1,1}}"', from=2-4, to=2-5]
	\arrow[from=3-1, to=3-2]
	\arrow[""{name=4, anchor=center, inner sep=0}, "{\X(u)\times 1}"', from=3-1, to=4-1]
	\arrow[""{name=5, anchor=center, inner sep=0}, "{\mZ(u)}", from=3-2, to=4-2]
	\arrow[from=3-4, to=3-5]
	\arrow[""{name=6, anchor=center, inner sep=0}, "{1\times \Y(m)}"', from=3-4, to=4-4]
	\arrow[""{name=7, anchor=center, inner sep=0}, "{\mZ(u)}", from=3-5, to=4-5]
	\arrow["{f_{1,1}}"', from=4-1, to=4-2]
	\arrow["{f_{1,1}}"', from=4-4, to=4-5]
	\arrow["\simeq"{description}, draw=none, from=0, to=1]
	\arrow["\simeq"{description}, draw=none, from=2, to=3]
	\arrow["\simeq"{description}, draw=none, from=4, to=5]
	\arrow["\simeq"{description}, draw=none, from=6, to=7]
\end{tikzcd}\]

Rewriting $\otimes_X:=\X(m)$, $u_X:=\X(u)(1)$ and similarly for $Y$ and $Z$ and writing simply $f:=f_{1,1}$, we get that for any $x,x'\in X$, $y,y'\in Y$, we get that the $2$-cells above evaluated on $x,x'$, $y,y'$ have the following form:
\begin{itemize}
    \item $f(x,y)\otimes_Zf(x',y)\simeq f(x\otimes_X x',y)$, $f(x,y)\otimes_Zf(x,y')\simeq f(x,y\otimes_Y y')$,
    \item $u_Z\simeq f(u_X,y)$, $u_Z\simeq f(x,u_Y)$.
\end{itemize}
Therefore, we arrived at a notion of \textit{bilinear functors} between symmetric monoidal categories. 
\end{example}

\begin{corr} \label{corr: closed operad}
   The $(\infty,2)$-operad above is closed.
\end{corr}
\begin{proof}
    We can immediately see that $\Mul_2(\X,\Y;\mZ)$ coincides with the one defined in \ref{subsec: closed structure I}. The statement then follows from Proposition \ref{prop: closed structure without operad}.
\end{proof}

\subsection{$(2,2)$-multicategory structure on $\Mod_l(\T,\Cat)$} \label{subsec: 4.3}
Finally, we are going to construct a $2$-multicategory structure on the $(2,2)$-category $\Mod_l(\T,\Cat)$ for $\T$ equipped with a lax commutativity. In order to do so, we cannot use the previous results: $\V:=(\Cat^{\mathfrak{m}}_{(2,2)},\otimes_{s,l}, \Fun_{s,l}(-,-))$ is not a monoidal $3$-category nor even $2$-category. It is not even monoidal in the enriched sense, i.e. as enriched over itself, simply because $\otimes_{s,l}$ is not a symmetric monoidal structure. There is certain way around it via considering $\V$ is bienriched over itself, but then it is not at all clear how to extract a $(2,2)$-multicategory out of this. 

Instead the enriched considerations, we are going to utilize the following simple observation.

\begin{lemma}
    Let $\C, \D$ be marked $(2,2)$-categories, $x,x'\in \C$, $y,y'$ in $\D$. Then there is a natural diagram 
    % https://q.uiver.app/#q=WzAsMixbMCwwLCJcXEMoeCx4JylcXG90aW1lc197cyxsfVxcRCh5LHknKSJdLFsxLDAsIlxcQ1xcb3RpbWVzX3tzLGx9XFxEKCh4LHkpLCh4Jyx5JykpIl0sWzAsMSwiXFxjb21wXntvcH0iLDIseyJjdXJ2ZSI6M31dLFswLDEsIlxcY29tcCIsMCx7ImN1cnZlIjotM31dLFszLDIsIiIsMCx7InNob3J0ZW4iOnsic291cmNlIjoyMCwidGFyZ2V0IjoyMH19XV0=
\[\begin{tikzcd}
	{\C(x,x')\times\D(y,y')} & {\C\otimes_{s,l}\D\,((x,y),(x',y'))}
	\arrow[""{name=0, anchor=center, inner sep=0}, "{\comp^{op}}"', curve={height=18pt}, from=1-1, to=1-2]
	\arrow[""{name=1, anchor=center, inner sep=0}, "\comp", curve={height=-18pt}, from=1-1, to=1-2]
	\arrow[between={0.2}{0.8}, Rightarrow, from=1, to=0]
\end{tikzcd}\]
where the $1$-cells are induced by composition (resp. composition in the reverse direction) and the $2$-cell is induced by the Gray cells.
\end{lemma}
\begin{proof}
    The lemma follows from the fact that we have two lax functors $\C\times\D \to \C\otimes_{s,l}\D$ an this is their effect on hom-categories, but let us write it more explicitly. Define $\comp$, $\comp^{op}$ on objects by $\comp(f,g):=(f,1)\circ (1,g)$, $\comp^{op}(f,g):=(1,g)\circ (f,1)$ as in (\ref{dia: gray square}). Similarly on morphisms: for $\alpha\colon f\Rightarrow f'$ in $\C(x,x')$ and $\beta\colon g\Rightarrow g'$ in $\D(y,y')$, define $\comp(\alpha,\beta):=(\alpha,1)\circ (1,\beta)$, $\comp^{op}(\alpha,\beta):=(1,\beta)\circ (\alpha,1)$. The Gray cells $\Sigma_{fg}$ are then by definition maps $\comp(f,g)\to \comp(g,f)$. The equality $\comp^{op}(\alpha,\beta)\circ \Sigma_{fg}=\Sigma_{f'g'}\circ \comp(\alpha,\beta)$ follows from the coherences for the Gray cells. 
\end{proof}

\begin{remark}
    The map $\comp$ comes from a marked version of the universal cubical lax functor $\C\times \D\to \C\otimes_{s,l}\D$.
\end{remark}

In the following, let us write $[\A,\B]:=\Fun_{s,l}(\A,\B)$ for any pair of marked $(2,2)$-categories. This notation is natural as this is the internal hom in $\V=(\Cat_{(2,2)}^{\mathfrak{m}},\otimes_{s,l})$. Let us also write simply $\otimes$ instead of $\otimes_{s,l}$.

\begin{construction}
    Consider any functors of marked $(2,2)$-categories $F,F'\colon \A\to \B$, $G,G'\colon \A'\to \B'$. As $\V$ is enriched over itself, we obtain a map \begin{align*}\mathfrak{comp}\colon [\A,\B](F,F')\times [\A',\B'](G,G')\xrightarrow{\comp}[\A,\B]\otimes[\A',\B']((F,G),(F',G'))\\ \to [\A\otimes \A', \B\otimes \B'](F\otimes G, F'\otimes G').\end{align*}
\end{construction}

\begin{definition}
    Define a $(2,2)$-category $\Fun_{\V}^{\mathfrak{res}}(\Delta^1,\V)$ as follows: \begin{itemize}
        \item Objects are morphisms $F\colon \A_0\to \A_1$ in $\V$, i.e. functors of marked $(2,2)$-categories, i.e. morphisms $\Delta^1\to \V$.
        \item For $F\colon \A_0\to \A_1$, $G\colon\B_0\to \B_1$, the hom-category $\langle F,G\rangle$ is defined to be $$\coprod_{J_0\colon \A_{0}\to \B_0, J_1\colon \A_1\to \B_1} [\A_0,\B_1](G J_0,J_1F).$$ 
        \item For a third functor $H\colon \C_0\to \C_1$, the composition $\langle G,H\rangle\times \langle F,G\rangle \to \langle F,H\rangle$ is induced by the maps \begin{align*}
            [\A_0,\B_1](GJ_0,J_1F)\times [\B_0,\C_1](HK_0, K_1G)\xrightarrow{K_{1,*}\times J_0^*}[\A_0,\C_1](K_1GJ_0,K_1J_1F)\times [\A_0,\C_1](HK_0J_0, K_1GJ_0)\\
            \xrightarrow{\circ} [\A_0,\C_1](HK_0J_0,K_1J_1F).
        \end{align*} 
    \end{itemize}
    Here, $J_0^*$ is the composition $[\B_0,\C_1]\otimes 1\xrightarrow{\id\otimes J_0}[\B_0,\C_1]\otimes [\A_0,\B_0]\to [\A_0,\C_1]$ and similarly with $K_{1,*}$.
\end{definition}

\begin{remark} \label{remark: Gray structure}
    Let us explain the notation $\Fun_{\V}^{\mathfrak{res}}(\Delta^1,\V)$. There should exist a $\V$-enriched functor category $\Fun_{\V}(\Delta^1,\V)$ where objects are morphisms in $\V$, i.e. functors of marked $(2,2)$-categories $F\colon \A_0\to \A_1$, and the hom-object $[[F,G]]$ is defined as the following comma object: % https://q.uiver.app/#q=WzAsNCxbMCwxLCJbXFxCLFxcRF0iXSxbMSwwLCJbXFxBLFxcQ10iXSxbMSwxLCJbXFxBLFxcRF0iXSxbMCwwLCJbW0YsR11dIl0sWzEsMiwiR18qIl0sWzAsMiwiRl4qIiwyXSxbMywwXSxbMywxXSxbMSwwLCIiLDEseyJzaG9ydGVuIjp7InNvdXJjZSI6MzAsInRhcmdldCI6MzB9LCJsZXZlbCI6Mn1dXQ==
\[\begin{tikzcd}
	{[[F,G]]} & {[\A_0,\B_0]} \\
	{[\A_1,\B_1]} & {[\A_0,\B_1]}
	\arrow[from=1-1, to=1-2]
	\arrow[from=1-1, to=2-1]
	\arrow[between={0.3}{0.7}, Rightarrow, from=1-2, to=2-1]
	\arrow["{G_*}", from=1-2, to=2-2]
	\arrow["{F^*}"', from=2-1, to=2-2]
\end{tikzcd}\]
Explicitly, $[[F,G]]$ looks as follows:
\begin{itemize}
    \item Objects are triples $(J_0\colon \A_0\to \B_0, J_1\colon \A_1\to \B_1, \eta\colon GJ_0\Rightarrow J_1F)$ where $\eta$ is an $(s,l)$-natural transformation.
    \item A $1$-cell $(J_0,J_1,\eta)\to (J'_0,J'_1,\eta)$ is a triple $(\alpha_0\colon J_0\Rightarrow J'_0, \alpha_1\Rightarrow J_1\to J_1',t\colon G_*\alpha_0\Rrightarrow F^*\alpha_1)$ where $\alpha_0$, $\alpha_1$ are $(s,l)$-natural transformations and $t$ is a modification,
    \item A $2$-cell $(\alpha_0,\alpha_1, t)\Rightarrow(\alpha'_0,\alpha'_1,t')$ is a pair of modifications $s_0\colon \alpha_0\Rrightarrow\alpha_0'$, $s_1\colon \alpha_1\Rrightarrow\alpha_1'$ such that $F^*s_1\circ t=t'\circ G_* s_0.$
\end{itemize}
Then, we see that $\langle F, G\rangle$ arises as a wide $(2,2)$-subcategory of $[[F,G]]$ where the only $1$-cells we keep are ``icons'', i.e. triples $(\alpha_0,\alpha_1,t)$ with $\alpha_0$, $\alpha_1$ being identities. This forces all $2$-cells to be identities.
\end{remark}

\begin{proposition} \label{prop: it exists, after all}
    There is a monoidal structure $\circledast$ on $\Fun_{\V}^{\mathfrak{res}}(\Delta^1,\V)$ such that on objects, $F\circledast G=F\otimes_{s,l} G$ is a tensor product of morphisms in $\V$. This structure makes $\Fun_{\V}^{\mathfrak{res}}(\Delta^1,\V)$ into a monoidal $(2,2)$-category.
\end{proposition}
\begin{proof}
    We have specified how $\circledast$ looks on objects. To specify it on hom-categories, let $F\colon \A_0\to \A_1$, $F'\colon \A'_0\to \A'_1$, $G\colon \B_0\to \B_1$, $G'\colon \B'_0\to \B'_1$ be a quadruple of functors. We need to specify a map $\langle F,G\rangle\times\langle F',G'\rangle\to \langle F\otimes F',G\otimes G'\rangle$. To do so, pick any $J_i\colon \A_i\to \B_i$, $J_i'\colon \A'_i\to \B_i'$, $i=0,1$ and do the following composition:
    \begin{align*}
        [\A_0,\B_1](G J_0,J_1F)\times [\A'_0,\B'_1](G' J'_0,J'_1F')\xrightarrow{\mathfrak{comp}}[\A_0\otimes \A_0',\B_1\otimes \B_1'](GJ_0\otimes G'J'_0, J_1F\otimes J'_1F')\\ \cong [\A_0\otimes \A_0',\B_1\otimes \B_1']((G\otimes G')\circ (J_0\otimes J_0'), (J_1\otimes J_1')\circ (F\otimes F')).
    \end{align*}
    Here, we are using the interchange law in $\V$ to see that $(G\otimes G')\circ (J_0\otimes J_0')\cong GJ_0\otimes G'J'_0$ etc. The coherences for associativity and unitality of $\circledast$ follows from those of $\otimes$.
\end{proof}

In the definition below, we use the same conventions as in the introduction to Section \ref{sec: day conv}.
\begin{definition} \label{def: lax Day conv} Let $(\mM_0,m_0^\bullet, \gamma_\bullet^0)$, $(\mM_1,m_1^\bullet, \gamma_\bullet^1)$ be $E_1$-monoids in $\V=(\Cat^{\mathfrak{m}}_{(2,2)},\otimes_{s,l}$). Define a $(2,2)$-multicategory structure $\Hom^l_{\V}(\mM_0,\mM_1)^\otimes$ on $\Hom_{\V}(\mM_0,\mM_1)$ as follows:  \begin{enumerate}
    \item For any tuple of functors $X_1,\dots, X_n,Y\colon \mM_0\to \mM_1$, define the category $\Mul(X_1,\dots,X_n;Y)$ of \textit{$n$-ary multimaps} to be the category $$[\mM_0^{\otimes n},\mM_1](m_1^{ n}\circ \otimes_i X_i, Y\circ m_0^{ n}).$$ In other words, for $n\ge1$, a multimap $f\colon X_1,\dots, X_n\to Y$ is a cell as on the left below, and for $n=0$, a nullary map to $Y$ is a $2$-cell $m_1^0\to Ym_0^0$ (as on the right below). 
    % https://q.uiver.app/#q=WzAsOCxbMCwwLCJNXntcXG90aW1lcyBufV8xIl0sWzAsMSwiTV8xIl0sWzIsMSwiTV8yLCJdLFsyLDAsIk1ee1xcb3RpbWVzIG59XzIiXSxbNCwwLCJJIl0sWzMsMCwiSSJdLFs0LDEsIk1fMiJdLFszLDEsIk1fMSJdLFswLDEsIm1fMV57XFxvdGltZXMgbn0iXSxbMywyLCJtXzJee1xcb3RpbWVzIG59Il0sWzEsMiwiWSIsMl0sWzAsMywiWF8xXFxvdGltZXNcXGNkb3RzXFxvdGltZXMgWF9uIl0sWzUsNCwiIiwxLHsibGV2ZWwiOjIsInN0eWxlIjp7ImhlYWQiOnsibmFtZSI6Im5vbmUifX19XSxbNCw2LCJtXzJeMCJdLFs1LDcsIm1fMV4wIiwyXSxbNyw2LCJYIiwyXSxbMTEsMTAsImYiLDAseyJzaG9ydGVuIjp7InNvdXJjZSI6NDAsInRhcmdldCI6NDB9fV0sWzEyLDE1LCJ1IiwwLHsic2hvcnRlbiI6eyJzb3VyY2UiOjQwLCJ0YXJnZXQiOjQwfX1dXQ==
\[\begin{tikzcd}
	{\mM^{\otimes n}_0} && {\mM^{\otimes n}_1} & I & I \\
	{\mM_0} && {\mM_1,} & {\mM_0} & {\mM_1}
	\arrow[""{name=0, anchor=center, inner sep=0}, "{X_1\otimes\cdots\otimes X_n}", from=1-1, to=1-3]
	\arrow["{m_0^{\otimes n}}", from=1-1, to=2-1]
	\arrow["{m_1^{\otimes n}}", from=1-3, to=2-3]
	\arrow[""{name=1, anchor=center, inner sep=0}, equals, from=1-4, to=1-5]
	\arrow["{m_0^0}"', from=1-4, to=2-4]
	\arrow["{m_1^0}", from=1-5, to=2-5]
	\arrow[""{name=2, anchor=center, inner sep=0}, "Y"', from=2-1, to=2-3]
	\arrow[""{name=3, anchor=center, inner sep=0}, "Y"', from=2-4, to=2-5]
	\arrow["f", between={0.4}{0.6}, Rightarrow, from=0, to=2]
	\arrow["u", between={0.4}{0.6}, Rightarrow, from=1, to=3]
\end{tikzcd}\]

    \item Consider multimaps $f\colon Y_1,\dots,Y_n\to Z$ and $f_i\colon X_{i1},\dots,X_{ik_i}\to Y_i$, $i=1,\dots,n$. Define the composition $f(f_1,\dots,f_n)\colon X_{11},\dots,X_{nk_n}\to Z$ to be the pasting of $f$, $\circledast_i f_i$, and the coherences, as in the following diagram:

% https://q.uiver.app/#q=WzAsOCxbMSwxLCJNXntcXG90aW1lcyBufV8xIl0sWzEsMiwiTV8xIl0sWzMsMiwiTV8yIl0sWzMsMSwiTV57XFxvdGltZXMgbn1fMiJdLFsxLDAsIlxcYmlnb3RpbWVzX2kgTV8xXntcXG90aW1lcyBrX2l9Il0sWzMsMCwiXFxiaWdvdGltZXNfaSBNXzJee1xcb3RpbWVzIGtfaX0iXSxbMCwwLCJNXzFee1xcb3RpbWVzXFxzdW1faSBrX2l9Il0sWzQsMCwiTV8yXntcXG90aW1lc1xcc3VtX2kga19pfSJdLFswLDEsIm1fMV57bn0iXSxbMywyLCJtXzJeeyBufSJdLFsxLDIsIloiLDJdLFswLDMsIllfMVxcb3RpbWVzXFxjZG90c1xcb3RpbWVzIFlfbiJdLFs0LDAsIlxcb3RpbWVzX2kgbV8xXnsga19pfSJdLFs1LDMsIlxcb3RpbWVzX2kgbV8yXnsga19pfSJdLFs0LDUsIlxcb3RpbWVzX3tpfVxcb3RpbWVzX2pYX3tpan0iXSxbNiw0LCJcXGNvbmciXSxbNiwxLCJtXzFee1xcc3VtX2lrX2l9IiwyLHsiY3VydmUiOjN9XSxbNSw3LCJcXGNvbmciXSxbNywyLCJtXzJee1xcb3RpbWVzXFxzdW1faWtfaX0iLDAseyJjdXJ2ZSI6LTN9XSxbMTEsMTAsImYiLDAseyJzaG9ydGVuIjp7InNvdXJjZSI6NDAsInRhcmdldCI6NDB9fV0sWzE0LDExLCJcXG90aW1lc19pIGZfaSIsMCx7InNob3J0ZW4iOnsic291cmNlIjo0MCwidGFyZ2V0Ijo0MH19XSxbMTgsNSwiKFxcZ2FtbWFeMl97a18xLFxcZG90cyxrX259KV57LTF9IiwyLHsic2hvcnRlbiI6eyJzb3VyY2UiOjQwLCJ0YXJnZXQiOjQwfX1dLFsxMiwxNiwiXFxnYW1tYV4xX3trXzEsXFxkb3RzLGtfbn0iLDIseyJzaG9ydGVuIjp7InNvdXJjZSI6NDAsInRhcmdldCI6NDB9fV1d
\begin{equation} \label{dia: pasting multimaps} \begin{tikzcd}
	{M_1^{\otimes\sum_i k_i}} & {\bigotimes_i M_1^{\otimes k_i}} && {\bigotimes_i M_2^{\otimes k_i}} & {M_2^{\otimes\sum_i k_i}} \\
	& {M^{\otimes n}_1} && {M^{\otimes n}_2} \\
	& {M_1} && {M_2}
	\arrow["\cong", from=1-1, to=1-2]
	\arrow[""{name=0, anchor=center, inner sep=0}, "{m_1^{\sum_ik_i}}"', curve={height=18pt}, from=1-1, to=3-2]
	\arrow[""{name=1, anchor=center, inner sep=0}, "{\otimes_{i}\otimes_jX_{ij}}", from=1-2, to=1-4]
	\arrow[""{name=2, anchor=center, inner sep=0}, "{\otimes_i m_1^{ k_i}}", from=1-2, to=2-2]
	\arrow["\cong", from=1-4, to=1-5]
	\arrow["{\otimes_i m_2^{ k_i}}", from=1-4, to=2-4]
	\arrow[""{name=3, anchor=center, inner sep=0}, "{m_2^{\otimes\sum_ik_i}}", curve={height=-18pt}, from=1-5, to=3-4]
	\arrow[""{name=4, anchor=center, inner sep=0}, "{Y_1\otimes\cdots\otimes Y_n}", from=2-2, to=2-4]
	\arrow["{m_1^{n}}", from=2-2, to=3-2]
	\arrow["{m_2^{ n}}", from=2-4, to=3-4]
	\arrow[""{name=5, anchor=center, inner sep=0}, "Z"', from=3-2, to=3-4]
	\arrow["{\circledast_i f_i}", between={0.4}{0.6}, Rightarrow, from=1, to=4]
	\arrow["{\gamma^1_{k_1,\dots,k_n}}"', between={0.4}{0.6}, Rightarrow, from=2, to=0]
	\arrow["{(\gamma^2_{k_1,\dots,k_n})^{-1}}"', between={0.4}{0.6}, Rightarrow, from=3, to=1-4]
	\arrow["f", between={0.4}{0.6}, Rightarrow, from=4, to=5]
\end{tikzcd}\end{equation}
\end{enumerate}
\end{definition}

This is well-defined and associative up to isomorphism by Proposition \ref{prop: it exists, after all}.

Choosing $\mM_0$ to be a Lawvere $(2,2)$-theory $\T$ with a lax commutativity $\mu$ and $\mM_1$ to be $\Cat$ with its cartesian product, we denote $\Mod_l(\T,\Cat)^\otimes:=\Hom_{\V}^l(\T,\Cat)^\otimes$.

\begin{example}
    As always, we are going to analyze the structure above for $\T=\T_{cmon}^\flat=\T_{E_\infty}$. Let $X,$ $Y$, $Z$ be $\Cat$-valued models, i.e. symmetric monoidal categories, and denote the corresponding functors $\T_{E_\infty}\to \Cat$ by $\X$, $\Y$, $\mZ$, respectively. As $m\colon 2\to 1$, $u\colon 0\to 1$ form a basis of $\T_{E_\infty}$, we get from Lemma \ref{lemma: binary} that a lax binary map $f\colon \X,\Y\to \mZ$ is determined by the following data: 
    % https://q.uiver.app/#q=WzAsMTYsWzAsMCwiWF4yXFx0aW1lcyBZIl0sWzAsMSwiWFxcdGltZXMgWSJdLFsxLDAsIlpeMiJdLFsxLDEsIloiXSxbMywwLCJYXFx0aW1lcyBZXjIiXSxbMywxLCJYXFx0aW1lcyBZIl0sWzQsMCwiWl4yIl0sWzQsMSwiWiJdLFswLDIsIjFcXHRpbWVzIFkiXSxbMCwzLCJYXFx0aW1lcyBZIl0sWzEsMiwiMSJdLFsxLDMsIloiXSxbMywyLCJYXFx0aW1lcyAxIl0sWzMsMywiWFxcdGltZXMgWSJdLFs0LDIsIjEiXSxbNCwzLCJaIl0sWzAsMSwiXFxYKG0pXFx0aW1lczEiLDJdLFsyLDMsIlxcbVoobSkiXSxbMSwzLCJmX3sxLDF9IiwyXSxbMCwyLCJmX3syLDF9Il0sWzQsNSwiMVxcdGltZXMgXFxZKG0pIiwyXSxbNiw3LCJcXG1aKG0pIl0sWzQsNiwiZl97MSwyfSJdLFs1LDcsImZfezEsMX0iLDJdLFs4LDksIlxcWCh1KVxcdGltZXMgMSIsMl0sWzksMTEsImZfezEsMX0iLDJdLFs4LDEwXSxbMTAsMTEsIlxcbVoodSkiXSxbMTIsMTMsIjFcXHRpbWVzIFxcWShtKSIsMl0sWzEyLDE0XSxbMTMsMTUsImZfezEsMX0iLDJdLFsxNCwxNSwiXFxtWih1KSJdLFsxOSwxOCwiZl97bSwxfSIsMCx7InNob3J0ZW4iOnsic291cmNlIjoyMCwidGFyZ2V0IjoyMH19XSxbMjIsMjMsImZfezEsbX0iLDAseyJzaG9ydGVuIjp7InNvdXJjZSI6MjAsInRhcmdldCI6MjB9fV0sWzI2LDI1LCJmX3t1LDF9IiwwLHsic2hvcnRlbiI6eyJzb3VyY2UiOjIwLCJ0YXJnZXQiOjIwfX1dLFsyOSwzMCwiZl97MSx1fSIsMCx7InNob3J0ZW4iOnsic291cmNlIjoyMCwidGFyZ2V0IjoyMH19XV0=
\[\begin{tikzcd}
	{X^2\times Y} & {Z^2} && {X\times Y^2} & {Z^2} \\
	{X\times Y} & Z && {X\times Y} & Z \\
	{1\times Y} & 1 && {X\times 1} & 1 \\
	{X\times Y} & Z && {X\times Y} & Z
	\arrow[""{name=0, anchor=center, inner sep=0}, "{f_{2,1}}", from=1-1, to=1-2]
	\arrow["{\X(m)\times1}"', from=1-1, to=2-1]
	\arrow["{\mZ(m)}", from=1-2, to=2-2]
	\arrow[""{name=1, anchor=center, inner sep=0}, "{f_{1,2}}", from=1-4, to=1-5]
	\arrow["{1\times \Y(m)}"', from=1-4, to=2-4]
	\arrow["{\mZ(m)}", from=1-5, to=2-5]
	\arrow[""{name=2, anchor=center, inner sep=0}, "{f_{1,1}}"', from=2-1, to=2-2]
	\arrow[""{name=3, anchor=center, inner sep=0}, "{f_{1,1}}"', from=2-4, to=2-5]
	\arrow[""{name=4, anchor=center, inner sep=0}, from=3-1, to=3-2]
	\arrow["{\X(u)\times 1}"', from=3-1, to=4-1]
	\arrow["{\mZ(u)}", from=3-2, to=4-2]
	\arrow[""{name=5, anchor=center, inner sep=0}, from=3-4, to=3-5]
	\arrow["{1\times \Y(m)}"', from=3-4, to=4-4]
	\arrow["{\mZ(u)}", from=3-5, to=4-5]
	\arrow[""{name=6, anchor=center, inner sep=0}, "{f_{1,1}}"', from=4-1, to=4-2]
	\arrow[""{name=7, anchor=center, inner sep=0}, "{f_{1,1}}"', from=4-4, to=4-5]
	\arrow["{f_{m,1}}", between={0.2}{0.8}, Rightarrow, from=0, to=2]
	\arrow["{f_{1,m}}", between={0.2}{0.8}, Rightarrow, from=1, to=3]
	\arrow["{f_{u,1}}", between={0.2}{0.8}, Rightarrow, from=4, to=6]
	\arrow["{f_{1,u}}", between={0.2}{0.8}, Rightarrow, from=5, to=7]
\end{tikzcd}\]

Rewriting $\otimes_X:=\X(m)$, $u_X:=\X(u)(1)$ and similarly for $Y$ and $Z$ and writing simply $f:=f_{1,1}$, we get that for any $x,x'\in X$, $y,y'\in Y$, we get that the $2$-cells above evaluated on $x,x'$, $y,y'$ have the following form:
\begin{itemize}
    \item $f(x,y)\otimes_Zf(x',y)\to f(x\otimes_X x',y)$, $f(x,y)\otimes_Zf(x,y')\to f(x,y\otimes_Y y')$,
    \item $u_Z\to f(u_X,y)$, $u_Z\to f(x,u_Y)$.
\end{itemize}
Therefore, we get lax binary functors give us a notion of \textit{lax bilinear functors} between symmetric monoidal categories. 
\end{example}

Finally, we are going to show that the $2$-multicategory $\underline{\Mod}_l(\T,\Cat)$ is closed.

\begin{theorem} \label{thm: closed structure}
    Let $\T$ be Lawvere $(2,2)$-theory with a lax commutativity $\mu$. Then the $(2,2)$-multicategory ${\Mod}_l(\T,\Cat)^\otimes$ is (right) closed, with the internal Hom being $\underline{\Hom}_l(\X,\Y)$. In other words, we have the isomorphism (functorial in all variables) $$\Mul_2(\X,\Y;\mZ)\cong \Hom_l(\X,\underline{\Hom}_l(\Y,\mZ)).$$
\end{theorem}

The proof can be done by hand using Lemma \ref{lemma: binary} below and some other observations. We present a more formal proof using lax ends; we refer the reader to \cite{Hir22} or Appendix A for more details on these. For convenience, let us just recall the following fact: for any two $2$-functors $F,G\colon \C\to \D$ of $2$-categories, the category of lax natural transformations $F\Rightarrow G$ and modifications between them, denoted by $\Nat_l(F,G)$, can be computed as a lax end, i.e. $\Nat_l(F,G)\cong \int_c^{lax}\D(Fc, Gc)$.

 In particular, for any $1$-cell $\alpha$ in $\C$, we have the following lax square: 
% https://q.uiver.app/#q=WzAsNCxbMCwwLCJcXE5hdF9sKEYsRykiXSxbMSwwLCJcXEQoRmMsR2MpIl0sWzEsMSwiXFxEKEZjLEdjJykiXSxbMCwxLCJcXEQoRmMnLEdjJykiXSxbMCwxLCJcXGV2X2MiXSxbMSwyLCJHXFxhbHBoYV8qIl0sWzAsMywiXFxldl97Yyd9IiwyXSxbMywyLCJGXFxhbHBoYV4qIiwyXSxbNCw3LCJcXHBoaV9cXGFscGhhIiwwLHsic2hvcnRlbiI6eyJzb3VyY2UiOjIwLCJ0YXJnZXQiOjIwfX1dXQ==
\[\begin{tikzcd}
	{\Nat_l(F,G)} & {\D(Fc,Gc)} \\
	{\D(Fc',Gc')} & {\D(Fc,Gc')}
	\arrow[""{name=0, anchor=center, inner sep=0}, "{\ev_c}", from=1-1, to=1-2]
	\arrow["{\ev_{c'}}"', from=1-1, to=2-1]
	\arrow["{G\alpha_*}", from=1-2, to=2-2]
	\arrow[""{name=1, anchor=center, inner sep=0}, "{F\alpha^*}"', from=2-1, to=2-2]
	\arrow["{\phi_\alpha}", between={0.2}{0.8}, Rightarrow, from=0, to=1]
\end{tikzcd}\]
We are going to make us of (a version of) the $2$-cell $\phi_\alpha$ later.

\begin{proof}
 Using the fact that the category of lax natural transformations can be computed using lax ends \cite[Prop. 3.2.]{Hir22}, we can do the following computation (where $\int_a^{lax}$ stands for a lax end): \begin{align}\Nat_l(\times \circ (\X\otimes \Y),\mZ\mu)&\cong
    \int_{a,b}^{lax} \Fun(X^a\times Y^b,Z^{ab})\nonumber\\
    &\cong \int_{a}^{lax} \int_b^{lax}\Fun(X^a\times Y^b,Z^{ab})\label{eq: Fubini}\\
    &\cong \int_{a}^{lax} \int_b^{lax} \Fun(X^a,\Fun(Y^b,Z^{ab}))\nonumber\\
    &\cong \int_a^{lax}\Fun(X^a,\int^{lax}_b\Fun(Y^b,Z^{ab})) \label{eq: ends}\\
    &\cong \int_a^{lax} \Fun(X^a,\Nat_l(\Y,\mZ^a))\nonumber\\
    &\cong \Nat_l(\X,\Nat_l(\Y,\mZ^\bullet))\nonumber
\end{align}

The line (\ref{eq: ends}) is \cite[Prop. 3.3]{Hir22}. The line (\ref{eq: Fubini}) is the Fubini rule for lax ends; however, the usual Fubini rule cannot be applied: for that, we would need to have a lax wedge $\T^{op}\times \T^{op}\times \T\times \T \to \Cat$, but what we really have is a lax wedge $\T^{op}\times \T^{op}\times \T\otimes_{s,l} \T \to \Cat$. We treat this subtle point in the Appendix A where we prove a generalized version of the Fubini rule, see Proposition \ref{prop: Fubini 2}.

In the equivalence we have just proved, those natural trasformations that are strict on the product projections clearly correspond to each other on both sides. Thus, we are done.
\end{proof}

In the rest of this subsection, we explain more concretely how the data of a lax homomorphism $g\colon \X\to \underline{\Hom}_{l}(\Y,\mZ)$ correspond to a data of lax binary map $f\colon \X,\Y\to \mZ$.

\begin{lemma} \label{lemma: binary}
    A lax binary map map $f\colon \X,\Y\to \mZ$ is uniquely determined by the $1$-cell $f_{1,1}\colon X\times Y\to Z$ and the families of $2$-cells $f_{1,\alpha}$, $f_{\alpha,1}$ as below.
    % https://q.uiver.app/#q=WzAsOCxbMCwwLCJYXm5cXHRpbWVzIFkiXSxbMCwxLCJYXFx0aW1lcyBZIl0sWzIsMCwiWl57bn0iXSxbMiwxLCJaLCJdLFs0LDAsIlhcXHRpbWVzIFlebiJdLFs2LDAsIlpee259Il0sWzQsMSwiWFxcdGltZXMgWSJdLFs2LDEsIloiXSxbMCwxLCJcXFgoXFxhbHBoYSlcXHRpbWVzIDEiLDJdLFswLDIsImZfe24sMX0iXSxbMSwzLCJmX3sxLDF9IiwyXSxbMiwzLCJcXG1aKFxcYWxwaGEpIl0sWzQsNSwiZl97MSxufSJdLFs0LDYsIjFcXHRpbWVzIFxcWShcXGFscGhhKSIsMl0sWzYsNywiZl97MSwxfSIsMl0sWzUsNywiXFxtWihcXGFscGhhKSJdLFs5LDEwLCJmX3tcXGFscGhhLDF9IiwwLHsic2hvcnRlbiI6eyJzb3VyY2UiOjIwLCJ0YXJnZXQiOjIwfX1dLFsxMiwxNCwiZl97MSxcXGFscGhhfSIsMCx7InNob3J0ZW4iOnsic291cmNlIjoyMCwidGFyZ2V0IjoyMH19XV0=
\[\begin{tikzcd}
	{X^n\times Y} && {Z^{n}} && {X\times Y^n} && {Z^{n}} \\
	{X\times Y} && {Z,} && {X\times Y} && Z
	\arrow[""{name=0, anchor=center, inner sep=0}, "{f_{n,1}}", from=1-1, to=1-3]
	\arrow["{\X(\alpha)\times 1}"', from=1-1, to=2-1]
	\arrow["{\mZ(\alpha)}", from=1-3, to=2-3]
	\arrow[""{name=1, anchor=center, inner sep=0}, "{f_{1,n}}", from=1-5, to=1-7]
	\arrow["{1\times \Y(\alpha)}"', from=1-5, to=2-5]
	\arrow["{\mZ(\alpha)}", from=1-7, to=2-7]
	\arrow[""{name=2, anchor=center, inner sep=0}, "{f_{1,1}}"', from=2-1, to=2-3]
	\arrow[""{name=3, anchor=center, inner sep=0}, "{f_{1,1}}"', from=2-5, to=2-7]
	\arrow["{f_{\alpha,1}}", between={0.2}{0.8}, Rightarrow, from=0, to=2]
	\arrow["{f_{1,\alpha}}", between={0.2}{0.8}, Rightarrow, from=1, to=3]
\end{tikzcd}\]
\end{lemma}

\begin{proof}
    A lax binary map $f\colon \X,\Y\to \mZ$ is a particular kind of a natural transformation $\X\otimes_{s,l}\Y\to \mZ\mu$ where $\mu$ is the lax commutativity structure on $\T$. Therefore, $f$ a priori comprises data of $1$-cells $f_{m,k}$ and $2$-cells $f_{\alpha,k}$, $f_{m,\beta}$, satisfying some equations. We will show that the data $f_{1,1}$, $f_{\alpha,1}$ and $f_{1,\alpha}$ are enough to determine the rest.
    
    As the diagram below is required to commute (where $p_i$, $q_j$ are projections to $i$-th or $j$-the factor, respectively), we see that any $1$-cell $f_{m,n}$ is determined by $f_{1,1}$. 
    % https://q.uiver.app/#q=WzAsNCxbMCwwLCJYXm1cXHRpbWVzIFlebiJdLFswLDEsIlhcXHRpbWVzIFkiXSxbMiwwLCJaXnttbn0iXSxbMiwxLCJaIl0sWzAsMSwiXFxYKHBfaSlcXHRpbWVzIFxcWShxX2opIiwyXSxbMCwyLCJmX3ttLG59Il0sWzEsMywiZl97MSwxfSIsMl0sWzIsMywiXFxtWihwX2lcXG90aW1lcyBxX2opIl1d
\[\begin{tikzcd}
	{X^m\times Y^n} && {Z^{mn}} \\
	{X\times Y} && Z
	\arrow["{f_{m,n}}", from=1-1, to=1-3]
	\arrow["{\X(p_i)\times \Y(q_j)}"', from=1-1, to=2-1]
	\arrow["{\mZ(p_i\otimes q_j)}", from=1-3, to=2-3]
	\arrow["{f_{1,1}}"', from=2-1, to=2-3]
\end{tikzcd}\]
     The $2$-cells $f_{\alpha,n}$ for any $n$ are determined by $f_{\alpha,1}$ in a similar fashion. To see that, again choose any projection $p_i\colon k\to 1$; then we obtain the cube as below with all the side faces commutative. The case of the cells $f_{n,\alpha}$ is the same.
     % https://q.uiver.app/#q=WzAsMTQsWzAsMSwiWF5uXFx0aW1lcyBZXmsiXSxbMSwwLCJaXntua30iXSxbMiwxLCJaXmsiXSxbMSwyLCJYXFx0aW1lcyBZXmsiXSxbMSwzLCJYXFx0aW1lcyBZIl0sWzAsMiwiWF5uXFx0aW1lcyBZIl0sWzIsMiwiWiJdLFszLDEsIlheblxcdGltZXMgWV5rIl0sWzQsMCwiWl57bmt9Il0sWzUsMSwiWl5rIl0sWzMsMiwiWF5uXFx0aW1lcyBZIl0sWzQsMywiWFxcdGltZXMgWSJdLFs1LDIsIloiXSxbNCwxLCJaXm4iXSxbMCwxLCJmX3tuLGt9Il0sWzEsMiwiXFxtWl5rKFxcYWxwaGEpIl0sWzAsMywiXFxYKFxcYWxwaGEpXFx0aW1lcyAxIl0sWzMsMiwiZl97MSxrfSIsMl0sWzMsNCwiMVxcdGltZXMgXFxZKHBfaSkiLDFdLFswLDUsIjFcXHRpbWVzIFxcWShwX2kpIiwyXSxbNSw0LCJcXFgoXFxhbHBoYSlcXHRpbWVzIDEiLDJdLFsyLDYsIlxcbVoocF9pKSIsMV0sWzQsNiwiZl97MSwxfSIsMl0sWzcsOCwiZl97bixrfSJdLFs4LDksIlxcbVpeayhcXGFscGhhKSJdLFs3LDEwLCIxXFx0aW1lcyBcXFkocF9pKSIsMV0sWzEwLDExLCJcXFgoXFxhbHBoYSlcXHRpbWVzIDEiLDJdLFs5LDEyLCJcXG1aKHBfaSkiXSxbMTEsMTIsImZfezEsMX0iLDJdLFsxMCwxMywiZl97biwxfSJdLFsxMywxMiwiXFxtWihcXGFscGhhKSJdLFs4LDEzLCJcXG1aXm4ocF9pKSIsMV0sWzI5LDI4LCJmX3tcXGFscGhhLDF9IiwwLHsic2hvcnRlbiI6eyJzb3VyY2UiOjQwLCJ0YXJnZXQiOjQwfX1dLFsxNCwxNywiZl97XFxhbHBoYSxufSIsMCx7InNob3J0ZW4iOnsic291cmNlIjo0MCwidGFyZ2V0Ijo0MH19XSxbMjEsMjUsIj0iLDEseyJzaG9ydGVuIjp7InNvdXJjZSI6MjAsInRhcmdldCI6MjB9LCJzdHlsZSI6eyJib2R5Ijp7Im5hbWUiOiJub25lIn0sImhlYWQiOnsibmFtZSI6Im5vbmUifX19XV0=
\[\begin{tikzcd}
	& {Z^{nk}} &&& {Z^{nk}} \\
	{X^n\times Y^k} && {Z^k} & {X^n\times Y^k} & {Z^n} & {Z^k} \\
	{X^n\times Y} & {X\times Y^k} & Z & {X^n\times Y} && Z \\
	& {X\times Y} &&& {X\times Y}
	\arrow["{\mZ^k(\alpha)}", from=1-2, to=2-3]
	\arrow["{\mZ^n(p_i)}"{description}, from=1-5, to=2-5]
	\arrow["{\mZ^k(\alpha)}", from=1-5, to=2-6]
	\arrow[""{name=0, anchor=center, inner sep=0}, "{f_{n,k}}", from=2-1, to=1-2]
	\arrow["{1\times \Y(p_i)}"', from=2-1, to=3-1]
	\arrow["{\X(\alpha)\times 1}", from=2-1, to=3-2]
	\arrow[""{name=1, anchor=center, inner sep=0}, "{\mZ(p_i)}"{description}, from=2-3, to=3-3]
	\arrow["{f_{n,k}}", from=2-4, to=1-5]
	\arrow[""{name=2, anchor=center, inner sep=0}, "{1\times \Y(p_i)}"{description}, from=2-4, to=3-4]
	\arrow["{\mZ(\alpha)}", from=2-5, to=3-6]
	\arrow["{\mZ(p_i)}", from=2-6, to=3-6]
	\arrow["{\X(\alpha)\times 1}"', from=3-1, to=4-2]
	\arrow[""{name=3, anchor=center, inner sep=0}, "{f_{1,k}}"', from=3-2, to=2-3]
	\arrow["{1\times \Y(p_i)}"{description}, from=3-2, to=4-2]
	\arrow[""{name=4, anchor=center, inner sep=0}, "{f_{n,1}}", from=3-4, to=2-5]
	\arrow["{\X(\alpha)\times 1}"', from=3-4, to=4-5]
	\arrow["{f_{1,1}}"', from=4-2, to=3-3]
	\arrow[""{name=5, anchor=center, inner sep=0}, "{f_{1,1}}"', from=4-5, to=3-6]
	\arrow["{f_{\alpha,n}}", between={0.4}{0.6}, Rightarrow, from=0, to=3]
	\arrow["{=}"{description}, draw=none, from=1, to=2]
	\arrow["{f_{\alpha,1}}", between={0.4}{0.6}, Rightarrow, from=4, to=5]
\end{tikzcd}\]
\end{proof}

Now, we know that $g$ is determined by the data $g_1\colon X\to \Hom_l(\Y,\mZ)$ together with $2$-cells $g_\alpha$ as in the leftmost square below; we then define $f_{1,1},$ $f_{\alpha,1}$ using currying from the rectangle on the left.  

% https://q.uiver.app/#q=WzAsMTAsWzAsMCwiWF5uIl0sWzAsMSwiWCJdLFsxLDAsIlxcSG9tX2woXFxZLFxcbVpebikiXSxbMiwwLCJcXEZ1bihZLFpebikiXSxbMSwxLCJcXEhvbV9sKFxcWSxcXG1aKSJdLFsyLDEsIlxcRnVuKFksWiksIl0sWzQsMCwiWF5uXFx0aW1lcyBZIl0sWzUsMCwiWl5uIl0sWzUsMSwiWiJdLFs0LDEsIlhcXHRpbWVzIFkiXSxbMCwxLCJcXFgoXFxhbHBoYSkiLDJdLFswLDIsImdfMV5uIl0sWzIsMywiXFxldl8xIl0sWzEsNCwiZ18xIiwyXSxbMiw0LCJcXHdpZGV0aWxkZXtcXG1afShcXGFscGhhKV8qIl0sWzQsNSwiXFxldl8xIiwyXSxbMyw1LCJcXG1aKFxcYWxwaGEpXyoiXSxbNiw3LCJmX3tuLDF9Il0sWzcsOCwiXFxtWihcXGFscGhhKSJdLFs2LDksIlxcWChcXGFscGhhKVxcdGltZXMgMSIsMl0sWzksOCwiZl97MSwxfSIsMl0sWzExLDEzLCJnX3tcXGFscGhhfSIsMCx7InNob3J0ZW4iOnsic291cmNlIjoyMCwidGFyZ2V0IjoyMH19XSxbMTcsMjAsImZfe1xcYWxwaGEsMX0iLDAseyJzaG9ydGVuIjp7InNvdXJjZSI6MjAsInRhcmdldCI6MjB9fV0sWzE2LDE5LCIiLDAseyJzaG9ydGVuIjp7InNvdXJjZSI6NDAsInRhcmdldCI6NDB9LCJsZXZlbCI6MSwic3R5bGUiOnsidGFpbCI6eyJuYW1lIjoiYXJyb3doZWFkIn0sImJvZHkiOnsibmFtZSI6InNxdWlnZ2x5In19fV1d
\[\begin{tikzcd}
	{X^n} & {\Hom_l(\Y,\mZ^n)} & {\Fun(Y,Z^n)} && {X^n\times Y} & {Z^n} \\
	X & {\Hom_l(\Y,\mZ)} & {\Fun(Y,Z),} && {X\times Y} & Z
	\arrow[""{name=0, anchor=center, inner sep=0}, "{g_1^n}", from=1-1, to=1-2]
	\arrow["{\X(\alpha)}"', from=1-1, to=2-1]
	\arrow["{\ev_1}", from=1-2, to=1-3]
	\arrow["{\widetilde{\mZ}(\alpha)_*}", from=1-2, to=2-2]
	\arrow[""{name=1, anchor=center, inner sep=0}, "{\mZ(\alpha)_*}", from=1-3, to=2-3]
	\arrow[""{name=2, anchor=center, inner sep=0}, "{f_{n,1}}", from=1-5, to=1-6]
	\arrow[""{name=3, anchor=center, inner sep=0}, "{\X(\alpha)\times 1}"', from=1-5, to=2-5]
	\arrow["{\mZ(\alpha)}", from=1-6, to=2-6]
	\arrow[""{name=4, anchor=center, inner sep=0}, "{g_1}"', from=2-1, to=2-2]
	\arrow["{\ev_1}"', from=2-2, to=2-3]
	\arrow[""{name=5, anchor=center, inner sep=0}, "{f_{1,1}}"', from=2-5, to=2-6]
	\arrow["{g_{\alpha}}", between={0.2}{0.8}, Rightarrow, from=0, to=4]
	\arrow[between={0.4}{0.6}, squiggly, tail reversed, from=1, to=3]
	\arrow["{f_{\alpha,1}}", between={0.2}{0.8}, Rightarrow, from=2, to=5]
\end{tikzcd}\]

Then, observe that the map $g_1\colon X\to \Hom_l(\Y,\mZ)$ is uniquely determined by the composite with the evaluation at 1 $\ev_1 g_1\colon X\to \Fun(Y,Z)$ together with the composite $2$-cells $\phi_\alpha g_1$ as in the left diagram below.
% https://q.uiver.app/#q=WzAsOSxbMCwxLCJcXEhvbV9sKFxcWSxcXG1aKSJdLFsxLDEsIlxcSG9tX2woWV5uLFpebikiXSxbMSwyLCJcXEhvbV9sKFlebixaKSJdLFswLDIsIlxcSG9tX2woWSxaKSJdLFswLDAsIlgiXSxbMywxLCJYXFx0aW1lcyBZXm4iXSxbNCwxLCJaXm4iXSxbNCwyLCJaIl0sWzMsMiwiWFxcdGltZXMgWSJdLFswLDEsIlxcZXZfbiJdLFsxLDIsIlxcbVooXFxhbHBoYSlfKiJdLFswLDMsIlxcZXZfMSIsMl0sWzMsMiwiXFxZKFxcYWxwaGEpXioiLDJdLFs0LDAsImdfMSJdLFs4LDcsImZfezEsMX0iLDJdLFs1LDYsImZfezEsbn0iXSxbNiw3LCJcXG1aKFxcYWxwaGEpIl0sWzUsOCwiMVxcdGltZXMgXFxZKFxcYWxwaGEpIiwyXSxbOSwxMiwiXFxwaGlfXFxhbHBoYSIsMCx7InNob3J0ZW4iOnsic291cmNlIjoyMCwidGFyZ2V0IjoyMH19XSxbMTAsMTcsIiIsMCx7InNob3J0ZW4iOnsic291cmNlIjo0MCwidGFyZ2V0Ijo0MH0sImxldmVsIjoxLCJzdHlsZSI6eyJ0YWlsIjp7Im5hbWUiOiJhcnJvd2hlYWQifSwiYm9keSI6eyJuYW1lIjoic3F1aWdnbHkifX19XSxbMTUsMTQsImZfezEsXFxhbHBoYX0iLDIseyJzaG9ydGVuIjp7InNvdXJjZSI6MjAsInRhcmdldCI6MjB9fV1d
\[\begin{tikzcd}
	X \\
	{\Hom_l(\Y,\mZ)} & {\Hom_l(Y^n,Z^n)} && {X\times Y^n} & {Z^n} \\
	{\Hom_l(Y,Z)} & {\Hom_l(Y^n,Z)} && {X\times Y} & Z
	\arrow["{g_1}", from=1-1, to=2-1]
	\arrow[""{name=0, anchor=center, inner sep=0}, "{\ev_n}", from=2-1, to=2-2]
	\arrow["{\ev_1}"', from=2-1, to=3-1]
	\arrow[""{name=1, anchor=center, inner sep=0}, "{\mZ(\alpha)_*}", from=2-2, to=3-2]
	\arrow[""{name=2, anchor=center, inner sep=0}, "{f_{1,n}}", from=2-4, to=2-5]
	\arrow[""{name=3, anchor=center, inner sep=0}, "{1\times \Y(\alpha)}"', from=2-4, to=3-4]
	\arrow["{\mZ(\alpha)}", from=2-5, to=3-5]
	\arrow[""{name=4, anchor=center, inner sep=0}, "{\Y(\alpha)^*}"', from=3-1, to=3-2]
	\arrow[""{name=5, anchor=center, inner sep=0}, "{f_{1,1}}"', from=3-4, to=3-5]
	\arrow["{\phi_\alpha}", between={0.2}{0.8}, Rightarrow, from=0, to=4]
	\arrow[between={0.4}{0.6}, squiggly, tail reversed, from=1, to=3]
	\arrow["{f_{1,\alpha}}"', between={0.2}{0.8}, Rightarrow, from=2, to=5]
\end{tikzcd}\]

By currying, we obtain the lax square on the right and we define $f_{1,\alpha}$ to be the $2$-cell filling that square. Since lax homomorphisms $\X\to \underline{\Hom}_l(\Y,\mZ)$ are precisely those natural transformations $\X\to \Nat_l(\Y,\mZ^{\bullet})$ uniquely determined by $\ev_1 g_1$, $\ev_1 g_\alpha$ and $\phi_\alpha g_1$, the above is a bijective correspondence.

\subsection{Fox's theorem}

In this subsection, we are finally going to show that if a Lawvere $2$-theory $\T$ is equipped with with a $w$-commutativity, the induced endofunctor $\underline{\IntAlg}$ is equipped with a comonad structure.

First, let $\T$ be a Lawvere $(2,2)$-theory equipped with a lax or colax commutativity $\mu$ -- without loss of generality, let us assume lax. Consider the endofunctor $\underline{\IntAlg}\colon \Mod_{\lax}(\T,\Cat)\to \Mod_{\lax}(\T,\Cat)$. We are now going to prove certain generalization of Fox's theorem.

\begin{theorem} \label{thm: fox}
    Suppose $\T$ is equipped with a lax commutativity $\mu$. Then the endofunctor $\underline{\IntAlg}$ has a structure of a 2-comonad. Moreover, if $\T$ is lax idempotent, this 2-comonad is idempotent.
\end{theorem}

\begin{proof}
    %Denote again by $*$ the constant functor $\T\to\Cat$ with value the terminal category. 
    As the diagram below commutes, $*$ has a unique monoid structure $m\colon *,*\to *$, $u\colon ()\to *$ in the $(2,2)$-multicategory $\underline{\Mod}_{\lax}(\T,\Cat)$.
    % https://q.uiver.app/#q=WzAsNCxbMCwwLCJcXFRcXG90aW1lc197dyxzfVxcVCJdLFsyLDAsIlxcQ2F0XFxvdGltZXNfe3csc31cXENhdCJdLFswLDEsIlxcVCJdLFsyLDEsIlxcQ2F0Il0sWzAsMSwiKlxcb3RpbWVzKiJdLFsxLDMsIlxcdGltZXMiXSxbMiwzLCIqIiwyXSxbMCwyLCJcXG11IiwyXV0=
\begin{equation} \label{dia: point} \begin{tikzcd}
	{\T\otimes_{w,s}\T} && {\Cat\otimes_{w,s}\Cat} \\
	\T && \Cat
	\arrow["{*\otimes*}", from=1-1, to=1-3]
	\arrow["\mu"', from=1-1, to=2-1]
	\arrow["\times", from=1-3, to=2-3]
	\arrow["{*}"', from=2-1, to=2-3]
\end{tikzcd}\end{equation}
    
    From that, it formally follows\footnote{We have learned this fact, including the proof, from John Bourke.} that $\underline{\Hom}_w(*,-)$ is a $2$-comonad in the underlying $2$-category of unary maps, i.e. $\Mod_w(\T,\Cat)$. In particular, for any object $\Y$, the comultiplication component $\Delta_{\Y}$ and the counit component $\epsilon_{\Y}$ are determined by the following commutative squares:

   % https://q.uiver.app/#q=WzAsOSxbMCwwLCJcXE11bF8xKFxcWCxcXHVuZGVybGluZXtcXEhvbX1fdygqLFxcWSkpIl0sWzEsMiwiXFxNdWxfMyhYLCosKjtcXFkpLCJdLFswLDIsIlxcTXVsXzIoXFxYLCo7XFxZKSJdLFsxLDEsIlxcTXVsXzIoXFxYLCo7XFx1bmRlcmxpbmV7XFxIb219X3coKixcXFkpKSJdLFsxLDAsIlxcTXVsXzEoXFxYLFxcdW5kZXJsaW5le1xcSG9tfV93KCosXFx1bmRlcmxpbmV7XFxIb219X3coKixcXFkpKSkiXSxbMCwzLCJcXE11bF8xKFxcWCxcXHVuZGVybGluZXtcXEhvbX1fdygqLFxcWSkpIl0sWzEsMywiXFxNdWxfMShcXFgsXFxZKSJdLFswLDQsIlxcTXVsXzIoXFxYLCo7XFxZKSJdLFsxLDQsIlxcTXVsXzEoXFxYLFxcWSkuIl0sWzIsMSwibV4qIl0sWzAsMiwiXFxjb25nIiwyXSxbMywxLCJcXGNvbmciXSxbNCwzLCJcXGNvbmciXSxbMCw0LCJcXERlbHRhX3tcXFl9XioiXSxbNSw3LCJcXGNvbmciLDJdLFs1LDYsIlxcZXBzaWxvbl97XFxZfV4qIl0sWzYsOCwiIiwyLHsibGV2ZWwiOjIsInN0eWxlIjp7ImhlYWQiOnsibmFtZSI6Im5vbmUifX19XSxbNyw4LCJ1XioiLDJdXQ==
\[\begin{tikzcd}
	{\Mul_1(\X,\underline{\Hom}_w(*,\Y))} & {\Mul_1(\X,\underline{\Hom}_w(*,\underline{\Hom}_w(*,\Y)))} \\
	& {\Mul_2(\X,*;\underline{\Hom}_w(*,\Y))} \\
	{\Mul_2(\X,*;\Y)} & {\Mul_3(X,*,*;\Y),} \\
	{\Mul_1(\X,\underline{\Hom}_w(*,\Y))} & {\Mul_1(\X,\Y)} \\
	{\Mul_2(\X,*;\Y)} & {\Mul_1(\X,\Y).}
	\arrow["{\Delta_{\Y}^*}", from=1-1, to=1-2]
	\arrow["\cong"', from=1-1, to=3-1]
	\arrow["\cong", from=1-2, to=2-2]
	\arrow["\cong", from=2-2, to=3-2]
	\arrow["{m^*}", from=3-1, to=3-2]
	\arrow["{\epsilon_{\Y}^*}", from=4-1, to=4-2]
	\arrow["\cong"', from=4-1, to=5-1]
	\arrow[equals, from=4-2, to=5-2]
	\arrow["{u^*}"', from=5-1, to=5-2]
\end{tikzcd}\]
    
    As both $m,$ $u$ are strict, so are $\Delta_{\Y}$ and $\epsilon_{\Y}$. Therefore, to see that $\Delta$ is an isomorphism, it is enough to see that for any $\X$, the evaluation at $1$ $$\Delta_{\X,1}\colon \Hom_w(*,\X)\to \Hom_w(*,\underline{\Hom}_w(*,\X))\cong \Mul_2(*,*;\X)$$ is invertible. We claim that there is an isomorphism $\Mul_2(*,*;\X)\cong \Hom_w(\widetilde{*},\widetilde{X)}$. Indeed, using the commutative diagram (\ref{dia: point}), we have \begin{align*}
        \Mul_2(*,*;\X)&\cong \Mod_w(\T\otimes_{s,w}\T,\Cat)(\times\circ *\otimes_{s,l}*, \X\circ \mu)\\
        &\cong \Mod_w(\T\otimes_{s,w}\T,\Cat)(*\circ \mu, \X\circ \mu)\\
        &\cong \Mod_w(\T,\Mod_w(\T,\Cat))(\widetilde{*},\widetilde{\X}).
    \end{align*} 

    If $\T$ is lax idempotent, we obtain by Lemma \ref{lemma: idempotence} that $\Hom_w(*,\X)\to \Hom_w(\widetilde{*},\widetilde{\X})$ is an isomorphism. As one can easily identify this map with $\Delta_{\X,1}$, this finishes the proof.
\end{proof}

\begin{example}
    Consider the example $\T=\T_{E_\infty}$. Recall that for any $w\in W$, the $2$-category $\Mod_w(\T_{E_\infty},\Cat)$ is the $2$-category of symmetric monoidal categories, symmetric $w$-monoidal functors, and monoidal natural transformations. Since $\T$ is pseudocommutative and pseudoidempotent, we get that $\underline{\IntAlg}$ is an idempotent $2$-comonad on $\Mod_p(\T,\Cat)$. The (slight strenghtening of) result of Fox identifies the $2$-category of $\underline{\IntAlg}$-coalgebras and their pseudohomomorphism with $\Cat^\coprod$, the 2-category of categories with finite coproducts and coproduct-preserving functors.

    As pseudocommutativity implies lax commutativity, we obtain a $2$-comonad $\underline{\IntAlg}'$ on $\Mod_l(\T,\Cat)$ as well. We can again identify the coalgebras with categories with finite products, but now the homomorphisms of coalgebras correspond to \textit{all} functors (as any functor between cocartesian monoidal categories is automatically lax monoidal).
\end{example}

\begin{example}
    Consider the Lawvere $1$-theory $\T_*$ for pointed objects, i.e. it is given by freely adjoining a map $u\colon 0\to 1$ to $\F^{op}$. We have a basis $B=\{u\}$ and we see that $\T_*$ is commutative and idempotent. Models of $\T$ in $\Cat$ are pointed categories, and we can identify the $2$-category $\Mod_w(\T,\Cat)$ with the $w$-slice under a point which we denote by $\Cat_{*/,w}$: objects are functors $*\to \C$ and morphims a triangles filled with a $w$-cell. For a pointed category $*\xrightarrow{x}\C$, we can see that the category of internal algebras is precisely the slice category $\C_{x/}$, and one can identify $\underline{\IntAlg}$-coalgebras and their homomorphisms with the category $\Cat^\emptyset$ of categories with an initial object and functors preserving initial objects (up to isomorphism). %Moreover, considering homomorphisms of $\underline{\IntAlg}_l$-coalgebras are all the functors of the underlying categories, same as in the previous example.
\end{example}

We can generalize the Fox's theorem to the setting of $\infty$-categories with some extra assumptions -- probably, neither of those are necessary and we would like to get rid of them in future work.

Let $\T$ be a Lawvere $(\infty,2)$-theory with a strict commutativity $\mu$. By Definition \ref{def: variations of intalg}, we can consider endofunctors $$\underline{\IntAlg},\underline{\IntCoalg}\colon \Mod_{s}(\T,\Cat)\to \Mod_{s}(\T,\Cat)$$ sending a model $\X\colon \T\to\Cat$ to $\underline{\Hom}_l(*,\X)$, $\underline{\Hom}_{l^*}(*,\X)$, respectively. 

\begin{proposition} \label{prop: restricting fox}
    Suppose that the following holds:
    \begin{itemize}
        \item the $(\infty,2)$-operad $\Mod_{s}(\T,\Cat_{\infty})^\otimes\to \Delta^{op}$ is in fact a monoidal $(\infty,2)$-category,a cocartesian $2$-fibration, 
        \item there is an object $\P\in \Alg(\Mod_{s}(\T,\Cat_{\infty})^\otimes)$ and a lax homomorphism of models $*\to \P$ inducing an isomorphism $\underline{\Hom}_s(\P,\X)\simeq \underline{\Hom}_l(*,\X)$.
    \end{itemize}
    Then $\underline{\IntAlg}$ is a comonad. If $\T$ is a pseudoidempotent $(2,2)$-theory, this comonad is idempotent.
\end{proposition}

\begin{proof}
    We adapt the proof of the previous version. The fact that $\P$ is a monoid implies that $\underline{\Hom}_s(\P,-)$ is a comonad; we prove that more generally in a Lemma below. To show the second part, we can proceed the same as in  the previous proof.
\end{proof}

\begin{lemma}
    Let $(\V,\otimes, I,[-,-])$ be a closed monoidal category and $X$ a monoid in $\V$. Then the endofunctor $[X,-]\colon \V\to \V$ inherits a comonad structure.
\end{lemma}

\begin{proof}
    The functor $\V\times \V\to \V$ sending $(X,Y)$ to $\X\otimes X$ gives $\V$ a structure of a module over itself. As $\Fun(\V,\V)$ with the monoidal structure given by composition is the endomorphsim object in sense of Lurie \cite[4.7.1]{Lur17}, the corresponding functor $\V\to \Fun(\V,\V)$ sending $X$ to $X\otimes -$ is monoidal. As $\V$ is closed, this functor factors through the full (monoidal) subcategory $\Fun^{\ladj}(\V,\V)$ spanned by left adjoints. Applying $(-)^{op}$, we obtain a functor $\V^{op}\to \Fun^{\ladj}(\V,\V)^{op}\simeq \Fun^{\radj}(\V,\V)^{\text{rev}}$; the latter is monoidal equivalence by \cite[Thm. B]{HHLN23} (by taking $B=*$). This functor sends $X$ to $[X,-]$ and it is monoidal, so it sends comonoids in $\V^{op}$ -- i.e., monoids in $\V$ -- to comonoids in $\Fun^{\radj}(\V,\V)^{\text{rev}}$ -- i.e., comonads.
\end{proof}

\section{Open questions} \label{sec: future}

There is a number of loose ends in this paper, some of them already mentioned. Aside from those, there is a number of other interesting questions. Let us mention some of them.

%It would be interesting to explore some other examples of $\mF$-sketches, for example the sketch for fibrations \cite[Example 5.15]{ABK24}.

\subsection{Lurie's tensor product and $w$-commutativity}

    In case that $\T$ is a Lawvere $(\infty,1)$-theory and $w=s$, we can formulate the results on commutativity differently using Lurie's tensor product $\otimes_{\Lur}$ on the $\infty$-category $\Pr^L$ of presentable $\infty$-categories and left adjoint functors. Denoting by $\mathcal{S}$ the $\infty$-category of $\infty$-groupoids, we get that $\Mod_{\T} :=\Mod_s(\T,\mathcal{S})$ is presentable and more generally, for any presentable $\infty$-category $\C$, we get that $\Mod_s(\T,\C)$ is again presentable and equivalent to the Lurie's tensor product $\Mod_{\T}\otimes_{\Lur} \C$; these results are due to \cite[Appendix B]{GGN16}. Therefore, we can write $\Mod_s(\T,\Mod_{\T})\simeq \Mod_{\T}\otimes_{\Lur}\Mod_{\T}$. Under these identifications, commutativity on $\T$ corresponds to a coalgebra structure in $\Pr^R$ on $\Mod_{\T}$ with counit being the forgetful functor $\Mod_{\T}\to \mathcal{S}$ and comultiplication $\Mod_{\T}\to \Mod_{\T}\otimes_{\Lur}\Mod_{\T}$ being $\widetilde{(-)}$. This observation, formulated in a slightly different language, was made by Berman \cite{Ber20}. We expect similar results to hold for $(\infty,2)$-theories.

\subsection{2-dimensional commutativity for monads}
A notion of pseudocommutativity for $(2,2)$-monads was developed by Hyland and Power \cite[Def. 5]{HP02}, and it consists of the invertible $2$-cells $\gamma_{A,B}$ below (where $t,$ $t^*$ are the left and right strength), natural in $A, B$ and satisfying plethora of axioms. Dropping the invertibility, one would obtain the notion of (co)lax commutativity.

    % https://q.uiver.app/#q=WzAsNixbMCwwLCJUQVxcdGltZXMgVEIiXSxbMSwwLCJUKEFcXHRpbWVzIFRCKSJdLFsyLDAsIlReMihBXFx0aW1lcyBCKSJdLFsyLDEsIlQoQVxcdGltZXMgQikiXSxbMCwxLCJUKFRBXFx0aW1lcyBCKSJdLFsxLDEsIlReMihBXFx0aW1lcyBCKSJdLFswLDEsInReKiJdLFsxLDIsIlQodCkiXSxbMCw0LCJ0IiwyXSxbNCw1LCJUKHReKikiLDJdLFsyLDMsIlxcbXVfe0FcXHRpbWVzIEJ9Il0sWzUsMywiXFxtdV97QVxcdGltZXMgQn0iLDJdLFsxLDUsIlxcZ2FtbWFfe0EsQn0iLDAseyJzaG9ydGVuIjp7InNvdXJjZSI6MzAsInRhcmdldCI6MzB9LCJsZXZlbCI6Mn1dXQ==
\begin{equation}\label{dia: monads}
\begin{tikzcd}
	{TA\times TB} & {T(A\times TB)} & {T^2(A\times B)} \\
	{T(TA\times B)} & {T^2(A\times B)} & {T(A\times B)}
	\arrow["{t^*}", from=1-1, to=1-2]
	\arrow["t"', from=1-1, to=2-1]
	\arrow["{T(t)}", from=1-2, to=1-3]
	\arrow["{\gamma_{A,B}}", between={0.3}{0.7}, Rightarrow, from=1-2, to=2-2]
	\arrow["{\mu_{A\times B}}", from=1-3, to=2-3]
	\arrow["{T(t^*)}"', from=2-1, to=2-2]
	\arrow["{\mu_{A\times B}}"', from=2-2, to=2-3]
\end{tikzcd}\end{equation}

  As Lawvere $(2,2)$-theories correspond to finitary $(2,2)$-monads on $\Cat$ by \cite{Pow99}, the natural expectation is that the notions of $w$-commutativity for theories and monads coincide. Indeed, for pseudocommutativity, one could show it by checking that the $(2,2)$-multicategory structure on the $(2,2)$-category of models we developed in this paper coincides with the one defined by Hyland and Power in \textit{loc. cit.}
  
  Unlike our definition of $w$-commutativity, which works for arbitrary $(\infty,2)$-theories, the definition of $w$-commutativity for monads is not suitable to the generalization to $(\infty,2)$-monads. The question is whether we can arrive at more conceptual definition of $w$-commutativity for monads that would transfer to the higher setting. Instead of answering that, we outline one possible approach to make the definition of $w$-commutativity for monads looks more similar to the one for theories.

    Having a Lawvere $(2,2)$-theory $\T$ with a pseudocommutativity given by $2$-cells $\sigma_{\alpha\beta}$, consider the corresponding $(2,2)$-monad $TX:=\int^nX^n\times \T(n,1)$; in particular, $Tn=\T(n,1)$. Consider also the $2$-functor $SX:=\int^{m,n}X^{mn}\times Tm \times Tn$. We have two natural transformations $\otimes_l, \otimes_r\colon S\Rightarrow T$ with components $SX\to TX$ defined by maps $X^{mk}\times Tm\times Tk\to X^{mk}\times T{mk}$ which are identity in the first component and send $(\alpha, \beta)\in Tm\times Tk$ to $\alpha\otimes\beta$, $\beta\otimes \alpha$, respectively. The $2$-cells $\sigma_{\beta\alpha}$ produce a modification $m\colon \otimes_l\to \otimes_r.$ Then, define a natural transformation $d_{X,Y}\colon TX\times TY \to S(X\times Y)$ induced by $X^m\times Tm \times Y^k\times Tk\to (X\times Y)^{mk} \times Tm\times Tk$, i.e. sending the quadruple $[(x_i),\alpha,(y_j),\beta]$ to the triple $[(x_i,y_j),\alpha,\beta]$. % https://q.uiver.app/#q=WzAsNCxbMCwwLCJcXENhdFxcdGltZXNcXENhdCJdLFsxLDAsIlxcQ2F0XFx0aW1lc1xcQ2F0Il0sWzAsMSwiXFxDYXQiXSxbMSwxLCJcXENhdCJdLFswLDEsIihULFQpIl0sWzAsMiwiXFx0aW1lcyIsMl0sWzEsMywiXFx0aW1lcyJdLFsyLDMsIlMiLDJdLFs0LDcsImQiLDAseyJzaG9ydGVuIjp7InNvdXJjZSI6MjAsInRhcmdldCI6MjB9fV1d
\[\begin{tikzcd}
	\Cat\times\Cat & \Cat\times\Cat \\
	\Cat & \Cat
	\arrow[""{name=0, anchor=center, inner sep=0}, "{(T,T)}", from=1-1, to=1-2]
	\arrow["\times"', from=1-1, to=2-1]
	\arrow["\times", from=1-2, to=2-2]
	\arrow[""{name=1, anchor=center, inner sep=0}, "S"', from=2-1, to=2-2]
	\arrow["d", shorten <=4pt, shorten >=4pt, Rightarrow, from=0, to=1]
\end{tikzcd}\]

We would now hope that the cells $\gamma_{X,Y}$ given by the composition% https://q.uiver.app/#q=WzAsMyxbMCwwLCJUWFxcdGltZXMgVFkiXSxbMSwwLCJTKFhcXHRpbWVzIFkpIl0sWzIsMCwiVChYXFx0aW1lcyBZKSJdLFswLDEsImRfe1gsWX0iXSxbMSwyLCJcXG90aW1lc19sIiwwLHsiY3VydmUiOi0zfV0sWzEsMiwiXFxvdGltZXNfciIsMix7ImN1cnZlIjozfV0sWzQsNSwibSIsMCx7InNob3J0ZW4iOnsic291cmNlIjoyMCwidGFyZ2V0IjoyMH19XV0=
\[\begin{tikzcd}
	{TX\times TY} & {S(X\times Y)} & {T(X\times Y)}
	\arrow["{d_{X,Y}}", from=1-1, to=1-2]
	\arrow[""{name=0, anchor=center, inner sep=0}, "{\otimes_l}", curve={height=-18pt}, from=1-2, to=1-3]
	\arrow[""{name=1, anchor=center, inner sep=0}, "{\otimes_r}"', curve={height=18pt}, from=1-2, to=1-3]
	\arrow["m", between={0.2}{0.8}, Rightarrow, from=0, to=1]
\end{tikzcd}\]
define a pseudocommutativity on $T$ in the sense of Hyland and Power. It seems that something similar has been done in a recent paper by Manco \cite{Man25}. 

\subsection{More on Fox's theorem}

 In this paper, we have not touched the second part of Fox's theorem, which says that algebras for the comonad $\underline{\CMon}$ are cartesian monoidal categories. Note that it is a \textit{property} of a category to be cartesian monoidal, i.e., the $\underline{\CMon}$-coalgebras are algebras for a lax idempotent $2$-monad on $\Cat$ (the finite product completion $2$-monad). Recently, a similar result was proved by Kudzman-Blais for linearly distributive categories \cite{Kud25}. \vspace{0.5em}   
 
    In our general setting, denote by $Q$ the comonad $$\underline{\IntAlg}\colon \Mod_w(\T,\Cat)\to \Mod_w(\T,\Cat),$$ where $\T$ is a theory with a $w$-commutativity. We can consider the forgetful functor $U\colon Q\text{-}\mathsf{Coalg}\to \Mod_w(\T,\Cat)\to \Cat$; under some mild assumptions, this is monadic and it holds for the corresponding monad $P$ on $\Cat$ that $P\text{-}\mathsf{Alg}\cong Q\text{-}\mathsf{Coalg}$.\footnote{We have learned this from John Bourke.} For $\T_{E_\infty}$, $P$ is the monad for cartesian categories. The question then becomes: when is $P$ lax idempotent? (That generalizes the observation that being coalgebra for $Q$ is a property, as lax idempotent monads on $\Cat$ capture properties.)

\appendix
\section{Fubini's rule for lax ends}

The aim of this appendix is to fill in the gap in the proof of Theorem \ref{thm: closed structure} by inspecting lax ends more carefully. As the original reference for lax ends is in French \cite{Boz76}, we follow the note of Hirata \cite{Hir22}.

\begin{remark} \label{remark: higher lax ends}
    In the theory of ``strict'' ends, the important result is that we can write an end as a limit indexed over the twisted arrow ($2$-)category; for any $T\colon \A^{op}\times \A\to \C$, we get $\int_a T(a,a)=\lim Fp$ where $p\colon \Tw(\A)\to \A$ is a Grothendieck construction of the two-sided Hom-functor $\A^{op}\times \A\to \Cat$. This makes the theory of ends easier and suitable for generalization to $\infty$-categorical setting. Unfortunately, lax ends cannot be written as a lax limit over the twisted arrow category (at least not naively) and hence their theory is more delicate. In particular, the text below contains a lot of explicit diagram chasing and is not suitable for $\infty$-categorical generalization.

    Nevertheless, our hope is that this defect can be resolved using $\mF$-categories (or marked $(\infty,2)$-categories). In particular, after introducing an $\mF$-category structure on $\A^{op}\times \A$ such that tight $1$-cells are of the form $(f,1)$ and lifting that to an $\mF$-category structure on the twisted arrow $2$-category $\Tw(\A)$, it seems that a lax end $\int_a^{lax}T(a,a)$ can be computed as a ``partially lax limit'' of $Fp$, meaning a right adjoint to the constant functor $\C\to \Fun_{s,l}(\Tw(\A),\C).$ We hope to investigate this in future work.
\end{remark}

\begin{definition} \cite[Def. 2.1]{Hir22}
    For a 2-functor $T\colon \A^{op}\times \A\to \C$, a \textit{lax wedge} $x\to T$ consists of an object $c$ of $\C$, $1$-cells $\sigma_a\colon c\to T(a,a)$ for each object $a$ of $\A$, and for each $1$-cell $f\colon a\to b$ in $\A$, the following $2$-cell: 
% https://q.uiver.app/#q=WzAsNCxbMCwwLCJjIl0sWzEsMCwiVChhLGEpIl0sWzAsMSwiVChiLGIpIl0sWzEsMSwiVChhLGIpIl0sWzAsMSwiXFxzaWdtYV9hIl0sWzAsMiwiXFxzaWdtYV9iIiwyXSxbMSwzLCJUKDEsZikiXSxbMiwzLCJUKGYsMSkiLDJdLFs0LDcsIlxcc2lnbWFfZiIsMCx7InNob3J0ZW4iOnsic291cmNlIjo0MCwidGFyZ2V0Ijo0MH19XV0=
\[\begin{tikzcd}
	c & {T(a,a)} \\
	{T(b,b)} & {T(a,b)}
	\arrow[""{name=0, anchor=center, inner sep=0}, "{\sigma_a}", from=1-1, to=1-2]
	\arrow["{\sigma_b}"', from=1-1, to=2-1]
	\arrow["{T(1,f)}", from=1-2, to=2-2]
	\arrow[""{name=1, anchor=center, inner sep=0}, "{T(f,1)}"', from=2-1, to=2-2]
	\arrow["{\sigma_f}", between={0.4}{0.6}, Rightarrow, from=0, to=1]
\end{tikzcd}\]
Moreover, we require $\sigma_{\id}=\id$ and the following two kinds of equations: 
% https://q.uiver.app/#q=WzAsOCxbMSwwLCJjIl0sWzIsMSwiVChhLGEpIl0sWzAsMSwiVChhJyxhJykiXSxbMSwyLCJUKGEsYScpIl0sWzQsMCwiYyJdLFs1LDEsIlQoYSxhKSJdLFs0LDIsIlQoYSxhJykiXSxbMywxLCJUKGEnLGEnKSJdLFswLDIsIlxcc2lnbWFfe2EnfSIsMl0sWzAsMSwiXFxzaWdtYV97YX0iXSxbMSwzLCJUKDEsZikiLDAseyJjdXJ2ZSI6LTN9XSxbMiwzLCJUKGYsMSkiLDEseyJjdXJ2ZSI6LTN9XSxbMiwzLCJUKGcsMSkiLDIseyJjdXJ2ZSI6M31dLFsxLDcsIj0iLDEseyJzdHlsZSI6eyJib2R5Ijp7Im5hbWUiOiJub25lIn0sImhlYWQiOnsibmFtZSI6Im5vbmUifX19XSxbNCw1LCJcXHNpZ21hX2EiXSxbNCw3LCJcXHNpZ21hX3thJ30iLDJdLFs3LDYsIlQoZywxKSIsMix7ImN1cnZlIjozfV0sWzUsNiwiVCgxLGYpIiwwLHsiY3VydmUiOi0zfV0sWzUsNiwiVCgxLGcpIiwxLHsiY3VydmUiOjN9XSxbMTEsMTIsIlQoXFxhbHBoYSwxKSIsMix7InNob3J0ZW4iOnsic291cmNlIjoyMCwidGFyZ2V0IjoyMH19XSxbOSwxMSwiXFxzaWdtYV9mIiwyLHsic2hvcnRlbiI6eyJzb3VyY2UiOjMwLCJ0YXJnZXQiOjMwfX1dLFsxNywxOCwiVCgxLFxcYWxwaGEpIiwyLHsic2hvcnRlbiI6eyJzb3VyY2UiOjIwLCJ0YXJnZXQiOjIwfX1dLFsxNCwxNiwiXFxzaWdtYV9nIiwyLHsic2hvcnRlbiI6eyJzb3VyY2UiOjMwLCJ0YXJnZXQiOjQwfX1dXQ==
\begin{equation} \label{dia: lax wedge 1}\begin{tikzcd}
	& c &&& c \\
	{T(a',a')} && {T(a,a)} & {T(a',a')} && {T(a,a)} \\
	& {T(a,a')} &&& {T(a,a')}
	\arrow["{\sigma_{a'}}"', from=1-2, to=2-1]
	\arrow[""{name=0, anchor=center, inner sep=0}, "{\sigma_{a}}", from=1-2, to=2-3]
	\arrow["{\sigma_{a'}}"', from=1-5, to=2-4]
	\arrow[""{name=1, anchor=center, inner sep=0}, "{\sigma_a}", from=1-5, to=2-6]
	\arrow[""{name=2, anchor=center, inner sep=0}, "{T(f,1)}"{description}, curve={height=-18pt}, from=2-1, to=3-2]
	\arrow[""{name=3, anchor=center, inner sep=0}, "{T(g,1)}"', curve={height=18pt}, from=2-1, to=3-2]
	\arrow["{=}"{description}, draw=none, from=2-3, to=2-4]
	\arrow["{T(1,f)}", curve={height=-18pt}, from=2-3, to=3-2]
	\arrow[""{name=4, anchor=center, inner sep=0}, "{T(g,1)}"', curve={height=18pt}, from=2-4, to=3-5]
	\arrow[""{name=5, anchor=center, inner sep=0}, "{T(1,f)}", curve={height=-18pt}, from=2-6, to=3-5]
	\arrow[""{name=6, anchor=center, inner sep=0}, "{T(1,g)}"{description}, curve={height=18pt}, from=2-6, to=3-5]
	\arrow["{\sigma_f}"', between={0.3}{0.7}, Rightarrow, from=0, to=2]
	\arrow["{\sigma_g}"', between={0.3}{0.6}, Rightarrow, from=1, to=4]
	\arrow["{T(\alpha,1)}"', between={0.2}{0.8}, Rightarrow, from=2, to=3]
	\arrow["{T(1,\alpha)}"', between={0.2}{0.8}, Rightarrow, from=5, to=6]
\end{tikzcd}\end{equation}

% https://q.uiver.app/#q=WzAsMTQsWzAsMSwiYyJdLFsxLDAsIlQoYSxhKSJdLFsxLDEsIlQoYScsYScpIl0sWzEsMiwiVChhJycsYScnKSJdLFsyLDAsIlQoYSxhJykiXSxbMiwyLCJUKGEnLGEnJykiXSxbNCwxLCJjIl0sWzUsMCwiVChhLGEpIl0sWzUsMiwiVChhJycsYScnKSJdLFs2LDAsIlQoYSxhJykiXSxbNiwyLCJUKGEnLGEnJykiXSxbNiwxLCJUKGEsYScnKSJdLFsyLDEsIlQoYSxhJycpIl0sWzMsMSwiPSJdLFswLDEsIlxcc2lnbWFfe2F9Il0sWzAsMiwiXFxzaWdtYV97YSd9IiwxXSxbMCwzLCJcXHNpZ21hX3thJyd9IiwyXSxbMSw0LCJUKDEsZikiXSxbMiw0LCJUKGYsMSkiLDJdLFszLDUsIlQoZywxKSIsMl0sWzIsNSwiVCgxLGcpIl0sWzYsNywiXFxzaWdtYV97YX0iXSxbNiw4LCJcXHNpZ21hX3thJyd9IiwyXSxbNyw5LCJUKDEsZikiXSxbOCwxMCwiVChnLDEpIiwyXSxbNyw4LCJcXHNpZ21hX3tnZn0iLDIseyJzaG9ydGVuIjp7InNvdXJjZSI6NDAsInRhcmdldCI6NDB9LCJsZXZlbCI6Mn1dLFs3LDExLCJUKDEsZ2YpIiwyXSxbOSwxMSwiVCgxLGcpIl0sWzgsMTEsIlQoZ2YsMSkiXSxbMTAsMTEsIlQoZiwxKSIsMl0sWzQsMTIsIlQoMSxnKSJdLFs1LDEyLCJUKGYsMSkiLDJdLFsxNCwxOCwiXFxzaWdtYV9mIiwwLHsic2hvcnRlbiI6eyJzb3VyY2UiOjQwLCJ0YXJnZXQiOjQwfX1dLFsxNiwyMCwiXFxzaWdtYV97Z30iLDIseyJzaG9ydGVuIjp7InNvdXJjZSI6NDAsInRhcmdldCI6NDB9fV1d
\begin{equation} \label{dia: lax wedge 2}\begin{tikzcd}
	& {T(a,a)} & {T(a,a')} &&& {T(a,a)} & {T(a,a')} \\
	c & {T(a',a')} & {T(a,a'')} & {=} & c && {T(a,a'')} \\
	& {T(a'',a'')} & {T(a',a'')} &&& {T(a'',a'')} & {T(a',a'')}
	\arrow["{T(1,f)}", from=1-2, to=1-3]
	\arrow["{T(1,g)}", from=1-3, to=2-3]
	\arrow["{T(1,f)}", from=1-6, to=1-7]
	\arrow["{T(1,gf)}"', from=1-6, to=2-7]
	\arrow["{\sigma_{gf}}"', between={0.4}{0.6}, Rightarrow, from=1-6, to=3-6]
	\arrow["{T(1,g)}", from=1-7, to=2-7]
	\arrow[""{name=0, anchor=center, inner sep=0}, "{\sigma_{a}}", from=2-1, to=1-2]
	\arrow["{\sigma_{a'}}"{description}, from=2-1, to=2-2]
	\arrow[""{name=1, anchor=center, inner sep=0}, "{\sigma_{a''}}"', from=2-1, to=3-2]
	\arrow[""{name=2, anchor=center, inner sep=0}, "{T(f,1)}"', from=2-2, to=1-3]
	\arrow[""{name=3, anchor=center, inner sep=0}, "{T(1,g)}", from=2-2, to=3-3]
	\arrow["{\sigma_{a}}", from=2-5, to=1-6]
	\arrow["{\sigma_{a''}}"', from=2-5, to=3-6]
	\arrow["{T(g,1)}"', from=3-2, to=3-3]
	\arrow["{T(f,1)}"', from=3-3, to=2-3]
	\arrow["{T(gf,1)}", from=3-6, to=2-7]
	\arrow["{T(g,1)}"', from=3-6, to=3-7]
	\arrow["{T(f,1)}"', from=3-7, to=2-7]
	\arrow["{\sigma_f}", between={0.4}{0.6}, Rightarrow, from=0, to=2]
	\arrow["{\sigma_{g}}"', between={0.4}{0.6}, Rightarrow, from=1, to=3]
\end{tikzcd}\end{equation}
\end{definition} 
%These pieces of data then need to satisfy certain equations. We define the lax end $\int^{lax}_a T(a,a)$ to be the universal lax wedge $\int^{lax}_a T(a,a)\to T$.

\begin{lemma} \cite[Prop. 3.2]{Hir22}
    For $2$-functors $F,G\colon \A\to \C$, we get $$\Nat_l(F,G)\cong \int^{lax}_a\C(Fa,Ga).$$
\end{lemma}

There is also the following version of the Fubini rule, although we will need something more general.

\begin{lemma}
    (Fubini's rule I) \cite[Prop. 3.10]{Hir22} Let $T\colon \A^{op}\times \B^{op}\times \A\times \B\to \Cat$ be $2$-functor. The we have the following isomorphisms: $$\int^{lax}_a\int^{lax}_bT(a,b,a,b)\cong \int^{lax}_{a,b} T(a,b,a,b)\cong \int^{lax}_b\int^{lax}_aT(a,b,a,b).$$
\end{lemma}

The proof of this lemma is straightforward and proof of the version we need will be analogous. However, we will see that it is subtle to even formulate the generalized Fubini's rule correctly. To that end, we introduce the following $\mF$-category structure on $\A^{op}\times \A$: the tight $1$-cells are cells of the form $(f,1)$ for any $1$-cell $f$ in $\A$.

\begin{lemma}
    Taking a lax end gives us a $2$-functor $$\int_a^{lax}\colon \Fun_{s,l}(\A^{op}\times \A, \Cat)\to \Cat.$$
\end{lemma}

\begin{proof}
    The proof is straightforward; we will only show the functoriality in the semi-strict natural transformations and leave out the verification of functoriality in modifications (as we also do not need it in our application). If $T,S\colon \A^{op}\times\A\to \Cat$ are $2$-functors and $c\to T$ is a lax wedge, a semi-strict natural transformation $\eta\colon T\to S$ gives us the following diagram:
    % https://q.uiver.app/#q=WzAsNyxbMCwxLCJjIl0sWzEsMCwiVChhLGEpIl0sWzEsMiwiVChhJyxhJykiXSxbMiwxLCJUKGEsYScpIl0sWzIsMCwiUyhhLGEpIl0sWzMsMSwiUyhhLGEnKSJdLFsyLDIsIlMoYScsYScpIl0sWzAsMSwiXFxzaWdtYV9hIl0sWzEsNCwiXFxldGFfe2EsYX0iXSxbNCw1LCJTKDEsZikiXSxbMSwzLCJUKDEsZikiLDJdLFswLDIsIlxcc2lnbWFfYiIsMl0sWzIsMywiVChmLDEpIl0sWzYsNSwiUyhmLDEpIiwyXSxbMyw1LCJcXGV0YV97YSxhJ30iXSxbMiw2LCJcXGV0YV97YScsYSd9IiwyXSxbMSwyLCJcXHNpZ21hX2YiLDIseyJzaG9ydGVuIjp7InNvdXJjZSI6NDAsInRhcmdldCI6NDB9LCJsZXZlbCI6Mn1dLFsxMiwxMywiPSIsMSx7InNob3J0ZW4iOnsic291cmNlIjoyMCwidGFyZ2V0IjoyMH0sInN0eWxlIjp7ImJvZHkiOnsibmFtZSI6Im5vbmUifSwiaGVhZCI6eyJuYW1lIjoibm9uZSJ9fX1dLFs4LDE0LCJcXGV0YV97MSxmfSIsMCx7InNob3J0ZW4iOnsic291cmNlIjo0MCwidGFyZ2V0Ijo0MH19XV0=
\[\begin{tikzcd}
	& {T(a,a)} & {S(a,a)} \\
	c && {T(a,a')} & {S(a,a')} \\
	& {T(a',a')} & {S(a',a')}
	\arrow[""{name=0, anchor=center, inner sep=0}, "{\eta_{a,a}}", from=1-2, to=1-3]
	\arrow["{T(1,f)}"', from=1-2, to=2-3]
	\arrow["{\sigma_f}"', between={0.4}{0.6}, Rightarrow, from=1-2, to=3-2]
	\arrow["{S(1,f)}", from=1-3, to=2-4]
	\arrow["{\sigma_a}", from=2-1, to=1-2]
	\arrow["{\sigma_{a'}}"', from=2-1, to=3-2]
	\arrow[""{name=1, anchor=center, inner sep=0}, "{\eta_{a,a'}}", from=2-3, to=2-4]
	\arrow[""{name=2, anchor=center, inner sep=0}, "{T(f,1)}", from=3-2, to=2-3]
	\arrow["{\eta_{a',a'}}"', from=3-2, to=3-3]
	\arrow[""{name=3, anchor=center, inner sep=0}, "{S(f,1)}"', from=3-3, to=2-4]
	\arrow["{\eta_{1,f}}", between={0.4}{0.6}, Rightarrow, from=0, to=1]
	\arrow["{=}"{description}, draw=none, from=2, to=3]
\end{tikzcd}\]
We define the lax wedge $\eta_*\sigma\colon c\to S$ by putting $(\eta_*\sigma)_a:=\eta_{a,a}\sigma_a,$ and defining $(\eta_*\sigma)_f$ to be the composite $2$-cell from the diagram above. (Here we can see why the $\mF$-category structure on $\A^{op}\times \A$ was introduced: a fully lax transformation with some non-invertible $2$-cell $\eta_{f,1}$ would not produce a $2$-cell $S(1,f)\circ \eta_{a,a}\circ \sigma_a\to S(f,1)\circ \eta_{a',a'}\circ \sigma_b$.) This is indeed a lax wedge: equations of the form (\ref{dia: lax wedge 1}), (\ref{dia: lax wedge 2}) follow from the naturality of $\eta$ with respect to $2$-cells and compositions, respectively. By the universal property of lax ends, we get a map $\int_a^{lax}\eta_{a,a}\colon \int_a^{lax} T(a,a)\to \int_a^{lax} S(a,a)$.
\end{proof}

To formulate the Fubini rule, consider two $2$-categories $\A,\B$. We have a lax end $$\int_{a,b}^{lax}\colon \Fun_{s,l}((\A\otimes_l \B)^{op}\times (\A\otimes_l \B), \Cat)\to \Cat.$$ As the canonical map $\A\otimes \B\to \A\times \B$ is an epimorphism, we can view $$\Fun_{s,l}(\A^{op}\times \B^{op}\times(\A\otimes_l\B),\Cat)$$ as the full subcategory of the former. 

\begin{lemma} \label{lemma: iso of tensor products}
    Using the $\mF$-category structure on $\A^{op}\times \A$, $\B^{op}\times \B$ discussed above, the permutation of the middle two coordinates extends to an equivalence  $$j\colon (\A^{op}\times \A)\otimes_{s,l}(\B^{op}\times\B)\xrightarrow{\sim} \A^{op}\times\B^{op}\times (\A\otimes_l\B).$$
\end{lemma}

\begin{proof}
    We only need to check that the non-trivial Gray cells match. Let $f\colon a_2\to a_2'$, $g\colon b_2\to b_2'$ be any $1$-cells in $\A, \B$, respectively. Then for any other map $f'\colon a_1\to a_1'$, we have $(f',f)=(f',1)\circ (1,f)$ where the first map is tight, and similarly in $\B$, so all the non-trivial Gray cells on the left hand side are generated by the ones below:
    %https://q.uiver.app/#q=WzAsOCxbMCwwLCIoYV8xLGJfMSxhXzIsYl8yKSJdLFsxLDAsIihhXzEsYl8xLGFfMicsYl8yKSJdLFsxLDEsIihhXzEsYl8xLGFfMicsYl8yJykiXSxbMCwxLCIoYV8xLGJfMSxhXzIsYl8yJykiXSxbMywwLCIoYV8xLGFfMixiXzEsYl8yKSJdLFszLDEsIihhXzEsYV8yLGJfMSxiXzInKSJdLFs0LDAsIihhXzEsYV8yJyxiXzEsYl8yKSJdLFs0LDEsIihhXzEsYV8yJyxiXzEsYl8yJykiXSxbMCwxLCIoMSwxLGYsMSkiXSxbMSwyLCIoMSwxLDEsZykiXSxbMywyLCIoMSwxLGYsMSkiLDJdLFswLDMsIigxLDEsMSxnKSIsMl0sWzQsNSwiKDEsMSwxLGcpIiwyXSxbNiw3LCIoMSwxLDEsZykiXSxbNSw3LCIoMSxmLDEsMSkiLDJdLFs0LDYsIigxLGYsMSwxKSJdLFs4LDEwLCJcXFNpZ21hX3tmZ30iLDAseyJzaG9ydGVuIjp7InNvdXJjZSI6NDAsInRhcmdldCI6NDB9fV0sWzE1LDE0LCJcXFNpZ21hJ197Zmd9IiwwLHsic2hvcnRlbiI6eyJzb3VyY2UiOjQwLCJ0YXJnZXQiOjQwfX1dXQ==
% https://q.uiver.app/#q=WzAsNCxbMCwwLCIoYV8xLGFfMixiXzEsYl8yKSJdLFswLDEsIihhXzEsYV8yLGJfMSxiXzInKSJdLFsxLDAsIihhXzEsYV8yJyxiXzEsYl8yKSJdLFsxLDEsIihhXzEsYV8yJyxiXzEsYl8yJykiXSxbMCwxLCIoMSwxLDEsZykiLDJdLFsyLDMsIigxLDEsMSxnKSJdLFsxLDMsIigxLGYsMSwxKSIsMl0sWzAsMiwiKDEsZiwxLDEpIl0sWzcsNiwiXFxTaWdtYSdfe2ZnfSIsMCx7InNob3J0ZW4iOnsic291cmNlIjo0MCwidGFyZ2V0Ijo0MH19XV0=
\[\begin{tikzcd}
	{(a_1,a_2,b_1,b_2)} & {(a_1,a_2',b_1,b_2)} \\
	{(a_1,a_2,b_1,b_2')} & {(a_1,a_2',b_1,b_2')}
	\arrow[""{name=0, anchor=center, inner sep=0}, "{(1,f,1,1)}", from=1-1, to=1-2]
	\arrow["{(1,1,1,g)}"', from=1-1, to=2-1]
	\arrow["{(1,1,1,g)}", from=1-2, to=2-2]
	\arrow[""{name=1, anchor=center, inner sep=0}, "{(1,f,1,1)}"', from=2-1, to=2-2]
	\arrow["{\Sigma_{fg}}", between={0.4}{0.6}, Rightarrow, from=0, to=1]
\end{tikzcd}\]
After permuting coordinates, these precisely match the Gray cells on the right hand side.
\end{proof}

Consider $T\colon \A^{op}\times \B^{op}\times (\A\otimes_l\B)\to \Cat$. We define $\int_a^{lax}\int_b^{lax}T(a,b,a,b)$ to be the value of $T$ precomposed with $j$ in the iterated lax end functor below. $$\Fun_{s,l}(\A^{op}\times \A,\Fun_{s,l}(\B^{op}\times \B,\Cat))\xrightarrow{(\int_b^{lax})_*} \Fun_{s,l}(\A^{op}\times \A,\Cat)\xrightarrow{\int_a^{lax}}\Cat.$$

We define $\int_{a,b}^{lax}T(a,b,a,b)$ to be the lax end of the precomposition of $T$ with the canonical quotient functor $(\A\otimes_l\B)^{op}\times (\A\otimes_l\B)\to\A^{op}\times \B^{op}\times (\A\otimes_l\B)$.

\begin{proposition} \label{prop: Fubini 2}
    (Fubini's rule II) For a $2$-functor $T\colon \A^{op}\times \B^{op}\times (\A\otimes_l\B)\to \Cat$, we have an isomorphism $$\int_{a,b}^{lax}T(a,b,a,b)\cong\int_a^{lax}\int_b^{lax}T(a,b,a,b).$$ 
\end{proposition}

\begin{proof}
    We compare the universal property of both lax ends. A map $x\to \int_{a}^{lax}\int_b^{lax}T(a,b,a,b)$ corresponds to a lax wedge $\tau\colon x\to \int_b^{lax}T(-,b,-,b)$, i.e. 1-cells $\tau_a$ and $2$-cells $\tau_f$ filling the leftmost square in the diagram below. 
   % https://q.uiver.app/#q=WzAsNyxbMCwxLCJ4Il0sWzEsMCwiXFxpbnRfYl57bGF4fVQoYSxiLGEsYikiXSxbMSwyLCJcXGludF9iXntsYXh9VChhJyxiLGEnLGIpIl0sWzIsMSwiXFxpbnRfYl57bGF4fVQoYSxiLGEnLGIpIl0sWzIsMiwiVChhJyxiLGEnLGIpIl0sWzIsMCwiVChhLGIsYSxiKSJdLFszLDEsIlQoYSxiLGEnLGIpIl0sWzAsMSwiXFx0YXVfYSJdLFswLDIsIlxcdGF1X3thJ30iLDJdLFsxLDMsIlxcaW50X2Jee2xheH1UKDEsYixmLGIpIiwxXSxbMiwzLCJcXGludF9iXntsYXh9VChmLGIsMSxiKSIsMV0sWzEsNSwiXFxyaG9ee2EsYX1fYiJdLFszLDYsIlxccmhvX2Jee2EsYSd9Il0sWzIsNCwiXFxyaG9fYl57YScsYSd9IiwyXSxbNSw2LCJUKDEsMSxmLDEpIl0sWzQsNiwiVChmLDEsMSwxKSIsMl0sWzcsMTAsIlxcdGF1X2YiLDAseyJzaG9ydGVuIjp7InNvdXJjZSI6NDAsInRhcmdldCI6NDB9fV1d
\[\begin{tikzcd}
	& {\int_b^{lax}T(a,b,a,b)} & {T(a,b,a,b)} \\
	x && {\int_b^{lax}T(a,b,a',b)} & {T(a,b,a',b)} \\
	& {\int_b^{lax}T(a',b,a',b)} & {T(a',b,a',b)}
	\arrow["{\rho^{a,a}_b}", from=1-2, to=1-3]
	\arrow["{\int_b^{lax}T(1,b,f,b)}"{description}, from=1-2, to=2-3]
	\arrow["{T(1,1,f,1)}", from=1-3, to=2-4]
	\arrow[""{name=0, anchor=center, inner sep=0}, "{\tau_a}", from=2-1, to=1-2]
	\arrow["{\tau_{a'}}"', from=2-1, to=3-2]
	\arrow["{\rho_b^{a,a'}}", from=2-3, to=2-4]
	\arrow[""{name=1, anchor=center, inner sep=0}, "{\int_b^{lax}T(f,b,1,b)}"{description}, from=3-2, to=2-3]
	\arrow["{\rho_b^{a',a'}}"', from=3-2, to=3-3]
	\arrow["{T(f,1,1,1)}"', from=3-3, to=2-4]
	\arrow["{\tau_f}", between={0.4}{0.6}, Rightarrow, from=0, to=1]
\end{tikzcd}\]
For a fixed pair $a,a'$, consider the universal lax wedge (in variable $b$) $\rho^{a,a'}\colon \int_b^{lax}T(a,b,a',b)\to T(a,-,a',-)$. We now claim that the following data give us a two-variable lax wedge $\sigma \colon x\to T$:
\begin{itemize}
    \item $\sigma_{a,b}:=\rho^{a,a}_b\circ\tau_a,$
    \item $\sigma_{\id_a, g}:=\rho^a_f\circ \tau_a$,
    \item $\sigma_{f,\id_b}$ to be the composite $2$-cell above.
\end{itemize}
This would finish the proof as we obtain that the maps $x\to \int_{a}^{lax}\int_b^{lax}T(a,b,a,b)$ correspond to map $x\to \int_{a,b}^{lax} T(a,b,a,b)$ and this correspondence is clearly a bijection. 

To prove the claim, we need to show that the equations of the form (\ref{dia: lax wedge 1}), (\ref{dia: lax wedge 2}) holds. Notice that if we were dealing with a wedge $(\A\times \B)^{op}\times (\A\times B)\to \Cat$, we would have to additionally prove that $\sigma_{\id_{a'}, g}\circ \sigma_{f,\id_b}=\sigma_{f,\id_{b'}}\circ \sigma_{\id_a,g}$ to see that the data above indeed define a lax wedge. In our case, it does not make sense to ask these $2$-cells to be equal; they are related by a $2$-cell coming form a Gray cell in $\A\otimes_l \B$.

Most of these verifications proceed the same way as in case of the classical Fubini rule; we refer to \cite{Hir22} for details. It essentially follows from observing that by the universal property of the lax end, $\tau_f$ is a modification of wedges. The only extra verification for us is given by the Gray cells $\Sigma_{fg}$. For reader's convenience, and to make up for the lack of details in \textit{loc. cit.}, we write down this case explicitly with all the details.

Consider $1$-cells $f\colon a\to a'$ in $\A$, $g\colon b\to b'$ in $\B$. We want to show the equation (\ref{dia: lax wedge 1}) holds where the $2$-cell $\alpha$ in question is the Gray cell $\Sigma_{fg}\colon (\id_{a'},g)\circ (f,\id_b)\Rightarrow (f,\id_{b'})\circ (\id_a,g)$. As our functor $T$ factors through $\A^{op}\times \B^{op}\times (\A\otimes_l \B)$, we get that $T(\Sigma_{fg},1,1)=1$, so what we need to really show is that the pasting of $T(1,1,\Sigma_{fg})$, $\sigma_{f,\id_{b'}}$ and $\sigma_{\id_a,g}$ is the same as pasting $\sigma_{\id_{a'},g}$ with $\sigma_{f,\id_b}$. Below, we are going to write down all the diagrams getting us from the latter to the former expression. To make the diagrams more eligible, all the cells coming from the map $f$ are depicted in the blue colour and all the maps coming from $g$ are depicted in yellow. 

Let us explain the equalities between the diagrams below:
\begin{enumerate}
    \item The equality (\ref{dia: Fubini 1})=(\ref{dia: Fubini 2}) follows from the universal property of lax ends, in particular from the properties of the map $\int_b^{lax} T(f,b,1,b)$.
    \item The diagrams (\ref{dia: Fubini 2}), (\ref{dia: Fubini 3}) are the same, the latter just has some extra identity $2$-cells in the middle.
    \item The equality (\ref{dia: Fubini 3})=(\ref{dia: Fubini 4}) is a bit difficult to read off of the diagrams given their monstrous shape, so we depicted the irrelevant parts of both diagrams in a lighter colors.  Then, one can observe the top right parts of the respective diagrams are two halves of a (deformed) commutative cube. We are going to add the equality of the subdiagrams to the end.
\end{enumerate} 

% https://q.uiver.app/#q=WzAsMTAsWzAsMSwieCJdLFsxLDAsIlxcaW50X2Jee2xheH1UKGEsYixhLGIpIl0sWzEsMiwiXFxpbnRfYl57bGF4fVQoYScsYixhJyxiKSJdLFsyLDEsIlxcaW50X2Jee2xheH1UKGEsYixhJyxiKSJdLFsyLDIsIlQoYScsYixhJyxiKSJdLFsyLDAsIlQoYSxiLGEsYikiXSxbMywxLCJUKGEsYixhJyxiKSJdLFszLDMsIlQoYScsYixhJyxiJykiXSxbMiwzLCJUKGEnLGInLGEnLGInKSJdLFs0LDIsIlQoYSxiLGEnLGInKSJdLFswLDEsIlxcdGF1X2EiXSxbMCwyLCJcXHRhdV97YSd9IiwyXSxbMSwzLCJcXGludF9iXntsYXh9VCgxLGIsZixiKSIsMSx7ImNvbG91ciI6WzIzMCw5OSw1MF19LFsyMzAsOTksNTAsMV1dLFsyLDMsIlxcaW50X2Jee2xheH1UKGYsYiwxLGIpIiwxLHsiY29sb3VyIjpbMjMwLDk5LDUwXX0sWzIzMCw5OSw1MCwxXV0sWzEsNSwiXFxyaG9ee2EsYX1fYiJdLFszLDYsIlxccmhvX2Jee2EsYSd9Il0sWzIsNCwiXFxyaG9fYl57YScsYSd9IiwyXSxbNSw2LCJUKDEsMSxmLDEpIiwwLHsiY29sb3VyIjpbMjMwLDk5LDUwXX0sWzIzMCw5OSw1MCwxXV0sWzQsNiwiVChmLDEsMSwxKSIsMix7ImNvbG91ciI6WzIzMCw5OSw1MF19LFsyMzAsOTksNTAsMV1dLFs0LDcsIlQoMSwxLDEsZykiLDAseyJjb2xvdXIiOlsyOCwxMDAsNTBdfSxbMjgsMTAwLDUwLDFdXSxbMiw4LCJcXHJob157YScsYSd9X3tiJ30iLDJdLFs4LDcsIlQoMSxnLDEsMSkiLDIseyJjb2xvdXIiOlsyOCwxMDAsNTBdfSxbMjgsMTAwLDUwLDFdXSxbNCw4LCJcXHJob19nXnthJyxhJ30iLDIseyJzaG9ydGVuIjp7InNvdXJjZSI6MzAsInRhcmdldCI6MzB9LCJsZXZlbCI6MiwiY29sb3VyIjpbMjgsMTAwLDUwXX0sWzI4LDEwMCw1MCwxXV0sWzYsOSwiVCgxLDEsMSxnKSIsMCx7ImNvbG91ciI6WzI4LDEwMCw1MF19LFsyOCwxMDAsNTAsMV1dLFs3LDksIlQoZiwxLDEsMSkiLDIseyJjb2xvdXIiOlsyMzAsOTksNTBdfSxbMjMwLDk5LDUwLDFdXSxbMTAsMTMsIlxcdGF1X2YiLDAseyJzaG9ydGVuIjp7InNvdXJjZSI6NDAsInRhcmdldCI6NDB9fV1d
\begin{equation}\label{dia: Fubini 1} \begin{tikzcd}
	& {\int_b^{lax}T(a,b,a,b)} & {T(a,b,a,b)} \\
	x && {\int_b^{lax}T(a,b,a',b)} & {T(a,b,a',b)} \\
	& {\int_b^{lax}T(a',b,a',b)} & {T(a',b,a',b)} && {T(a,b,a',b')} \\
	&& {T(a',b',a',b')} & {T(a',b,a',b')}
	\arrow["{\rho^{a,a}_b}", from=1-2, to=1-3]
	\arrow["{\int_b^{lax}T(1,b,f,b)}"{description}, color={rgb,255:red,1;green,43;blue,254}, from=1-2, to=2-3]
	\arrow["{T(1,1,f,1)}", color={rgb,255:red,1;green,43;blue,254}, from=1-3, to=2-4]
	\arrow[""{name=0, anchor=center, inner sep=0}, "{\tau_a}", from=2-1, to=1-2]
	\arrow["{\tau_{a'}}"', from=2-1, to=3-2]
	\arrow["{\rho_b^{a,a'}}", from=2-3, to=2-4]
	\arrow["{T(1,1,1,g)}", color={rgb,255:red,255;green,119;blue,0}, from=2-4, to=3-5]
	\arrow[""{name=1, anchor=center, inner sep=0}, "{\int_b^{lax}T(f,b,1,b)}"{description}, color={rgb,255:red,1;green,43;blue,254}, from=3-2, to=2-3]
	\arrow["{\rho_b^{a',a'}}"', from=3-2, to=3-3]
	\arrow["{\rho^{a',a'}_{b'}}"', from=3-2, to=4-3]
	\arrow["{T(f,1,1,1)}"', color={rgb,255:red,1;green,43;blue,254}, from=3-3, to=2-4]
	\arrow["{\rho_g^{a',a'}}"', color={rgb,255:red,255;green,119;blue,0}, between={0.3}{0.7}, Rightarrow, from=3-3, to=4-3]
	\arrow["{T(1,1,1,g)}", color={rgb,255:red,255;green,119;blue,0}, from=3-3, to=4-4]
	\arrow["{T(1,g,1,1)}"', color={rgb,255:red,255;green,119;blue,0}, from=4-3, to=4-4]
	\arrow["{T(f,1,1,1)}"', color={rgb,255:red,1;green,43;blue,254}, from=4-4, to=3-5]
	\arrow["{\tau_f}", color={rgb,255:red,1;green,43;blue,254}, between={0.4}{0.6}, Rightarrow, from=0, to=1]
\end{tikzcd}\end{equation}

% https://q.uiver.app/#q=WzAsMTAsWzAsMSwieCJdLFsxLDAsIlxcaW50X2Jee2xheH1UKGEsYixhLGIpIl0sWzEsMiwiXFxpbnRfYl57bGF4fVQoYScsYixhJyxiKSJdLFsyLDEsIlxcaW50X2Jee2xheH1UKGEsYixhJyxiKSJdLFsyLDAsIlQoYSxiLGEsYikiXSxbMywxLCJUKGEsYixhJyxiKSJdLFszLDMsIlQoYScsYixhJyxiJykiXSxbMiwzLCJUKGEnLGInLGEnLGInKSJdLFs0LDIsIlQoYSxiLGEnLGInKSJdLFszLDIsIlQoYSxiJyxhJyxiJykiXSxbMCwxLCJcXHRhdV9hIl0sWzAsMiwiXFx0YXVfe2EnfSIsMl0sWzEsMywiXFxpbnRfYl57bGF4fVQoMSxiLGYsYikiLDEseyJjb2xvdXIiOlsyMzAsOTksNTBdfSxbMjMwLDk5LDUwLDFdXSxbMiwzLCJcXGludF9iXntsYXh9VChmLGIsMSxiKSIsMSx7ImNvbG91ciI6WzIzMCw5OSw1MF19LFsyMzAsOTksNTAsMV1dLFsxLDQsIlxccmhvXnthLGF9X2IiXSxbMyw1LCJcXHJob19iXnthLGEnfSJdLFs0LDUsIlQoMSwxLGYsMSkiLDAseyJjb2xvdXIiOlsyMzAsOTksNTBdfSxbMjMwLDk5LDUwLDFdXSxbMiw3LCJcXHJob157YScsYSd9X3tiJ30iLDJdLFs3LDYsIlQoMSxnLDEsMSkiLDIseyJjb2xvdXIiOlsyOCwxMDAsNTBdfSxbMjgsMTAwLDUwLDFdXSxbNSw4LCJUKDEsMSwxLGcpIiwwLHsiY29sb3VyIjpbMjgsMTAwLDUwXX0sWzI4LDEwMCw1MCwxXV0sWzYsOCwiVChmLDEsMSwxKSIsMix7ImNvbG91ciI6WzIzMCw5OSw1MF19LFsyMzAsOTksNTAsMV1dLFszLDksIlxccmhvX3tiJ31ee2EsYSd9IiwyXSxbOSw4LCJUKDEsZywxLDEpIiwyLHsiY29sb3VyIjpbMjgsMTAwLDUwXX0sWzI4LDEwMCw1MCwxXV0sWzcsOSwiVChmLDEsMSwxKSIsMSx7ImNvbG91ciI6WzIzMCw5OSw1MF19LFsyMzAsOTksNTAsMV1dLFs1LDksIlxccmhvX2dee2EsYSd9IiwyLHsic2hvcnRlbiI6eyJzb3VyY2UiOjMwLCJ0YXJnZXQiOjMwfSwibGV2ZWwiOjIsImNvbG91ciI6WzI4LDEwMCw1MF19LFsyOCwxMDAsNTAsMV1dLFsxMCwxMywiXFx0YXVfZiIsMCx7InNob3J0ZW4iOnsic291cmNlIjo0MCwidGFyZ2V0Ijo0MH0sImNvbG91ciI6WzIzMCw5OSw1MF19LFsyMzAsOTksNTAsMV1dXQ==
\begin{equation}\label{dia: Fubini 2}\begin{tikzcd}
	& {\int_b^{lax}T(a,b,a,b)} & {T(a,b,a,b)} \\
	x && {\int_b^{lax}T(a,b,a',b)} & {T(a,b,a',b)} \\
	& {\int_b^{lax}T(a',b,a',b)} && {T(a,b',a',b')} & {T(a,b,a',b')} \\
	&& {T(a',b',a',b')} & {T(a',b,a',b')}
	\arrow["{\rho^{a,a}_b}", from=1-2, to=1-3]
	\arrow["{\int_b^{lax}T(1,b,f,b)}"{description}, color={rgb,255:red,1;green,43;blue,254}, from=1-2, to=2-3]
	\arrow["{T(1,1,f,1)}", color={rgb,255:red,1;green,43;blue,254}, from=1-3, to=2-4]
	\arrow[""{name=0, anchor=center, inner sep=0}, "{\tau_a}", from=2-1, to=1-2]
	\arrow["{\tau_{a'}}"', from=2-1, to=3-2]
	\arrow["{\rho_b^{a,a'}}", from=2-3, to=2-4]
	\arrow["{\rho_{b'}^{a,a'}}"', from=2-3, to=3-4]
	\arrow["{\rho_g^{a,a'}}"', color={rgb,255:red,255;green,119;blue,0}, between={0.3}{0.7}, Rightarrow, from=2-4, to=3-4]
	\arrow["{T(1,1,1,g)}", color={rgb,255:red,255;green,119;blue,0}, from=2-4, to=3-5]
	\arrow[""{name=1, anchor=center, inner sep=0}, "{\int_b^{lax}T(f,b,1,b)}"{description}, color={rgb,255:red,1;green,43;blue,254}, from=3-2, to=2-3]
	\arrow["{\rho^{a',a'}_{b'}}"', from=3-2, to=4-3]
	\arrow["{T(1,g,1,1)}"', color={rgb,255:red,255;green,119;blue,0}, from=3-4, to=3-5]
	\arrow["{T(f,1,1,1)}"{description}, color={rgb,255:red,1;green,43;blue,254}, from=4-3, to=3-4]
	\arrow["{T(1,g,1,1)}"', color={rgb,255:red,255;green,119;blue,0}, from=4-3, to=4-4]
	\arrow["{T(f,1,1,1)}"', color={rgb,255:red,1;green,43;blue,254}, from=4-4, to=3-5]
	\arrow["{\tau_f}", color={rgb,255:red,1;green,43;blue,254}, between={0.4}{0.6}, Rightarrow, from=0, to=1]
\end{tikzcd}\end{equation}

\begin{equation} \label{dia: Fubini 3}\begin{tikzcd}
	&& {T(a,b,a,b)} & {T(a,b,a',b)} \\
	& {\int_b^{lax}T(a,b,a,b)} & {\int_b^{lax}T(a,b,a',b)} \\
	\textcolor{rgb,255:red,99;green,99;blue,99}{x} & \textcolor{rgb,255:red,99;green,99;blue,99}{{\int_b^{lax}T(a,b,a',b)}} & {T(a,b',a,b')} \\
	\textcolor{rgb,255:red,102;green,102;blue,102}{{\int_b^{lax}T(a',b,a',b)}} && {T(a,b',a',b')} & {T(a,b,a',b')} \\
	& \textcolor{rgb,255:red,99;green,99;blue,99}{{T(a',b',a',b')}} & \textcolor{rgb,255:red,99;green,99;blue,99}{{T(a',b,a',b')}}
	\arrow["{T(1,1,f,1)}", color={rgb,255:red,1;green,43;blue,254}, from=1-3, to=1-4]
	\arrow["{T(1,1,1,g)}", color={rgb,255:red,255;green,119;blue,0}, from=1-4, to=4-4]
	\arrow["{\rho^{a,a}_b}", from=2-2, to=1-3]
	\arrow[draw={rgb,255:red,1;green,43;blue,254}, from=2-2, to=2-3]
	\arrow["{\int_b^{lax}T(1,b,f,b)}"{description}, color={rgb,255:red,98;green,124;blue,254}, from=2-2, to=3-2]
	\arrow["{\rho^{a,a}_{b'}}", from=2-2, to=3-3]
	\arrow[""{name=0, anchor=center, inner sep=0}, "{\rho_b^{a,a'}}", from=2-3, to=1-4]
	\arrow["{\rho_{b'}^{a,a'}}", shift left=5, curve={height=-30pt}, from=2-3, to=4-3]
	\arrow[""{name=1, anchor=center, inner sep=0}, "{\tau_a}", color={rgb,255:red,99;green,99;blue,99}, from=3-1, to=2-2]
	\arrow["{\tau_{a'}}"', color={rgb,255:red,99;green,99;blue,99}, from=3-1, to=4-1]
	\arrow["{\rho^{a,a'}_{b'}}"{description}, color={rgb,255:red,99;green,99;blue,99}, from=3-2, to=4-3]
	\arrow["{T(1,1,f,1)}"{description}, color={rgb,255:red,1;green,43;blue,254}, from=3-3, to=4-3]
	\arrow[""{name=2, anchor=center, inner sep=0}, "{\int_b^{lax}T(f,b,1,b)}"{description}, color={rgb,255:red,98;green,124;blue,254}, from=4-1, to=3-2]
	\arrow["{\rho^{a',a'}_{b'}}"', color={rgb,255:red,102;green,102;blue,102}, from=4-1, to=5-2]
	\arrow["{T(1,g,1,1)}"', color={rgb,255:red,255;green,119;blue,0}, from=4-3, to=4-4]
	\arrow["{T(f,1,1,1)}"{description}, color={rgb,255:red,98;green,124;blue,254}, from=5-2, to=4-3]
	\arrow["{T(1,g,1,1)}"', color={rgb,255:red,255;green,171;blue,97}, from=5-2, to=5-3]
	\arrow["{T(f,1,1,1)}"', color={rgb,255:red,98;green,124;blue,254}, from=5-3, to=4-4]
	\arrow["{\rho_g^{a,a'}}", color={rgb,255:red,255;green,119;blue,0}, between={0.3}{0.7}, Rightarrow, from=0, to=4-4]
	\arrow["{\tau_f}", color={rgb,255:red,98;green,124;blue,254}, between={0.4}{0.6}, Rightarrow, from=1, to=2]
\end{tikzcd}\end{equation}

\begin{equation} \label{dia: Fubini 4}\begin{tikzcd}
	&& {T(a,b,a,b)} & {T(a,b,a',b)} \\
	& {\int_b^{lax}T(a,b,a,b)} && {T(a,b,a,b')} \\
	\textcolor{rgb,255:red,92;green,92;blue,92}{x} & \textcolor{rgb,255:red,99;green,99;blue,99}{{\int_b^{lax}T(a,b,a',b)}} & {T(a,b',a,b')} \\
	\textcolor{rgb,255:red,102;green,102;blue,102}{{\int_b^{lax}T(a',b,a',b)}} && {T(a,b',a',b')} & {T(a,b,a',b')} \\
	& \textcolor{rgb,255:red,99;green,99;blue,99}{{T(a',b',a',b')}} & \textcolor{rgb,255:red,102;green,102;blue,102}{{T(a',b,a',b')}}
	\arrow["{T(1,1,f,1)}", color={rgb,255:red,1;green,43;blue,254}, from=1-3, to=1-4]
	\arrow["{T(1,1,1,g)}"{description}, color={rgb,255:red,255;green,119;blue,0}, from=1-3, to=2-4]
	\arrow["{\rho_g^{a,a}}"', color={rgb,255:red,255;green,119;blue,0}, between={0.4}{0.6}, Rightarrow, from=1-3, to=3-3]
	\arrow["{T(1,1,\Sigma_{fg})}"{description}, Rightarrow, from=1-4, to=2-4]
	\arrow["{T(1,1,1,g)}", shift left=5, color={rgb,255:red,255;green,119;blue,0}, curve={height=-30pt}, from=1-4, to=4-4]
	\arrow["{\rho^{a,a}_b}", from=2-2, to=1-3]
	\arrow["{\int_b^{lax}T(1,b,f,b)}"{description}, color={rgb,255:red,98;green,124;blue,254}, from=2-2, to=3-2]
	\arrow["{\rho^{a,a}_{b'}}", from=2-2, to=3-3]
	\arrow["{T(1,1,f,1)}"', color={rgb,255:red,1;green,43;blue,254}, from=2-4, to=4-4]
	\arrow[""{name=0, anchor=center, inner sep=0}, "{\tau_a}", color={rgb,255:red,133;green,133;blue,133}, from=3-1, to=2-2]
	\arrow["{\tau_{a'}}"', color={rgb,255:red,120;green,120;blue,120}, from=3-1, to=4-1]
	\arrow["{\rho^{a,a'}_{b'}}"{description}, color={rgb,255:red,102;green,102;blue,102}, from=3-2, to=4-3]
	\arrow["{T(1,g,1,1)}", color={rgb,255:red,255;green,119;blue,0}, from=3-3, to=2-4]
	\arrow["{T(1,1,f,1)}"{description}, color={rgb,255:red,1;green,43;blue,254}, from=3-3, to=4-3]
	\arrow[""{name=1, anchor=center, inner sep=0}, "{\int_b^{lax}T(f,b,1,b)}"{description}, color={rgb,255:red,113;green,136;blue,254}, from=4-1, to=3-2]
	\arrow["{\rho^{a',a'}_{b'}}"', color={rgb,255:red,110;green,110;blue,110}, from=4-1, to=5-2]
	\arrow["{T(1,g,1,1)}"', color={rgb,255:red,255;green,119;blue,0}, from=4-3, to=4-4]
	\arrow["{T(f,1,1,1)}"{description}, color={rgb,255:red,98;green,124;blue,254}, from=5-2, to=4-3]
	\arrow["{T(1,g,1,1)}"', color={rgb,255:red,255;green,171;blue,97}, from=5-2, to=5-3]
	\arrow["{T(f,1,1,1)}"', color={rgb,255:red,98;green,124;blue,254}, from=5-3, to=4-4]
	\arrow["{\tau_f}", color={rgb,255:red,128;green,149;blue,254}, between={0.4}{0.6}, Rightarrow, from=0, to=1]
\end{tikzcd}\end{equation}

Note that diagrams (\ref{dia: Fubini 3}) and (\ref{dia: Fubini 4}) are partially the same and these parts are depicted in the lighter colors above. The equality of the other parts of the diagrams follows from universal properties of lax ends, in particular of the map $\int_b^{lax} T(1,b,f,b)$. We rewrite this equality once again in a nicer shape below.

% https://q.uiver.app/#q=WzAsOCxbMSwwLCJcXGludF9iXntsYXh9VChhLGIsYSxiKSJdLFswLDEsIlQoYSxiLGEsYikiXSxbMiwxLCJUKGEsYicsYSxiJykiXSxbMSwyLCJUKGEsYixhLGInKSJdLFsyLDMsIlQoYSxiJyxhJyxiJykiXSxbMSw0LCJUKGEsYixhJyxiJykiXSxbMCwzLCJUKGEsYixhJyxiKSJdLFszLDIsIj0iXSxbMCwxLCJcXHJob19iXnthLGF9IiwyXSxbMCwyLCJcXHJob197Yid9XnthLGF9Il0sWzEsMywiVCgxLDEsMSxnKSIsMSx7ImNvbG91ciI6WzI4LDEwMCw1MF19LFsyOCwxMDAsNTAsMV1dLFsyLDMsIlQoMSwxLDEsZykiLDEseyJjb2xvdXIiOlsyOCwxMDAsNTBdfSxbMjgsMTAwLDUwLDFdXSxbMiw0LCJUKDEsMSxmLDEpIiwwLHsiY29sb3VyIjpbMjM3LDk4LDUwXX0sWzIzNyw5OCw1MCwxXV0sWzQsNSwiVCgxLDEsMSxnKSIsMCx7ImNvbG91ciI6WzI4LDEwMCw1MF19LFsyOCwxMDAsNTAsMV1dLFszLDUsIlQoMSwxLGYsMSkiLDEseyJjb2xvdXIiOlsyMzcsOTgsNTBdfSxbMjM3LDk4LDUwLDFdXSxbNiw1LCJUKDEsMSwxLGcpIiwyLHsiY29sb3VyIjpbMjgsMTAwLDUwXX0sWzI4LDEwMCw1MCwxXV0sWzEsNiwiVCgxLDEsZiwxKSIsMix7ImNvbG91ciI6WzIzNyw5OCw1MF19LFsyMzcsOTgsNTAsMV1dLFs5LDEwLCJcXHJob19nXnthLGF9IiwyLHsic2hvcnRlbiI6eyJzb3VyY2UiOjMwLCJ0YXJnZXQiOjMwfSwiY29sb3VyIjpbMjgsMTAwLDUwXX0sWzI4LDEwMCw1MCwxXV0sWzEyLDE0LCJUKDEsMSxcXFNpZ21hX3tmZ30pIiwyLHsic2hvcnRlbiI6eyJzb3VyY2UiOjQwLCJ0YXJnZXQiOjQwfX1dXQ==
\[\begin{tikzcd}
	& {\int_b^{lax}T(a,b,a,b)} \\
	{T(a,b,a,b)} && {T(a,b',a,b')} \\
	& {T(a,b,a,b')} && {=} \\
	{T(a,b,a',b)} && {T(a,b',a',b')} \\
	& {T(a,b,a',b')}
	\arrow["{\rho_b^{a,a}}"', from=1-2, to=2-1]
	\arrow[""{name=0, anchor=center, inner sep=0}, "{\rho_{b'}^{a,a}}", from=1-2, to=2-3]
	\arrow[""{name=1, anchor=center, inner sep=0}, "{T(1,1,1,g)}"{description}, color={rgb,255:red,255;green,119;blue,0}, from=2-1, to=3-2]
	\arrow["{T(1,1,f,1)}"', color={rgb,255:red,3;green,15;blue,252}, from=2-1, to=4-1]
	\arrow["{T(1,1,1,g)}"{description}, color={rgb,255:red,255;green,119;blue,0}, from=2-3, to=3-2]
	\arrow[""{name=2, anchor=center, inner sep=0}, "{T(1,1,f,1)}", color={rgb,255:red,3;green,15;blue,252}, from=2-3, to=4-3]
	\arrow[""{name=3, anchor=center, inner sep=0}, "{T(1,1,f,1)}"{description}, color={rgb,255:red,3;green,15;blue,252}, from=3-2, to=5-2]
	\arrow["{T(1,1,1,g)}"', color={rgb,255:red,255;green,119;blue,0}, from=4-1, to=5-2]
	\arrow["{T(1,1,1,g)}", color={rgb,255:red,255;green,119;blue,0}, from=4-3, to=5-2]
	\arrow["{\rho_g^{a,a}}"', color={rgb,255:red,255;green,119;blue,0}, between={0.3}{0.7}, Rightarrow, from=0, to=1]
	\arrow["{T(1,1,\Sigma_{fg})}"', between={0.4}{0.6}, Rightarrow, from=2, to=3]
\end{tikzcd}\]

% https://q.uiver.app/#q=WzAsOCxbMiwwLCJcXGludF9iXntsYXh9VChhLGIsYSxiKSJdLFsxLDEsIlQoYSxiLGEsYikiXSxbMywxLCJUKGEsYicsYSxiJykiXSxbMywzLCJUKGEsYicsYScsYicpIl0sWzIsNCwiVChhLGIsYScsYicpIl0sWzEsMywiVChhLGIsYScsYikiXSxbMiwyLCJcXGludF9iXntsYXh9VChhLGIsYScsYikiXSxbMCwyLCI9Il0sWzAsMSwiXFxyaG9fYl57YSxhfSIsMl0sWzAsMiwiXFxyaG9fe2InfV57YSxhfSJdLFsyLDMsIlQoMSwxLGYsMSkiLDAseyJjb2xvdXIiOlsyMzcsOTgsNTBdfSxbMjM3LDk4LDUwLDFdXSxbMyw0LCJUKDEsMSwxLGcpIiwwLHsiY29sb3VyIjpbMjgsMTAwLDUwXX0sWzI4LDEwMCw1MCwxXV0sWzUsNCwiVCgxLDEsMSxnKSIsMix7ImNvbG91ciI6WzI4LDEwMCw1MF19LFsyOCwxMDAsNTAsMV1dLFsxLDUsIlQoMSwxLGYsMSkiLDIseyJjb2xvdXIiOlsyMzcsOTgsNTBdfSxbMjM3LDk4LDUwLDFdXSxbNiw1LCJcXHJob19iXnthLGEnfSIsMl0sWzYsMywiXFxyaG9fe2InfV57YSxhJ30iXSxbMCw2LCJcXGludF9iXntsYXh9VCgxLGIsZixiKSIsMSx7ImNvbG91ciI6WzIzNyw5OCw1MF19LFsyMzcsOTgsNTAsMV1dLFsxNSwxMiwiXFxyaG9fZ157YSxhJ30iLDAseyJzaG9ydGVuIjp7InNvdXJjZSI6MzAsInRhcmdldCI6MzB9LCJjb2xvdXIiOlsyOCwxMDAsNTBdfSxbMjgsMTAwLDUwLDFdXV0=
\[\begin{tikzcd}
	&& {\int_b^{lax}T(a,b,a,b)} \\
	& {T(a,b,a,b)} && {T(a,b',a,b')} \\
	{=} && {\int_b^{lax}T(a,b,a',b)} \\
	& {T(a,b,a',b)} && {T(a,b',a',b')} \\
	&& {T(a,b,a',b')}
	\arrow["{\rho_b^{a,a}}"', from=1-3, to=2-2]
	\arrow["{\rho_{b'}^{a,a}}", from=1-3, to=2-4]
	\arrow["{\int_b^{lax}T(1,b,f,b)}"{description}, color={rgb,255:red,3;green,15;blue,252}, from=1-3, to=3-3]
	\arrow["{T(1,1,f,1)}"', color={rgb,255:red,3;green,15;blue,252}, from=2-2, to=4-2]
	\arrow["{T(1,1,f,1)}", color={rgb,255:red,3;green,15;blue,252}, from=2-4, to=4-4]
	\arrow["{\rho_b^{a,a'}}"', from=3-3, to=4-2]
	\arrow[""{name=0, anchor=center, inner sep=0}, "{\rho_{b'}^{a,a'}}", from=3-3, to=4-4]
	\arrow[""{name=1, anchor=center, inner sep=0}, "{T(1,1,1,g)}"', color={rgb,255:red,255;green,119;blue,0}, from=4-2, to=5-3]
	\arrow["{T(1,1,1,g)}", color={rgb,255:red,255;green,119;blue,0}, from=4-4, to=5-3]
	\arrow["{\rho_g^{a,a'}}", color={rgb,255:red,255;green,119;blue,0}, between={0.3}{0.7}, Rightarrow, from=0, to=1]
\end{tikzcd}\]

\end{proof}

\bibliographystyle{alpha}  % Generates citations like [CI19]
\bibliography{bibliography}  % File name of your .bib file

 \end{document}